\documentclass[11pt]{amsart}

 \usepackage{amsfonts,graphics,amsmath,amsthm,amsfonts,amscd,amssymb,amsmath,latexsym,multicol,
 mathrsfs}
\usepackage{epsfig,url}
\usepackage{flafter}
\usepackage{fancyhdr}
\usepackage{hyperref}
\hypersetup{colorlinks=true, linkcolor=black}

\makeatletter

\def\jobis#1{FF\fi
  \def\predicate{#1}%
  \edef\predicate{\expandafter\strip@prefix\meaning\predicate}%
  \edef\job{\jobname}%
  \ifx\job\predicate
}

\makeatother

\if\jobis{proposal}%
\else
\fi

 \usepackage[matrix, arrow]{xy}

\DeclareMathOperator{\Bs}{Bs}

\DeclareMathOperator{\Div}{Div}

\DeclareMathOperator{\Pic}{Pic}

\DeclareMathOperator{\Supp}{Supp}

\DeclareMathOperator{\vol}{vol}

 \numberwithin{equation}{subsection}
 \numberwithin{footnote}{subsection}

 \newtheorem{cor}[subsection]{Corollary}
 \newtheorem{lem}[subsection]{Lemma}
 \newtheorem{prop}[subsection]{Proposition}
 \newtheorem{thm}[subsection]{Theorem}
 \newtheorem{conj}[subsection]{Conjecture}

{
    \newtheoremstyle{upright}%
        {8pt plus2pt minus4pt}%
        {8pt plus2pt minus4pt}%
        {\upshape}%
        {}%
        {\bfseries\scshape}%
        {}%
        {1em}%
        {}%
\theoremstyle{upright}

 \newtheorem{constr}[subsection]{Construction}

}

 \newcommand{\N}{\mathbb N}
 \newcommand{\PP}{\mathbb P}
 
 \newcommand{\Q}{\mathbb Q}
 \newcommand{\R}{\mathbb R}
 \newcommand{\Z}{\mathbb Z}
 \newcommand{\bir}{\dashrightarrow}
 \newcommand{\rddown}[1]{\left\lfloor{#1}\right\rfloor} 

\title{\large A\MakeLowercase{nti-pluricanonical systems on} F\MakeLowercase{ano varieties}}
\author{\large C\MakeLowercase{aucher} B\MakeLowercase{irkar}}
\date{\today}
\begin{document}
\maketitle

\begin{abstract}
In this paper, we study the linear systems $|-mK_X|$ on Fano varieties $X$ with klt singularities. 
In a given dimension $d$, we prove $|-mK_X|$ is non-empty and contains an element 
with ``good singularities" for some natural number $m$ depending only on $d$; if in addition $X$ is $\epsilon$-lc 
for some $\epsilon>0$, 
then we show that we can choose $m$ depending only on $d$ and $\epsilon$ so that $|-mK_X|$ defines a birational map.
Further, we prove Shokurov's conjecture on boundedness of complements, 
and show that certain classes of Fano varieties form bounded families.   
\end{abstract}

\tableofcontents


\section{\bf Introduction}

We work over an algebraically closed field of characteristic zero. Given a smooth projective variety $W$, 
the minimal model program predicts that $W$ is birational to a projective variety $Y$ with canonical singularities 
such that either $K_Y$ is ample, or $Y$ admits a fibration whose general fibres $X$ are Calabi-Yau varieties 
or Fano varieties (here we consider Calabi-Yau varieties in a weak sense, that is, by requiring $K_X\equiv 0$
without the vanishing $h^i(X,\mathcal{O}_X)=0$ for $0<i<\dim X$ which is assumed in some other contexts). 
In other words, one may say that, birationally, every variety is in some sense constructed  
from varieties $X$ with good singularities such that either $K_X$ is ample or numerically trivial or anti-ample.
So it is quite natural to study such special varieties with the hope of obtaining some sort of classification theory. 
They are also very interesting in moduli theory, differential geometry, arithmetic geometry, and mathematical physics. 

When $X$ is one-dimensional the linear system $|K_X|$  determines its geometry to a large extent. However, in higher dimension, 
one needs to study $|mK_X|$ or $|-mK_X|$ for all $m\in \N$ (depending on the type of $X$) 
in order to investigate the geometry of $X$. If $K_X$ is ample, 
then there is $m$ depending only on the dimension such that $|mK_{X}|$ defines a birational 
embedding into some projective space, by Hacon-M$^{\rm c}$Kernan [\ref{HM-bir-bnd}] and Takayama [\ref{Takayama}]. 
If $K_X$ is numerically trivial, there is $m$ such that $|mK_X|$ is non-empty but it is not 
clear whether we can choose $m$ depending only on the dimension. 
When $K_X$ is anti-ample, that is when $X$ is Fano, in this paper we study boundedness and singularity 
properties of the linear systems $|-mK_X|$ 
in a quite general setting in conjunction with Shokurov's theory of complements.\\

{\textbf{\sffamily{Effective non-vanishing.}}}
Our first result is a consequence of boundedness of complements (\ref{t-bnd-compl-usual} below). 
We state it separately because it involves little technicalities.

\begin{thm}\label{t-eff-non-van-fano}
Let $d$ be a natural number. Then there is a natural 
number $m$ depending only on $d$ such that if $X$ is any Fano 
variety of dimension $d$ with klt singularities, then the linear system $|-mK_X|$ is non-empty, that is, 
$h^0(-mK_X)\neq 0$. Moreover, the linear system 
contains a divisor $M$ such that $(X,\frac{1}{m}M)$ has lc singularities.
\end{thm}

 Obviously the statement also holds if we replace Fano with the more general notion of weak Fano, that is, if 
$-K_X$ is nef and big. The theorem was proved by Shokurov in dimension two [\ref{shokurov-surf-comp}].\\

{\textbf{\sffamily{Effective birationality for $\epsilon$-lc Fano varieties.}}}
If we bound the singularities of $X$, we then have a much stronger statement than the 
non-vanishing of \ref{t-eff-non-van-fano}.

\begin{thm}\label{t-eff-bir-e-lc}
Let $d$ be a natural number and $\epsilon>0$ a real number. Then there is a natural 
number $m$ depending only on $d$ and $\epsilon$ such that if $X$ is any $\epsilon$-lc weak Fano 
variety of dimension $d$, then $|-mK_X|$ defines a birational map.
\end{thm}

Note that $m$ indeed depends on $d$ as well as $\epsilon$ because the theorem implies 
the volume $\vol(-K_X)$ is bounded from below 
by $\frac{1}{m^d}$. Without the $\epsilon$-lc assumption, $\vol(-K_X)$ can get arbitrarily small 
or large [\ref{HMX2}, Example 2.1.1]. In dimension $2$, the theorem is a consequence of Alexeev [\ref{Alexeev}], 
and in dimension $3$, special cases are proved by Jiang [\ref{Jiang}] using different methods. 
Paolo Cascini informed us that he and James M$^{\rm{c}}$Kernan have independently proved the 
theorem for canonical singularities, that is when $\epsilon=1$, using quite different methods.\\

{\textbf{\sffamily{Boundedness of certain classes of Fano varieties.}}} 
Fano varieties come in two flavours: \emph{non-exceptional} and \emph{exceptional}. A Fano variety $X$ is non-exceptional if there is 
$0\le P\sim_\Q -K_X$ such that $(X,P)$ is not klt. Otherwise we say $X$ is exceptional.
In the non-exceptional case we can create non-klt centres which sometimes can be used to do induction, i.e. 
lift sections and complements from such centres (eg, see \ref{p-bnd-compl-non-klt} below). 
We do not have that luxury in the exceptional case. Instead 
we show that there is a ``limited number" of them, that is:

\begin{thm}\label{t-BAB-exc-usual}
Let $d$ be a natural number. Then the set of exceptional weak Fano varieties of dimension $d$ 
forms a bounded family. 
\end{thm}

Exceptional pairs and exceptional generalised polarised pairs can be defined similarly. 
We will extend \ref{t-BAB-exc-usual} to such pairs (see \ref{t-BAB-exc} below) which is important for our proofs.
In a different direction we have: 

\begin{thm}\label{t-BAB-good-boundary}
Let $d$ be a natural number, and $\epsilon$ and $\delta$ be positive real numbers. 
Consider projective varieties $X$ equipped with a boundary $B$ such that: 
\begin{itemize}
\item $(X,B)$ is $\epsilon$-lc of dimension $d$, 

\item $B$ is big and $K_X+B\sim_\R 0$, and 

\item the coefficients of $B$ are more than or equal to $\delta$.\\
\end{itemize}
Then the set of such $X$ forms a bounded family.
\end{thm}

Note that $B$ is not necessarily $\R$-Cartier. Bigness means $B\sim_\R H+D$ where $H$ is an ample 
$\R$-divisor and $D$ is an effective $\R$-divisor.

Hacon and Xu proved the theorem assuming the coefficients of $B$ belong to a fixed DCC set of rational 
numbers [\ref{HX}, Theorem 1.3] relying on the special case when $-K_X$ is ample [\ref{HMX2}, Corollary 1.7].
The theorem can be viewed as a special case of the following conjecture due to Alexeev and the Borisov brothers.

\begin{conj}[BAB]\label{conj-BAB}
Let $d$ be a natural number and $\epsilon$ a positive real number. Then the set of 
$\epsilon$-lc Fano varieties $X$ of dimension $d$ forms a bounded family. 
\end{conj}

The conjecture is often stated in the log case for pairs $(X,B)$ but it is not 
hard to reduce it to the above version. 
Using the results and ideas developed in 
this paper the conjecture is proved in the sequel [\ref{B-BAB}].
The conjecture was previously known  in dimension two  [\ref{Alexeev}], for smooth 
$X$ [\ref{KMM-smooth-fano}],  for toric $X$  
[\ref{A-L-Borisov}],  for threefolds of Picard number one and terminal singularities 
[\ref{kawamata-term-3-folds}],  
for threefolds with canonical singularities [\ref{KMMT-can-3-folds}], in dimension
 three with $K_X$ having bounded Cartier index [\ref{A-Borisov}],  in any dimension with $K_X$ having bounded Cartier index  [\ref{HMX2}, Corollary 1.8], for spherical varieties [\ref{Alexeev-Brion}], and the case 
 [\ref{HMX2}, Corollary 1.7] mentioned above.

Next we show \ref{conj-BAB} in lower dimension implies a weak form of \ref{conj-BAB}, more precisely: 

\begin{thm}\label{t-BAB-to-bnd-vol}
Let $d$ be a natural number and $\epsilon$ a positive real number. 
Assume Conjecture \ref{conj-BAB} holds in dimension $d-1$. Then there is a 
number $v$ depending only on $d$ and $\epsilon$ such that if $X$ is an $\epsilon$-lc weak Fano variety 
of dimension $d$, then $\vol(-K_X)\le v$. In particular, such $X$ are birationally bounded. 
\end{thm}

In dimension $3$, the boundedness of $\vol(-K_X)$ was proved by Lai [\ref{Lai}] for $X$ of Picard number one
and by Jiang [\ref{Jiang-2}] in general who also proves the birational boundedness of $X$ in [\ref{Jiang}] using 
different methods. Theorem \ref{t-BAB-to-bnd-vol} gives new proofs of their results since 
Conjecture \ref{conj-BAB} is known in dimension $2$. The theorem is one of the crucial inductive steps
of the proof of \ref{conj-BAB} in [\ref{B-BAB}].\\ 

{\textbf{\sffamily{Boundedness of complements.}}}
Shokurov introduced the theory of complements while investigating threefold log flips [\ref{Shokurov-log-flips}, \S 5]. 
The theory originates from his earlier work on anticanonical systems on Fano threefolds [\ref{Shokurov-anticanonical}].  
The notion of complement involves both boundedness and singularities of the linear systems $|-mK_X|$.
It is actually defined in the more general setting of pairs. 
See \ref{ss-compl} for relevant definitions.

The following theorem was conjectured by Shokurov [\ref{shokurov-surf-comp}, Conjecture 1.3] who proved it in 
dimension $2$ [\ref{shokurov-surf-comp}, Theorem 1.4] (see also [\ref{PSh-II}, Corollary 1.8], 
and [\ref{Shokurov-log-flips}] for some cases). 

\begin{thm}\label{t-bnd-compl-usual}
Let $d$ be a natural number and $\mathfrak{R}\subset [0,1]$ be a finite set of rational numbers.
Then there exists a natural number $n$ 
depending only on $d$ and $\mathfrak{R}$ satisfying the following.  
Assume $(X,B)$ is a projective pair such that 
\begin{itemize}

\item $(X,B)$ is lc of dimension $d$,

\item $B\in \Phi(\mathfrak{R})$, that is, the coefficients of $B$ are in $\Phi(\mathfrak{R})$, 

\item $X$ is of Fano type, and 

\item $-(K_{X}+B)$ is nef.\\
\end{itemize}
Then there is an $n$-complement $K_{X}+{B}^+$ of $K_{X}+{B}$ 
such that ${B}^+\ge B$. Moreover, the complement is also an $mn$-complement for any $m\in \N$. 

\end{thm}

Here Fano type means $(X,G)$ is klt and $-(K_X+G)$ is 
ample for some boundary $G$, and $\Phi(\mathfrak{R})$ is the set of numbers of the form $1-\frac{r}{l}$ 
with $r\in\mathfrak{R}$ and $l\in\N$.  Note that the theorem in particular says that  
$$
-n(K_X+B)\sim n(B^+-B)\ge 0,
$$ 
hence the linear system $|-n(K_{X}+B)|$ is not empty.

Prokhorov and Shokurov [\ref{PSh-I}][\ref{PSh-II}] prove various inductive statements 
regarding complements. They  [\ref{PSh-II}, Theorem 1.4] also show that \ref{t-bnd-compl-usual}  
follows from two conjectures in the same dimension: 
the BAB conjecture (\ref{conj-BAB} above)
and the adjunction conjecture for fibre spaces [\ref{PSh-II}, Conjectue 7.13]. In dimension 
$3$, one only needs the BAB [\ref{PSh-II}, Corollary 1.7], and this also can be dropped 
if in addition we assume $(X,B)$ is non-exceptional.    
In this paper we replace the BAB with its special cases 
\ref{t-BAB-exc-usual} and \ref{t-BAB-good-boundary}, 
and we replace the adjunction conjecture with the theory of generalised polarised pairs developed in [\ref{BZh}].
See [\ref{Shokurov-adjunction}] for Shokurov's work on adjunction.

It is expected that Theorem \ref{t-bnd-compl-usual} holds for more general boundary coefficients. However, this is not 
well-understood even in dimension $2$.\\

{\textbf{\sffamily{Boundedness of complements in the relative setting.}}}
 Complements are also defined in the relative setting (see \ref{ss-compl}) for a given 
 contraction $f\colon X\to Z$, that is, a projective morphism 
 with $f_*\mathcal{O}_X=\mathcal{O}_Z$. 
In particular, when $X\to Z$ is the identity morphism, boundedness of complements is simply a local statement 
about singularities near a point on $X$. 

\begin{thm}\label{t-bnd-compl-usual-local}
Let $d$ be a natural number and $\mathfrak{R}\subset [0,1]$ be a finite set of rational numbers.
Then there exists a natural number $n$
depending only on $d$ and $\mathfrak{R}$ satisfying the following.  
Assume $(X,B)$ is a pair and $X\to Z$ is a contraction such that 
\begin{itemize}
\item $(X,B)$ is lc of dimension $d$ and $\dim Z>0$,

\item $B\in \Phi(\mathfrak{R})$, 

\item $X$ is of Fano type over $Z$, and 

\item $-(K_{X}+B)$ is nef over $Z$.\\
\end{itemize}
Then for any point $z\in Z$, there is an $n$-complement $K_{X}+{B}^+$ of $K_{X}+{B}$ 
over $z$ such that ${B}^+\ge B$. Moreover, the complement is also an $mn$-complement for any $m\in \N$. 
\end{thm}

The theorem was proved by Shokurov [\ref{shokurov-surf-comp}]  in dimension $2$, and by 
Prokhorov and Shokurov [\ref{PSh-I}] in dimension $3$.
They also essentially show that \ref{t-bnd-compl-usual-local} in dimension $d$ follows from \ref{t-bnd-compl-usual} 
in dimension $d-1$ [\ref{PSh-I}, Theorem 3.1].

In the local situation, Theorem \ref{t-bnd-compl-usual-local} implies a boundedness result about singularities.

\begin{cor}\label{cor-bnd-index}
Let $d$ be a natural number and $\mathfrak{R}\subset [0,1]$ be a finite set of rational numbers.
Then there exists a natural number $n$ depending only on $d$ and $\mathfrak{R}$ satisfying the following.  
Assume $(X,B)$ is a pair and $V\subset X$ is a subvariety such that 
\begin{itemize}
\item $(X,B)$ is lc  of dimension $d$,

\item $B\in \Phi(\mathfrak{R})$,

\item $V$ is a non-klt centre of $(X,B)$, and

\item $(X,\Delta)$ is klt near the generic point of $V$ for some $\Delta$.\\
\end{itemize}
Then $n(K_{X}+{B})$ is Cartier near the generic point of $V$.
\end{cor}

Note that existence of $\Delta$ is equivalent to saying that $X$ is of Fano type 
over itself perhaps after shrinking it around the generic point of $V$.\\

{\textbf{\sffamily{Boundedness of complements for generalised polarised pairs.}}}
We prove boundedness of complements in the even more general setting of generalised polarised pairs. 
This is of independent interest but also fundamental to our proofs. 
It gives enough flexibility to apply induction on dimension, 
unlike Theorem \ref{t-bnd-compl-usual}. For the relevant definitions, see \ref{ss-gpp} and \ref{ss-compl}.

\begin{thm}\label{t-bnd-compl}
Let $d$ and $p$ be natural numbers and $\mathfrak{R}\subset [0,1]$ be a finite set of rational numbers.
Then there exists a natural number $n$ 
depending only on $d,p$, and $\mathfrak{R}$ satisfying the following. 
Assume $(X',B'+M')$ is a projective generalised polarised pair with data $\phi\colon X\to X'$ and $M$ 
 such that 
\begin{itemize}
\item $(X',B'+M')$ is generalised lc of dimension $d$,

\item $B'\in \Phi(\mathfrak{R})$ and $pM$ is b-Cartier,

\item $X'$ is of Fano type, and 

\item $-(K_{X'}+B'+M')$ is nef.\\
\end{itemize}
Then there is an $n$-complement $K_{X'}+{B'}^++M'$ of $K_{X'}+{B'}+M'$ 
such that ${B'}^+\ge B'$. Moreover, the complement is also an $mn$-complement for any $m\in \N$.
\end{thm}

Here $pM$ being b-Cartier simply means that its pullback to some resolution of $X$ is Cartier.
Generalised polarised pairs behave similar to usual pairs in many ways. 
For example also see the effective birationality results for polarised pairs of general type established 
in [\ref{BZh}].\\

{\textbf{\sffamily{Boundedness of exceptional pairs.}}}
As mentioned earlier, in order to carry out our inductive arguments we need boundedness of exceptional pairs 
as in the next result.

\begin{thm}\label{t-BAB-exc}
Let $d$ and $p$ be natural numbers and $\mathfrak{R}\subset [0,1]$ be a finite set of rational numbers. 
Consider the pairs $(X',B'+M')$ as in Theorem \ref{t-bnd-compl} which are exceptional.
Then the set of the $(X',B')$ is log bounded.\\
\end{thm}

{\textbf{\sffamily{Structure of the paper.}}}
We outline the organisation of the paper. In Section 2, we gather some of the tools used in the paper, and 
prove certain basic results.  In Section 3, we discuss various types of adjunction, recall some 
known results, and prove some new results (eg, \ref{l-sub-bnd-on-gen-subvar}, \ref{p-lift-section-lcc}) crucial for later sections. 
In Section 4, we prove \ref{t-eff-bir-e-lc} under some additional assumptions (\ref{p-eff-bir-delta-2}, \ref{p-eff-bir-tau}). In Section 5, we prove \ref{t-BAB-good-boundary}. In Section 6, we develop the theory of 
complements for generalised polarised pairs, 
and prove various inductive statements (eg, \ref{p-bnd-compl-non-big}, \ref{p-bnd-compl-non-klt}), and discuss behaviour 
of boundary coefficients under adjunction for fibre spaces (\ref{l-fib-adj-dcc}). In Section 7, we study exceptional pairs and 
treat \ref{t-BAB-exc-usual} and \ref{t-BAB-exc} inductively (eg, \ref{l-exc-bnd-sing}, \ref{l-from-compl-to-BAB-exc-usual}, 
\ref{l-exc-bnd-vol}, \ref{p-from-compl-to-BAB-exc}), and give a criterion for a family of 
Fano varieties to be bounded (\ref{p-from-bnd-lct-to-bnd-var}) which is also a crucial ingredient of the proof of 
\ref{conj-BAB} in [\ref{B-BAB}]. In Section 8, we discuss complements in the relative setting 
(\ref{p-bnd-compl-usual-local}). In Section 9, we prove \ref{t-BAB-to-bnd-vol}.
Finally, in Section 10, we give the proofs of all the main results except those proved earlier.
It is worth mentioning that some of the results stated for varieties only can be easily extended to the log case 
(eg, \ref{t-eff-bir-e-lc}, \ref{t-BAB-to-bnd-vol}) but for simplicity we treat the non-log case only.\\

{\textbf{\sffamily{Sketch of some proofs.}}}
The main tools used in this paper are the minimal model program [\ref{kollar-mori}][\ref{BCHM}],  
the theory of complements [\ref{PSh-I}][\ref{PSh-II}][\ref{shokurov-surf-comp}], creating families   
of non-klt centres using volumes [\ref{HMX2}][\ref{HMX}][\ref{kollar-sing-pairs}, \S 6], and the 
theory of generalised polarised pairs [\ref{BZh}]. 
We give a brief account of some of the ideas of the proof of boundedness of complements (Theorem 
\ref{t-bnd-compl-usual}) and effective birationality (Theorem \ref{t-eff-bir-e-lc}). 

We start with boundedness of complements (Theorem \ref{t-bnd-compl-usual}). 
Pick a sufficiently small $\epsilon\in(0,1)$. If $(X,B)$ is non-exceptional, then we can modify 
it and assume it is not klt. On the other hand, if $(X,B)$ is exceptional, then we can bound 
its singularities, hence assume it is $\epsilon$-lc perhaps after decreasing $\epsilon$. Define $\Theta$ 
to be the same as $B$ except that we replace each coefficient in $(1-\epsilon,1)$ with $1$.  
Run an MMP on $-(K_X+\Theta)$ and let $X'$ be the resulting model and $\Theta'$ 
be the pushdown of $\Theta$. Since $X$ is of Fano type, we can run MMP on any divisor on $X$. 

As a consequence of local and global ACC [\ref{HMX2}, Theorems 1.1 and 1.5] (in practice we 
need their generalisations to generalised pairs [\ref{BZh}, Theorems 1.5 and 1.6]), we can show that 
the MMP does not contract any component of $\rddown{\Theta}$, $(X',\Theta')$ is lc, 
and $-(K_{X'}+\Theta')$ is nef (\ref{p-B'-to-Theta'}). It is enough to construct a bounded complement for 
$K_{X'}+\Theta'$.  Replacing $(X,B)$ with $(X',\Theta')$ and making further modifications, 
we can assume that the coefficients of $B$ belong to $\mathfrak{R}$ and  
that one of the following cases occurs:

\begin{enumerate}
\item $-(K_X+B)$ is nef and big, and $B$ has a component $S$ with coefficient $1$ which is of Fano type, or 

\item $K_X+B\equiv 0$ along a fibration $f\colon X\to T$, or 

\item $(X,B)$ is {exceptional}.
\end{enumerate}
These cases require very different inductive treatments.

Case (1): First apply {divisorial adjunction} to define $K_S+B_S=(K_X+B)|_S$. 
The coefficients of $B_S$ happen to be in  $\Phi(\mathfrak{S})$ 
for some fixed finite set $\mathfrak{S}$ of rational numbers. By induction on dimension $K_S+B_S$ has an $n$-complement for 
some bounded $n$. The idea then is to lift the complement to $X$ using vanishing theorems.
In the simplest case when $(X,B)$ is log smooth and ${B}=S$, we look at the exact sequence 
$$
H^0(-n(K_X+B))\to H^0( -n(K_X+B)|_S)\to H^1(-n(K_X+B)-S)=0
$$
where the vanishing follows from the Kawamata-Viehweg vanishing theorem noting that 
$$
-n(K_X+B)-S=K_X-n(K_X+B)-(K_X+B)=K_X-(n+1)(K_X+B).
$$ 
Since $K_S+B_S$ has an $n$-complement, the middle space in the above sequence 
is non-trivial which implies the left hand side is also non-trivial by lifting the section corresponding 
to the complement. One then argues that the lifted section gives an 
  $n$-complement for $K_X+B$. 

Case (2):  Apply the {{canonical bundle formula}} (also called adjunction for fibre spaces) to  write 
$$
K_X+B\sim_\Q f^*(K_T+B_T+M_T)
$$
where $B_T$ is the {discriminant divisor} and $M_T$ is the {moduli divisor}. It turns out that 
the coefficients of $B_T$ are in $\Phi(\mathfrak{S})$ 
for some fixed finite set $\mathfrak{S}$ of rational numbers, and that $pM_T$ is integral for some bounded number $p\in \N$. 
Now we want to find a complement for $K_T+B_T+M_T$ and pull it back to $X$. 
There is a serious issue here: $(T,B_T+M_T)$ is not a pair but it is a 
generalised polarised  pair. Thus we actually need to construct complements 
 in the more general setting of generalised polarised pairs. This makes life a lot more difficult 
but fortunately everything turns out to work. Once we have a bounded complement for 
$K_T+B_T+M_T$ we pull it back to get a bounded complement for $K_X+B$.

Case (3):  
For simplicity assume $B=0$ and that $X$ is a Fano variety. If we can show $X$ belongs to a 
bounded family, then we would be done. Actually we need something weaker: effective birationality. 
Assume we have already proved Theorem \ref{t-eff-bir-e-lc}. Then  
there is a bounded number $m\in\N$ such that $|-mK_X|$ defines a birational map. Pick 
$M\in |-mK_X|$ and let $B^+=\frac{1}{m}M$. Since $X$ is exceptional, $(X,B^+)$ is automatically klt, 
hence  $K_X+B^+$ is an $m$-complement. Although this gives some idea of how 
one may get a bounded complement but in practice we cannot give a complete 
proof of Theorem \ref{t-eff-bir-e-lc} before proving \ref{t-bnd-compl-usual}. The two theorems are 
actually proved together. Moreover, since we need to construct complements for generalised polarised pairs, 
treating the exceptional case in that setting is much harder.

Now we give a brief sketch of the proof of the effective birationality theorem (\ref{t-eff-bir-e-lc}).
Let $m\in\N$ be the smallest number such that $|-mK_X|$ defines a birational map, and 
let $n\in\N$ be a number such that ${\rm vol}(-nK_X)>(2d)^d$. Initially we take $n$ to be the smallest 
such number but  we will modify it during the proof. We want to show that $m$ is bounded from above.
The idea is first to show that $\frac{m}{n}$ is bounded from above, and then at the end show that 
$m$ is bounded. 

Applying a standard technique we can create a covering family $\mathcal{G}$ of subvarieties of $X$ such that 
if $x,y\in X$ are any pair of general closed points, then there is $0\le \Delta \sim_\Q -(n+1)K_X$ 
and $G\in\mathcal{G}$ such that $(X,\Delta)$ is lc at $x$ with the unique non-klt centre $G$ containing $x$, 
and $(X,\Delta)$ is not klt at $y$. 

Assume $\dim G=0$ for all $G$. Then $G=\{x\}$ is an isolated 
non-klt centre. Using multiplier ideal sheaves and vanishing theorems we can lift sections from $G$ and show  that 
$|-nK_X|$ defines a  birational map after replacing $n$ with a bounded multiple, hence in particular 
$\frac{m}{n}$ is bounded from above in this case. 

Now lets assume all $G$ 
have positive dimension. If ${\rm vol}(-mK_X|_{G})$ is too large, then using some 
elementary arguments, we can replace $n$ and create a new non-klt centre $G'$ containing $x$ but with $\dim G'<\dim G$. 
Thus we can replace $G$ with $G'$ and apply induction on dimension of $G$. We can then assume 
${\rm vol}(-mK_X|_{G})$ is bounded from above. 

Similar to the previous paragraph,  we can cut $G$ and decrease its dimension if 
${\rm vol}(-nK_X|_{G})$ is bounded from below. Showing this lower boundedness 
is the hard part. Although $G$ is not necessarily a divisor and 
although the singularities of $(X,\Delta)$ away from 
$x$ maybe quite bad but still there is a kind of adjunction formula, that is, if $F$ is the normalisation of $G$, 
then we can write  
$$
(K_X+\Delta)|_F\sim_\Q K_F+\Theta_F+P_F
$$ 
where $\Theta_F$ is a divisor with coefficients in a fixed DCC set $\Psi\subset [0,1]$ depending 
only on $d$, and  $P_F$ is pseudo-effective. Replacing $n$ with $2n$ and adding to $\Delta$ 
we can easily make $P_F$ big and effective. 

Now we would ideally want to apply induction on $d$ but the difficulty is that $F$ may not be Fano, in fact, it 
can be any type of variety. Another issue is that the singularities of $(F,\Theta_F+P_F)$ can be pretty bad. To 
overcome these difficulties we use the fact that ${\rm vol}(-mK_X|_{G})$ is bounded from above. 
From this boundedness one can deduce that there is a bounded projective log 
smooth pair $(\overline{F},\Sigma_{\overline{F}})$
 and a birational map $\overline{F}\bir F$ such that $\Sigma_{\overline{F}}$ is reduced 
containing the exceptional divisor of $\overline{F}\bir F$ and the support of the birational 
transform of $\Theta_F$ (and other relevant divisors).

Surprisingly, the worse the 
singularities of $(F,\Theta_F+P_F)$ the better because we can then produce divisors on $\overline{F}$ 
with bounded ``degree" but with arbitrarily small lc thresholds which would contradict a result 
about singularities (\ref{p-non-term-places}).
Indeed assume $(F,\Theta_F+P_F)$ is not klt. A careful study of the above adjunction formula allows to 
write $K_F+\Lambda_F:=K_X|_F$ where $\Lambda_F\le \Theta_F$ and $(F,\Lambda_F)$ is sub-$\epsilon$-lc. 
Put $I_F=\Theta_F+P_F-\Lambda_F$. Then 
$$
I_F=K_F+\Theta_F+P_F-K_F-\Lambda_F\sim_\Q (K_X+\Delta)|_F-K_X|_F=\Delta|_F\sim_\Q -(n+1)K_X|_F.
$$
Moreover, $K_F+\Lambda_F+I_F$ is ample. 

Let $\phi\colon F'\to F$ and $\psi\colon F'\to \overline{F}$ be a common resolution.  
Pull back $K_F+\Lambda_F+I_F$ to $F'$ and then push it down to $\overline{F}$ and write it as 
$K_{\overline{F}}+\Lambda_{\overline{F}}+I_{\overline{F}}$. Then the above ampleness gives 
$$
\phi^*(K_F+\Lambda_F+I_F)\le \psi^*(K_{\overline{F}}+\Lambda_{\overline{F}}+I_{\overline{F}})
$$
which implies that $(\overline{F},\Lambda_{\overline{F}}+I_{\overline{F}})$ is not sub-klt. From this one 
deduces that $(\overline{F},\Gamma_{\overline{F}}+I_{\overline{F}})$ is not klt where 
$\Gamma_{\overline{F}}=(1-\epsilon)\Sigma_{\overline{F}}$. Finally, one argues that the degree 
of $I_{\overline{F}}$ gets arbitrarily small if ${\rm vol}(-nK_X|_{G})$ gets arbitrarily small, and 
this contradicts the result on singularities mentioned above.

If singularities of $(F,\Theta_F+P_F)$ are good, then we again face some serious difficulties. 
Very roughly, in this case, we lift sections from $F$ to $X$ (\ref{p-lift-section-lcc}) and use this section 
to modify $\Delta$ so that  $(F,\Theta_F+P_F)$ 
has bad singularities, hence we reduce the problem to the above arguments. 
This shows $\frac{m}{n}$ is bounded.

Finally, we still need to bound $m$. This can be done by arguing that $\vol(-mK_X)$ is bounded 
from above and use this to show $X$ is birationally bounded, and then work on the bounded model.\\

{\textbf{\sffamily{Acknowledgements.}}}
This work was partially supported by a grant of the Leverhulme Trust.
Part of this work was done while visiting National Taiwan University in May and August 2015 
with the support of the Mathematics Division (Taipei Office) of the National
Center for Theoretical Sciences, and arranged by Jungkai Chen. I would like to thank them for their 
hospitality. I am indebted to Vyacheslav V. Shokurov for teaching me the theory of complements, and for his comments. 
I am grateful to Florin Ambro, Yifei Chen, Jingjun Han, and Jinsong Xu for their comments and corrections.
Thanks to Christopher Hacon and the referee for their long list of suggestions, comments, and 
corrections which improved this paper 
considerably.

\section{\bf Preliminaries}

All the varieties in this paper are quasi-projective over a fixed algebraically closed field of characteristic zero
unless stated otherwise. The set of natural numbers $\N$ is the set of positive integers, so it does not contain $0$.

\subsection{Contractions}
In this paper a \emph{contraction} refers to a projective morphism $f\colon X\to Y$ of varieties 
such that $f_*\mathcal{O}_X=\mathcal{O}_Y$ ($f$ is not necessarily birational). In particular, $f$ has connected fibres and 
if $X\to Z\to Y$ is the Stein factorisation of $f$, then $Z\to Y$ is an isomorphism. Moreover, 
if $X$ is normal, then $Y$ is also normal.

\subsection{Hyperstandard sets}\label{ss-dcc-sets}
For a subset $V\subseteq \R$ and a number $a\in\R$, we define $V^{\ge a}=\{v\in V \mid v\ge a\}$. 
We similarly define $V^{\le a}, V^{<a}$, and $V^{>a}$.
 
Let $\mathfrak{R}$ be a subset of $[0,1]$. Following [\ref{PSh-II}, 3.2] we define 
$$
\Phi(\mathfrak{R})=\left\{1-\frac{r}{m} \mid r\in \mathfrak{R},~ m\in \N\right\}
$$
to be the set of \emph{hyperstandard multiplicities} associated to $\mathfrak{R}$. We usually assume  
$0,1\in \mathfrak{R}$ without mention, so $\Phi(\mathfrak{R})$ includes $\Phi({\{0,1\}})$, the set of usual 
\emph{standard multiplicities}. Note that if we add  
$1-r$ to $\mathfrak{R}$ for each $r\in\mathfrak{R}$, then we get
$\mathfrak{R}\subset \Phi(\mathfrak{R})$.

Now assume $\mathfrak{R}\subset [0,1]$ is a finite set of rational numbers. Then 
$\Phi(\mathfrak{R})$ is a DCC set of rational numbers whose only accumulation point is $1$. 
We define $I=I(\mathfrak{R})$ to be the smallest natural number so that $Ir\in \Z$ 
for every $r\in \mathfrak{R}$. If $n\in \N$ is divisible by $I(\mathfrak{R})$, 
then $nb\le \rddown{(n+1)b}$ for every $b\in \Phi(\mathfrak{R})$ [\ref{PSh-II}, Lemma 3.5].

\subsection{Divisors}
Let $X$ be a normal variety, and let $M$ be an $\R$-divisor on $X$. 
We denote the coefficient of a prime divisor $D$ in $M$ by $\mu_DM$. If every non-zero coefficient of 
$M$ belongs to a set $\Phi\subseteq \R$, we write $M\in \Phi$. Writing $M=\sum m_iM_i$ where 
$M_i$ are the distinct irreducible components, the notation $M^{\ge a}$ means 
$\sum_{m_i\ge a} m_iM_i$, that is, we ignore the components with coefficient $<a$. One similarly defines $M^{\le a}, M^{>a}$, and $M^{<a}$. 
    
We say $M$ is \emph{b-Cartier} if it is $\Q$-Cartier and if there is a birational contraction 
$\phi\colon W\to X$ from a normal variety such that $\phi^*M$  is Cartier.    
    
Now let $f\colon X\to Z$ be a morphism to a normal variety. We say $M$ is \emph{horizontal} over $Z$ 
if the induced map $\Supp M\to Z$ is dominant, otherwise we say $M$ is \emph{vertical} over $Z$. If $N$ is an $\R$-Cartier 
divisor on $Z$, we often denote $f^*N$ by $N|_X$. 

Again let $f\colon X\to Z$ be a morphism to a normal variety, and let $M$ and $L$ be $\R$-Cartier divisors on $X$. 
We say $M\sim L$ over $Z$ (resp. $M\sim_\Q L$ over $Z$)(resp. $M\sim_\R L$ over $Z$) if there is a Cartier  
(resp. $\Q$-Cartier)(resp. $\R$-Cartier) divisor $N$ on $Z$ such that $M-L\sim f^*N$  
(resp. $M-L\sim_\Q f^*N$)(resp. $M-L\sim_\R f^*N$). For a point $z\in Z$, we say $M\sim L$ over $z$ if   
$M\sim L$ over $Z$ perhaps after shrinking $Z$ around $z$. The properties $M\sim_\Q L$ and $M\sim_\R L$ over $z$ 
are similarly defined.

For a birational map $X\bir X'$ (resp. $X\bir X''$)(resp. $X\bir X'''$)(resp. $X\bir Y$) 
whose inverse does not contract divisors, and for 
an $\R$-divisor $M$ on $X$ we usually denote the pushdown of $M$ to $X'$ (resp. $X''$)(resp. $X'''$)(resp. $Y$) 
by $M'$ (resp. $M''$)(resp. $M'''$)(resp. $M_Y$).

\begin{lem}\label{l-relatively-trivial-divisors}
Let $f\colon X\to Z$ be a contraction between normal varieties and let $M$ be a Weil divisor on $X$.
Assume that $M\sim 0$ over $z$ for each $z\in Z$. Then $M\sim 0/Z$.
\end{lem}
\begin{proof}
Since $M\sim 0$ over $z$ for each $z\in Z$, the sheaves $\mathcal{O}_X(M)$ and 
$f_*\mathcal{O}_X(M)$ are invertible. Moreover, the 
induced morphism $f^*f_*\mathcal{O}_X(M)\to \mathcal{O}_X(M)$ is surjective, hence 
it is an isomorphism. There is a Cartier divisor $N$ such that $f_*\mathcal{O}_X(M)\simeq \mathcal{O}_Z(N)$. 
Then $M\sim f^*N$, so $M\sim 0/Z$.

\end{proof}

\subsection{Linear systems}\label{ss-lin-systems}

Let $X$ be a normal variety and let $M$ be an $\R$-divisor on $X$. 
We usually write $H^i(M)$ instead of $H^i(X,\mathcal{O}_X(\rddown{M}))$. We can describe 
$H^0(M)$ in terms of rational functions on $X$ as 
$$
H^0(M)=\{0\neq \alpha\in K \mid\Div(\alpha)+M\ge 0 \}\cup \{0\}
$$
where $K$ is the function field of $X$ and $\Div(\alpha)$ is the divisor associated to $\alpha$.

Assume $h^0(M)\neq 0$. 
The \emph{linear system} $|M|$ is defined as 
$$
|M|=\{N\mid 0\le N\sim M\}=\{\Div(\alpha)+M \mid 0\neq \alpha\in H^0(M)\}.
$$
Note that $|M|$ is not equal to $|\rddown{M}|$ unless $M$ is integral.
The \emph{fixed part} of $|M|$ is the $\R$-divisor $F$ with the property: 
if $G\ge 0$ is an $\R$-divisor and $G\le N$ for every $N\in |M|$, then $G\le F$. In particular, $F\ge 0$. 
We then define the \emph{movable part} of $|M|$ to be $M-F$ which is defined up to 
linear equivalence. If $\langle M\rangle:=M-\rddown{M}$, then the fixed part of $|M|$ is equal to 
$\langle M\rangle$ plus the fixed part of $|\rddown{M}|$. Moreover, if $0\le G\le F$, then 
the fixed and  movable parts of $|M-G|$ are $F-G$ and $M-F$, respectively. 

Note that it is clear from the definition that the movable part of $|M|$ is an integral divisor 
but the fixed part is only an $\R$-divisor.

\begin{lem}\label{l-mov-part-lin-system}
Let $X$ be a normal variety and let $M$ be an $\R$-Cartier $\R$-divisor on $X$. 
Assume $\phi\colon Y\to X$ is a projective birational morphism from a normal variety.  
Let $F$ be the fixed part of $|M|$ and $F_Y$ be the fixed part of $|M_Y|$ where $M_Y=\phi^*M$. Then 
\begin{itemize}
\item $\phi_*|M_Y|=|M|$,

\item $\phi_*F_Y=F$, and

\item if $\phi$ is a sufficiently high resolution, then  $|M_Y-F_Y|$ is base point free.
\end{itemize}
\end{lem}
\begin{proof}
If $N_Y\in |M_Y|$, then $0\le N_Y\sim M_Y$, hence $0\le \phi_*N_Y\sim M$ which means $\phi_*N_Y\in |M|$.
On the other hand, if $N\in |M|$, then $0\le N\sim M$, hence $0\le \phi^*N\sim M_Y$ which means $\phi^*N\in |M_Y|$. 
Thus $\phi_*|M_Y|=|M|$.

Since $F_Y\le N_Y$ for every $N_Y\in |M_Y|$, we get $\phi_*F_Y\le N$ for every $N\in |M|$, hence 
$\phi_*F_Y\le F$. On the other hand, $|M_Y-F_Y|$ is movable, that is, it is base point free outside a 
codimension two closed subset of $Y$. Then $|M-\phi_*F_Y|$ is also base point free outside a codimension 
two closed subset of  $X$ which  means $\phi_*F_Y\ge F$. 
Thus $\phi_*F_Y=F$. 

Now let $\psi \colon W\to X$ be a resolution and let $F_W$ be the fixed part of $|M_W|$ where $M_W=\psi^*M$. 
By Hironaka's work, there is a higher resolution $\pi\colon Y\to W$ such that if $L_Y$ is the movable part of 
$|\pi^*(M_W-F_W)|$, then $|L_Y|$ is base point free. Pick $N_Y\in |M_Y|$ and let $N_W=\pi_*N_Y$. 
Since $F_W\le N_W$ we get $\pi^*F_W\le \pi^*N_W=N_Y$, hence $\pi^*F_W\le F_Y$. Then 
 the movable part of $|M_Y|$ coincides with the movable part of $|M_Y-\pi^*F_W=\pi^*(M_W-F_W)|$.
Therefore,   $|M_Y-F_Y|=|L_Y|$ is base point free.

\end{proof}

\subsection{b-divisors}\label{ss-b-divisor}

We recall some definitions regarding b-divisors but not in full generality. 
Let $X$ be a variety. A \emph{b-$\R$-Cartier b-divisor over $X$} is the choice of  
a projective birational morphism 
$Y\to X$ from a normal variety and an $\R$-Cartier divisor $M$ on $Y$ up to the following equivalence: 
 another projective birational morphism $Y'\to X$ from a normal variety and an $\R$-Cartier divisor
$M'$ defines the same b-$\R$-Cartier  b-divisor if there is a common resolution $W\to Y$ and $W\to Y'$ 
on which the pullbacks of $M$ and $M'$ coincide.  

A b-$\R$-Cartier  b-divisor  represented by some $Y\to X$ and $M$ is \emph{b-Cartier} if  $M$ is 
b-Cartier, i.e. its pullback to some resolution is Cartier.

\subsection{Pairs}
In this paper a \emph{sub-pair} $(X,B)$ consists of a normal quasi-projective variety $X$ and an $\R$-divisor 
$B$ such that $K_X+B$ is $\R$-Cartier. 
If the coefficients of $B$ are at most $1$ we say $B$ is a 
\emph{sub-boundary}, and if in addition $B\ge 0$, 
we say $B$ is a \emph{boundary}. A sub-pair $(X,B)$ is called a \emph{pair} if $B\ge 0$ (we allow coefficients 
of $B$ to be larger than $1$ for practical reasons).

Let $\phi\colon W\to X$ be a log resolution of a sub-pair $(X,B)$. Let $K_W+B_W$ be the 
pulback of $K_X+B$. The \emph{log discrepancy} of a prime divisor $D$ on $W$ with respect to $(X,B)$ 
is $1-\mu_DB_W$ and it is denoted by $a(D,X,B)$.
We say $(X,B)$ is \emph{sub-lc} (resp. \emph{sub-klt})(resp. \emph{sub-$\epsilon$-lc}) 
if $a(D,X,B)$ is $\ge 0$ (resp. $>0$)(resp. $\ge \epsilon$) for every $D$. When $(X,B)$ 
is a pair we remove the sub and say the pair is lc, etc. Note that if $(X,B)$ is a lc pair, then 
the coefficients of $B$ necessarily belong to $[0,1]$. Also if $(X,B)$ is $\epsilon$-lc, then 
automatically $\epsilon\le 1$ because $a(D,X,B)=1$ for most $D$. 

Let $(X,B)$ be a sub-pair. A \emph{non-klt place} of $(X,B)$ is a prime divisor $D$ on 
birational models of $X$ such that $a(D,X,B)\le 0$. A \emph{non-klt centre} is the image on 
$X$ of a non-klt place. When $(X,B)$ is lc, a non-klt centre is also called an 
\emph{lc centre}.

\subsection{Minimal model program (MMP)}\label{ss-MMP} 
We will use standard results of the minimal model program (cf. [\ref{kollar-mori}][\ref{BCHM}]). 
Assume $(X,B)$ is a pair and $X\to Z$ is a projective morphism. 
 Assume $H$ is an ample$/Z$ $\R$-divisor and that $K_X+B+H$ is nef$/Z$. Suppose $(X,B)$ is klt or 
that it is $\Q$-factorial dlt. We can run an MMP$/Z$ on $K_X+B$ with scaling of $H$. If 
$(X,B)$ is klt and if either $K_X+B$ or $B$ is big$/Z$, then the MMP terminates [\ref{BCHM}]. If $(X,B)$ 
is $\Q$-factorial dlt, then in general we do not know whether the MMP terminates but 
we know that in some step of the MMP we reach a model $Y$ on which $K_Y+B_Y$, 
the pushdown of $K_X+B$, is a limit of movable$/Z$ $\R$-divisors: indeed, if the MMP terminates, then 
the claim is obvious; otherwise the MMP produces an infinite sequence $X_i\bir X_{i+1}$ 
of flips and a decreasing sequence $\lambda_i$ of numbers in $(0,1]$ such that 
$K_{X_i}+B_i+\lambda_iH_i$ is nef$/Z$; by [\ref{BCHM}][\ref{B-lc-flips}, Theorem 1.9], $\lim\lambda_i=0$; 
in particular, if $Y:=X_1$, then $K_Y+B_Y$ is the limit of the movable$/Z$ $\R$-divisors 
$K_Y+B_Y+\lambda_i H_Y$.

\subsection{Fano pairs}
Let $(X,B)$ be a pair and $X\to Z$ a contraction. We say $(X,B)$ is \emph{log Fano} over $Z$ 
if it is lc and $-(K_X+B)$ is ample over $Z$; if $B=0$ 
we just say $X$ is Fano over $Z$. The pair is called \emph{weak log Fano} over $Z$ if it is lc 
and $-(K_X+B)$ is nef and big over $Z$; 
if $B=0$ we say $X$ is \emph{weak Fano} over $Z$.
We say $X$ is \emph{of Fano type} over $Z$ if $(X,B)$ is klt weak log Fano over $Z$ for some choice of $B$;
it is easy to see this is equivalent to existence of a big$/Z$ $\Q$-boundary (resp. $\R$-boundary) 
$\Gamma$ so that $(X,\Gamma)$ is klt and $K_X+\Gamma \sim_\Q 0/Z$ (resp. $\sim_\R$ instead of $\sim_\Q$).

Assume $X$ is of Fano type over $Z$. Then we can run the MMP 
over $Z$ on any $\R$-Cartier $\R$-divisor $D$ on $X$ which ends with some model $Y$ [\ref{BCHM}]. If $D_Y$ is nef over $Z$,
we call $Y$ a \emph{minimal model} over $Z$ for $D$. If $D_Y$ is not nef$/Z$, then 
there is a $D_Y$-negative extremal contraction $Y\to T/Z$ with $\dim Y>\dim T$ and we call 
$Y$ a \emph{Mori fibre space} over $Z$ for $D$.

\begin{lem}\label{l-FT-over-curve}
Let $(X,B)$ be an lc pair and $f\colon X\to Z$ be a contraction onto a smooth curve. 
Assume $X$ is of Fano type over some non-empty open set $U\subseteq Z$. Further assume $B$ is a $\Q$-boundary,
$K_X+B\sim_\Q 0/Z$, and that the generic point of each non-klt centre of $(X,B)$ is mapped into $U$. 
Then $X$ is of Fano type over $Z$.
\end{lem}
\begin{proof}
This proof which differs from the original proof, was suggested by Florin Ambro. 
Since $X$ is of Fano type over $U$, there is a $\Q$-divisor $\Gamma$ with coefficients in $[0,1)$ 
such that  $\Gamma$ is big over $U$, $(X,\Gamma)$ is klt over $U$, and $K_X+\Gamma\sim_\Q 0$ over $U$.  
Then $K_X+\Gamma\sim_\Q D/Z$ for some $\Q$-divisor $D$ which is vertical$/Z$. 
Since $Z$ is a curve, we can easily find a $\Q$-divisor $C$ on $Z$ so that if we replace 
$D$ with $D+f^*C$, then the support of $D$ is 
mapped into $Z\setminus U$ and that $D\le 0$.  

Now $K_X+\Gamma-D\sim_\Q 0/Z$ and $\Gamma-D\ge 0$.
Since the generic point of each non-klt centre of $(X,B)$ is mapped into $U$ and since 
$(X,\Gamma)$ is klt over $U$, the pair  
$$
(X,(1-t)B+t(\Gamma-D))
$$
is klt if $t>0$ is a sufficiently small rational number. Therefore, $X$ is of Fano type over $Z$ as 
$$
K_X+(1-t)B+t(\Gamma-D)\sim_\Q 0/Z
$$ 
and 
$(1-t)B+t(\Gamma-D)$ is big over $Z$.

\end{proof}

\begin{lem}\label{l-FT-base-contraction}
Let $X$ be a projective variety of Fano type, and let $f\colon X\to Z$ be a contraction 
where $\dim Z>0$. Then $Z$ is of Fano type.  
\end{lem}
\begin{proof}
There is a big $\Q$-boundary $\Gamma$ such that $(X,\Gamma)$ is klt and $K_X+\Gamma \sim_\Q 0$. 
In particular, $X$ is normal, hence $Z$ is normal too. Since $\Gamma$ is big, we can modify it 
and assume $\Gamma\ge H\ge 0$ for some ample $\Q$-divisor $H$. In turn we can modify $H$, hence  
assume $\Gamma\ge H\ge f^*A$ for some ample $\Q$-divisor $A\ge 0$. 
Let $\Delta:=\Gamma-f^*A$. Then $K_X+\Delta\sim_\Q 0/Z$, hence by [\ref{ambro-lc-trivial}, Theorem 0.2], 
there is $\Delta_Z$ such that $K_X+\Delta\sim_\Q f^*(K_Z+\Delta_Z)$ and $(Z,\Delta_Z)$ is klt. 
Now letting $\Gamma_Z=\Delta_Z+A'$ where $A'\sim_\Q A$ is general we see that 
$K_Z+\Gamma_Z\sim_\Q 0$ and $(Z,\Delta_Z)$ is klt. Thus $Z$ is of Fano type.

\end{proof}

\subsection{Generalised polarised pairs}\label{ss-gpp}
For the basic theory of generalised polarised pairs see [\ref{BZh}, Section 4].
Below we recall some of the main notions and discuss some basic properties.\\

(1)
A \emph{generalised polarised pair} consists of 
\begin{itemize}
\item a normal variety $X'$ equipped with a projective
morphism $X'\to Z$, 

\item an $\R$-divisor $B'\ge 0$ on $X'$, and 

\item a b-$\R$-Cartier  b-divisor over $X'$ represented 
by some projective birational morphism $X \overset{\phi}\to X'$ and $\R$-Cartier divisor
$M$ on $X$
\end{itemize}
such that $M$ is nef$/Z$ and $K_{X'}+B'+M'$ is $\R$-Cartier,
where $M' := \phi_*M$. 

We usually refer to the pair by saying $(X',B'+M')$ is a  generalised pair with 
data $X\overset{\phi}\to X'\to Z$ and $M$. Since a b-$\R$-Cartier b-divisor is defined birationally (see \ref{ss-b-divisor}), 
in practice we will often replace $X$ with a resolution and replace $M$ with its pullback.
When $Z$ is not relevant we usually drop it
 and do not mention it: in this case one can just assume $X'\to Z$ is the identity. 
When $Z$ is a point we also drop it but say the pair is projective. 

Now we define generalised singularities.
Replacing $X$ we can assume $\phi$ is a log resolution of $(X',B')$. We can write 
$$
K_X+B+M=\phi^*(K_{X'}+B'+M')
$$
for some uniquely determined $B$. For a prime divisor $D$ on $X$ the \emph{generalised log discrepancy} 
$a(D,X',B'+M')$ is defined to be $1-\mu_DB$. 

We say $(X',B'+M')$ is 
\emph{generalised lc} (resp. \emph{generalised klt})(resp. \emph{generalised $\epsilon$-lc}) 
if for each $D$ the generalised log discrepancy $a(D,X',B'+M')$ is $\ge 0$ (resp. $>0$)(resp. $\ge \epsilon$).
A \emph{generalised non-klt centre} of $(X',B'+M')$ is the image of a prime divisor 
$D$ on birational models of $X'$ with $a(D,X',B'+M')\le 0$, and 
the \emph{generalised non-klt locus} of the generalised pair is the union of all the generalised non-klt centres.  

(2)
Let $(X',B'+M')$ be a generalised pair as in (1).
We say $(X',B'+M')$ is \emph{generalised dlt} if it is generalised lc and if 
$\eta$ is the generic point of any generalised non-klt centre of 
$(X',B'+M')$, then $(X',B')$ is log smooth near $\eta$ and $M=\phi^*M'$  holds over a neighbourhood of $\eta$.  
 If in addition the connected 
components of $\rddown{B'}$ are irreducible, we say the pair is \emph{generalised plt}.
Note that when $M=0$, then $(X',B')$ is {generalised dlt} iff it is dlt in the usual sense. 

The generalised dlt property is preserved under the MMP. Indeed, assume $(X',B'+M')$ is {generalised dlt} and 
that $X'\bir X''/Z$ is a divisorial contraction or a flip with respect to $K_{X'}+B'+M'$. Replacing $\phi$ 
we can assume $X\bir X''$ is a morphism. Let $B'',M''$ be the pushdowns of $B',M'$ and consider 
$(X'',B''+M'')$ as a generalised pair with data $X\to X''\to Z$ and $M$.  Then $(X'',B''+M'')$
is also generalised dlt because it is generalised lc and because $X'\bir X''$ is an isomorphism over the generic point of 
any generalised non-klt center of   $(X'',B''+M'')$.

(3)
Let $(X',B'+M')$ be a generalised pair as in (1) and let $\psi\colon X''\to X'$ be a projective birational 
morphism from a normal variety. Replacing $\phi$ we can assume $\phi$ factors through 
$\psi$. We then let $B''$ and $M''$ be the pushdowns of 
$B$ and $M$ on $X''$ respectively. In particular,  
$$
K_{X''}+B''+M''=\psi^*(K_{X'}+B'+M').
$$
If $B''\ge 0$, then $(X'',B''+M'')$ is also a generalised pair 
with data $X\to X''\to Z$ and $M$. If $(X'',B''+M'')$ is $\Q$-factorial generalised dlt and if 
every exceptional prime divisor of $\psi$ appears in $B''$ with coefficients one, then we say 
$(X'',B''+M'')$ is a \emph{$\Q$-factorial generalised dlt model} of $(X',B'+M')$. Such models exist if 
$(X',B'+M')$ is generalised lc, by [\ref{BZh}, Lemma 4.5].

(4)
Let $(X',B'+M')$ be a generalised pair as in (1).
Assume that $D'$ on $X'$ is an effective $\R$-divisor and that $N$ on $X$ is an $\R$-divisor which is
nef$/Z$ such that $D'+N'$ is $\R$-Cartier where $N'=\phi_*N$.
The \emph{generalised lc threshold} of $D'+N'$
with respect to $(X',B'+M')$ is defined as
$$
\sup \{s \mid \mbox{$(X',B'+sD'+M'+sN')$ is generalised lc}\}
$$
where the pair in the definition comes with data $X\overset{\phi}\to X'\to Z$ and $M+sN$.

(5)
We prove a connectedness principle similar to the usual one.

\begin{lem}[Connectedness principle]\label{l-connectedness}
Let $(X',B'+M')$ be a generalised pair with data $X\overset{\phi}\to X'\to Z$ and $M$ 
where $X'\to Z$ is a contraction. 
Assume $-(K_{X'}+B'+M')$ is nef and big over $Z$. Then the generalised non-klt locus of 
$(X',B'+M')$ is connected near each fibre of $X'\to Z$.
\end{lem}
\begin{proof}
We can assume $\phi$ is a log resolution. Write 
$$
K_X+B+M=\phi^*(K_{X'}+B'+M').
$$
The generalised non-klt locus of $(X',B'+M')$ is just $\phi(\Supp B^{\ge 1})$.
We can write 
$$
-(K_{X}+B+M)\sim_\R A+C/Z
$$
where $A$ is ample and $C\ge 0$. Replacing $X$ with a higher resolution and 
replacing $A,C$ with their pullbacks we can assume $\phi$ is a log resolution of $(X',B'+C')$ 
where $C'=\phi_*C$: note that here we initially  
replace $A,C$ with their pullbacks to the new resolution but then $A$ may no longer be ample 
although it is nef and big; we then 
perturb $A,C$ in the exceptional components so that $A$ is ample again.  
Pick a sufficiently small $\epsilon>0$, let $G\sim_\R M+\epsilon A/Z$ be general with  coefficients 
less than $1$, and let 
$\Delta=B+\epsilon C+G$.
Then $K_X+\Delta\sim_\R 0/X'$, so $K_X+\Delta=\phi^*(K_{X'}+\Delta')$. 
Moreover, $\Supp B^{\ge 1}=\Supp \Delta^{\ge 1}$.
Thus the non-klt locus of the pair $(X',\Delta')$   
is equal to the generalised non-klt locus of $(X',B'+M')$. Therefore, the result follows from the usual 
connectedness principle [\ref{Kollar-flip-abundance}, Theorem 17.4] because 
$$
-(K_{X'}+\Delta')\sim_\R -(1-\epsilon) (K_{X'}+B'+M')/Z
$$ 
is nef and big over $Z$.
\end{proof}

(6)
Let $(X',B'+M')$ be a projective generalised klt pair. Assume $A':=-(K_{X'}+B'+M')$ is nef and big. 
We show $X'$ is of Fano type. Using the notation of (1),  
let $A=\phi^*A'$. Then $A\sim_\R H+G$ where $H$ is ample and $G\ge 0$. 
Take a small $\epsilon>0$ and a general $C\sim_\R \epsilon H+M$.  Then 
$$
K_X+B+\epsilon G+C\sim_\R K_X+B+M+\epsilon A=\phi^*(K_{X'}+B'+M'+\epsilon A'),
$$
hence if we let $\Delta=B+\epsilon G+C$, then 
$K_X+\Delta=\phi^*(K_{X'}+\Delta')$ which shows $(X',\Delta')$ is klt. 
Since $-(K_{X'}+\Delta')$ is nef and big, $X'$ is of Fano type.

(7)
Let $(X',B'+M')$ be a projective generalised lc pair where $X'$ is of Fano type 
and $-(K_{X'}+B'+M')$ is nef. 
Assume $X''\to X'$ is a birational morphism from a normal projective variety. 
Let  $K_{X''}+B''+M''$ be the pullback of $K_{X'}+B'+M'$ where $B''$ is the pushdown of 
$B$ and $M''$ is the pushdown of $M$.
We show $X''$ is of Fano type too, assuming every exceptional$/X'$ component of  $B''$ 
has positive coefficient. There is a $\Q$-boundary $\Gamma'$ such that 
$(X',\Gamma')$ is klt and $-(K_{X'}+\Gamma')$ is nef and big. Let $K_{X''}+\Gamma''$ 
be the pullback of $K_{X'}+\Gamma'$. Let $\Delta''=(1-t)\Gamma''+tB''$ for some  
$t\in (0,1)$ sufficiently close to $1$. Then $(X'',\Delta''+tM'')$ is generalised klt 
and $-(K_{X''}+\Delta''+tM')$ is nef and big. Now apply (5).

\subsection{Exceptional and non-exceptional pairs}\label{ss-exc-non-exc}
(1)
Let $(X,B)$ be a projective pair such that $K_X+B+P\sim_\R 0$ for some $\R$ divisor 
$P\ge 0$. We say the pair is \emph{non-exceptional} (resp. \emph{strongly non-exceptional}) if 
we can choose $P$ so that $(X,B+P)$ is not klt (resp. not lc). We say the pair is 
\emph{exceptional} if $(X,B+P)$ is klt for every choice of $P$. 

(2)
Now let $(X',B'+M')$ be a projective generalised  pair with 
data $\phi\colon X\to X'$ and $M$. Assume $K_{X'}+B'+M'+P'\sim_\R 0$ for some $\R$-divisor $P'\ge 0$. 
We say the pair is \emph{non-exceptional} (resp. \emph{strongly non-exceptional}) if 
we can choose $P'$ so that $(X',B'+P'+M')$ is not generalised klt (resp. not generalised lc). We say the pair is 
\emph{exceptional} if $(X',B'+P'+M')$ is generalised klt for every choice of $P'$. Here we consider 
$(X',B'+P'+M')$ as a generalised pair with data $\phi\colon X\to X'$ and $M$.


\begin{lem}\label{l-s-non-exc-change-to-Q-div}
Let $(X',B'+M')$ be a projective generalised pair such that $B'+M'$ is a $\Q$-divisor. 
If $(X',B'+M')$ is non-exceptional (resp. strongly non-exceptional), then there is a $\Q$-divisor $P'\ge 0$ such that 
$K_{X'}+B'+M'+P'\sim_\Q 0$ and $({X'},B'+M'+P')$ is not generalised klt (resp. generalised lc).
\end{lem}
\begin{proof}
By definition, there is an $\R$-divisor $P'\ge 0$ such that 
$K_{X'}+B'+M'+P'\sim_\R 0$ and $({X'},B'+M'+P')$ is not generalised klt (resp. generalised lc).
There exist numbers $0\neq r_i\in \R$ and rational functions $\alpha_i$ on $X'$ such that 
$$
P'=\sum r_i\Div(\alpha_i)-(K_{X'}+B'+M').
$$
Consider the set $V$ of $\R$-divisors 
$$
\sum s_i\Div(\alpha_i)-(K_{X'}+B'+M').
$$
where $s_i\in \R$ are arbitrary numbers. This is a rational affine space inside the space $W$ of $\R$-divisors 
generated by the components of $K_{X'}+B'+M'$ and the components of all the $\Div(\alpha_i)$. 
On the other hand, the space $U$ of $\R$-divisors generated by all the components of $P'$ is also 
a rational affine subspace of $W$. The intersection $U\cap V$ is a rational affine subspace of $W$ and $P'\in U\cap V$. 
Therefore, there exist real numbers $a_j\in [0,1]$ with $\sum a_j=1$ and effective $\Q$-divisors $P_j'\in U\cap V$ 
sufficiently close to $P'$ (in terms of coefficients) such that $P'=\sum a_jP_j'$. By construction,
 $K_{X'}+B'+M'+P_j'\sim_\Q 0$ for each $j$, and 
$$
K_{X'}+B'+M'+P'=\sum a_j(K_{X'}+B'+M'+P_j').
$$  
Thus there is $j$ such that $({X'},B'+M'+P_j')$ is not generalised klt (resp. generalised lc).
Now replace $P'$ with $P_j'$.
 
\end{proof}

(3) 
The next lemma is useful to keep track of the exceptionality property when considering a birational model 
of $X'$.

\begin{lem}\label{l-exc-compare-bigger-model}
Let $(X',B'+M')$ be a projective generalised pair with data $X\overset{\phi}\to X'$ and $M$, and let
$(X'',B''+M'')$ be a projective generalised pair with data $X\overset{\psi}\to X''$ and $M$ (here 
$X$ and $M$ are the same, in particular, $\psi\phi^{-1}\colon X'\bir X''$ is a birational map).  
Assume 
$$
\psi^*(K_{X''}+B''+M'')\ge \phi^*(K_{X'}+B'+M').
$$
If $(X',B'+M')$ is exceptional, then $(X'',B''+M'')$ is also exceptional.
\end{lem}
\begin{proof}
Assume $(X'',B''+M'')$ is not exceptional. Then there is 
$$
0\le P''\sim_\R -(K_{X''}+B''+M'')
$$
such that  $({X''},B''+P''+M'')$ is not generalised klt. Thus by the inequality in the statement, there is  
$$
0\le P'\sim_\R -(K_{X'}+B'+M')
$$
such that $({X'},B'+P'+M')$ is not generalised klt, contradicting the exceptionality of $(X',B'+M')$.

\end{proof}

\subsection{Complements}\label{ss-compl}
(1) 
We first recall the definition for usual pairs.
Let $(X,B)$ be a pair where $B$ is a boundary and let $X\to Z$ be a contraction. 
Let $T=\rddown{B}$ and $\Delta=B-T$. 
An \emph{$n$-complement} of $K_{X}+B$ over a point $z\in Z$ is of the form 
$K_{X}+{B}^+$ such that over some neighbourhood of $z$ we have the following properties:
\begin{itemize}
\item $(X,{B}^+)$ is lc, 

\item $n(K_{X}+{B}^+)\sim 0$, and 

\item $n{B}^+\ge nT+\rddown{(n+1)\Delta}$.\\
\end{itemize}
From the definition one sees that 
$$
-nK_{X}-nT-\rddown{(n+1)\Delta}\sim n{B}^+-nT-\rddown{(n+1)\Delta}\ge 0
$$
over some neighbourhood of $z$ which in particular means the linear system 
$$
|-nK_{X}-nT-\rddown{(n+1)\Delta}|
$$
is not empty over $z$. Moreover, if $B^+\ge B$, then $-n(K_X+B)\sim n(B^+-B)$ over $z$, hence 
$|-n(K_X+B)|$ is non-empty over $z$.

(2) 
Now let $(X',B'+M')$ be a projective generalised pair with data $\phi\colon X\to X'$ 
and $M$ where $B'\in [0,1]$. 
Let $T'=\rddown{B'}$ and $\Delta'=B'-T'$. 
An \emph{$n$-complement} of $K_{X'}+B'+M'$ is of the form $K_{X'}+{B'}^++M'$ where 
\begin{itemize}
\item $(X',{B'}^++M')$ is generalised lc,

\item $n(K_{X'}+{B'}^++M')\sim 0$ and $nM$ is b-Cartier, and 

\item $n{B'}^+\ge nT'+\rddown{(n+1)\Delta'}$.\\
\end{itemize}
From the definition one sees that 
$$
-nK_{X'}-nT'-\rddown{(n+1)\Delta'}-nM'\sim n{B'}^+-nT'-\rddown{(n+1)\Delta'}\ge 0
$$
which in particular means the linear system 
$$
|-nK_{X'}-nT'-\rddown{(n+1)\Delta'}-nM'|
$$
is not empty. Moreover, if $B'^+\ge B'$, then $-n(K_{X'}+B'+M')\sim n(B'^+-B')$, hence 
$|-n(K_{X'}+B'+M')|$ is non-empty.

We can also define complements for generalised pairs in the relative setting but for simplicity 
we will not deal with those.

\subsection{Bounded families of pairs}\label{ss-bnd-couples}
A \emph{couple} $(X,D)$ consists of a normal projective variety $X$ and a  divisor 
$D$ on $X$ whose non-zero coefficients are all equal to $1$, i.e. $D$ is a reduced divisor. 
The reason we call $(X,D)$ a couple rather than a pair is that we are concerned with 
$D$ rather than $K_X+D$ and we do not want to assume $K_X+D$ to be $\Q$-Cartier 
or with nice singularities. Two couples $(X,D)$ and $(X',D')$ are isomorphic if 
there is an isomorphism $X\to X'$ mapping $D$ onto $D'$.

We say that a set $\mathcal{P}$ of couples  is \emph{birationally bounded} if there exist 
finitely many projective morphisms $V^i\to T^i$ of varieties and reduced divisors $C^i$ on $V^i$ 
such that for each $(X,D)\in \mathcal{P}$ there exist an $i$, a closed point $t\in T^i$, and a 
birational isomorphism $\phi\colon V^i_t\bir X$ such that $(V^i_t,C^i_t)$ is a couple and 
$E\le C_t^i$ where 
$V_t^i$ and $C_t^i$ are the fibres over $t$ of the morphisms $V^i\to T^i$ and $C^i\to T^i$ respectively, and $E$ is the sum of the 
birational transform of $D$ and the reduced exceptional divisor of $\phi$.
We say $\mathcal{P}$ is \emph{bounded} if we can choose $\phi$ to be an isomorphism. 
  
A set $\mathcal{R}$ of projective pairs $(X,B)$ is said to be \emph{log birationally bounded} (respectively \emph{log bounded}) 
if the set of the corresponding couples $(X,\Supp B)$ is birationally bounded (respectively bounded).
Note that this does not put any condition on the coefficients of $B$, eg we are not requiring the 
coefficients of $B$ to be in a finite set. If $B=0$ for all the $(X,B)\in\mathcal{R}$ we usually remove the 
log and just say the set is birationally bounded (resp. bounded).

\begin{lem}\label{l-bnd-couples-bnd-degree}
Let $d,r\in \N$. Assume $\mathcal{P}$ is a set of couples $(X,D)$ where $X$ is of dimension $d$
and there is a very ample divisor $A$ on $X$ with $A^d\le r$ and $A^{d-1}D\le r$. Then $\mathcal{P}$ is bounded.
\end{lem}
\begin{proof} 
The very ample divisor $A$ gives an embedding of $X$ into some  $\PP^n$. When  $X$ is 
nondegenerate, i.e. $X$ is not contained in any hyperplane in  $\PP^n$, it is well-known that 
$n-d+1\le r$ [\ref{Eisenbud-Harris}, Proposition 0]. Therefore, we can assume $n$ is bounded depending only on $d,r$. 
We view both $X$ and $D$ as cycles on $\PP^n$.
By representability of the Chow functor on well-defined families of cycles 
[\ref{kollar-rational-curves}, Chapter I, Theorem 3.21], there exist reduced schemes $R,S$ 
and reduced closed subschemes $W\subseteq \PP^n\times R$ and $G\subseteq \PP^n\times S$ so that if  
$p\colon W\to R$ and $q\colon G\to S$ denote projections, then for each $(X,D)\in \mathcal{P}$ 
there are closed points $r\in R$ and $s\in S$ such that $X,D$ are isomorphic to the reduction of the 
fibres of $p,q$ over $r,s$, respectively. Using stratification and replacing  
$\mathcal{P}$ accordingly, we can assume $R,S$ are integral and that 
$p,q$ are surjective. 

Since $X$ is integral, we can assume $W$ is integral too and that all the fibres of $p$ are integral. On the other hand, 
we can assume that each component of $G$ is mapped onto  $S$. This ensures that the generic fibre of 
$q$ is reduced. Since  we work in characteristic zero, the geometric generic fibre is reduced too. 
Therefore, we can assume all the fibres of $q$ are reduced, hence in particular, $X,D$ are isomorphic to the 
fibres of $p,q$ over $r,s$, respectively.

Let $T=R\times S$ and consider $V:=W\times S$ and $C:=G\times R$ as closed subsets of $\PP^n\times T$.
Considering projections gives projective morphisms $h\colon V\to T$ and $e\colon C\to T$ such that
if  $X,D,r,s$ are as above, then 
the fibres of $h,e$ over $t=(r,s)$ are isomorphic to $X$ and $D$, respectively. 
Let $U\subset T$ be the points parameterising the elements of $\mathcal{P}$. 
Replacing $T$ with the closure of $U$ and using stratification, etc, as above and replacing $\mathcal{P}$ accordingly, 
we can assume  that $U$ is dense in $T$. We can assume $V$ is still integral, and $C$ is still reduced and 
that each of its components dominates $T$.
Then $C\subset V$ because for each $t\in U$, the fibre of $C\to T$ over $t$ is inside the fibre of $V\to T$. 
Now each $(X,D)\in \mathcal{P}$ is the ``fibre" of $(V,C)$ over some closed point $t$.

\end{proof}

Note that in the definition of a bounded set of couples we have $\phi^{-1}D\le C_t^i$ but equality may not 
hold in general. The proof of the previous lemma shows that we can choose the families so that equality holds. 
For reference we put this into the next lemma.

\begin{lem}\label{l-couples-bnd-full-fib}
Assume $\mathcal{P}$ is a bounded set of couples. Then under the above notation, 
we can choose $V^i, C^i$ and $V^i\to T^i$  
so that for each $(X,D)\in \mathcal{P}$ there exist an $i$ and a closed point $t\in T^i$ such that 
 $(V^i_t,C^i_t)$ is a couple isomorphic to $(X,D)$.
 \end{lem}
\begin{proof}
From the definition of bounded families we can see that there is $n$ depending only on $\mathcal{P}$ 
such that for each $(X,D)\in \mathcal{P}$ we can embed $X$ in $\PP^n$ so that degree of $X$ and 
degree of $D$ (calculated with respect to a hyperplane) are both bounded. Now the claim follows from the 
proof of Lemma \ref{l-bnd-couples-bnd-degree}.

\end{proof}

\begin{lem}\label{l-bnd-fam-intersection}
Let $\mathcal{P}$ be a bounded set of couples and $e\in \R^{>0}$. Then there is a finite set $I\subset \R$ 
depending only on $\mathcal{P}$ and $e$ satisfying the following. Let $(X,D)\in \mathcal{P}$ and  
assume $R\ge 0$ is a non-zero integral divisor on $X$ such that $K_X+D+rR\equiv 0$ for some real number $r\ge e$. 
Then $r\in I$.
\end{lem}
\begin{proof}
We can choose an effective very ample Cartier divisor $A$ on $X$ such that $(X,A+D)$ belongs to a 
bounded set of couples $\mathcal{Q}$ determined by some presentation of $\mathcal{P}$ (as in \ref{ss-bnd-couples}). 
Moreover, realising  
$A^{d-1}$ as a $1$-cycle inside the smooth locus of $X$, the intersection number $L\cdot A^{d-1}$ 
makes sense and is an integer for any integral divisor $L$. On the other hand, $(K_X+D)\cdot A^{d-1}$ belongs to some finite 
set depending only on $\mathcal{Q}$. Thus from 
$$
(K_X+D+rR)\cdot A^{d-1}=0
$$
and the assumption $r\ge e$ we deduce that $R\cdot A^{d-1}$ is bounded from above. Therefore,  
$r$ belongs to some fixed finite set $I$. Choosing $I$ to be minimal with this property, 
it depends only on $\mathcal{P}$ and $e$.

\end{proof}

\subsection{Cartier index in bounded families}

\begin{lem}\label{l-bnd-couples-bnd-Cartier-index}
 Let $\mathcal{P}$ be a bounded set of couples. Then there is a natural number $I$ 
 depending only on $\mathcal{P}$ satisfying the following. Assume $X$ is projective 
 with klt singularities and $M\ge 0$ an integral $\Q$-Cartier 
 divisor on $X$ so that $(X,\Supp M)\in \mathcal{P}$. Then $IK_{X}$ and $IM$ are Cartier. 
\end{lem}
\begin{proof}
We will use MMP similar to [\ref{HX}, Proposition 2.4].
 Assume there is a sequence $X_i,M_i$ of pairs and divisors as in the lemma such that if $I_i$ 
 is the smallest natural number so that $I_iK_{X_i}$ and $I_iM_i$ are Cartier, 
 then the $I_i$ form a strictly increasing sequence of numbers.  Perhaps after replacing 
 the sequence with a subsequence, by Lemma \ref{l-couples-bnd-full-fib}, 
we can assume there is a projective morphism $V\to T$ of varieties, a reduced divisor $C$ on 
$V$, and a dense set of closed points $t_i\in T$ such that $X_i$ is the fibre of $V\to T$ over 
$t_i$ and  $\Supp M_i$ is the fibre of $C\to T$ over $t_i$. Since $X_i$ are normal, 
replacing $V$ with its normalisation 
and replacing $C$ with its inverse image with reduced structure, we can assume $V$ is normal.

Let $\phi\colon W\to V$ be a resolution of $V$ and let $\Delta$ be  the 
reduced exceptional divisor of $\phi$. Running an MMP$/V$ on $K_W+\Delta$ with scaling of an 
ample divisor, we reach a model $V'$ on which $K_{V'}+\Delta'$ is a limit of movable$/V$ divisors (\ref{ss-MMP}).  
Let $V'\to V$ be the induced morphism and $X_i',\Delta_i'$ be the fibres of $V'\to T$ and $\Delta'\to T$ 
over $t_i$, respectively (note that $\Delta_i'=\Delta'|_{X_i'}$ and since we work in 
characteristic zero, we can assume $\Delta_i'$ is reduced). 
Now we can assume $X_i'$ are general fibres of $V'\to T$, 
hence $\Delta_i'$ is the reduced exceptional divisor of $X_i'\to X_i$.  Since $X_i$ is klt, 
we can write the pullback of $K_{X_i}$ to $X_i'$ as $K_{X_i'}+\Theta_i'$  where $\Theta_i'$ is exceptional 
with coefficients strictly less than $1$. But then since $X_i'$ are general fibres, 
$$
\Delta_i'-\Theta_i'=K_{X_i'}+\Delta_i'-(K_{X_i'}+\Theta_i')\sim_\Q K_{X_i'}+\Delta_i'/X_i
$$ 
is a limit of movable$/X_i$
 divisors, hence $\Delta_i'-\Theta_i'\le 0$ by the general negativity lemma [\ref{B-lc-flips}, Lemma 3.3] 
which in turn implies $\Delta_i'=\Theta_i'=0$ as $\Delta_i'$ is reduced. Thus $X_i'\to X_i$ is a small contraction.

There is a $\Q$-divisor $\Gamma_i'\ge 0$ which is anti-ample over $X_i$. Rescaling it we can 
assume $(X_i',\Gamma_i')$ is klt. In particular, $X_i'\to X_i$ is a $K_{X_i'}+\Gamma_i'$-negative 
contraction of an extremal face of the Mori-Kleiman cone of $X_i'$ 
because by the previous paragraph $K_{X_i'}$ is the pullback of $K_{X_i}$. 
Thus by the cone theorem [\ref{kollar-mori}, Theorem 3.7], the 
Cartier index of $K_{X_i'}$ (resp, $M_i'$) and $K_{X_i}$ (resp. $M_i$) are the same where 
$M_i'$ is the pullback of $M_i$. 

Now since $X_i'$ is a general fibre, $K_{X_i'}=K_{V'}|_{X_i'}$ which shows that the Cartier index of $K_{X_i'}$ is bounded. 
Moreover, if $C'\subset V'$ denotes the birational transform of $C$, then $\Supp M_i'$ is the fibre of $C'\to T$ over $t_i$.
Thus replacing $V,C$ with $V',C'$ we can replace  
$X_i$ with $X_i'$ and replace $M_i$ with $M_i'$, hence assume $V$ is $\Q$-factorial, so $C$ is $\Q$-Cartier. 

Pick $I$ so that $IC$ is Cartier. Let $D_i=\Supp M_i$. Then  $D_i=C|_{X_i}$, hence 
$ID_i$ is Cartier. This gives a contradiction if $M_i$ are all irreducible. 
In general, let $h_i\in\N$ be the largest number such that $M_i-h_iD_i\ge 0$. Then $M_i-h_iD_i$ has 
at least one component less than $M_i$. Thus we can apply induction on the number of components of $M_i$ 
which is a bounded number.  

\end{proof}

\begin{lem}\label{l-bnd-couples-bnd-Cartier-index-2}
 Let $d,r$ be natural numbers. 
 Then there is a natural number $I$ depending only on $d,r$ 
 satisfying the following. Suppose $X$ is a projective variety with klt singularities and 
$A$ is very ample on $X$ with $A^d\le r$ where $d=\dim X$. If $L$ is a nef integral divisor on $X$ with  
$A^{d-1}L\le r$, then $IL$ is Cartier. 
\end{lem}
\begin{proof}
We can assume $\dim X>1$ otherwise the statement holds trivially. Since $A$ is very ample and 
$A^d\le r$,  $X$ belongs to a bounded family of varieties and $A^{d-1}K_X$ is bounded from above.
Replacing $A$ with a bounded multiple and changing it linearly, and replacing $r$ accordingly, 
we can assume $(X,A)$ is dlt and $K_X+A$ is ample.
We can assume $L_A:=L|_A$ is integral. Moreover, $L_A$ is nef, $(A|_A)^{d-2}L_A\le r$ and $(A|_A)^{d-1}\le r$. 
 Thus by induction, 
there is a natural number $l>1$ depending only on $d,r$ such that $lK_A$ and $lL_A$ are Cartier. 
Moreover, since 
$$
K_A+L_A=(K_X+A+L)|_A
$$ 
is ample, we can choose $l$ such that $l(K_A+L_A)$ is base point free, 
by effective base point freeness. 

Now 
$$
h^1(l(K_X+A+L)-A)=0
$$ 
by the Kawamata-Viehweg vanishing theorem as 
$$
l(K_X+A+L)-A=K_X+L+(l-1)(K_X+A+L).
$$
 Thus lifting sections from $A$, we deduce that  
$$
h^0(l(K_X+A+L))>0.
$$ 
So $l(K_X+A+L)\sim N$ where $N$ is effective and 
$A^{d-1}N$ is bounded from above. Then $(X,\Supp N)$ belongs to a bounded family of couples, 
hence by Lemma \ref{l-bnd-couples-bnd-Cartier-index}, there is $J$ depending only on 
$d,r$ so that $JN$ and $JK_X$ are both Cartier. Therefore, $lJL$ is Cartier too. 
Now let $I=lJ$.

\end{proof}

The next lemma is useful to prove boundedness of certain birational models of weak Fano varieties.

\begin{lem}\label{l-bnd-Fano-change-model}
Let $\mathcal{P}$ be a bounded set of klt weak Fano varieties. Let $\mathcal{Q}$ be the set of 
normal projective varieties $Y$ such that 
\begin{itemize}
\item $K_Y$ is $\Q$-Cartier,

\item there is $X\in \mathcal{P}$ and a birational map $Y\bir X$, and 

\item if $\phi\colon W\to X$ and $\psi\colon W\to Y$ is a common resolution, then 
$\phi^*K_X\ge \psi^*K_Y$.
\end{itemize}

Then $\mathcal{Q}$ is bounded.
\end{lem}
\begin{proof}
Let $Y,X$ be as in the statement. By Lemma \ref{l-bnd-couples-bnd-Cartier-index}, 
there is $m\in\N$ depending only on $\mathcal{P}$ such that $-mK_X$ is Cartier. 
By the effective base point free theorem [\ref{kollar-ebpf}], we can assume $|-mK_X|$ is 
base point free. Thus $K_X$ has a klt $m$-complement $K_X+B^+$. Since 
$\phi^*K_X\ge \psi^*K_Y$, taking the crepant pullback of $K_X+B^+$ to $Y$ gives a 
klt $m$-complement $K_Y+B_Y^+$ of $K_Y$ where $B_Y^+\ge \psi_*\phi^*B^+$ is big. Now apply 
[\ref{HX}, Theorem 1.3].  

\end{proof}

\subsection{Families of subvarieties}\label{ss-cov-fam-subvar}
Let $X$ be a normal projective variety. A \emph{bounded family $\mathcal{G}$ of subvarieties} of $X$ 
is a family of (closed) subvarieties such that there are finitely many morphisms 
$V^i\to T^i$ of projective varieties together with morphisms $V^i\to X$ 
such that $V^i\to X$ embeds in $X$ the fibres of $V^i\to T^i$ over closed points, and 
each member of the family $\mathcal{G}$ is isomorphic to a fibre of 
some $V^i\to T^i$ over some closed point. Note that we can replace the $V^i\to T^i$ 
so that we can assume the set of points of $T^i$ corresponding to members of $\mathcal{G}$ 
is dense in $T^i$.
We say the family $\mathcal{G}$ is a \emph{covering family of subvarieties} of $X$ if 
the union of its members contains some non-empty open subset of $X$. In particular, this means 
$V^i\to X$ is surjective for at least one $i$.
When we say $G$ is a \emph{general member of $\mathcal{G}$} 
we mean there is $i$ such that $V^i\to X$ is surjective,  the set $A$ of points of $T^i$ corresponding 
to members of $\mathcal{G}$ is dense in $T^i$, and $G$ is the fibre of 
$V^i\to T^i$ over a general point of $A$ (in particular, $G$ is among the general fibres of 
$V^i\to T^i$). 

Note that our definition of a bounded family here is compatible with 
 \ref{ss-bnd-couples}. Indeed assume $\mathcal{G}$ is a family of subvarieties of $X$
which is bounded according to the definition in \ref{ss-bnd-couples}. Then there are finitely many possible 
Hilbert polynomials (with respect to a fixed ample divisor on $X$) of the members of the family.
Consider the Hilbert scheme  $H$ of $X$ given by the previously fixed finitely many polynomials, 
 and take the universal family $\mathcal{H}\to H$. 
There are closed subvarieties $T^i$  of $H$ and irreducible components 
$V^i$ of the reduction of $\mathcal{V}^i$, where $\mathcal{V}^i=T^i\times_H \mathcal{H}\to T^i$ is the induced family, 
so that each $G\in \mathcal{G}$ is isomorphic to a fibre of ${V}^i\to T^i$ 
over some closed point. 
By choosing the $T^i$ carefully, we can assume that, for each $i$, the members of $\mathcal{G}$ correspond to a 
dense set of fibres of $V^i\to T^i$. Since we obtained $V^i\to T^i$ from the Hilbert scheme, we have an
induced morphism $V^i\to X$ which embeds in $X$ the fibres of $V^i\to T^i$ over closed points. 
Therefore, $\mathcal{G}$ is a bounded family of subvarieties 
according to the definition in the last paragraph.

The next lemma is useful in applications when we want to replace $V^i\to T^i$ so that $V^i\to X$ becomes generically finite 
(eg, see the proof of \ref{l-sub-bnd-on-gen-subvar}).

\begin{lem}\label{l-hypersurface-section-family-subvar}
Let $f\colon V\to T$ be a contraction between smooth projective varieties and 
$g\colon V\to X$ a surjective morphism to a normal projective variety. 
 Let $t$ be a closed point of $T$ and $F$ the fibre of $f$ over $t$. 
Further assume 

$(1)$ the induced map $F\to X$ is birational onto its image,  

$(2)$ $f$ is smooth over $t$,  

$(3)$ $g$ is smooth over $g(\eta_F)$, and $g(\eta_F)$ is a smooth point of $X$ where $\eta_F$ is the generic point of $F$.\\

Let $S$ be a general hypersurface section of $T$ of sufficiently large degree passing through $t$, 
let $U=f^*S$, and assume $U\to X$ is surjective. 
Then $U$ and $S$ are smooth,  $U\to S$ is smooth over $t$, and $U\to X$ is smooth over $g(\eta_F)$.
\end{lem}
\begin{proof}
Let 
$v$ be a general closed point of $F$, and let $G$ be the fibre of $g$ over $x:=g(v)$. The 
scheme-theoretic intersection $F\cap G$ is the fibre of $F\to X$ over $x$, hence it 
is the reduced point $\{v\}$, by (1). In particular, 
$\mathcal{T}_{F,v}\cap \mathcal{T}_{G,v}=\{0\}$ where $\mathcal{T}_{F,v}$ and $\mathcal{T}_{G,v}$ 
are the tangent spaces to $F$ and $G$ 
at $v$. On the other hand, $F\cap G$ is also the fibre of the induced morphism 
$G\to T$ over $t$, so the fibre is again just $\{v\}$. 
Now since $f$ is smooth over $t$, we have an exact sequence of tangent spaces 
$$
0 \to \mathcal{T}_{F,v}\to \mathcal{T}_{V,v}\to  \mathcal{T}_{T,t} \to 0.
$$
Thus the kernel of the map $\mathcal{T}_{G,v}\to \mathcal{T}_{T,t}$ is 
$\mathcal{T}_{F,v}\cap \mathcal{T}_{G,v}$, hence $\mathcal{T}_{G,v}\to \mathcal{T}_{T,t}$ is injective. 
Therefore, $G\to T$ is a closed immersion near $v$, by Lemma \ref{l-closed-immersion} below.

Since $S$ is a general hypersurface section of $T$ of sufficiently large degree passing through $t$ 
and since $T$ is smooth, $S$ is smooth too.  Moreover, since $G$ is smooth by (3),  
 $U\cap G$ is smooth as well: indeed $U\cap G$ is smooth outside 
$v$ by [\ref{Hartshorne}, Chapter III, Remark 10.9.2] and also smooth at $v$ as $G$ is smooth and 
$G\to T$ is a closed immersion near $v$.  
  
By construction, $U\to S$ is smooth over $t$, and $U$ is a smooth variety outside $F$. 
Moreover, $U$ is smooth at every point of $F$ 
as $F$ is the fibre of $U\to S$ over $t$ which is smooth and $t\in S$ is smooth. Therefore, $U$ is smooth. 
 
Since $g$ is smooth near $G$, $\dim G=\dim V-\dim X$. So 
$$
\dim U\cap G=\dim G-1=\dim V-\dim X-1=\dim U-\dim X
$$ 
where 
we think of the scheme-theoretic intersection $U\cap G$ as the fibre of $U\to X$ over $x$. 
Thus  $U\to X$ is flat over $x$ by [\ref{Hartshorne}, Chapter III, Exercise 10.9] which 
in turn implies $U\to X$ is smooth over $x$ [\ref{Hartshorne}, Chapter III, Exercise 10.2] 
as $U\cap G$ is smooth. Therefore, $U\to X$ is smooth over $g(\eta_F)$.

\end{proof}

\begin{lem}\label{l-closed-immersion}
Let $h\colon X \to Z$ be a projective morphism of normal varieties, $x\in X$ be a closed point, 
and $z=h(x)$. Assume $h^{-1}\{z\}=\{x\}$, 
and assume the map on tangent spaces $\mathcal{T}_{X,x}\to \mathcal{T}_{Z,z}$ is injective. 
Then $h$ is a closed immersion near $x$.  
\end{lem}
\begin{proof}
Since $h^{-1}\{z\}=\{x\}$, by considering the Stein factorisation of $h$, we can see that $h$ is a finite 
morphism over $z$. Thus shrinking $X,Z$ we can assume $h$ is finite. 
If $U$ is an open neighbourhood of $x$, then $z\not\in h(X\setminus U)$ because 
$h^{-1}\{z\}=\{x\}$, hence $V=Z\setminus h(X\setminus U)$ is an open neighbourhood of $z$. 
Moreover, $h^{-1}V\subseteq U$: indeed if $y\notin U$, then $y\in X\setminus U$, so 
$h(y)\in h(X\setminus U)$, hence $h(y)\notin V$ which implies $y\notin h^{-1}V$. 
Thus every open neighbourhood of $x$ contains the inverse image of some open neighbourhood of 
$z$. This implies the induced map on stalks 
$(h_*\mathcal{O}_X)_z\to \mathcal{O}_{X,x}$ is an isomorphism. Thus 
$\mathcal{O}_{X,x}$ is a finitely generated $\mathcal{O}_{Z,z}$-module as $h_*\mathcal{O}_X$ is coherent. 

On the other hand, since the map $\mathcal{T}_{X,x}\to \mathcal{T}_{Z,z}$ is injective, 
the dual map $m_z/m_z^2\to m_x/m_x^2$ is surjective where $m_z$ and $m_x$ are the maximal ideals of 
$\mathcal{O}_{Z,z}$ and $\mathcal{O}_{X,x}$. Now apply [\ref{Hartshorne}, II, Lemma 7.4] to 
show the homomorphism $\mathcal{O}_{Z,z}\to \mathcal{O}_{X,x}$ is surjective which implies  
$\mathcal{O}_Z\to h_*\mathcal{O}_X$ is surjective near $z$. Therefore, $h$ is a closed immersion near $x$ 
since $h$ is finite. 
\end{proof}

\subsection{Potentially birational divisors}
Let $X$ be a normal projective variety and let $D$ be a big
$\Q$-Cartier $\Q$-divisor on $X$. We say that $D$ is \emph{potentially
birational} [\ref{HMX2}, Definition 3.5.3] 
if for any pair $x$ and $y$ of general closed points of $X$, possibly
switching $x$ and $y$, we can find $0 \le  \Delta \sim_\Q (1 - \epsilon)D$ 
for some $0 < \epsilon < 1$ such that 
$(X, \Delta)$ is not klt at $y$ but $(X, \Delta)$ is lc at $x$ and
$\{x\}$ is a non-klt centre. 

If $D$ is potentially birational, then $|K_X+\lceil D\rceil|$ defines a birational map 
[\ref{HMX}, Lemma 2.3.4].

\subsection{Non-klt centres}\label{ss-non-klt-centres}
In this subsection we study non-klt centres of pairs. In several places in this paper we use a 
standard technique to create covering families of non-klt centres.

(1)
The next statement is a variant of [\ref{HMX2}, Lemma 3.2.3]. 

\begin{lem}\label{l-unique-lc-place}
Let $(X,B)$ be a projective pair where $B$ is a $\Q$-boundary, and let $D\ge 0$ be an ample $\Q$-divisor. 
Let $x,y\in X$ be closed points, and assume $(X,B)$ is klt near $x$, $(X,B+D)$ is lc near $x$ with a unique non-klt
centre $G$ containing $x$, and $(X,B+D)$ is not klt near $y$. Then there exist rational numbers $0\le t\ll s\le 1$ and  
a $\Q$-divisor $0\le E\sim_\Q tD$ such that $(X,B+sD+E)$ is not klt near $y$ but it is lc near $x$ with 
a unique non-klt place, and the centre of this non-klt place is $G$.   
\end{lem}
\begin{proof}
Let $\phi\colon W\to X$ be a log resolution. 
Then $\phi^*D\sim_\Q A+C$ where $A\ge 0$ is ample and $C\ge 0$. Let 
$C'=\phi_*C$ and $D'=\phi_*(A+C)$.  Replacing $X$ with a higher resolution we can assume $\phi$ is a log 
resolution of $(X,B+D+C')$: note that here we pull back $A,C$ to the new resolution, so $A$ may no longer be 
ample but it is nef and big, hence perturbing coefficients in the exceptional components 
we can make $A$ ample again. Changing $A$ 
up to $\Q$-linear equivalence we can assume $A$ is general, so $\phi$ is a log resolution of 
$(X,B+D+D')$.

Write 
$$
K_W+\Gamma_{s,t}=\phi^*(K_{X}+B+sD+tD').
$$ 
Let $T$ be the sum of the components of $\rddown{\Gamma_{1,0}^{\ge 0}}$ 
whose image contains $x$. By assumption, $\phi(S)=G$ for every component $S$ of $T$. 
Now pick $t>0$ sufficiently small and let 
$s$ be the lc threshold of $D$ with respect to $(X,B+tD')$ near $x$. Then $s$ is sufficiently 
close to $1$. Moreover, 
$$
\rddown{\Gamma_{s,t}^{\ge 0}}\subseteq \rddown{\Gamma_{1,t}^{\ge 0}}=\rddown{\Gamma_{1,0}^{\ge 0}},
$$  
so any component of $\rddown{\Gamma_{s,t}^{\ge 0}}$  whose image contains $x$, is a component of $T$. 
Now possibly after perturbing the coefficients of $C$ and replacing $A$ accordingly, we can assume 
$\rddown{\Gamma_{s,t}^{\ge 0}}$ has only one component $S$ 
such that $x\in \phi(S)$. Since $S$ is a component of $T$, we have $\phi(S)=G$.

If  $G$ contains $y$, then let $E=tD'$, hence $(X,B+sD+E)$ is not klt near $y$. 
We can then assume $G$ does not contain $y$. Since $(X,B+D)$ is not klt near $y$, 
it has a non-klt centre $J\neq G$ containing $y$. By assumption, 
$G$ is the only non-klt centre containing $x$, so $J$ does not contain $x$. Thus there is an effective 
$\tilde{D}\sim_\Q D$ containing $J$ but not containing $x$. In particular, we can choose a 
small $\alpha>0$ (depending on $s$) so that $(X,B+sD+tD'+\alpha \tilde{D})$  is not lc near $y$. Now  
let $E=tD'+\alpha \tilde{D}$ and rename $t+\alpha$ to $t$.

\end{proof}

(2)
Let $X$ be a normal projective variety of dimension $d$ and $D$ an ample $\Q$-divisor. 
Assume $\vol(D)>(2d)^d$. Then there is a bounded family of subvarieties of $X$ such that 
for each pair $x,y\in X$ of general closed points, there is a member $G$ of the family and 
there is $0\le \Delta\sim_\Q D$ such that $(X,\Delta)$ is lc near $x$ with a unique non-klt place 
whose centre contains $x$, that centre is $G$, 
and $(X,\Delta)$ is not klt at $y$ [\ref{HMX2}, Lemma 7.1].

Now assume $A$ is an ample and effective $\Q$-divisor. 
Pick a pair $x,y\in X$ of general closed points and let $\Delta$ and $G$ be as above chosen for $x,y$.  
 If $\dim G=0$, or if $\dim G>0$ and 
$\vol(A|_G)\le d^d$, then we let $G':=G$ and 
let $\Delta':=\Delta+A$. On the other hand, if $\dim G>0$ and $\vol(A|_G)> d^d$,  then  
there is $0\le \Delta'\sim_\Q \Delta+A$ and there is a 
proper subvariety $G'\subset G$ such that $(X,\Delta')$ is lc near $x$ with a unique non-klt place 
whose centre contains $x$, that centre is $G'$,
and $(X,\Delta')$ is not klt at $y$, by [\ref{kollar-sing-pairs}, Theorem 6.8.1 and 6.8.1.3] 
and by Lemma \ref{l-unique-lc-place}.  Repeating this process $d-1$ times, we find 
$0\le \Delta^{(d-1)}\sim_\Q D+(d-1)A$ 
and a proper subvariety $G^{(d-1)}\subset G$ such that $(X,\Delta^{(d-1)})$ is lc near $x$ with a unique non-klt place 
whose centre contains $x$, that centre is $G^{(d-1)}$, and $(X,\Delta^{(d-1)})$ is not klt at $y$, and either $\dim G^{(d-1)}=0$ 
or $\vol(A|_{G^{(d-1)}})\le d^d$. In particular, all such centres $G^{(d-1)}$ form a bounded family of 
subvarieties of $X$. 

(3) 
We will need the next lemma in section 3 when we define adjunction on non-klt places.

\begin{lem}\label{l-unique-non-klt-place-connected-fibre}
Assume that 
\begin{itemize}
\item $(X,B)$ is an lc pair, 

\item $G\subset X$ is a subvariety with normalisation $F$,

\item $X$ is $\Q$-factorial near the generic point of $G$, and 

\item  there is a unique non-klt place of $(X,B)$ whose centre is $G$.
\end{itemize}

Then  if $(Y,B_Y)$ is a $\Q$-factorial dlt model of $(X,B)$ and $S$ is a component of $\rddown{B_Y}$ 
mapping onto $G$, then the induced morphism $h\colon S\to F$ is a contraction.
 Moreover,  the only non-klt centre of $(X,B)$ containing $G$ is $G$ itself.

\end{lem}
\begin{proof}
If $G$ is a divisor, then the claim is obvious, so we can assume $\dim G <\dim S$, in particular, 
$S$ is exceptional over $X$. 
Let $\Pi$ be the fibre of $Y\to X$ over a general closed point $g$ of $G$. Then $\Pi$ is connected, and
since $X$ is $\Q$-factorial near $g$,  
$\Pi$ is contained in the union of the exceptional divisors of $Y\to X$, hence contained in $\rddown{B_Y}$. 
Moreover, by the connectedness principle, $\rddown{B_Y}$ is connected near $\Pi$.

Let $R$ be the support of $(\rddown{B_Y}-S)|_S$. We show that the induced morphism $R\to F$ 
is not surjective. Assume not. Then some component $V$ of $R$ 
maps onto $G$. As $V$ is a non-klt centre of $(Y,B_Y)$, there is a non-klt place of $(Y,B_Y)$ with centre $V$, 
hence there is a non-klt place of $(X,B)$ with centre $G$ 
and this place is not $S$. This is a contradiction. 

We claim that $\Pi$ is contained in $S$. Assume not. Then 
since $\Pi$ is connected, and since $\rddown{B_Y}$ is connected near $\Pi$ and $\Pi\subset \rddown{B_Y}$, 
there is a component $T$ of $\rddown{B_Y}-S$ such that $\Supp T|_S$ intersects $\Pi$. This contradicts 
the previous paragraph. Thus $\Pi$ is contained in $S$. In particular, $\Pi$ is the fibre of $h$ over $g$, 
hence $h$ has connected general fibres. This implies $h$ is a contraction as $F$ is normal. Finally, 
applying the previous paragraph once more 
shows that in fact $\Pi$ does not intersect  $\rddown{B_Y}-S$. Therefore, 
 no lc centre of $(X,B)$ contains $G$ other than $G$ itself.

\end{proof}

\subsection{Pseudo-effective thresholds}

\begin{lem}\label{l-p-eff-thr-bnd-fam}
Let $\mathcal{P}$ be a log bounded set of log smooth projective pairs $(X,B)$. 
Then there is a number $\lambda>0$ such that if $(X,B)\in \mathcal{P}$ and if  
$K_X$ is not pseudo-effective, then $K_X+\lambda B$ is not pseudo-effective. 
\end{lem}
\begin{proof}

Perhaps after replacing $\mathcal{P}$, we can assume there exist a smooth projective 
morphism $f\colon V\to T$ of smooth varieties and a reduced divisor $S$ on $V$
with simple normal crossing singularities such that if $(X,B)\in\mathcal{P}$, 
then $X$ is a fibre of $f$ over some closed point and 
$\Supp B$ is inside the restriction of $S$ to $X$. Replacing $B$, we can assume $B=S|_X$. 
Moreover, adding to $S$ a general ample$/T$ divisor, we can assume  $S$ 
and $K_V+S$ are ample$/T$.

Let 
$$
t=\inf \{s \mid K_V+sS ~~\mbox{is pseudo-effective over $T$}\}.
$$
We can assume $t>0$.
Run an MMP on $K_V$ over $T$ with scaling of $S$ which terminates by [\ref{BCHM}]. 
This gives a minimal model $(V',tS')$  of $(V,tS)$ over $T$. Then  
$K_{V'}+tS'$ is semi-ample over $T$ but not big, hence it defines a non-birational contraction 
$V'\to W/T$. Fix some $0<\lambda<t$. Assume $F$ is a general fibre of $f$ and $F'$ the corresponding fibre 
of $V'\to T$. Then $K_{F'}+\lambda S'|_{F'}$ is not pseudo-effective, hence 
 $K_F+\lambda S|_F$ is also not pseudo-effective.

\end{proof}

\subsection{Numerical Kodaira dimension}

Let $D$ be an $\R$-Cartier $\R$-divisor on a normal projective variety $X$. The \emph{numerical 
Kodaira dimension} of $D$ denoted by $\kappa_\sigma(D)$ following [\ref{Nakayama}]. We replace $X$ by a 
resolution and replace $D$ by its pullback. If $D$ is not pseudo-effective, we let $\kappa_\sigma(D)=-\infty$. 
Otherwise $\kappa_\sigma(D)$ is the largest integer $r$ such that 
$$
\limsup_{m\to \infty} \frac{h^0(\rddown{mD}+A)}{m^r}>0
$$
for some ample Cartier divisor $A$. 

When $D$ is pseudo-effective, a kind of Zariski decomposition $D=P_\sigma(D)+N_\sigma(D)$ 
is defined in [\ref{Nakayama}] where $P_\sigma(D)$ is pseudo-effective and $N_\sigma(D)\ge 0$.

\begin{lem}\label{l-non-van-kappa(K)=0}
Let $\mathcal{P}$ be a bounded set of smooth projective varieties $X$
with $\kappa_\sigma(K_X)=0$. Then there is a number $l\in \N$ such that 
$h^0(lK_X)\neq 0$ for every $X\in \mathcal{P}$. 
\end{lem}
\begin{proof}
Perhaps after replacing $\mathcal{P}$, we can assume there is a smooth projective morphism $f\colon V\to T$ of 
smooth varieties such that 
every $X\in\mathcal{P}$ appears as a fibre of $f$ over some closed point. 
 Then  by [\ref{HMX}, Theorem 1.8 (2)], $\kappa_\sigma(K_F)=0$ for every fibre $F$ of $f$.  
Applying [\ref{Gongyo}] to the geometric generic fibre and applying semi-continuity of cohomology 
shows that there is $l\in \N$ such that $h^0(lK_{F})\neq 0$ for every fibre $F$. 

\end{proof}

\subsection{Volume of divisors}

Recall that the volume of an $\R$-divisor $D$ on a normal projective variety $X$ of dimension $d$ is defined 
as 
$$
\vol(D)=\limsup_{m\to \infty} \frac{h^0(\rddown{mD})}{m^d/d!}. 
$$

\begin{lem}\label{l-vol-kappa>0}
Let $X$ be a $\Q$-factorial normal projective variety of dimension $d$, 
$D$ be an  $\R$-divisor with $\kappa_\sigma(D)>0$, and $A$ be
an ample $\Q$-divisor. Then $\lim_{m\to \infty} \vol(mD+A)=\infty$.
\end{lem}
\begin{proof}
Replacing $X$ with a resolution and replacing $A$ appropriately, we can assume $X$ is smooth.
Replacing $D$ with $P_\sigma(D)$, we can assume $N_\sigma(D)=0$. Let $C$ be a curve cut out 
by general members of $|lA|$ for some sufficiently divisible $l\in\N$.
By  [\ref{Nakayama}, Chapter V, Theorem 1.3], we can assume $C$ does not intersect $\Bs|\rddown{mD}+A|$ for any $m\in\N$.
In particular, by  Lemma \ref{l-mov-part-lin-system}, for each $m$,
there is a resolution $\phi\colon W\to X$ such that the movable part $M$ of $|\phi^*(\rddown{mD}+A)|$ is base point 
free and the support of the fixed part $F$  does not intersect $\phi^{-1}C$.  
Then  
$$
\vol({mD}+2A)\ge \vol(\rddown{mD}+2A)=\vol(\phi^*(\rddown{mD}+2A)) \ge 
$$
$$
 \vol(M+\phi^*A)\ge M\cdot (\phi^*A)^{d-1}=\phi^*(\rddown{mD}+A)\cdot (\phi^*A)^{d-1}
=(\rddown{mD}+A)\cdot C.
$$
Since $\kappa_\sigma(D)>0$, we have $D\cdot C>0$, so the intersection number 
$(\rddown{mD}+A)\cdot C$ is not a bounded function of $m$, hence $\vol({mD}+2A)$ is not bounded.
Thus $\lim_{m\to \infty} \vol({2mD}+2A)=\infty$ which implies the lemma.

\end{proof}

\begin{lem}\label{l-vol-kappa(K)>0}
Let $d\in \N$ and let $\mathcal{P}$ be a set of pairs $(X,A)$ where $X$ is smooth projective 
of dimension $\le d$ with $\kappa_\sigma(K_X)>0$ and $A$ is very ample. 
Then for each $q\in \N$ there is $p\in\N$ such that  for every $(X,A)\in \mathcal{P}$ we have $\vol(pK_X+A)>q$. 
\end{lem}
\begin{proof}
If the statement is not true, then there exist a number $q\in\N$, a strictly increasing 
sequence of numbers $p_i\in\N$, 
and a sequence of pairs $(X_i,A_i)\in \mathcal{P}$ such that $\vol(p_iK_{X_i}+A_i)\le q$ for every $i$.
Since $A_i$ is very ample, we may assume $A_i$ is smooth, and 
since $\vol(A_i)\le q$, the pairs $(X_i,A_i)$ form a log bounded family. 
Replacing the sequence, we can assume there is a projective morphism 
$f\colon V\to T$ of smooth varieties 
and a reduced divisor $A\ge 0$ on $V$ which is simple normal crossing over $T$ and such that for each $i$, 
$X_i$ appears as a fibre of $f$ over some closed point and that $A_i=A|_{X_i}$. 

Fix a general fibre $F$ of $f$ and let $A_F=A|_F$. 
Then by [\ref{HMX}, Theorem 1.8 (3)], 
for  each $i$, 
$$
\vol\left(K_F+\frac{1}{p_i} A_F\right)=\vol\left(K_{X_i}+\frac{1}{p_i} A_i\right),
$$
hence 
$$
\vol(p_iK_F+A_F)=\vol(p_iK_{X_i}+A_i)\le q.
$$
Pick $l$ so that $lA_F\sim H_F+L_F$ where $H_F$ is ample and $L_F$ is big. Then 
$\vol(lp_iK_F+H_F)\le l^dq$. This contradicts Lemma \ref{l-vol-kappa>0}.

\end{proof}

\subsection{The restriction exact sequence}

Let $X$ be a normal projective variety, $S$ be a normal prime divisor and $L$ be an integral 
$\Q$-Cartier divisor on $X$. We have the exact sequence 
$$
0\to \mathcal{O}_X(-S)\to \mathcal{O}_X \to \mathcal{O}_S\to 0
$$
from which we get the exact sequence 
$$
\mathcal{O}_X(L)\otimes_{\mathcal{O}_X}\mathcal{O}_X(-S)\to \mathcal{O}_X(L)\otimes_{\mathcal{O}_X} \mathcal{O}_X \to \mathcal{F}:=\mathcal{O}_X(L)\otimes_{\mathcal{O}_X} \mathcal{O}_S\to 0
$$
which may not be exact on the left. This in turn gives the exact sequence 
$$
0\to \mathcal{O}_X(L-S)\to \mathcal{O}_X(L) \to \mathcal{F}\to 0.
$$
Note that $\mathcal{F}$ is an $\mathcal{O}_S$-module.

\begin{lem}\label{l-restriction-sequence}
Assume $(X,B)$ is dlt for some boundary $B$ and that $S$ is $\Q$-Cartier.
Let $U$ be the largest open subset of $X$ on which $L$ is Cartier. 
If the codimension of the complement of $S\cap U$ in $S$ is at least two, then 
$\mathcal{F}\simeq \mathcal{O}_S(L|_S)$. 
\end{lem}
\begin{proof}
This is a generalisation of [\ref{kollar-mori}, Proposition 5.26].
Note that  $L|_S$ is well-defined up to linear equivalence and it is an integral divisor.
Let $A$ be an ample Cartier divisor on $X$. Since $(X,B)$ is dlt and $S$ is $\Q$-Cartier, by duality 
[\ref{kollar-mori}, Corollary 5.27] and Serre vanishing, 
$h^i(L-nA)=0$ and  $h^i(L-S-nA)=0$ 
for any $i<d=\dim X$ and any $n\gg 0$. Thus using the third exact sequence above we 
get the exact sequence 
$$
H^i(\mathcal{O}_X(L-nA)) \to H^i(\mathcal{F}(-nA))\to H^{i+1}(\mathcal{O}_X(L-S-nA))
$$
and the vanishing $h^i(\mathcal{F}(-nA))=0$ 
for  $i<d-1$ and any $n\gg 0$. This implies $\mathcal{F}$ is a Cohen-Macaulay sheaf 
on $S$ [\ref{kollar-mori}, Corollary 5.72]. In particular, 
$\mathcal{F}$ is $\rm S_2$, hence $\mathcal{F}$ is determined  
by $\mathcal{F}|_{U\cap S}$. On the other hand, $\mathcal{O}_S(L|_S)$ 
 is reflexive, hence it is also determined by its restriction to $S\cap U$.
Now the result follows from the fact that on $S\cap U$, the two sheaves $\mathcal{F}$ and $\mathcal{O}_S(L|_S)$ 
are isomorphic. 

\end{proof}

\subsection{Descent of nef divisors}

\begin{lem}\label{l-nef-div-descent}
Let $f\colon X\to Z$ be a contraction from a smooth projective variety to a normal projective 
variety with rationally connected general fibres. Assume $M$ is a nef Cartier divisor on $X$ 
such that $M\sim_\Q 0$ on the generic fibre of $f$. Then there exist resolutions 
$\phi\colon W\to X$ and $\psi\colon V\to Z$ such that the induced map $W\bir V$ is a morphism and $\phi^*M\sim 0/V$. 
\end{lem}
\begin{proof}
Since $M$ is nef and $M\sim_\Q 0$ on the generic fibre of $f$, $M\sim_\Q N$ for some $N$ whose support 
is vertical$/Z$, i.e. its components do not intersect the generic fibre. Thus if $A$ is the pullback of a 
sufficiently ample divisor on $Z$, then $\kappa(A+M)=\dim Z$ and $\kappa(A-M)=\dim Z$. In particular, 
$$
\nu(A+M)=\kappa_\sigma(A+M)\le \kappa_\sigma(A+M+A-M)=\dim Z
$$
which means $A+M$ is a nef and \emph{good} divisor. Therefore, applying  [\ref{kaw-lin-sys-mmodel}, Proposition 2.1], 
 we can find $\phi$ and $\psi$ so that $\phi^* (A+M)\sim_\Q 0/V$, hence 
$\phi^* M\sim_\Q 0/V$. 

Replacing $X$ with 
$W$ and replacing $Z$ with $V$, we can assume $M\sim_\Q 0/Z$.
Since the general fibres of $f$ are rationally connected, $K_X$ is not pseudo-effective over $Z$.
 Running an MMP$/Z$ on $K_X$ with scaling of some ample divisor as in [\ref{BCHM}], we get a Mori fibre space $X'\to T'/Z$.  
Since $M$ is Cartier and $M\sim_\Q 0/Z$, $M'$ is Cartier, 
by the cone theorem [\ref{kollar-mori}, Theorem 3.7],  where $M'$ is the pushdown of $M$. 
Moreover, $M'\sim 0/T'$ again by [\ref{kollar-mori}, Theorem 3.7], hence $M'$ is the pullback of some nef Cartier divisor $N'$ 
on $T'$. Let $T$ be a resolution of $T'$. Then the general fibres of $T\to Z$ are rationally connected as 
they are dominated by the general fibres of $f$. Now replace $X$ with $T$ and replace $M$ with the 
pullback of $N'$ to $T$. Then apply induction on dimension.

\end{proof}

\subsection{Pairs with large boundaries}

\begin{lem}\label{l-large-boundaries}
Let $(X,B)$ be a projective $\Q$-factorial dlt pair of dimension $d$, and let $M$ be a nef Cartier divisor. 
Let $a>2d$ be a real number. 
 Then any MMP on $K_X+B+aM$ is $M$-trivial, i.e. the extremal rays in the process intersect $M$ trivially. 
If $M$ is big, then $K_X+B+aM$ is also big. 
\end{lem}
\begin{proof}
The fact that any MMP on $K_X+B+aM$ is $M$-trivial follows from boundedness of 
extremal rays [\ref{kawamata-bnd-ext-ray}]. The Cartier condition and the nefness of $M$ is preserved in the 
process by the cone theorem [\ref{kollar-mori}, Theorem 3.7]. 
Now assume $M$ is big. It is enough to show $K_X+aM$ is big. 
If $K_X+aM$ is not big, then $K_X+(a-\delta)M$ is not pseudo-effective for any small $\delta>0$.
Then we can run an MMP on $K_X+(a-\delta)M$ which 
terminates with a Mori fibre space $Y\to T$ [\ref{BCHM}]. By boundedness of extremal rays 
[\ref{kawamata-bnd-ext-ray}], $M_Y\equiv 0/T$ 
which is not possible as $M_Y$ is big. 

\end{proof}

\subsection{Divisors with log discrepancy close to $0$}

\begin{lem}\label{l-ep-to-0}
Let $d\in\N$ and $\Phi\subset [0,1]$ be a DCC set. Then there is $\epsilon>0$ depending only 
on $d$ and $\Phi$ such that if $(X,B)$ is a projective pair and $D$ is a prime divisor on 
birational models of $X$ satisfying 
\begin{itemize}
\item $(X,B)$ is lc of dimension $d$ and $(X,0)$ is klt,

\item $K_X+B\sim_\R 0$ and $B\in \Phi$, and

\item $a(D,X,B)<\epsilon$,\\
\end{itemize}
then $a(D,X,B)=0$.
\end{lem}
\begin{proof}
If the lemma does not hold, then there exist a decreasing sequence $\epsilon_i$ of numbers approaching $0$ 
and a sequence $(X_i,B_i),D_i$ of pairs and divisors as in the statement such that $0<a(D_i,X_i,B_i)<\epsilon_i$. 
If $D_i$ is a divisor on $X_i$, we let $X_i'\to X_i$ be the identity morphism. If not, then since 
$(X_i,0)$ is klt, there is a birational morphism $X_i'\to X_i$ extracting 
only $D_i$. Let $K_{X_i'}+B_i'$ be the pullback of $K_{X_i}+B_i$, and let 
$b_i=1-a(D_i,X_i,B_i)$ which is  the coefficient of $D_i$ in $B_i'$. Note that $B_i'\in \Phi':=\Phi\cup \{b_i \mid i\in\N\}$.
Replacing the sequence, we can assume $\Phi'$ is a DCC set. 
Now we get a contradiction, by [\ref{HMX2}, Theorem 1.5],  because $\{b_i \mid i\in\N\}$ is not finite. 

\end{proof}

\subsection{Boundary coefficients close to $1$}

\begin{prop}\label{p-B'-to-Theta'}
Let $d,p\in\N$ and let $\Phi\subset [0,1]$ be a DCC set. Then there is $\epsilon\in\R^{>0}$ depending only 
on $d,p,\Phi$ satisfying the following. 
Let $(X',B'+M')$ be a generalised pair of dimension $d$ with data $\phi\colon X\to X'\to Z$ 
and $M$ such that 
\begin{itemize}
\item $B'\in \Phi\cup (1-\epsilon,1]$ and $pM$ is b-Cartier,

\item  $-(K_{X'}+B'+M')$ is a limit of movable$/Z$ $\R$-divisors,

\item there is 
$$
0\le P'\sim_\R -(K_{X'}+B'+M')/Z
$$ 
such that $({X'},B'+P'+M')$ is generalised lc, and 

\item $X'$ is $\Q$-factorial of Fano type$/Z$.\\
\end{itemize}
Let $\Theta'$ be the boundary whose coefficients are 
the same as $B'$ except that we replace each coefficient in $(1-\epsilon,1)$ with $1$. That is, 
$\Theta'=(B')^{\le 1-\epsilon}+\lceil (B')^{>1-\epsilon}\rceil$.
Run an MMP$/Z$ on $-(K_{X'}+\Theta'+M')$ and let $X''$ be the resulting model. Then:\\

\begin{enumerate}
\item $(X',\Theta'+M')$ is generalised lc,
 
\item the MMP does not contract any component of $\rddown{\Theta'}$,

\item $-(K_{X''}+\Theta''+M'')$ is nef$/Z$, and 

\item $({X''},\Theta''+M'')$ is generalised lc.
\end{enumerate}
\end{prop}

\begin{proof}
Note that  $X'$ is $\Q$-factorial of Fano type$/Z$, so we can run MMP on any divisor on $X'$.

(1)
Assume this is not true. Then there exist a decreasing sequence $\epsilon_i$ of numbers approaching $0$, 
and a sequence $(X_i',B_i'+M_i')$ of generalised pairs as in the statement such that if $\Theta_i'$ is the 
divisor derived from $B_i'$ using $\epsilon_i$, then $(X_i',\Theta_i'+M_i')$ is not generalised lc.
There exist boundaries $B_i'\le \Gamma_i'\le \Theta_i'$ with $\Gamma_i'\in \Phi\cup (1-\epsilon_i,1]$ 
and a component $D_i'$ of $\Gamma_i'$ with coefficient $t_i\in (1-\epsilon_i,1)$ such that  
 $t_i$ is the generalised lc threshold of $D_i'$ with respect to $(X_i',\Gamma_i'-t_iD_i'+M_i')$. 
Replacing the sequence we can assume the union of the coefficients of all the $\Gamma_i'$ 
is a DCC set. Then we get a contradiction by the ACC for generalised lc thresholds [\ref{BZh}, Theorem 1.5].

(2) If this is not true, then 
there exist a decreasing sequence $\epsilon_i$ of numbers approaching $0$, 
and a sequence $(X_i',B_i'+M_i')$ of generalised pairs as in the statement such that if $\Theta_i'$ is the 
divisor derived from $B_i'$ using $\epsilon_i$, then the MMP contracts some component $S_i'$ of $\rddown{\Theta_i'}$.
Since all our assumptions are preserved under the MMP, we may assume the first step of the 
MMP is a birational contraction $\pi_i\colon X_i'\to \tilde{X}_i'$ which contracts $S_i'$. 
Let $R_i'$ be the corresponding extremal ray.

As $-(K_{X'}+B'+M')$ is a limit of movable$/Z$ $\R$-divisors, $-(K_{X_i'}+B_i'+M_i')\cdot R_i'\ge 0$. 
Moreover, $S_i'\cdot R_i'<0$ and if $\bar{B}'_i$ is the same as $B_i'$ except that we increase the coefficient of 
$S_i'$ to $1$, then 
$$
B_i'\le \bar{B}'_i\le \Theta_i'~~ \mbox{and} ~~ -(K_{X_i'}+\bar{B}_i'+M_i')\cdot R_i'\ge 0.
$$
Therefore, there exist boundaries $\bar{B}_i'\le \Gamma_i'\lneq \Theta_i'$ with $\Gamma_i'\in \Phi\cup (1-\epsilon_i,1]$ 
such that $-(K_{X_i'}+\Gamma_i'+M_i')\cdot R_i'=0$ and $S_i'$ is a component of $\rddown{\Gamma_i'}$. Moreover, 
there is a component $D_i'$ of $\Gamma_i'$ with coefficient $t_i\in (1-\epsilon_i,1)$ such that  
$D_i'\cdot R_i'>0$. Thus 
$$
0=a(S_i',X_i',\Gamma_i'+M_i')=a(S_i',\tilde{X}_i',\tilde{\Gamma}_i'+\tilde{M}_i')
$$
but 
$$
a(S_i',\tilde{X}_i',\tilde{\Gamma}_i'+\delta_i\tilde{D}_i'+\tilde{M}_i')=-\mu_{S_i'}\pi_i^*\delta_i{D}_i'<0
$$
for any $\delta_i>0$ where $\tilde{\Gamma}_i',\tilde{D}_i',\tilde{M}_i'$ are the pushdowns of 
${\Gamma}_i',{D}_i',{M}_i'$. This means that $t_i$ is the generalised lc 
threshold of $\tilde{D}_i'$ with respect to $(\tilde{X}_i',\tilde{\Gamma}_i'-t_i\tilde{D}_i'+\tilde{M}_i')$. 
Again this contradicts [\ref{BZh}, Theorem 1.5].

(3)
Assume this is not true. Then 
there exist a decreasing sequence $\epsilon_i$ of numbers approaching $0$, 
and a sequence $(X_i',B_i'+M_i')$ of generalised pairs as in the statement such that if $\Theta_i'$ is the 
divisor derived from $B_i'$ using $\epsilon_i$, then the MMP ends with a Mori fibre space, that is, there 
is an extremal non-birational contraction 
$X_i''\to T_i''/Z$ which is $-(K_{X_i''}+\Theta_i''+M_i'')$-negative. Under our assumptions 
 $-(K_{X_i''}+B_i''+M_i'')$ is nef over $T_i''$. 
So there exist boundaries $B_i'\le \Gamma_i'\lneq \Theta_i'$ with $\Gamma_i'\in \Phi\cup (1-\epsilon_i,1]$ 
such that 
$$
-(K_{X_i''}+\Gamma_i''+M_i'')\equiv 0/T_i''.
$$
 Moreover, 
there is a component $D_i'$ of $\Gamma_i'$ with coefficient $t_i\in (1-\epsilon_i,1)$ such that  
$D_i'$ is ample over $T_i''$. In particular, the set of the coefficients of the horizontal$/T_i''$ components 
of all the $\Gamma_i''$ put together is not a finite set but we can assume it to be DCC. 
Thus by restricting to the general fibres of $X_i''\to T_i''$  
we get a contradiction in view of the global ACC for generalised pairs [\ref{BZh}, Theorem 1.6].

(4)
 This follows from (1) since the assumptions of the proposition are stable under the MMP.

\end{proof}


\section{\bf Adjunction}

In this section we discuss various kinds of adjunction that are needed in the subsequent sections. 
In general adjunction is relating the (log) canonical divisors of two varieties that are somehow related.
We will consider adjunction for a prime divisor on a variety (divisorial adjunction) and more 
generally for a non-klt centre on a variety (adjunction on non-klt centres), and for certain 
fibrations (adjunction for fibre spaces).

\subsection{Divisorial adjunction}\label{ss-q-adjunction} 

(1) 
Let $(X',B'+M')$ be a generalised pair with data $X \overset{\phi}\to X'$ and $M$.
Assume that $S'$ is the normalisation of a component of ${B'}$ with coefficient $1$, 
and that $S$ is its birational transform on $X$. 
Replacing $X$ we may assume $\phi$ is a log resolution of $(X',B')$.
Write
$$
K_X+B+M=\phi^*(K_{X'}+B'+M'),
$$
let $B_S=(B-S)|_S$ and pick $M_S\sim_\R M|_S$. We get 
$$
K_S+B_S+M_S\sim_\R (K_X+B+M)|_S.
$$
Let $\psi$ be the induced morphism $S\to S'$ and let $B_{S'}=\psi_*B_S$ and $M_{S'}=\psi_*M_S$.
Then we get
$$
K_{S'}+B_{S'}+M_{S'}\sim_\R (K_{X'}+B'+M')|_{S'}
$$
which we refer to as  \emph{generalised (divisorial) adjunction}. 
Note that  
$$
K_{S}+B_{S}+M_{S}=\psi^*(K_{S'}+B_{S'}+M_{S'}).
$$
When $M=0$, we recover the usual \emph{divisorial adjunction}.

Assume $(X',B'+M')$ is generalised lc. Then the coefficients of $B_{S'}$  belong to $[0,1]$ [\ref{BZh},  Remark 4.8]. 
We then consider $(S',B_{S'}+M_{S'})$ as
a generalised pair which comes with data $S\overset{\psi}\to S'$ and $M_S$.
It is clear from the construction that $(S',B_{S'}+M_{S'})$ is generalised lc.

(2) 
The above generalised adjunction formula is not unique, that is, 
although $B_{S'}$ is uniquely determined depending on both $B'$ and $M$ but 
$M_S$ is not unique: if $S$ is a component of $M$, then in general $M|_S$ is only defined up to 
$\R$-linear equivalence. However, in some situations we can choose $M_S$ so that it satisfies certain properties. 
For example assume $pM$ is b-Cartier for some $p\in\N$. Then we can choose  $M_S$ such that 
$pM_S$ is b-Cartier. Indeed, we can assume $\phi$ is a log resolution of $(X',B')$ and that $pM$ is Cartier, and then 
we can choose $M_S$ such that $pM_S\sim (pM)|_S$. In particular, $pM_S$ is Cartier. Moreover,  
$$
p(K_X+B+M)\sim L
$$
where $L$ is an $\R$-divisor whose support does not contain $S$. This means that 
$p(K_X+B+M)|_S$ is well-defined up to linear equivalence which in turn implies that 
$p(K_{X'}+B'+M')|_{S'}$ is also well-defined up to linear equivalence. Therefore, we get  
$$
p(K_{S'}+B_{S'}+M_{S'})\sim p(K_{X'}+B'+M')|_{S'}.
$$

(3) 
Next we prove inversion of adjunction similar to the one for usual pairs.

\begin{lem}[Generalised inversion of adjunction]\label{l-inv-adjunction}
Let $(X',B'+M')$ be a $\Q$-factorial generalised pair with data $X\overset{\phi}\to X'$ and $M$. 
Assume $S'$ is a component of ${B'}$ with coefficient $1$, and that $(X',S')$ is plt. Let  
$$
K_{S'}+B_{S'}+M_{S'}\sim_\R (K_{X'}+B'+M')|_{S'}
$$
be given by generalised adjunction. If $({S'},B_{S'}+M_{S'})$ is generalised lc, then 
$(X',B'+M')$ is generalised lc near $S'$.
\end{lem}
\begin{proof}
Assume $(X',B'+M')$ is not generalised lc near $S'$.
 We can assume $\phi$ is a log resolution. Write 
$$
K_X+B+M=\phi^*(K_{X'}+B'+M').
$$
By assumption $\phi(B^{>1})$ intersects $S'$.
 Pick $\alpha\in (0,1)$ sufficiently close to $1$,   
let $\Gamma'=(1-\alpha)S'+\alpha B'$, and write 
$$
K_X+\Gamma+\alpha M=\phi^*(K_{X'}+\Gamma'+\alpha M').
$$
Then $({X'},\Gamma'+\alpha M')$ is not generalised lc near $S'$.
On the other hand, since $(X',S')$ is plt and $({S'},B_{S'}+M_{S'})$ is generalised lc, 
$({S'},\Gamma_{S'}+\alpha M_{S'})$ is generalised klt where
$$
K_{S'}+\Gamma_{S'}+\alpha M_{S'}\sim_\R(K_{X'}+\Gamma'+\alpha M')|_{S'}
$$
is generalised adjunction. Thus replacing $B'$ with $\Gamma'$ and $M$ with $\alpha M$, 
we can assume $({S'},B_{S'}+M_{S'})$ is generalised klt. 

Pick an ample divisor $A\ge 0$ and an effective divisor $C\ge 0$ on $X$ such that $A+C\sim 0/X'$ 
and that $S$ is not a component of $C$.
Let $\epsilon>0$ be small and pick a general $0\le G\sim_\R \epsilon A+M/X'$. 
Let $\Delta:=B+G+\epsilon C$. Then $K_X+\Delta\sim_\R 0/X'$, hence 
$K_X+\Delta=\phi^*(K_{X'}+\Delta')$ where $\Delta'$ is the pushdown of $\Delta$. 
In particular, $(X',\Delta')$ is not lc near $S'$ as $\Delta\ge B$. On the other hand, 
by assumption,  $(S,B_{S})$ is sub-klt where $K_{S}+B_{S}=(K_{X}+B)|_{S}$.   
This implies $(S,\Delta_S)$ is also sub-klt by construction of $\Delta$ where $K_S+\Delta_S=(K_X+\Delta)|_S$.
Therefore, letting $K_{S'}+\Delta_{S'}=(K_{X'}+\Delta')|_{S'}$ we see that $(S',\Delta_{S'})$ is klt 
because $K_S+\Delta_S$ is the pullback of $K_{S'}+\Delta_{S'}$.
 This contradicts the usual inversion of 
adjunction [\ref{kawakita}].  

\end{proof}

(4) 
The next result is similar to [\ref{PSh-II}, Proposition 3.9] and [\ref{BZh}, Proposition 4.9].

\begin{lem}\label{l-div-adj-dcc}
Let $p\in \N$ and $\mathfrak{R}\subset [0,1]$ be a finite set of rational numbers.
Then there is a finite set of rational numbers $\mathfrak{S}\subset [0,1]$ 
depending only on $p$ and $\mathfrak{R}$ satisfying the following. 
Assume 
\begin{itemize}

\item $(X',B'+M')$ is generalised lc of dimension $d$ with data $X\overset{\phi} \to X'$ and $M$,

\item $S'$ is the normalisation of a component of $\rddown{B'}$,  

\item $B'\in  \Phi(\mathfrak{R})$ and $pM$ is b-Cartier, and 

\item $(S',B_{S'}+M_{S'})$ is the generalised pair determined by generalised adjunction
$$
K_{S'}+B_{S'}+M_{S'}\sim_\R (K_{X'}+B'+M')|_{S'}.
$$ 
\end{itemize}

Then $B_{S'}\in  \Phi(\mathfrak{S})$. 
\end{lem}
\begin{proof}
By taking hypersurface sections, we can reduce the lemma to the case when $X'$ is a surface. 
In this case $K_{X'}+B'$ is automatically $\R$-Cartier and $(X',B')$ is lc [\ref{BZh}, Remark 4.8]. 
Let $V$ be a closed point on $S'$ and assume $\mu_VB_{S'}<1$. In particular, if we use usual adjunction 
to write $K_{S'}+\tilde{B}_{S'}=(K_{X'}+B')|_{S'}$, then $\mu_V\tilde{B}_{S'}\le \mu_VB_{S'}<1$.
Then by inversion of adjunction on surfaces [\ref{Shokurov-log-flips}, Corollary 3.12], 
$(X',B')$ is plt near the image of $V$. Thus shrinking $X'$ we can assume $S'\to X'$ 
maps $S'$ isomorphically onto its image, and that $X'$ is $\Q$-factorial since $\dim X'=2$. 

Assume $B'=\sum_{i=1}^{n} b_iB_i'$ where $B_i'$ are irreducible components. By 
[\ref{Shokurov-log-flips}, Proposition 3.9 and Corollary 3.10], 
there is $l\in \N$ depending only on $(X',S')$ such that the Cartier index near $V$ of any Weil divisor on $X'$ 
divides $l$ and that 
$$
\mu_V\tilde{B}_{S'}=1-\frac{1}{l}+\sum_{i=1}^{n} \frac{b_i\alpha_i}{l}
$$ 
where $\alpha_i\in \Z^{\ge 0}$. 

Replacing $\phi$ we can assume $X,S$ are smooth. Write $\phi^*M'=M+E$. Since $M$ is nef$/X'$, 
$E$ is effective. Since $pM$ is b-Cartier 
and $X$ is smooth, $pM$ is Cartier, so $pM'$ is integral and $lpM'$ is Cartier. Thus $lpE$ is Cartier. 
Now by definition $B_{S'}=\tilde{B}_{S'}+E_{S'}$ where $E_{S'}$ is the pushdown of $E|_S$. 
In particular, $lpE_{S'}$ is integral. Therefore, 
$$
\mu_VB_{S'}=1-\frac{1}{l}+\sum_{i=1}^{n} \frac{b_i\alpha_i}{l}+\frac{\beta}{pl}
$$ 
for some $\beta \in \Z^{\ge 0}$. 

Expanding $\mathfrak{R}$ we can assume $\frac{p-1}{p}\in \mathfrak{R}$, hence 
$\frac{1}{p}\in \Phi(\mathfrak{R})$. Put $\alpha_{n+1}:=\beta$ and $b_{n+1}:=\frac{1}{p}$. 
So we can write  
$$
\mu_VB_{S'}=1-\frac{1}{l}+\sum_{i=1}^{n+1} \frac{b_i\alpha_i}{l}.
$$ 
For each $i$ there is $r_i\in \mathfrak{R}$ 
and $m_i\in \N$ such that $b_i=1-\frac{r_i}{m_i}$. Let $s:=1-\sum_{i=1}^{n+1} b_i{\alpha_i}$. Then
$\mu_VB_{S'}=1-\frac{s}{l}$ and $s> 0$.
Removing the zero terms and re-arranging the indexes we can assume $s=1-\sum_{i=1}^{t} b_i{\alpha_i}$, and  
that $b_i\alpha_i>0$ for every $i$ otherwise $\mu_VB_{S'}=1-\frac{1}{l}$ is just a 
standard coefficient. 
Note that since $s>0$, we have $\sum_{i=1}^{t} b_i{\alpha_i}<1$, hence 
$\alpha_i$ are bounded. Moreover, $b_i\ge 1-\frac{1}{m_i}$ which means 
$b_i\ge \frac{1}{2}$ if $m_i>1$. 

Now assume $t=1$. Then either $m_1=1$, 
or $m_1>1$ and $\alpha_1=1$, hence in any case $s=\frac{s'}{m_1}$ where there are only finitely many possibilities 
for $s'$. Thus 
$$
\mu_VB_{S'}=1-\frac{s}{l}=1-\frac{s'}{m_1l}\in \Phi(\mathfrak{S})
$$ 
if we choose $\mathfrak{S}$ so that 
it contains all such $s'$. We can then assume $t>1$. If $m_i=1$ for every $i$, then there are finitely many possibilities 
for 
$$
s=1-\sum_{i=1}^{t} (1-r_i){\alpha_i}
$$ 
and so we can choose $\mathfrak{S}$ to contain all such $s$.
But if some $m_i>1$, say  $m_1>1$, then $\alpha_1=1$ and $m_i=1$ for $i>1$. Then 
$$
s=\frac{r_1}{m_1}-\sum_{i=2}^{t} (1-r_i){\alpha_i},
$$
 hence $m_1$ is bounded, so $s=\frac{s'}{m_1}$ 
for some $s'$ for which there are only finitely many possibilities. As before taking 
$\mathfrak{S}$ so that it contains all such $s'$, we get $\mu_VB_{S'}\in \Phi(\mathfrak{S})$.

\end{proof}

\subsection{Adjunction for fibre spaces}\label{ss-adj-fib-spaces}
(1)
Let $(X,B)$ be a projective sub-pair and let $f\colon X\to Z$ be a contraction with $\dim Z>0$ 
such that $(X,B)$ is sub-lc near the 
generic fibre of $f$, and $K_X+B\sim_\R 0/Z$. 
Below we recall a construction based on  [\ref{kaw-subadjuntion}] giving 
a kind of canonical bundle formula which we refer 
to as \emph{adjunction for fibre spaces} [\ref{ambro-adj}][\ref{PSh-II}, \S 7].

For each prime divisor $D$ on $Z$ we let 
$t_D$ be the lc threshold of $f^*D$ with respect to $(X,B)$ over the generic point of $D$, that is, 
$t_D$ is the largest number so that $(X,B+t_Df^*D)$ is sub-lc over the generic point of $D$. 
Of course $f^*D$ may not be well-defined everywhere but at least it is defined over the 
smooth locus of $Z$, in particular, near the generic point of $D$, and that is all we need. 
Next let $b_D=1-t_D$, and then define $B_Z=\sum b_DD$ where the sum runs over all the 
prime divisors on $Z$. 

By assumption, $K_X+B\sim_\R f^*L_Z$ for some $\R$-Cartier 
$\R$-divisor $L_Z$ on $Z$. Letting $M_Z=L_Z-(K_Z+B_Z)$ we get 
$$
K_X+B\sim_\R f^*(K_Z+B_Z+M_Z).
$$
We call $B_Z$ the \emph{discriminant part} and ${M_Z}$ the \emph{moduli part} of adjunction. 
Obviously $B_Z$ is uniquely determined but $M_Z$  is determined only up to $\R$-linear equivalence 
because it depends on the choice of $L_Z$.

Now let $\phi\colon X'\to X$ and $\psi\colon Z'\to Z$ be birational morphisms from normal 
projective varieties and assume the induced map $f'\colon X'\bir Z'$ is a morphism. 
Let $K_{X'}+B'=\phi^*(K_X+B)$. Similar to above we can define a discriminant divisor $B_{Z'}$, and 
taking $L_{Z'}=\psi^*L_Z$ gives a moduli divisor $M_{Z'}$ so that 
$$
K_{X'}+B'\sim_\R f'^*(K_{Z'}+B_{Z'}+M_{Z'})
$$ 
and $B_Z=\psi_*B_{Z'}$ and $M_Z=\psi_*M_{Z'}$.

(2) 
We want to show $M_Z$ depends only on $(X,B)$ near the generic fibre of $f$, even birationally. 
We make this more precise. 
In addition to the assumptions of (1), suppose we are given another projective sub-pair $(\overline{X},\overline{B})$ 
and a contraction $\overline{X}\to Z$ such that $K_{\overline{X}}+\overline{B}\sim_\R 0/Z$. 
Moreover, suppose we have a birational map $X\bir \overline{X}/Z$, and suppose there is 
a common resolution of $X$ and $\overline{X}$ on which the pullback of $K_X+B$ and $K_{\overline{X}}+\overline{B}$ 
are equal near the generic fibre over $Z$. 
As in (1) we can define the discriminant and moduli parts 
$\overline{B}_{Z'}$ and $\overline{M}_{Z'}$ of adjunction on the birational model $Z'$ of $Z$, for the pair $(\overline{X},\overline{B})$ 
and the contraction $\overline{X}\to Z$.

\begin{lem}\label{l-mod-part-generically}
Under the above notation and assumptions, $M_{Z'}\sim_\R \overline{M}_{Z'}$.
\end{lem}
\begin{proof}
Replacing both $X$ and $\overline{X}$ with a common resolution over $Z'$, replacing $K_X+B$ and 
$K_{\overline{X}}+\overline{B}$ with their crepant pullbacks, and replacing $Z$ with $Z'$ we can assume 
$\overline{X}=X$, $Z'=Z$, and that $\overline{B}=B$ near the generic fibre. Then $\overline{B}-B$ 
is vertical over $Z$ and $\overline{B}-B\sim_\R 0/Z$. So $\overline{B}-B=f^*P$ for some $P$ on $Z$.
Therefore, by definition of the discriminant part, $\overline{B}_{Z}=B_{Z}+P$ 
from which we get  $\overline{M}_{Z}\sim_\R M_{Z}$ because 
$$
K_Z+B_Z+M_Z+P\sim_\R K_Z+\overline{B}_{Z}+\overline{M}_{Z}.
$$
\end{proof}

(3) When $(X,B)$ is lc over the generic point of $Z$  the moduli part has nice properties. 

\begin{thm}\label{t-adj-fib-spaces}
Under the notation and assumptions of $(1)$, suppose $(X,B)$ is lc near the generic fibre of $f$. 
Let $Z'$ be a sufficiently high resolution of $Z$. Then 

$(\rm i)$ $M_{Z'}$ is pseudo-effective, and 

$(\rm ii)$ if $B$ is a $\Q$-divisor, then 
$M_{Z'}$ is nef and for any birational morphism $Z''\to Z'$ from a normal projective 
variety, $M_{Z''}$ is the pullback of $M_{Z'}$.
\end{thm}
\begin{proof}
Let  $\Gamma$ be the divisor obtained from $B$ by  
removing components that are vertical$/Z$. Since $(X,B)$ is lc over the generic point of $Z$, 
we have $\Gamma\ge 0$ as any component of $B$ with negative coefficient is vertical$/Z$, and 
over the generic point of $Z$ we have $\Gamma=B$.
Let $\phi\colon W\to X$ be a log resolution of $(X,B)$ and let $\Gamma_W$ be the sum of the horizontal$/Z$ part of the 
reduced exceptional divisor of $\phi$
and the birational transform of $\Gamma$. We can choose $W$ so that every lc centre of $(W,\Gamma_W)$ is 
horizontal$/Z$.
Running an MMP on $K_W+\Gamma_W$ over $X$ we 
reach a model $X'$ so that $(X',\Gamma')$ is a $\Q$-factorial dlt model of $(X,B)$ over the generic 
point of $Z$. Let $K_{X'}+B'$ be the pullback of $K_X+B$. 
Replacing $(X,B)$ with $(X',B')$ and replacing $(X,\Gamma)$ with $(X',\Gamma')$
 we can assume $(X,\Gamma)$ is $\Q$-factorial dlt with every lc centre horizontal$/Z$, 
and that every component of ${\Gamma}$ is horizontal$/Z$. 

We prove (ii) first. 
By [\ref{B-lc-flips}, Theorem 1.4]
we can replace $X$ with a model on which $K_X+\Gamma$ is semi-ample over $Z$, hence 
defining a contraction $X\to Y/Z$. Since $K_X+\Gamma\sim_\Q 0/Y$, by Lemma \ref{l-mod-part-generically}, we can 
replace $Z$ with $Y$ and replace $B$ with $\Gamma$, hence assume $(X,B)$ is $\Q$-factorial dlt.  
The claim then follows from [\ref{FG-lc-trivial}] (this relies on 
[\ref{ambro-lc-trivial}] in the klt case).

Now we prove (i). By standard arguments we can find $\Q$-divisors $B^i$ and numbers $\alpha_i\in [0,1]$ 
with $\sum \alpha_i=1$ such that $B=\sum \alpha_iB^i$, $(X,B^i)$ is lc over the generic point 
of $Z$, and $K_X+B^i\sim_\Q 0/Z$.
Let $B^i_{Z}$ (resp. $B^i_{Z'}$) and $M^i_{Z}$ (resp. $M^i_{Z'}$) be the discriminant and moduli parts on $Z$ (resp. $Z')$ 
defined for $(X,B^i)$ over $Z$.  
Let $D$ be a prime divisor on $Z$ and let $t_i$ be the lc threshold of $f^*D$ with respect to 
$(X,B^i)$ over the generic point of $D$. Then $(X,B+(\sum \alpha_it_i)f^*D)$ is lc 
over the generic point of $D$, hence $t\ge \sum \alpha_it_i$ where $t$ is the lc 
threshold of $f^*D$ with respect to $(X,B)$ over the generic point of $D$.
Then 
$$
1-t\le 1-\sum \alpha_it_i=\sum \alpha_i-\sum \alpha_it_i=\sum \alpha_i(1-t_i)
$$ 
from which we deduce $B_{Z}\le \sum \alpha_iB^i_{Z}$. 
 Similarly we can prove $B_{Z'}\le \sum \alpha_iB^i_{Z'}$. Then from 
$$
K_{Z'}+B_{Z'}+M_{Z'}\sim_\R \sum \alpha_i(K_{Z'}+B_{Z'}^i+M_{Z'}^i)
$$ 
we deduce that $M_{Z'}\sim_\R \sum \alpha_iM^i_{Z'}+P$ 
for some $P\ge 0$. 
Therefore, $M_{Z'}$ is pseudo-effective as each $M^i_{Z'}$ is nef.
  
\end{proof}

(4) Next we relate the singularities of $(X,B)$ and $(Z,B_Z+M_Z)$ in a rough sense.

\begin{lem}\label{l-adj-fib-eps-lc}
Let $\epsilon\in \R$. Under the notation and assumptions of $(1)$, suppose there is a prime divisor $S$ on 
some birational model of $X$ such that the log discrepancy $a(S,X,B)\le \epsilon$ 
and that $S$ is vertical over $Z$. Then there is a resolution $Z'\to Z$ and a component $T$ of 
$B_{Z'}$ with coefficient $\ge 1-\epsilon$. 
\end{lem}
\begin{proof}
Pick resolutions $X'\to X$ and $Z'\to Z$ so that the induced map $f'\colon X'\bir Z'$ is a morphism 
and so that $S$ is a prime divisor on $X'$ and its image on $Z'$ is a prime divisor $T$. 
Let $K_{X'}+B'$ be the pullback of $K_X+B$. Since $a(S,X,B)\le \epsilon$, 
the coefficient of $S$ in $B'$ is at least $1-\epsilon$. Thus the lc threshold of $f'^*T$ with 
respect to $(X',B')$ over the generic point of $T$ is at most $\epsilon$. Therefore, 
the coefficient of $T$ in $B_{Z'}$ is at least $1-\epsilon$.

\end{proof}

\subsection{Adjunction on non-klt centres}
In this subsection we recall a kind of adjunction introduced in {[\ref{HMX2}, Theorem 4.2]} similar 
to the so-called sub-adjunction formula [\ref{kaw-subadjuntion}], and then  
prove some new results in this direction. 

\begin{constr}\label{constr-adjunction-non-klt-centre}
Assume the following setting:
\begin{itemize}
\item $(X,B)$ is a projective klt pair, 

\item $G\subset X$ is a subvariety with normalisation $F$,

\item $X$ is $\Q$-factorial near the generic point of $G$,

\item $\Delta\ge 0$ is an $\R$-Cartier divisor on $X$, and 

\item  $(X, B + \Delta)$ is lc near the generic point of $G$, 
and there is a unique non-klt place of this pair whose centre is $G$.
\end{itemize}\

We will define an $\R$-divisor $\Theta_F$ on $F$ with coefficients in $[0,1]$. This then gives an 
 adjucntion formula
$$
K_F + \Theta_F+P_F\sim_\R (K_X + B + \Delta)|_F
$$
where in general $P_F$ is determined only up to $\R$-linear equivalence. 

Let $\Gamma$ be the sum of $(B+\Delta)^{<1}$ and the support of $(B+\Delta)^{\ge 1}$. 
Put $N=B+\Delta-\Gamma$ which is supported in $\rddown{\Gamma}$.
Let $\phi\colon W\to X$ be a log resolution of $(X,B+\Delta)$ and let $\Gamma_W$ be the sum of the reduced exceptional 
divisor of $\phi$ and the birational transform of $\Gamma$. Let 
$$
N_W=\phi^*(K_X+B+\Delta)-(K_W+\Gamma_W).
$$
Then $\phi_*N_W=N\ge 0$ and $N_W$ is supported in $\rddown{\Gamma_W}$. 
Now run an MMP$/X$ on $K_W+\Gamma_W$ with scaling of some 
ample divisor $H$. We reach a model $Y$ on which $K_Y+\Gamma_Y$ is a limit of movable$/X$ $\R$-divisors 
(\ref{ss-MMP}). 
Applying the general negativity lemma (cf. [\ref{B-lc-flips}, Lemma 3.3]), we deduce $N_Y\ge 0$. 
In particular, if $U\subseteq X$ is the set of points where $(X,B+\Delta)$ is lc, then 
$N_Y=0$ over $U$ and $(Y,\Gamma_Y)$ is a $\Q$-factorial dlt model of $(X,B+\Delta)$ over $U$. 
By assumption, $(X,B+\Delta)$ is lc but not klt at the generic point of $G$. By 
Lemma \ref{l-unique-non-klt-place-connected-fibre}, no non-klt centre of the pair contains $G$ apart from $G$ 
itself, hence we can assume there is a unique
component $S$ of $\rddown{\Gamma_Y}$ mapping onto $G$. 
Moreover,  $G$ is not inside the image of $N_Y$.

Let $h\colon S\to F$ be the morphism induced by $S\to G$. By 
Lemma \ref{l-unique-non-klt-place-connected-fibre}, $h$ is a contraction.
By divisorial adjunction we can write 
$$
K_S+\Gamma_S+N_S=(K_Y+\Gamma_Y+N_Y)|_S\sim_\R 0/F
$$
where $N_S=N_Y|_S$ is vertical over $F$. If $S$ is exceptional over $X$, then
let $\Sigma_Y$ be the sum of the exceptional$/X$ divisors on $Y$ plus 
the birational transform of $B$. Otherwise let $\Sigma_Y$ 
be the sum of the exceptional$/X$ divisors on $Y$ plus 
the birational transform of $B$ plus $(1-\mu_GB)S$. In any case, $S$ is a 
component of $\rddown{\Sigma_Y}$ and $\Sigma_Y\le \Gamma_Y$.
Applying adjunction again we get 
$K_S+\Sigma_S=(K_Y+\Sigma_Y)|_S$. Obviously $\Sigma_S\le \Gamma_S$. 

Now we define $\Theta_F$: 
for each prime divisor $D$ on $F$, let $t$ be the lc threshold of $h^*D$ with respect to 
$(S,\Sigma_S)$ over the generic point of $D$, and then let $\mu_D\Theta_F:=1-t$. 
Note that $h^*D$ is defined only over the generic 
point of $D$ as $D$ may not be $\Q$-Cartier. Same applies to similar definitions below, e.g. 
see proof of \ref{l-subadjunction-integral-div}.

\end{constr}

\begin{thm}[{[\ref{HMX2}, Theorem 4.2]}]\label{t-subadjunction}
Let $d\in\N$ and let $\Phi$ be a subset of $[0, 1]$ which contains $1$. 
Let $X,B,\Delta,G,F,\Theta_F,P_F$ be as in Construction 
\ref{constr-adjunction-non-klt-centre} and assume $\dim X=d$ and $B\in \Phi$.
Then the coefficients of $\Theta_F$ belong to 
$$
\Psi:=\{a \, | \, 1-a \in {\rm LCT}_{d-1}(D(\Phi))\} \cup \{1\}
$$
and $P_F$ is pseudo-effective.
 
Now suppose in addition that $G$ is a general member of a covering family of subvarieties of $X$.
Let $\psi \colon F' \to F$ be a log resolution of $(F,\Theta_F)$, and let $\Theta_{F'}$ be the sum
of the birational transform of $\Theta_F$ and the reduced exceptional divisor of $\psi$. Then
$$
K_{F'} + \Theta_{F'} \ge (K_X + B) |_{F'}.
$$
\end{thm}

For convenience we explain some of the notation in the theorem and also prove the claim of pseudo-effectivity of 
$P_F$ in the next paragraph. The set $D(\Phi)$ is a set associated to $\Phi$ with the properties: 
$D(\Phi)$ is DCC iff $\Phi$ is DCC 
and the coefficients of $\Sigma_S$ in Construction \ref{constr-adjunction-non-klt-centre} 
belong to $D(\Phi)$ (see [\ref{HMX2}, 3.4] and [\ref{PSh-II}, Proposition 3.9]). 
The set ${\rm LCT}_{d-1}(D(\Phi))$ stands for the set of all lc thresholds of integral effective divisors 
with respect to pairs $(S,\Gamma)$ of dimension $d-1$ such that $\Gamma\in D(\Phi)$.
In particular, if $\Phi$ is DCC, then $\Psi$ is also DCC [\ref{HMX2}, Theorem 1.1].

Now recall that $P_F$ is defined by the relation  
$$
K_F+\Theta_F+P_F\sim_\R (K_X+B+\Delta)|_F.
$$
Using the notation of \ref{constr-adjunction-non-klt-centre}, we have 
$$
K_S+\Gamma_S+N_S\sim_\R h^*(K_F+\Theta_F+P_F).
$$
Let $\Delta_F$ and $R_F$ be the discriminant and the moduli parts of adjunction for $(S,\Gamma_S+N_S)$ over $F$.
Then $\Delta_F+R_F\sim_\R \Theta_F+P_F$. 
Since $\Sigma_S\le \Gamma_S+N_S$, we have $\Theta_F\le \Delta_F$, hence  $P_F-R_F\sim_\R \Delta_F-\Theta_F\ge 0$.  
By Theorem \ref{t-adj-fib-spaces}, $R_F$ is pseudo-effective which implies $P_F$ is pseudo-effective too.

In the rest of this subsection we prove further results regarding the above adjunction. 
The proof of \ref{l-sub-bnd-on-gen-subvar} contains a proof of the last claim of \ref{t-subadjunction} (when 
$G$ is a member of a covering family) based on the proof of [\ref{HMX2}, Theorem 4.2] but with more details.

\begin{lem}\label{l-subadjunction-integral-div}
Let $(X,B), \Delta, G,F,$ and $\Theta_F,$ be as in Construction \ref{constr-adjunction-non-klt-centre}. 
Let $M\ge 0$ be a $\Q$-Cartier $\Q$-divisor on $X$ with coefficients $\ge 1$  
and such that $G\nsubseteq \Supp M$. Then for every component $D$ of $M_F:=M|_F$ we have $\mu_D(\Theta_F+M_F)\ge 1$. 
\end{lem}
\begin{proof}
We will use the notation introduced in Construction \ref{constr-adjunction-non-klt-centre}.
Let $\Sigma_Y':=\Sigma_Y+M_Y$ and $\Sigma_S'=\Sigma_S+M_S$ where $M_Y=M|_Y$ and $M_S=M_Y|_S$. 
We define $\Theta_F'$ similar to $\Theta_F$. That is, for each prime divisor $D$ on $F$, 
let $t'$ be the lc threshold of $h^*D$ with respect to 
$(S,\Sigma_S')$ over the generic point of $D$, and then let $\mu_D\Theta_F':=1-t'$. 
It is easy to see $t'+\mu_DM_F=t$. This means $\Theta_F'=\Theta_F+M_F$. On the other hand, 
each component $L$ of $M_Y$ is either exceptional or non-exceptional over $X$: in the former 
case $L$ is a component of $\rddown{\Sigma_Y}$ as $\rddown{\Sigma_Y}$ contains all the exceptional 
divisors, and in the latter case $L$ is a component of $\rddown{M_Y}$ as the coefficients of $M$ are at least $1$.
Combining this with the fact that $S$ is not a component of $M_Y$, we get 
$$
\Supp M_Y\subseteq \Supp \rddown{\Sigma_Y+M_Y-S}=\Supp \rddown{\Sigma_Y'-S}.
$$ 
Therefore, 
$$
\Supp M_S\subseteq (\Supp \rddown{\Sigma_Y'-S})|_S \subseteq \Supp \rddown{\Sigma_S'}
$$ 
where the second inclusion follows from the following property of adjunction: if $T$ is a reduced divisor on $Y$ not 
containing $S$ and if $K_S+T_S=(K_S+S+T)|_S$ is given by adjunction, then $\Supp T|_S\subseteq \Supp \rddown{T_S}$. 

Now as $h^*M_F=M_S$,  
if $D$ is a component of $M_F$, then every component of $h^*D$ mapping onto 
$D$ is a component of $\rddown{\Sigma_S'}$. Thus $t'\le 0$ where as above $t'$ is the  
lc threshold of $h^*D$ with respect to 
$(S,\Sigma_S')$ over the generic point of $D$. Therefore, 
$$
\mu_D(\Theta_F+M_F)=\mu_D\Theta_F'=1-t'\ge 1.
$$
\end{proof}

\begin{lem}\label{l-sub-bnd-on-gen-subvar}
Let $(X,B), \Delta, G,F,$ and $\Theta_F$ be as in Construction \ref{constr-adjunction-non-klt-centre}. 
Assume that $G$ is a general member of a covering family of subvarieties, and 
assume $(X,B)$ is $\epsilon$-lc for some $\epsilon> 0$. 
Then there is a sub-boundary $B_F$ on $F$ such that $K_{F}+B_{F}=(K_X+B)|_{F}$,  
that $(F,B_F)$ is sub-$\epsilon$-lc, and $B_{F}\le \Theta_{F}$.
\end{lem}
\begin{proof}
We first give a short summary of the proof. We consider the covering family to which $G$ belongs, and then 
derive another family $W'\to R'$ such that the induced morphism  $W'\to X$ is generically finite.
Next we let $K_{W'}+B_{W'}$ be the pullback of $K_X+B$, and restricting to the fibre $F'$ of $W'\to R'$ 
corresponding to $G$, we get $K_{F'}+B_{F'}=(K_{W'}+B_{W'})|_{F'}$. Pushing down $B_{F'}$ to $F$ 
gives $B_F$. Afterwards we show $(F,B_F)$ is sub-$\epsilon$-lc. The rest of the proof is devoted to showing 
$B_F\le \Theta_F$ by relating the coefficients of $\Theta_F$ to lc thresholds of certain divisors on $X$.\\ 

\emph{Step 1.}
In this step we consider the covering family to which $G$ belongs.
There is a contraction $\tilde{f}\colon \tilde{V}\to \tilde{T}$ of projective varieties 
such that $G$ is  a general fibre of $\tilde{f}$, and there is a surjective morphism $\tilde{V}\to X$ whose restriction  
to each fibre of $\tilde{f}$ over a closed point is a closed immersion (see \ref{ss-cov-fam-subvar}). 
Taking normalisations of $\tilde{V}$ and $\tilde{T}$, we get a 
contraction ${f}\colon {V}\to {T}$ of normal projective varieties 
such that the fibre of $f$ corresponding to $G$ is just $F$ the normalisation of $G$.
Taking resolutions of $V$ and $T$ we get a contraction $f'\colon V'\to T'$ of smooth 
projective varieties. Letting $F'$ be the fibre of $f'$ corresponding to $F$, 
we see that the induced morphism ${F}'\to X$ is birational onto its image. Moreover, we can assume
 there is a Cartier divisor $P\ge 0$ on $X$ whose support contains $\Supp B$ and the 
singular locus of $X$ and such that $Q':=\Supp g'^*P$ is relatively simple normal crossing over some open subset of $T'$ 
where $g'$ is the induced map $V'\to X$.\\

\emph{Step 2.} 
In this step we consider another covering family which is generically finite over $X$.
Fixing $G$, by construction, $f'$ is smooth over $t'=f'(F')$, $g'$ is smooth over $g'(\eta_{F'})$, 
and $g'(\eta_{F'})$ is smooth on $X$.   
If $g'$ is generically finite, then let $W':=V'$ and $R':=T'$. If not,  
applying Lemma \ref{l-hypersurface-section-family-subvar}, 
there is a general smooth hypersurface section $H'$ of $T'$ passing through $t'$ 
so that if $U'=f'^*H'$, then $U'$ and $H'$ are smooth, $U'\to H'$ is smooth over $t'$, and 
$U'\to X$ is surjective and smooth over $g'(\eta_{F'})$. Repeating this we get a 
smooth subvariety $R'$ of $T'$ passing through $t'$ so that the induced family $W'\to R'$ 
is smooth over $t'$, $W'$ is smooth, $W'\to X$ is surjective and generically finite and \'etale over $g'(\eta_{F'})$.
Let $Q_{W'}=Q'|_{W'}$. Then by construction, $Q_{W'}|_{F'}=Q'|_{F'}$ 
is reduced and simple normal crossing. Therefore, near $F'$ the divisor $Q_{W'}$ is reduced and 
any prime divisor $C'$ on $F'$ is contained in at most one component of $Q_{W'}$.\\

\emph{Step 3.}
In this step we define a sub-boundary $B_{F'}$ whose pushdown to $F$ gives $B_{F}$.
Let $K_{W'}+B_{W'}$ be the pullback of $K_X+B$ to $W'$. Here $K_{W'}$ and $B_{W'}$ are uniquely determined 
as Weil divisors (we assume we have already 
fixed a choice of $K_X$). Let $W'\to W\to X$ be the Stein factorisation of $W'\to X$. By the Riemann-Hurwitz formula, 
each coefficient of $B_W$ is $\le 1-\epsilon$ where $K_W+B_W$ is the pullback of $K_X+B$. Thus 
$(W,B_W)$ is sub-$\epsilon$-lc by [\ref{kollar-mori}, Proposition 5.20], hence $(W',B_{W'})$ is 
also sub-$\epsilon$-lc. 
On the other hand, by our choice of $P$, any component of $B_{W'}$ with positive coefficient is mapped into $P$, hence 
it is a component of $Q_{W'}$. 
Moreover, since $G$ was chosen general, it is not inside $\Supp K_X\cup \Supp P$. 
Thus since $W'\to X$ 
is \'etale over $g'(\eta_{F'})$, $F'$ is not inside $\Supp B_{W'}$ nor inside $\Supp K_{W'}$. 
Defining $B_{F'}=B_{W'}|_{F'}$ we get $K_{F'}+B_{F'}=(K_{W'}+B_{W'})|_{F'}$ where $K_{F'}=K_{W'}|_{F'}$ 
follows from the fact that $W'\to R'$ is smooth near $F'$ (note that $K_{F'}$ is determined as a Weil divisor). 
Let $K_F+B_F$ be the pushdown of $K_{F'}+B_{F'}$ which satisfies $K_F+B_F=(K_X+B)|_F$.\\

\emph{Step 4.}
In this step we show $(F,B_F)$ is sub-$\epsilon$-lc. It is enough to  
show $(F',B_{F'})$ is sub-$\epsilon$-lc. This in turn follows if show that 
$(F',A_{F'})$ is $\epsilon$-lc where $A_{W'}:=B_{W'}^{\ge 0}$ and $A_{F'}=A_{W'}|_{F'}$. 
By  the previous step, $\Supp A_{W'}\subseteq \Supp Q_{W'}$, hence  
$$
\Supp A_{F'}=\Supp (A_{W'}|_{F'})\subseteq \Supp (Q_{W'}|_{F'})=\Supp (Q'|_{F'}).
$$ 
So $\Supp A_{F'}$  is simple normal crossing  
as $\Supp (Q'|_{F'})$ is simple normal crossing. Therefore, it is enough to show each coefficient of 
$A_{F'}$ is $\le 1-\epsilon$. 

Let $C'$ be a component of $A_{F'}$ with positive coefficient. 
Then there is a component $D'$ of $A_{W'}$ with positive coefficient so that $C'$ is a component of $D'|_{F'}$. 
Since $D'$ is a component of $Q_{W'}$, by the last sentence of Step 2, $D'$ is uniquely determined  
and $D'|_{F'}$ is reduced. Then 
$\mu_{D'}B_{W'}=\mu_{D'}A_{W'}=\mu_{C'}A_{F'}$. But $\mu_{D'}B_{W'}\le 1-\epsilon$ because 
$(W',B_{W'})$ is sub-$\epsilon$-lc, hence $\mu_{C'}A_{F'}\le 1-\epsilon$.\\ 

\emph{Step 5.}
In this step we compute the coefficients of $\Theta_F$ in terms of lc thresholds of certain divisors.
Assume $C$ is a prime divisor on $F$ with $\mu_CB_F>0$. Let $C'$ on $F'$ be the birational transform of $C$. 
Then $\mu_{C'}B_{F'}>0$. Thus there is a unique component $D'$ of $B_{W'}^{>0}$ with positive coefficient 
such that $C'$ is a component of $D'|_{F'}$. In particular, $\mu_{C'}B_{F'}\le \mu_{C'}A_{F'}=\mu_{D'}B_{W'}$ 
as in the previous step (note that we do not claim $\mu_{C'}B_{F'}=\mu_{D'}B_{W'}$ because we have not ruled out 
the possibility that another component of $B_{W'}$ with negative coefficient contains $C'$).  
Let $L=cP$ and let $L_{W'}=L|_{W'}$ and $L_{F'}=L|_{F'}$ where $c$ is the number so that 
 $\mu_{D'}L_{W'}=1$. 
Then $\mu_{C'}L_{F'}=\mu_{D'}L_{W'}=1$.
 
We use the notation introduced in Construction \ref{constr-adjunction-non-klt-centre}. 
Recall that we constructed a birational morphism $Y\to X$ which we denote by $\pi$, a boundary $\Sigma_Y$, 
 a component $S$ of $\rddown{\Sigma_Y}$ mapping onto $G$, and $\Sigma_S$ defined by 
 $K_S+\Sigma_S=(K_Y+\Sigma_Y)|_S$. Let $L_S=L|_S$ and 
let $t$ be the lc threshold of $L_S$ with respect to $(S,\Sigma_S)$ over the generic point $\eta_C$ of $C$. 
By construction, $\mu_C(L|_F)=\mu_{C'}L_{F'}=1$, so under $h\colon S\to F$ the divisor
$L_S$ is equal to $h^*C$ over $\eta_C$.
Therefore, $\mu_C\Theta_F=1-t$.\\  

\emph{Step 6.}
In this step we consider the lc threshold of $L$ on $X$. 
Let $s$ be the lc threshold of $L$ with respect to 
$(X,B)$ near $\nu(\eta_C)$ where $\nu$ denotes $F\to G$. 
Let $I$ be the minimal non-klt centre of $(X,B+sL)$ which contains  
$\nu(\eta_C)$. Since $G\nsubseteq \Supp L$,  we have $G\nsubseteq I$. 
Let $L_Y=\pi^*L$ and write $K_Y+B_Y=\pi^*(K_X+B)$. 
Let $I_Y$ be a non-klt centre of $(Y,B_Y+sL_Y)$ which maps onto  
$I$. Since  $\Sigma_Y\ge B_Y$,  
$I_Y$ is also a non-klt centre of $(Y,\Sigma_Y+sL_Y)$. Note that $I_Y\neq S$.\\

\emph{Step 7.}
In this step and the next step we assume $X$ is $\Q$-factorial. In this step we show $t\le s$. 
This follows if we show that some non-klt centre of $(S,\Sigma_S+sL_S)$ maps onto $C$.
 Let $\Pi$ be the fibre of $\pi$ over 
a general closed point of $\nu(C)$ and let $H$ be the corresponding fibre of $S\to G$. 
Then $\Pi$ is connected, and since $X$ is $\Q$-factorial, $\Pi$ is inside the union of the
exceptional divisors of $\pi$, hence $\Pi\subset \rddown{\Sigma_Y}$. Moreover, 
 $I_Y$ intersects $\Pi$.

Let $E$ be the connected component of $H$ which maps into $C$ under $S\to F$. If a component $R$ of 
$\rddown{\Sigma_Y}-S$ intersects $E$, then $R\cap S$ gives a non-klt centre of 
$(S,\Sigma_S+sL_S)$ mapping onto $C$ (note that $(\rddown{\Sigma_Y}-S)\cap S$ is vertical over $F$). 
We can then assume $\rddown{\Sigma_Y}-S$ does not intersect $E$.
Then $E=H=\Pi:$ indeed otherwise since $\Pi$ is connected, there is a curve 
$Z\subseteq \Pi$ such that $Z\nsubseteq E$ but $Z$ intersects $E$; this is not possible 
because $Z$ cannot be inside $\rddown{\Sigma_Y}-S$ as $E$ does not intersect $\rddown{\Sigma_Y}-S$, 
so $Z$ is inside $S$, hence it is inside $E$ as it intersects $E$, a contradiction.  

Now $I_Y$ intersects $E=\Pi$ by the first paragraph of this step, so by inversion of adjunction,
 $I_Y\cap S$ produces a 
non-klt centre of $(S,\Sigma_S+sL_S)$ intersecting $E$, and so the centre maps onto $C$ as required.
Thus we have proved $t\le s$.\\

\emph{Step 8.}
In this step we show $B_F\le \Theta_F$.
Since $(X,B+sL)$ is lc near  $\nu(\eta_C)$, the sub-pair $(W',B_{W'}+sL_{W'})$ is sub-lc over 
$\nu(\eta_C)$ which implies 
$$
\mu_{D'}B_{W'}+s=\mu_{D'}(B_{W'}+sL_{W'})\le 1
$$ 
where $D'$ is as in Step 5. Therefore, 
$$
\mu_CB_F+t\le \mu_{C'}B_{F'}+s\le \mu_{D'}B_{W'}+s\le 1
$$
where we use the inequality $ \mu_{C'}B_{F'}\le \mu_{D'}B_{W'}$ observed in Step 5.
Therefore, we get $\mu_CB_F\le 1-t=\mu_C\Theta_F$.\\

\emph{Step 9.}
In this final step we take care of the non-$\Q$-factorial case.
Assume $X$ is not $\Q$-factorial and let $\overline{X}\to X$ be a small $\Q$-factorialisation. 
Let $\overline{B},  \overline{\Delta}, \overline{G}$, etc, be the birational transforms of $B, \Delta, G$, etc. 
We can assume $Y\to X$ and $W'\to X$ both factor through $\overline{X}\to X$. 
Let $\overline{F}$ be the normalisation of $\overline{G}$. Then we have induced morphisms $S\to \overline{F}\to F$ 
where $\overline{F}\to F$ is birational.   
We can define $\Theta_{\overline{F}}$ whose pushdown to $F$ is just $\Theta_F$. 
Also we can write $K_{\overline{F}}+B_{\overline{F}}=(K_X+B)|_{\overline{F}}$ where the pushdown of 
$B_{\overline{F}}$ is just $B_F$. Thus it is enough to show $B_{\overline{F}}\le \Theta_{\overline{F}}$. 
Now apply the arguments of Steps 7 and 8 on ${\overline{X}}$.

\end{proof}

\subsection{Lifting sections from non-klt centres}
In this subsection we show that under suitable assumptions we can lift sections from a non-klt centre. This is 
a key ingredient of the proof of \ref{p-eff-bir-tau} in the next section. First we prove a lemma.

\begin{lem}\label{l-subadjunction-non-epsilon-lc}
Let $(X,B), \Delta, G,F, \Theta_F, P_F, S, \Gamma_S,N_S$ be as in Construction \ref{constr-adjunction-non-klt-centre}. 
Assume $P_F$ is big. Then 

$(1)$ if there is a prime divisor $D$ on birational models of $S$ 
such that $a(D,S,\Gamma_S+N_S)<\epsilon$ and that the centre of $D$ on $S$ is vertical over $F$ where 
$\epsilon\ge 0$, then we can choose $P_F\ge 0$ so that 
$(F,\Theta_F+P_F)$ is not $\epsilon$-lc;

$(2)$ if $(X,B+\Delta)$ has a non-klt centre intersecting $G$ but not equal to $G$, then 
for each $\delta>0$ we can choose $P_F\ge 0$ so that $(F,\Theta_F+P_F)$ is not $\delta$-lc.
\end{lem}
\begin{proof}
We use other notation introduced in Construction \ref{constr-adjunction-non-klt-centre}. 
Let $\Delta_Y:=\Gamma_Y+N_Y$ and let $\Delta_S:=\Gamma_S+N_S$.

(1)
 Applying Lemma \ref{l-adj-fib-eps-lc}, there is a high
resolution $F'\to F$ so that some component $T$ of $\Delta_{F'}$ has coefficient larger than $1-\epsilon$ 
where $\Delta_{F'}, R_{F'}$ are the discriminant and moduli divisors on $F'$ associated to $(S,\Delta_S)$ 
over $F$ (we will use $R_{F'}$ below).

Pick a sufficiently small number $t>0$. Since $P_F$ is big, we can assume $P_F=A_F+C_F$ 
where $A_F\ge 0$ is ample and $C_F\ge 0$.  Let $\Omega_F=\Theta_F+C_F$ and let 
$K_{F'}+\Omega_{F'}$ be the pullback of $K_F+\Omega_F$. Then the coefficient of $T$ in  
$t\Omega_{F'}+(1-t)\Delta_{F'}$
is more than $1-\epsilon$. As the moduli part $R_{F'}$ is pseudo-effective (\ref{t-adj-fib-spaces}), we can find 
$$
0\le J_{F'}\sim_\R tA_{F'}+(1-t)R_{F'}
$$ 
where $A_{F'}$ is the pullback of $A_F$. Therefore, 
$$
(F',t\Omega_{F'}+(1-t)\Delta_{F'}+J_{F'})
$$
is not sub-$\epsilon$-lc. Moreover, as $K_{F'}+\Delta_{F'}+R_{F'}$ is the pullback of 
$K_F+\Delta_F+R_F$ and as  $K_{F'}+\Omega_{F'}+A_{F'}$ is the pullback of 
$K_F+\Omega_F+A_F$, we see that 
$$
K_{F'}+t\Omega_{F'}+(1-t)\Delta_{F'}+J_{F'}\sim_\R t(K_{F'}+\Omega_{F'}+A_{F'})+(1-t)(K_{F'}+\Delta_{F'}+R_{F'})
$$
is $\R$-linearly trivial over $F$, hence it is the pullback of 
$$
K_{F}+t\Omega_{F}+(1-t)\Delta_{F}+J_{F}
$$
where $J_F$ is the pushdown of $J_{F'}$. Therefore, 
$$
(F,t\Omega_{F}+(1-t)\Delta_{F}+J_{F})
$$
is not $\epsilon$-lc.

On the other hand,  we have  
$$
t\Omega_{F}+(1-t)\Delta_{F}+J_{F}\sim_\R 
t\Theta_F+tC_F+(1-t)\Delta_{F}+tA_{F}+(1-t)R_{F}
$$
$$
= t\Theta_F+tP_F+(1-t)\Delta_{F}+(1-t)R_{F}\sim_\R \Theta_F+P_F.
$$
Moreover, by construction, $\Delta_F\ge \Theta_F$ because $\Delta_Y\ge \Gamma_Y\ge \Sigma_Y$ 
which gives $\Delta_S\ge \Sigma_S$ (here $\Sigma_Y$ and $\Sigma_S$ 
are as in \ref{constr-adjunction-non-klt-centre} which are used to define $\Theta_F$), hence 
$$
t\Omega_F+(1-t)\Delta_F\ge t\Theta_F+(1-t)\Delta_F\ge \Theta_F.
$$ 
Thus if we change $P_F$ to
$$
t\Omega_{F}+(1-t)\Delta_{F}+J_{F}-\Theta_F,
$$ 
then $P_F$ is effective and 
$
(F,\Theta_F+P_F)
$
is not $\epsilon$-lc.

(2) 
By assumption, some non-klt centre $H\neq G$
of $(X,B+\Delta)$ intersects $G$, hence there is a non-klt centre $Z\neq S$ of $(Y,\Delta_Y)$  
intersecting $S$ and mapping onto $H$, by the connectedness principle applied to $(Y,\Delta_Y)$ near fibres of 
$\pi$ over points in $H\cap G$. Thus some component of $\rddown{\Delta_Y-S}$ intersects $S$ 
as the non-klt locus of $(Y,\Delta_Y)$ is equal to $\rddown{\Delta_Y}=\rddown{\Gamma_Y}$ because 
$(Y,\Gamma_Y)$ is dlt and $\Supp N_Y\subseteq \rddown{\Gamma_Y}$. 
This in turn gives a component of 
$\rddown{\Delta_S}$ which is vertical over $F$ as there is a unique non-klt place of $(Y,\Delta_Y)$
with centre $G$. Applying (1) we can choose 
$P_F\ge 0$ in its $\R$-linear equivalence class such that 
the pair $(F,\Theta_F+P_F)$ is not $\delta$-lc.

\end{proof}

\begin{prop}\label{p-lift-section-lcc}
Let $d,r\in \N$ and $\epsilon\in \R^{>0}$. Then there is $l\in \N$ depending only on 
$d,r, \epsilon$ satisfying the following.
Let $(X,B), \Delta, G,F, \Theta_F,$ and $P_F$ be as in Construction \ref{constr-adjunction-non-klt-centre}, 
and assume $G$ is a general member of a 
covering family of subvarieties. Assume in addition that 
\begin{itemize}
\item $X$ is Fano of dimension $d$ and $B=0$,

\item $\Delta\sim_\Q -(n+1)K_X$ for some $n\in\N$,

\item $h^0(-nrK_X|_F)\neq 0$, and 

\item $P_F$ is big and for any choice of $P_F\ge 0$ in its $\R$-linear equivalence class 
the pair $(F,\Theta_F+P_F)$ is $\epsilon$-lc.\\
\end{itemize}
Then $h^0(-lnrK_X)\neq 0$.  

\end{prop}
\begin{proof}
We give a short description of the proof. The idea is to show that $G$ is an isolated non-klt centre. 
Moreover, if $\pi\colon Y\to X$ and $S,\Gamma_Y, N_Y$ are as in Construction \ref{constr-adjunction-non-klt-centre}, then 
$(Y,\Gamma_Y+N_Y)$ is plt near $S$. Next we pull back a section in $H^0(-nrK_X|_F)$ to $S$ and 
in turn lift the section to $H^0(\lceil -lnr\pi^*K_X\rceil )$ on $Y$ using vanishing theorems 
where $l$ is a bounded natural number. Finally 
we push the section down to $X$ to deduce $h^0(-lrnK_X)\neq 0$.\\

\emph{Step 1}.
 In this step we show $G$ is an isolated non-klt centre.
We use the notation of Construction \ref{constr-adjunction-non-klt-centre}. 
Remember that $(X,\Delta)$ is lc near the generic point of $G$. 
Also recall that $\rddown{\Gamma_Y}$ has a unique component $S$ mapping onto $G$.
Letting $\Delta_Y=\Gamma_Y+N_Y$ 
we have $K_Y+\Delta_Y=\pi^*(K_X+\Delta)$ where $\pi$ denotes $Y\to X$. 
Let $K_S+\Delta_S=(K_Y+\Delta_Y)|_S$.

 Assume $G$ is not an isolated non-klt centre. Then some non-klt centre $H\neq G$
of $(X,\Delta)$ intersects $G$. Applying Lemma \ref{l-subadjunction-non-epsilon-lc}, we can choose 
$P_F\ge 0$ in its $\R$-linear equivalence class such that 
the pair $(F,\Theta_F+P_F)$ is not $\epsilon$-lc, a contradiction. In particular,  
$(Y,\Delta_Y)$ is plt near $S$ and that no component of $\rddown{\Delta_Y-S}$ intersects $S$.\\

\emph{Step 2.}
In this step we show $E:=\pi^*(-nrK_X)$ is an integral divisor near $S$, and that it has bounded Cartier index near 
codimension one points on $S$.  
Since $G$ is a general member of a covering family of subvarieties of $X$, the generic 
point of $G$ is in the smooth locus of $X$, hence $K_X$ is Cartier near this point. Thus the coefficient of 
$S$ in $E$ is integral, hence any component of $E$ whose coefficient is not integral (if there is any) should be 
an exceptional divisor of $\pi$ other than $S$. But such exceptional divisors do not intersect $S$, 
by Step 1. Therefore, $E$ is integral near $S$. 

Let $V$ be a prime divisor on $S$. 
If $V$ is horizontal over $G$, then $E$ is Cartier near the generic point of $V$. 
Assume $V$ is vertical over $G$. We want to show the Cartier index of $E$ near the generic 
point of $V$ is bounded. Let $p$ be the Cartier index of $K_Y+S$ near the generic point of $V$. 
Then $\mu_V\Delta_S\ge 1-\frac{1}{p}$  and the Cartier index of $E$ near the generic point of $V$ 
divides $p$ [\ref{Shokurov-log-flips}, Proposition 3.9]. 
Therefore, $\frac{1}{p}\ge \epsilon$ otherwise $\mu_V\Delta_S>1-\epsilon$, hence $(S,\Delta_S)$ is not $\epsilon$-lc, so 
by Lemma \ref{l-subadjunction-non-epsilon-lc},
we can choose $P_F\ge 0$ so that $(F,\Theta_F+P_F)$ is not $\epsilon$-lc, a contradiction. 
This means that $p$ is bounded depending only on 
$\epsilon$, hence the Cartier index of $E$ near the generic point of $V$ is also bounded.\\
 
\emph{Step 3}.
In this step we use Kawamata-Viehweg vanishing to show $h^1(\lceil lE -\rddown{\Gamma_Y}-N_Y\rceil)=0$ 
where $l\in\N$ is at least $2$. Indeed pick $2\le l\in \N$ and define $L$ by the relation 
$$
\lceil lE -\rddown{\Gamma_Y}-N_Y\rceil=lE -\rddown{\Gamma_Y}-N_Y+L.
$$ 
Each component of $L$ is either an exceptional$/X$ component of $E$ or is a component of $N_Y$, and
in either case the component cannot be $S$.  
Then $N_Y+L$ is supported in $\rddown{\Gamma_Y}-S$ 
which is disjoint from $S$ by Step 1.  

On the other hand, the $\Q$-divisor 
$$
 I:=lE-(K_Y+\Delta_Y)=\pi^*(-lnrK_X)-\pi^*(K_X+\Delta)
 $$
 $$
 \sim_\Q\pi^*(-lnrK_X) +\pi^*(nK_X)=-(lnr-n)\pi^*K_X
$$ 
is nef and big. Now we can write  
$$
\lceil lE -\rddown{\Gamma_Y}-N_Y\rceil=lE -\rddown{\Gamma_Y}-N_Y+L
$$
$$
\sim_\Q K_Y+\Delta_Y+I-\rddown{\Gamma_Y}-N_Y+L=K_Y+\Gamma_Y-\rddown{\Gamma_Y}+L+I.
$$ 
Moreover, $(Y,\Gamma_Y-\rddown{\Gamma_Y}+L)$ is klt because $(Y,\Gamma_Y)$ 
is dlt and because $L$ is supported in $\rddown{\Gamma_Y}$ with coefficients less than $1$. 
Thus by the Kawamata-Viehweg vanishing theorem (cf. [\ref{kollar-mori}, Theorems 2.70 and Corollary 5.27]), 
we get 
$$
h^1(\lceil lE -\rddown{\Gamma_Y}-N_Y\rceil)=0.
$$\

\emph{Step 4.}
In this step we lift sections of $\lceil lE\rceil|_S$ to $Y$ from which we 
produce sections of $-lnrK_X$, for some bounded number $l$.
Let $U\subseteq X$ be the largest open set 
on which $\lceil lE -\rddown{\Gamma_Y}+S-N_Y\rceil$ is Cartier. By Step 2, we can choose a 
bounded  $l\ge 2$ so that no codimension two component of $X\setminus U$ is contained 
in $S$ where we use the fact that $\rddown{\Gamma_Y}-S+N_Y$ does not intersect $S$ and that 
$E$ is integral near $S$. Thus the codimension of $S\setminus (S\cap U)$ in $S$ is at least $2$. 
 Therefore, by Lemma \ref{l-restriction-sequence}, we have the exact sequence 
$$
0\to \mathcal{O}_Y(\lceil lE -\rddown{\Gamma_Y}-N_Y\rceil)\to \mathcal{O}_Y(\lceil lE -\rddown{\Gamma_Y}+S-N_Y\rceil) 
\to \mathcal{O}_S(\lceil lE\rceil|_S) \to 0.
$$  
Note that to apply the lemma we need $S$ to be $\Q$-Cartier which is the case as $Y$ is $\Q$-factorial by construction.

On the other hand,
$$
\lceil lE\rceil|_S=lE|_S=h^*(-lnrK_X|_F)
$$  
where $h$ denotes $S\to F$, and by assumption, $h^0(-lnrK_X|_F)\neq 0$. Thus  
$h^0(\lceil lE \rceil|_S)\neq 0$. Therefore, the above exact sequence combined with the vanishing at the end of 
Step 3 imply
$$
h^0(\lceil lE -\rddown{\Gamma_Y}+S-N_Y\rceil)\neq 0.
$$ 
This in turn gives $h^0(\lceil lE \rceil)\neq 0$ as $-\rddown{\Gamma_Y}+S-N_Y\le 0$, hence we get $h^0(-lnrK_X)\neq 0$.
  
\end{proof}

\section{\bf Effective birationality}

In this section we prove Theorem \ref{t-eff-bir-e-lc} under certain extra assumptions. 
These special cases are crucial for the proof of all the main results of this paper. 
One case is when we have an effective $\R$-divisor $B$ whose coefficients are bounded from below 
and $K_X+B\sim_\R 0$ (\ref{p-eff-bir-delta-2}), e.g. in practice when $X$ is non-exceptional we can find such $B$ using induction 
and complements. Another case is when singularities of $X$ are canonical or close to being 
canonical (\ref{p-eff-bir-tau}). Before we get to these results we make some technical preparations.

\subsection{Singularities in bounded families}
The next result roughly says that effective divisors with  ``degree" bounded from above cannot have too small lc thresholds, 
under appropriate assumptions. A far more general form of this is proved in [\ref{B-BAB}, Theorem 1.6] which is 
one of the key ingredients of the proof of BAB.

\begin{prop}\label{p-non-term-places}
Let $\epsilon\in \R^{>0}$ and let $\mathcal{P}$ be a bounded set of couples. 
Then there is $\delta\in \R^{>0}$ depending only on $\epsilon,\mathcal{P}$ satisfying 
the following. Let $(X,B)$ be a projective pair and let $T$ be a reduced divisor on $X$.
Assume 
\begin{itemize}
\item $(X,B)$ is $\epsilon$-lc and $(X,\Supp B+T)\in \mathcal{P}$, 

\item $L\ge 0$ is an $\R$-Cartier $\R$-divisor on $X$,  

\item $L\sim_\R N$ for some $\R$-divisor $N$ supported in $T$, and 

\item the absolute value of each coefficient of $N$ is at most $\delta$.\\
\end{itemize}
Then $(X,B+L)$ is klt. 
\end{prop}
\begin{proof}
We may assume all the members of $\mathcal{P}$ have the same dimension, say $d$. 
We prove the proposition by induction on $d$. Let $(X,B),T,L,N$ be as in the statement for some $\delta$, 
and assume $(X,B+L)$ is not klt. 
First assume $d=1$. Then $\deg L\ge \epsilon$, hence 
$\deg N\ge \epsilon$. This is not possible if $\delta$ is small enough 
because $\deg N\le \delta \deg T$ and $\deg T$ is bounded. 
So we can assume $d\ge 2$. 

We can find a log resolution 
$\phi\colon W\to X$ of $(X,B+T)$ such that if we write  
$$ 
K_W+B_W=\phi^*(K_X+B)+E 
$$
where $B_W\ge 0$ and $E\ge 0$ have no common components, and let $T_W$ be the sum of 
the birational transform of $T$ and the reduced exceptional divisor $\phi$, 
then $(W,B_W)$ is $\epsilon$-lc and $(W,\Supp B_W+T_W)$ belongs to a bounded set of couples $\mathcal{S}$ 
determined by some presentation of $\mathcal{P}$ (as in \ref{ss-bnd-couples}). 
Let $L_W=\phi^*L$ and $N_W=\phi^*N$. Then there is $m\in\N$ depending only on $\mathcal{P}, \mathcal{S}$ 
so that the absolute value of each 
coefficient of $N_W$ is $\le m\delta$. 
Therefore, we can replace $\mathcal{P}$ with $\mathcal{S}$ and replace 
$(X,B),T,L,N,\delta$ with $(W,B_W),T_W,L_W,N_W,m\delta$, 
hence assume from now on that $(X,\Supp B+T)$ is log smooth. 

We will argue that perhaps after replacing $\mathcal{P}$ and modifying $T,N$ we can assume that 
 $B$ and $T$ have no common components and $T$ is very ample. Write 
 $T=\sum_1^q T_i$ where $T_i$ are irreducible components. Since $\mathcal{P}$ is bounded, 
there is a reduced  
 divisor $\Gamma=\sum_1^{2q}\Gamma_i$ on $X$  such that $\Gamma_j$  
are distinct prime very ample divisors which are not components of $B$, and 
$T_i\sim \Gamma_{2i}-\Gamma_{2i-1}$. Moreover, we can assume the  couples $(X,\Supp B+\Gamma)$ 
belong to a bounded family. Using 
the linear equivalences $T_i\sim \Gamma_{2i}-\Gamma_{2i-1}$ we can modify $N$ so that it is now supported in $\Gamma$ but still 
the absolute value of each of its coefficients is at most $\delta$. Now replacing 
$T$ with $\Gamma$ and replacing $\mathcal{P}$ accordingly we can assume 
  $B$ and $T$ have no common components and that $T$ is very ample. 

Now assume $\delta$ is sufficiently small.
Let $D$ be a prime divisor on birational models of $X$ such that $a(D,X,B+L)\le 0$. 
We show $D$ is not a divisor on $X$. Since $T$ is very ample and $T^d$ is bounded, from 
$T^{d-1}\cdot L=T^{d-1}\cdot N\le \delta T^d$ we deduce that if $l$ is a coefficient of $L$, 
then $l\le \delta T^d$. This shows that $D$ cannot be a divisor on $X$ as 
we can assume $\mu_D(B+L)\le 1-\epsilon+\delta T^d<1$. 
Therefore, $D$ is exceptional$/X$.

Let $V$ be the centre of $D$ on $X$. 
First assume $V$ is not inside $\Supp B$. In this case we can remove $B$ and assume $B=0$. 
Let $H$ be a general member of $|rT|$ 
intersecting $V$ where $r$ is sufficiently large depending on $\mathcal{P}$. Then $H$ is irreducible and smooth, 
$(X,H)$ is plt but $(X,H+L)$ is not plt near any component of $V\cap H$. This implies 
$(H,L_H)$ is not klt near any component of $V\cap H$ where $L_H=L|_H$ [\ref{kollar-mori}, Theorem 5.50]. 
The divisor $T|_H$ is reduced unless $\dim V=2$ and $V$ is the intersection point of 
two components of $T$ in which case the coefficients of $T|_H$ are at most $2$. 
Letting $N_H=N|_H$, we see that  
$N_H$ is supported in $T_H:=\Supp T|_H$ with absolute value of coefficients $\le 2\delta$. 
 By construction, $(H,T_H)$ belongs to 
a bounded set of couples $\mathcal{Q}$. So applying induction we get a contradiction as 
$\delta$ can be chosen according to $\mathcal{Q}$. 

Now assume $V$ is inside some component $S$ of $B$. 
Let $\Delta=B+(1-b)S$ where $b=\mu_SB$. Then $(X,\Delta)$ is plt and $\rddown{\Delta}=S$. Moreover, 
since $l:=\mu_SL\le \delta T^d$ (as observed above),
 we can assume $l=\mu_SL< \epsilon \le 1-b$, hence $B+L\le \Delta+L-lS$. 
Thus $V$ is a non-klt centre of $(X,\Delta+L-lS)$, and since $\mu_S(\Delta+L-lS)=1$, 
$S$ is also a non-klt centre of $(X,\Delta+L-lS)$. Thus 
$(X,\Delta+L-lS)$ is not plt near $V$. 

Let $K_S+\Delta_S=(K_S+\Delta)|_S$ and $L_S=(L-lS)|_S$.
Then  $(S,\Delta_S)$ is $\epsilon$-lc but $(S,\Delta_S+L_S)$ is not klt. 
Perhaps after adding some components to $T$, we can assume $S\sim S'$ where 
$S'$ is supported in $T$ and with bounded absolute value of coefficients. 
By construction, $L_S\sim_\R N_S:=(N-lS')|_S$ is supported on $T_S:=T|_S$ and the absolute value of 
each coefficient of $N_S$ is at most $n\delta$ where $n\in \N$ depends only on $\mathcal{P}$. 
Moreover, $(S,\Supp \Delta_S+T_S)$ belongs to a bounded set of couples $\mathcal{R}$.  
Now applying induction on induction we again get a contradiction.

\end{proof}

\subsection{Log birational boundedness of certain pairs}

In various places in this paper we use the next statement to find a bounded birational model 
$\overline{X}$ of a given variety $X$, 
e.g. proofs of \ref{p-bnd-sing-on-non-klt-centre}, \ref{p-eff-bir-delta-2}, \ref{p-eff-bir-tau}. This allows us to translate problems 
about $X$ to problems about $\overline{X}$ which are then more tractable. 

\begin{prop}\label{p-log-bir-bnd-cert-pairs}
Let $d,v\in\N$ and let $\epsilon \in \R^{>0}$. 
Then there exist a number $c \in \R^{>0}$ and a 
bounded set of couples $\mathcal{P}$ depending only on $d,v,\epsilon$  satisfying the following.
Assume 
\begin{itemize}
\item $X$ is a normal projective variety of dimension $d$, 

\item $B\ge 0$ is an $\R$-divisor with coefficients in $\{0\}\cup [\epsilon,\infty)$, 

\item $M\ge 0$ is a nef $\Q$-divisor such that $|M|$ defines a birational map,

\item $M-(K_X+B)$ is pseudo-effective,

\item $\vol(M)< v$, and

\item if $D$ is a component of $M$, then $\mu_D(B+M)\ge 1$.\\
\end{itemize}

Then there is a projective log smooth couple $(\overline{X},{\Sigma}_{\overline{X}})\in \mathcal{P}$ 
and a birational map $\overline{X}\bir X$ such that 
\begin{enumerate}
\item  $\Supp {\Sigma}_{\overline{X}}$ contains the exceptional 
divisor of  $\overline{X}\bir X$ and the birational transform of $\Supp (B+{M})$;

\item if $X'\to X$ and $X'\to \overline{X}$ is a common resolution and 
$M_{\overline{X}}$ is the pushdown of $M_{X'}:=M|_{X'}$, then each coefficient of $M_{\overline{X}}$ is at most $c$;

\item there is a resolution $W\to X$ such that $M_W:=M|_W\sim A_W+R_W$ where $A_W$ is 
the movable part of $|M_W|$, $|A_W|$ is base point free, and if $X'\to X$ factors through $W$, then 
$A_{X'}:=A_W|_{X'}\sim 0/\overline{X}$.

\end{enumerate}
\end{prop}

\begin{proof}

First we give a short summary of the proof. Since 
$|M|$ defines a birational map, there is a resolution $\phi\colon W\to X$ such that 
$\phi^*M$ decomposes as the sum of a base point free movable part $A_W$ and fixed part $R_W$.
Since  $\vol(M)< v$,  
the contraction defined by $A_W$ gives a bounded birational model 
$\overline{X}$. To find ${\Sigma}_{\overline{X}}$ as in the statement, the idea is to argue that 
$\vol(K_X+\Sigma+2(2d+1)A)$ is bounded from above where $A$ is the pushdown of $A_W$ and 
$\Sigma$ is the support of $B+M$ union a divisor derived from a multiple of $A_W$. Applying  
[\ref{HMX}, Lemma 3.2] and [\ref{HMX}, Lemma 2.4.2(4)] would then produce the required 
$(\overline{X},{\Sigma}_{\overline{X}})$ after taking an appropriate resolution.\\

\emph{Step 1.}
In this step we introduce some basic notation.
Since $|M|$ defines a birational map, $M$ is big.
Moreover, by Lemma \ref{l-mov-part-lin-system}, there is a log resolution $\phi\colon W\to X$ of $(X,\Supp (B+M))$ such that 
$$
M_W:=\phi^*M\sim A_{W}+R_{W}
$$ 
where $A_W$ is the movable part of $|M_W|$, $|A_W|$ is based point free defining a birational contraction, 
and $R_W\ge 0$ is the fixed part. We can assume $A_W$ is general 
in the linear system $|A_W|$. We denote the pushdown of $A_{W},R_{W}$ to $X$ by $A,R$ respectively. 
Note that $R_{W}$ is only a $\Q$-divisor but $R$ is integral.\\ 

\emph{Step 2.}
In this step we define an auxiliary boundary $\Omega_W$ on $W$.  Decreasing $\epsilon$ 
we can assume $\epsilon\in(0,1)$.
Let $H_W\in |6dA_W|$ be general. Let $D$ be a prime divisor on $W$. Then let the coefficient of $D$ 
in $\Omega_W$ be  
$$ 
\mu_D\Omega_W:= \left\{
  \begin{array}{l l}
    1-\epsilon & \quad \text{if $D$ is exceptional$/X$,}\\
    1-\epsilon & \quad \text{if $D$ is a component of $M_W$},\\
   \epsilon & \quad \text{if $D$ is a component of $B^\sim$ but not of $M_W$},\\
    \frac{1}{2}& \quad \text{if $D=H_W$},\\
    0 & \quad \text{otherwise}
  \end{array} \right.
$$   
where $B^\sim$ is the birational transform of $B$.
The pair $(W,\Omega_W)$ is log smooth, and by Lemma \ref{l-large-boundaries}, 
$K_W+\Omega_W$ is big.\\

\emph{Step 3.}
The aim of this step is to show that $\vol(K_W+\Omega_W)$ is bounded from above. 
Since $M-(K_X+B)$ is pseudo-effective and since $\vol(M)<v$, 
$$
\vol(K_X+B+5dM)<\vol(6dM)<(6d)^dv
$$
hence the left hand side is bounded  from above. 
Now we claim
$$
\vol(K_X+\Omega)\le \vol(K_{X}+B+5dM)
$$ 
where $\Omega$ is the pushdown of $\Omega_{W}$. This follows if we 
show $B+5dM-\Omega$ is big.  

Let $D$ be a component of $\Omega$. Then either $D$ is a component of $M$ or a 
component of $B$ or $D=H$ the pushdown of $H_W$. In the first case,   
$$
\mu_D\Omega= 1-\epsilon<1\le \mu_D(B+M)
$$ 
where the inequality $1\le \mu_D(B+M)$ holds by assumption. 
If $D$ is as in the second case but not the 
first case, then 
$$
\epsilon=\mu_D\Omega\le \mu_DB\le \mu_D(B+M).
$$ 
So we get 
$$
B+M+\frac{1}{2} H-\Omega\ge 0.
$$ 
On the other hand,  $4dM-3dA$ is big, hence 
$$
4dM-\frac{1}{2}H\sim_\Q 4dM-3dA
$$  
is also big. This combined with the previous sentence implies the bigness of $B+5dM-\Omega$.  

We have shown that $\vol(K_{X}+\Omega)$ is bounded from above.
Thus since 
$$
\vol(K_W+\Omega_W)\le \vol(K_{X}+\Omega)
$$ 
we get the required boundedness of $\vol(K_{W}+\Omega_W)$.\\

\emph{Step 4.}
In this step we show the existence of $\mathcal{P}$ and $(\overline{X}, \Sigma_{\overline{X}})$ 
satisfying (1).  
To do this we need to show that $(W,\Omega_W)$ is log birationally bounded.
Let $\Sigma_{W}:=\Supp \Omega_{W}$. First we show  
$$
\vol(K_{W}+\Sigma_{W}+2(2d+1)A_{W})
$$ 
is bounded from above. Since $K_{W}+\Omega_{W}$ is 
big and since the coefficients of $\Omega_{W}$ belong to $\{\epsilon, \frac{1}{2}, 1-\epsilon\}$, 
there is $\alpha\in (0,1)$ depending only on $d$ and $\epsilon$ 
such that $K_{W}+\alpha \Omega_{W}$ is big [\ref{HMX2}, Lemma 7.3]. By definition of $\Omega_W$,
taking a large but bounded number $p$, we get  
$$
\vol(K_{W}+\Sigma_{W}+2(2d+1)A_{W})\le \vol(K_{W}+\Omega_{W}+p(1-\alpha) \Omega_{W})
$$
$$
\le \vol(K_{W}+\Omega_{W}+p(K_{W}+\alpha \Omega_{W})+p(1-\alpha) \Omega_{W})
$$
$$
\le \vol((1+p)(K_{W}+ \Omega_{W}))
$$
which shows the left hand side volume is bounded from above. 

By construction, $|A_W|$ is base point free defining a birational contraction. 
Thus by [\ref{HMX}, Lemma 3.2],
$\Sigma_W\cdot A_W^{d-1}$ is bounded from above.
Therefore, $(W,\Omega_W)$ is log birationally bounded by [\ref{HMX}, Lemma 2.4.2(4)] as the volume of 
$A_W$ is bounded. 

If $W\to \tilde{X}$ is the contraction defined by $A_{W}$, then 
$(\tilde{X}, \Sigma_{\tilde{X}})$ is log bounded where $\Sigma_{\tilde{X}}$ is the pushdown of 
$\Sigma_{W}$, and $\Sigma_{\tilde{X}}$ contains the exceptional divisor of 
$\tilde{X}\bir X$ and the birational transform of $B+M$. 
Thus there is a log resolution $\overline{X}\to \tilde{X}$ of $(\tilde{X}, \Sigma_{\tilde{X}})$ 
such that if $\Sigma_{\overline{X}}$ is the sum of the reduced exceptional divisor 
of $\overline{X}\to \tilde{X}$ and the birational 
transform of $\Sigma_{\tilde{X}}$, then $(\overline{X},\Sigma_{\overline{X}})$ is log smooth and belongs to 
a fixed bounded set of couples $\mathcal{P}$ depending only on $d,v,\epsilon$. Moreover, 
$\Sigma_{\overline{X}}$ contains the exceptional divisor of $\overline{X}\bir X$ and the birational 
transform of $B+M$.\\

\emph{Step 5.}
In this step we prove  (2) and (3). 
Take a common resolution $X'\to W$ and $X'\to \overline{X}$. Let $A_{X'},R_{X'},H_{X'}$ be the pullbacks of 
$A_{W},R_W, H_{W}$, and $A_{\overline{X}}, R_{\overline{X}}, H_{\overline{X}}$ be their pushdown to $\overline{X}$.
By construction, 
$$
H_{X'}\sim 6dA_{X'},  ~~A_{X'}\sim 0/\overline{X}, 
~~\mbox{and}~~\Sigma_{\overline{X}}\ge \frac{1}{2}H_{\overline{X}}.
$$ 
In particular,  there is a number $b\in\N$ depending only on $\mathcal{P}$ such that we can 
pick an ample Cartier divisor $C_{\overline{X}}$ 
 so that 
$bH_{\overline{X}}-C_{\overline{X}}$ is big. Then $bH_{{X'}}-C_{{X'}}$ is also big where 
$C_{X'}$ is the pullback of $C_{\overline{X}}$.
Thus  if $M_{X'}$ is the pullback of $M_W$ and if $M_{\overline{X}}$ is the pushdown of $M_{X'}$, then we have 
$$
M_{\overline{X}}\cdot C_{\overline{X}}^{d-1}=M_{X'}\cdot C_{X'}^{d-1}\le 
\vol(M_{X'}+C_{X'})\le \vol(M_{X'}+bH_{X'})\le \vol((1+6bd)M_{X'})
$$
where the first inequality uses the fact that $M_{X'}, C_{X'}$ are both nef. 
Therefore, $M_{\overline{X}}\cdot C_{\overline{X}}^{d-1}$ is bounded 
from above which implies the coefficients of $M_{\overline{X}}$ are bounded from above by some 
fixed number $c$. That is, (2) holds. Note that we have assumed that $X'\to X$ factors through $W\to X$ 
but (2) holds even if $X'\to X$ does not factor through $W\to X$ because 
$M_{\overline{X}}$ does not depend on the choice of the common resolution.

Finally, (3) holds as by construction $M_W\sim A_W+R_W$ where $A_W$ is the movable 
part of $|M_W|$, $|A_W|$ is base point free, and $A_{X'}\sim 0/\overline{X}$.

\end{proof}

\subsection{Boundedness of singularities on non-klt centres}

The next result is about boundedness of singularities on the normalisation of a non-klt centre 
in the context of adjunction as in \ref{constr-adjunction-non-klt-centre}. This is key to the proofs of 
\ref{p-eff-bir-delta-1}, \ref{p-eff-bir-tau}, \ref{p-bnd-vol-good-boundary}.

\begin{prop}\label{p-bnd-sing-on-non-klt-centre}
Let $d,v\in\N$ and $\epsilon, \epsilon'\in \R^{>0}$ with $\epsilon'<\epsilon<\frac{1}{2}$. Then there exists $t \in \R^{>0}$ 
depending only on $d,v,\epsilon,\epsilon'$ satisfying the following. Assume $X,C,M,\Delta,G,F,\Theta_F,P_F$ 
are as follows:
\begin{itemize}
\item $(X,C)$ is a projective $\epsilon$-lc pair of dimension $d$,

\item $C$ is $\R$-Cartier with coefficients in $\{0\}\cup [\epsilon, 1-\epsilon]$, 

\item $M$ is an ample integral divisor and $|M|$ defines a birational map,

\item $0\le \Delta\sim_\R \alpha M$ for some $0<\alpha<t$, 

\item  $K_X+C+\Delta$ is ample and $M-(K_X+ C+\Delta)$ is big,

\item $G$ is a member of a covering family of subvarieties of $X$, with normalisation $F$, 

\item there is a unique non-klt place of $(X,\Delta)$ whose centre is $G$, 

\item the adjunction formula
$$
K_F+\Theta_F+P_F\sim_\R (K_X+\Delta)|_F
$$
is as in \ref{constr-adjunction-non-klt-centre} assuming $P_F\ge 0$, and 

\item $\vol(M|_F)< v$.\\
\end{itemize}
Then  for any $0\le L_F\sim_\R M|_F$, the pair 
$$
(F,C|_F+\Theta_F+P_F+tL_F)
$$ 
is $\epsilon'$-lc.
\end{prop}
\begin{proof}

We first give a summary of the proof. Note that the adjunction formula in the statement is as in Construction 
 \ref{constr-adjunction-non-klt-centre} but with $B=0$. After looking at the adjunction formula more closely 
and letting $C_F:=C|_F$ and $M_F:=M|_F$, and using \ref{p-log-bir-bnd-cert-pairs}, 
we will find a log bounded birational model $(\overline{F},\Sigma_{\overline{F}})$ of $(F, \Supp (\Theta_F+C_F+M_F))$. 
We then argue that  
$$
K_F+C_F+\Theta_F+P_F+tL_F
$$ 
is ample, hence its singularities cannot be worse than singularities 
of its ``crepant pullback" to $\overline{F}$. At the end we apply \ref{p-non-term-places} to control singularities 
on $\overline{F}$.\\

\emph{Step 1.}
In this step we introduce some basic notation.
We will assume $\dim G>0$ otherwise the statement is vacuous.  Since $|M|$ defines a birational map, 
we can assume $M\ge 0$. 
Moreover, changing $M$ up to linear equivalence, by Lemma \ref{l-mov-part-lin-system}, 
there is a log resolution $\phi\colon W\to X$ of 
$(X,\Supp (C+M))$ such that we can write 
$$
M_W:=\phi^*M=A_{W}+R_{W}
$$ 
where $A_W$ is the movable part of $|M_W|$, $|A_W|$ is based point free defining a birational contraction, 
and $R_W\ge 0$ is the fixed part.  
We denote the pushdown of $A_{W},R_{W}$ to $X$ by $A,R$ respectively. 
Note that $R_{W}$ is only a $\Q$-divisor but $R$ is integral.\\

\emph{Step 2.}
In this step we have a closer look at the adjunction formula given in the statement,  and the related divisors. 
First note that since $G$ is a general member of a covering family (as in \ref{ss-cov-fam-subvar}), 
it is not contained in $\Supp (C+M)$, and $X$ is $\Q$-factorial near the generic point of $G$.
By Theorem \ref{t-subadjunction} (here we take $B=0$) and 
the ACC for lc thresholds [\ref{HMX2}, Theorem 1.1], the
coefficients of $\Theta_{F}$ are in a fixed DCC set $\Psi$ depending only on $d$.  

Since both $K_X+C$ and $C$ are $\R$-Cartier, $K_X$ is $\Q$-Cartier.
By Lemma \ref{l-sub-bnd-on-gen-subvar} 
(again here $B=0$), we can write $K_F+\Lambda_F=K_X|_F$ where $(F,\Lambda_{F})$ is sub-klt and 
$\Lambda_{F}\le \Theta_{F}$. On the other hand, since $G$ is not contained in $\Supp C$, 
the unique non-klt place of $(X,\Delta)$ whose centre is $G$ is also a  unique non-klt place of 
$(X,C+\Delta)$ whose centre is $G$. Thus 
applying Lemma \ref{l-sub-bnd-on-gen-subvar} 
once more (this time by taking $B=C$), we can write $K_F+\tilde{C}_F=(K_X+C)|_F$ where 
$(F,\tilde{C}_{F})$ is sub-$\epsilon$-lc. Note that  
$$
\tilde{C}_F=\Lambda_F+C|_F\le \Theta_F+C|_F.
$$\

\emph{Step 3.} 
Let $C_F:=C|_F$ and $M_F:=M|_F$.
In this step we show 
$$
(F, \Supp (\Theta_F+C_F+M_F))
$$ 
is log birationally bounded using Proposition \ref{p-log-bir-bnd-cert-pairs}.
Since $G$ is a general member of a covering family, we can choose a log resolution $F'\to F$ of the above pair  
such that we have an induced morphism $F'\to W$ and that $|A_{F'}|$ defines a birational contraction 
where $A_{F'}:=A_W|_{F'}$. Thus $|A_{F}|$ defines a birational map where $A_F$ is the 
pushdown of $A_{F'}$. This in turn implies $|M_F|$ defines a birational map because $A_F\le M_F$. 
Moreover,  
$$
K_F+{C}_F+\Theta_F+P_F\sim_\R (K_X+C+\Delta)|_F
$$ 
is ample, and by the generality of $G$, 
$$
M_F-(K_F+{C}_F+\Theta_F+P_F)\sim_\R (M-(K_X+C+\Delta))|_F
$$
is big which in turn implies 
$$
M_F-(K_F+{C}_F+\Theta_F)
$$
is big as well. 

On the other hand, 
by Lemma \ref{l-subadjunction-integral-div}, $\mu_D(\Theta_F+ M_F)\ge 1$ 
for any component $D$ of $M_F$. Applying the lemma once more, 
$\mu_D(\Theta_F+\frac{1}{\epsilon} C_F)\ge 1$ 
for any component $D$ of $C_F$ because  each non-zero coefficient of $\frac{1}{\epsilon} C$ 
is at least $1$.  In particular,  replacing $\epsilon$ with the minimum of $\Psi^{>0}\cup \{\epsilon\}$, we can assume 
the coefficients of $\Theta_F+C_F$ belong to $\{0\}\cup[\epsilon,\infty)$.

Now applying Proposition \ref{p-log-bir-bnd-cert-pairs} to $F, B_F:=\Theta_F+C_F, M_F$, 
there is a bounded set of couples $\mathcal{P}$ 
depending only on $d,v,\epsilon$ such that there is a projective log smooth couple 
$(\overline{F},{\Sigma}_{\overline{F}})\in \mathcal{P}$ 
and a birational map $\overline{F}\bir F$ satisfying: 
\begin{itemize}
\item  ${\Sigma}_{\overline{F}}$ contains the exceptional 
divisor of  $\overline{F}\bir F$ and the birational transform of $\Supp (\Theta_F+C_F+M_F)$, and 

\item if $f\colon F'\to F$ and $g\colon F'\to \overline{F}$ is a common resolution and 
$M_{\overline{F}}$ is the pushdown of $M_F|_{F'}$, then each coefficient of $M_{\overline{F}}$ is at most $c$.\\ 
\end{itemize}

\emph{Step 4.} 
In this step we compare log divisors on $F$ and $\overline{F}$.
 First define $ {\Gamma}_{\overline{F}}:=(1-\epsilon) {\Sigma}_{\overline{F}}$.
Let $K_{F'}+\tilde{C}_{F'}$ be the pullback of $K_{F}+\tilde{C}_{F}$ and let $K_{\overline{F}}+\tilde{C}_{\overline{F}}$
be the pushdown of $K_{F'}+\tilde{C}_{F'}$ to $\overline{F}$. We claim that $\tilde{C}_{\overline{F}}\le {\Gamma}_{\overline{F}}$. 
If $\tilde{C}_{\overline{F}}\le 0$, then the claim holds trivially. 
Assume  $\tilde{C}_{\overline{F}}$ has a component $D$ with positive coefficient. Then $D$ 
is either exceptional$/F$ or is a component of the  birational transform of $\tilde{C}_{F}$ with positive coefficient. 
In the former case, $D$ is a component of ${\Sigma}_{\overline{F}}$ because  ${\Sigma}_{\overline{F}}$
contains the exceptional divisor of  $\overline{F}\bir F$. In the latter case,  $D$ is a component of 
the birational transform of $\Theta_F+C_F$ because $\tilde{C}_F\le \Theta_F+C_F$ by Step 2, hence again 
$D$ is a component of ${\Sigma}_{\overline{F}}$ as it contains 
the birational transform of $\Supp (\Theta_F+C_F+M_F)$. Moreover, since $(F,\tilde{C}_{F})$ is sub-$\epsilon$-lc, 
the coefficient of $D$ in $\tilde{C}_{\overline{F}}$ is at most $1-\epsilon$, hence 
$\mu_D\tilde{C}_{\overline{F}}\le \mu_D{\Gamma}_{\overline{F}}$. We have then proved the claim 
$\tilde{C}_{\overline{F}}\le {\Gamma}_{\overline{F}}$.\\

\emph{Step 5.} 
In this step we define a divisor $I_F$ and compare singularities on $F$ and $\overline{F}$. 
Let $I_F:=\Theta_F+P_F-\Lambda_F$. By Step 2, $I_F\ge 0$. Pick $0\le L_F\sim_\R M_F$ and  
assume $t>0$. Let  $I_{\overline{F}}$ and $L_{\overline{F}}$ be the pushdowns of $I_F|_{F'}$ and 
$L_F|_{F'}$ to $\overline{F}$. Then 
$$
I_{{F}}+tL_{{F}}=\Theta_F+P_F-\Lambda_F+tL_F=K_F+\Theta_F+P_F-K_F-\Lambda_F+tL_F 
$$
$$
\sim_\R  (K_X+\Delta)|_F-K_X|_F+t M_{{F}}
\sim_\R \Delta|_F+t M_{{F}}\sim_\R (\alpha+t)M_F.
$$
Thus we get 
$$
I_{\overline{F}}+tL_{\overline{F}}\sim_\R (\alpha+t)M_{\overline{F}}.
$$

On the other hand, our assumptions ensure that 
$$
K_F+\tilde{C}_F+I_F+tL_F=K_F+C_F+\Lambda_F+\Theta_F+P_F-\Lambda_F+tL_F
$$
$$
=K_F+C_F+\Theta_F+P_F+tL_F\sim_\R (K_X+C+\Delta+tM)|_F
$$ 
is ample. Therefore, 
$$
f^*(K_F+\tilde{C}_F+I_F+tL_F)\le g^*(K_{\overline{F}}+\tilde{C}_{\overline{F}}+I_{\overline{F}}+tL_{\overline{F}})
$$
which implies that 
$$
(F,\tilde{C}_F+I_F+tL_F)
$$
is sub-$\epsilon'$-lc if 
 $$
 ({\overline{F}},\tilde{C}_{\overline{F}}+I_{\overline{F}}+tL_{\overline{F}})
 $$ 
 is sub-$\epsilon'$-lc.\\

\emph{Step 6.} 
In this step we finish the proof using Proposition \ref{p-non-term-places}. 
Pick $l\in\N$ such that ${(l-1)\epsilon}>l\epsilon'$.
Since the coefficients of $M_{\overline{F}}$ are bounded from above by $c$, by Step 3, 
applying Proposition \ref{p-non-term-places}, we deduce that 
$$
({\overline{F}},{\Gamma}_{\overline{F}}+lI_{\overline{F}}+ltL_{\overline{F}})
$$
is klt if $\alpha+t$ is sufficiently small depending only on $\mathcal{P},\epsilon,c$, recalling that 
$I_{\overline{F}}+tL_{\overline{F}}\sim_\R (\alpha+t)M_{\overline{F}}$ by the previous step. In particular, this holds 
assuming  $t>0$ is sufficiently small as $\alpha+t<2t$. From now on we assume $t$ is sufficiently small. Thus 
$$
({\overline{F}},{\Gamma}_{\overline{F}}+I_{\overline{F}}+tL_{\overline{F}})
$$ 
is ${\epsilon'}$-lc because for any prime divisor $D$ on birational models of $F$ we have 
$$
a(D,{\overline{F}},{\Gamma}_{\overline{F}}+I_{\overline{F}}+tL_{\overline{F}})
$$
$$
=\left(\frac{l-1}{l}\right)a(D,{\overline{F}},{\Gamma}_{\overline{F}})
+\frac{1}{l}a(D,{\overline{F}},{\Gamma}_{\overline{F}}+lI_{\overline{F}}+ltL_{\overline{F}})\ge \left(\frac{l-1}{l}\right)\epsilon>\epsilon'.
$$
This then 
 implies that 
 $$
 ({\overline{F}},\tilde{C}_{\overline{F}}+I_{\overline{F}}+tL_{\overline{F}})
 $$ 
 is sub-$\epsilon'$-lc as $\tilde{C}_{\overline{F}}\le \Gamma_{\overline{F}}$ by Step 4.
Therefore, by Step 5, 
$$
(F,\tilde{C}_F+I_F+tL_F)
$$
is also sub-$\epsilon'$-lc. In other words,
$$
(F,C_F+\Theta_F+P_F+tL_F)
$$ 
is  $\epsilon'$-lc.

\end{proof}

\subsection{Effective birationality for Fano varieties with good $\Q$-complements}

Our next result is an attempt to relate effective birationality on Fano varieties $X$ and the volume of $-K_X$.
This is  crucial for the proofs of \ref{p-eff-bir-delta-2} and \ref{p-eff-bir-tau}.

\begin{prop}\label{p-eff-bir-delta-1}
Let $d\in\N$ and $\epsilon,\delta \in \R^{>0}$. Then there exists a number $p\in\N$ depending only on $d,\epsilon,$ and $\delta$ 
satisfying the following. Assume 
\begin{itemize}
\item $X$ is an $\epsilon$-lc Fano variety of dimension $d$, 

\item $m\in \N$ is the smallest number such that $|-mK_X|$ defines a birational map,

\item  $n\in \N$ is a number such that $\vol(-nK_X)> (2d)^d$, and 

\item $nK_X+N\sim_\Q 0$ for some $\Q$-divisor $N\ge 0$ with coefficients $\ge \delta$.\\
\end{itemize}
Then $\frac{m}{n}<p$. 
\end{prop}

\begin{proof}
The idea of the proof is to apply the frequently used method of showing birationality 
of a linear system by creating non-klt centres. If the centres we create happen to be 
zero dimensional, we are ready immediately. Otherwise the centres are positive dimensional 
and we can cut them and decrease their dimension unless the volume of the restriction of 
$-nK_X$ to these centres is too small in which case we will get a contradiction 
by applying \ref{p-bnd-sing-on-non-klt-centre}.\\

\emph{Step 1.}
If the proposition does not hold, then there is a sequence $X_i,m_i,n_i,N_i$ of 
Fano varieties, numbers, and divisors as in the statement such that the 
numbers $\frac{m_i}{n_i}$ form a strictly increasing sequence approaching $\infty$. 
At the end of the proof we will use Proposition \ref{p-bnd-sing-on-non-klt-centre} to derive a contradiction.\\ 

\emph{Step 2.}
In this step we fix $i$ and create a family of non-klt centres on $X_i$. 
 Applying \ref{ss-non-klt-centres} (2),   
there is a  covering family of subvarieties of $X_i$ (as in \ref{ss-cov-fam-subvar})
such that for any two general closed points $x_i,y_i\in X_i$ we can choose a  member  
 $G_i$ of the family and choose a $\Q$-divisor 
$0\le \Delta_i\sim_\Q -(n_i+1)K_{X_i}$  so that $(X_i,\Delta_i)$ is lc near $x_i$ with a unique 
non-klt place whose centre contains $x_i$, that centre is $G_i$, and $(X_i,\Delta_i)$ is not klt near $y_i$. 
Note that since $x_i,y_i$ are general, we can assume $G_i$ is a general member of 
the family. Recall from \ref{ss-cov-fam-subvar} that this means the family is given by finitely many 
morphisms $V^j\to T^j$ of projective varieties with accompanying surjective morphisms $V^j\to X$ 
and that each $G_i$ is a general fibre of one of these morphisms. Moreover, we can assume 
the points of $T^j$ corresponding to such $G_i$ are dense in $T^j$. Let $d_i:=\max\{\dim V^j-\dim T^j\}$.

If $d_i=0$, that is, if $\dim G_i=0$ for all the $G_i$, then $-2(n_i+1)K_{X_i}$ is potentially birational, 
hence $|K_{X_i}-2(n_i+1)K_{X_i}|$ defines a 
birational map [\ref{HMX}, Lemma 2.3.4] which means $m_i\le 2n_i+1$ giving a contradiction as 
we can assume $m_i/n_i\gg 0$. Thus we can assume $d_i>0$, so  $\dim G_i>0$ for all the $G_i$  
appearing as general fibres of $V^j\to T^j$ for some $j$.

Define $l_i\in\N$ to be the smallest number so that $\vol(-l_iK_{X_i}|_{G_i})>d^d$ for all 
the $G_i$ with positive dimension. 
Then we can assume there is $j$ so that if $G_i$ is a general fibre of 
$V^j\to T^j$, then $G_i$ is positive dimensional and $\vol(-(l_i-1)K_{X_i}|_{G_i})\le d^d$. \\
 
\emph{Step 3.} 
In this step we reduce the problem to the case when $\vol(-m_iK_{X_i}|_{G_i})$ is bounded from above. 
Assume $\frac{l_i}{n_i}$ is bounded from above by some natural number $a$.
Then after replacing $n_i$ with $dan_i$ and applying the second paragraph of \ref{ss-non-klt-centres} (2), 
for each $i$, we can replace the positive dimensional $G_i$ with new ones of smaller dimension, and 
replace the family accordingly, hence decrease the number $d_i$.  
Repeating the process we get to the situation in which we can assume $\frac{l_i}{n_i}$ is an 
increasing sequence approaching $\infty$ otherwise we get the case $d_i=0$ which yields a contradiction 
as in Step 2. On the other hand, if $\frac{m_i}{l_i}$ is not bounded from above, then we can assume 
$\frac{m_i}{l_i}$ is an increasing sequence approaching $\infty$, hence 
we can replace $n_i$ with $l_i$ in which case $\frac{l_i}{n_i}$ is bounded and we 
can argue as before. So we can assume $\frac{m_i}{l_i}$ is bounded from above.

In order to get a contradiction in the following steps it suffices to consider only those $G_i$  
which are positive dimensional and $\vol(-(l_i-1)K_{X_i}|_{G_i})\le d^d$. By Step 2, there is 
a sub-family of such $G_i$ appearing as general fibres of some $V^j\to T^j$.  
\emph{From now on when we mention $G_i$ we assume it is positive dimensional and it 
satisfies the  inequality just stated}.
  In particular,  
$$
\vol(-m_iK_{X_i}|_{G_i})=\left(\frac{m_i}{l_i-1}\right)^d\vol(-(l_i-1)K_{X_i}|_{G_i})\le \left(\frac{m_i}{l_i-1}\right)^dd^d
$$ 
is bounded from above, so $\vol(-m_iK_{X_i}|_{G_i})<v$ for some fixed number $v$ independent of $i$.\\

\emph{Step 4.} 
Let $F_i$ be the normalisation of $G_i$. In this step we look at adjunction by restricting to $F_i$.
Since $G_i$ is a general member of a covering family, $X_i$ is $\Q$-factorial near the generic point 
of $G_i$. By Construction \ref{constr-adjunction-non-klt-centre} and Theorem \ref{t-subadjunction} (by taking 
$B=0$ and $\Delta=\Delta_i$), we can write 
$$
K_{F_i}+\Theta_{F_i}+P_{F_i}\sim_\R (K_{X_i}+\Delta_i)|_{F_i}
$$
where $P_{F_i}$ is pseudo-effective. Pick $0\le Q_i\sim_\Q -n_iK_{X_i}$ not containing $x_i$. By definition of 
$\Theta_{F_i}$, adding $Q_i$ to 
$\Delta_i$ does not change $\Theta_{F_i}$ but changes $P_{F_i}$ to $P_{F_i}+Q_i|_{F_i}$. Thus replacing 
$n_i$ with $2n_i$ and changing $P_{F_i}$ up to $\R$-linear equivalence we can assume $P_{F_i}$ is effective 
and big.\\

\emph{Step 5.}
In this step we get a contradiction by applying Proposition \ref{p-bnd-sing-on-non-klt-centre}. 
By construction, $K_{X_i}+\Delta_i\sim_\Q -n_iK_{X_i}$ is ample. 
Let $M_i:=-m_iK_{X_i}$ and $M_{F_i}:=M_i|_{F_i}$. Then we can assume 
$M_i-(K_{X_i}+\Delta_i)\sim_\Q -(m_i-n_i)K_{X_i}$ is also ample. Moreover, $\Delta_i\sim_\Q \frac{n_i+1}{m_i}M_i$.
On the other hand, 
since $G_i$ is general, it is not contained in $\Supp N_i$, so defining 
$N_{F_i}:=N_i|_{F_i}$ we get the effective divisor $L_{F_i}:=\frac{m_i}{n_i}N_{F_i}$. Since $N_i\sim_\Q -n_iK_{X_i}$, 
we get $L_{F_i}\sim_\R M_{F_i}$. 

Now let $t$ be the number given by Proposition \ref{p-bnd-sing-on-non-klt-centre} 
for the data $d,v,\epsilon, \epsilon'=\frac{\epsilon}{2}$ where we can assume $\epsilon<\frac{1}{2}$ 
by decreasing it if necessary. We can assume $\frac{n_i+1}{m_i}<t$ for every $i$.
Applying the proposition to $X_i,M_{i},\Delta_i, G_i,F_i,\Theta_{F_i},P_{F_i}$ (here we take $C_i=0$), we deduce that  
$$
(F_i, \Theta_{F_i}+P_{F_i}+tL_{F_i})
$$ 
is $\frac{\epsilon}{2}$-lc for every $i$.

Let $D$ be a component of $J_{F_i}:=\frac{1}{\delta}N_{F_i}$. Since each coefficient of $\frac{1}{\delta}N_i$ is 
at least $1$,  by Lemma \ref{l-subadjunction-integral-div},  
$$
\mu_D(\Theta_{F_i}+P_{F_i}+J_{F_i})\ge \mu_D(\Theta_{F_i}+J_{F_i})\ge 1.
$$
By the previous paragraph, $\mu_D\Theta_{F_i}\le 1-\frac{\epsilon}{2}$, hence $ \mu_DJ_{F_i}\ge \frac{\epsilon}{2}$. But then 
$$
\mu_DtL_{F_i}=\left(\frac{t\delta m_i}{n_i}\right)\mu_DJ_{F_i}=\frac{t\delta\epsilon m_i}{2n_i}\gg 0
$$ 
when $i$ is large. This contradicts the $\frac{\epsilon}{2}$-lc property in the previous paragraph.

\end{proof}

\begin{prop}\label{p-eff-bir-delta-2}
Let $d\in\N$ and $\epsilon,\delta \in \R^{>0}$. Then there exists a number $m\in\N$ depending 
only on $d,\epsilon$, and $\delta$ 
satisfying the following. Assume $X$ is an $\epsilon$-lc Fano variety of dimension $d$ such that 
$K_X+B\sim_\Q 0$ for some $\Q$-divisor $B\ge 0$ with coefficients $\ge \delta$. 
Then $|-mK_X|$ defines a birational map. 
\end{prop}
\begin{proof}
\emph{Step 1}.
In this step we set up the notation and bound certain volumes.
If the proposition is not true, then there is a sequence of Fano varieties $X_i$ and $\Q$-divisors $B_i$ 
satisfying the assumptions of the proposition but such that if $m_i\in\N$ is the smallest 
number so that $|-m_iK_{X_i}|$ defines a birational map, then the $m_i$ form a strictly increasing 
sequence approaching $\infty$. Let $n_i\in\N$ be the smallest number so that $\vol(-n_iK_{X_i})>(2d)^d$. 
Obviously the coefficients of $N_i:=n_iB_i$ are $\ge \delta$, and $n_iK_{X_i}+N_i\sim_\Q 0$.
Thus by Proposition \ref{p-eff-bir-delta-1}, there is a number $p\in\N$ independent of $i$ such that 
$\frac{m_i}{n_i}<p$. In particular, we can assume $n_i>1$.
Therefore, $\vol(-m_iK_{X_i})$ is bounded from above because 
$$
\vol(-m_iK_{X_i})=\left(\frac{m_i}{n_i-1}\right)^d\vol(-(n_i-1)K_{X_i})\le \left(\frac{m_i}{n_i-1}\right)^d(2d)^d.
$$\ 

\emph{Step 2}.
In this step we find a bounded birational model of $X$.
Let $M_i$ be a general element of $|-m_iK_{X_i}|$. 
We show $(X_i,\Supp(B_i+M_i))$ is log birationally bounded.
After replacing $\epsilon$ with $\min\{\epsilon,\delta\}$, we can assume the coefficients of 
$B_i$ belong to $\{0\}\cup [\epsilon,\infty)$. Applying Proposition \ref{p-log-bir-bnd-cert-pairs} 
to $X_i,B_i,M_i$, 
 there is a bounded set of couples $\mathcal{P}$ and a number $c\in\R^{>0}$ 
such that for each $i$ there is a projective log smooth couple $(\overline{X}_i,{\Sigma}_{\overline{X}_i})\in \mathcal{P}$ 
and a birational map $\overline{X}_i\bir X_i$ such that 
\begin{itemize}
\item  $\Supp {\Sigma}_{\overline{X}_i}$ contains the exceptional 
divisor of  $\overline{X}_i\bir X_i$ and the birational transform of $\Supp (B_i+{M}_i)$, and 

\item if $W_i\to X_i$ and $W_i\to \overline{X}_i$ is a common resolution and 
$M_{\overline{X}_i}$ is the pushdown of $M|_{W_i}$, then each coefficient of $M_{\overline{X}_i}$ is at most $c$. 
\end{itemize}\

\emph{Step 3.}
In this step we derive a contradiction using Proposition \ref{p-non-term-places}.
Let $K_{W_i}+\Lambda_{W_i}$ be the 
pullback of $K_{X_i}$ and let $K_{\overline{X}_i}+\Lambda_{\overline{X}_i}$ be its pushdown on $\overline{X}_i$.
The crepant pullback of 
$$
K_{X_i}+\Delta_i:=K_{X_i}+\frac{1}{m_i}M_{i}\sim_\Q 0
$$ 
to $\overline{X}_i$ is  
$$
K_{\overline{X}_i}+\Delta_{\overline{X}_i}:=
K_{\overline{X}_i}+\Lambda_{\overline{X}_i}+\frac{1}{m_i}M_{\overline{X}_i}\sim_\Q 0.
$$ 
Note that the coefficients of $\Lambda_{\overline{X}_i}$ are at most $1-\epsilon$ as $X_i$ is $\epsilon$-lc, 
and the support of 
$\Delta_{\overline{X}_i}$ is contained in $\Sigma_{\overline{X}_i}$.

By Step 2, if $m_i$ is sufficiently large, then the coefficients of 
$\frac{1}{m_i}M_{\overline{X}_i}$ are sufficiently small.
Therefore,  letting  $\Gamma_{\overline{X}_i}:=(1-\frac{\epsilon}{2})\Sigma_{\overline{X}_i}$ 
we have $\Delta_{\overline{X}_i}\le \Gamma_{\overline{X}_i}$ for $i\gg 0$. 
Now  let $L_{i}:=\frac{1}{\delta}B_i$ which has coefficients $\ge 1$. Then 
$({X_i},\Delta_i+L_{i})$ 
is not klt and   
$K_{X_i}+\Delta_i+L_{i}$ 
is ample. Therefore, if $L_{\overline{X}_i}$ is the pushdown of $L_{i}|_{X_i}$, then 
$(\overline{X}_i,\Delta_{\overline{X}_i}+L_{\overline{X}_i})$
is not sub-klt which in turn implies $(\overline{X}_i,\Gamma_{\overline{X}_i}+L_{\overline{X}_i})$ is not klt. 
This contradicts Proposition \ref{p-non-term-places} because 
$L_{\overline{X}_i}\sim_\R \frac{1}{\delta m_i}M_{\overline{X}_i}$.
    
\end{proof}

\subsection{Effective birationality for nearly canonical Fano varieties}

The next result treats one of the main special cases of \ref{t-eff-bir-e-lc} when $X$ has canonical or 
nearly canonical singularities. This is particularly useful when $X$  is exceptional and we cannot 
create deep singularities using divisors $0\le D\sim_\Q -K_X$, e.g. end of proof of \ref{l-from-compl-to-BAB-exc-usual} 
which is an inductive treatment of \ref{t-BAB-exc}.

\begin{prop}\label{p-eff-bir-tau}
Let $d\in \N$. Then there exist numbers $\tau\in (0,1)$ and $m\in\N$ depending only on $d$ 
satisfying the following. 
If $X$ is a $\tau$-lc Fano variety of dimension $d$,  
then $|-mK_{X}|$ defines a birational map.
\end{prop}
\begin{proof}
We first give a short summary of the proof. Steps 1-4 are quite similar to those of the proof of \ref{p-eff-bir-delta-1}. 
We take $n\in\N$ so that $\vol(-nK_X)>(2d)^d$ and much of the proof is spent on showing $\frac{m}{n}$ is bounded from above 
arguing by contradiction.
We create a covering family of non-klt centres $G$ on $X$. Again the 
difficult case is when the centres are positive dimensional. We argue 
that singularities on the normalisation $F$ of $G$ cannot be too bad. Next we find a smooth and bounded birational 
model $\overline{F}$ of $F$ and reduce to the situation when $\kappa(K_{\overline{F}})=0$. 
This in turn allows us to reduce to the case when $h^0(-rK_X|_F)>0$ for some bounded $r\in \N$. Next we use \ref{p-lift-section-lcc} 
to lift sections and produce some $N\ge 0$  with 
coefficients bounded from below and satisfying $nK_X+N\sim_\Q 0$. At this point we apply 
\ref{p-eff-bir-delta-1} to deduce that $\frac{m}{n}$ is bounded from above. 
At the last step we work on a bounded birational model $\overline{X}$ of $X$ 
and apply \ref{l-p-eff-thr-bnd-fam} to show $m$ is bounded from above.\\

\emph{Step 1}.
In this step we setup basic notation. 
If the proposition does not hold, then there is a sequence $X_i$ of Fano varieties 
of dimension $d$ and an increasing sequence $\epsilon_i$ of numbers in $(0,1)$ approaching 
$1$ such that $X_i$ is $\epsilon_i$-lc and if $m_i\in\N$ is the smallest 
number such that $|-m_iK_{X_i}|$ defines a birational map, then the $m_i$ form 
an increasing sequence approaching $\infty$. 
Let $n_i\in\N$ be a number so that $\vol(-n_iK_{X_i})>(2d)^d$. 
First we want to show $\frac{m_i}{n_i}$ is bounded from above. 
Assume this is not the case, so we can assume the  
$\frac{m_i}{n_i}$ form an increasing sequence approaching $\infty$. 
We will derive a contradiction by the end of Step 9. Finally in Step 10 we prove $m_i$ is bounded which is again 
a contradiction.\\

\emph{Step 2.}
In this step we fix $i$ and create a covering family of non-klt centres on $X_i$. 
Applying \ref{ss-non-klt-centres} (2),   
there is a  covering family of subvarieties of $X_i$ (as in \ref{ss-cov-fam-subvar})
such that for any two general closed points $x_i,y_i\in X_i$ we can choose a  member  
 $G_i$ of the family and choose a $\Q$-divisor 
$0\le \Delta_i\sim_\Q -(n_i+1)K_{X_i}$  so that $(X_i,\Delta_i)$ is lc near $x_i$ with a unique 
non-klt place whose centre contains $x_i$, that centre is $G_i$, and $(X_i,\Delta_i)$ is not klt near $y_i$. 
Note that since $x_i,y_i$ are general, we can assume $G_i$ is a general member of 
the family. Recall from \ref{ss-cov-fam-subvar} that this means the family is given by finitely many 
morphisms $V^j\to T^j$ of projective varieties with accompanying surjective morphisms $V^j\to X$ 
and that each $G_i$ is a general fibre of one of these morphisms. Moreover, we can assume 
the points of $T^j$ corresponding to such $G_i$ are dense in $T^j$. Let $d_i:=\max\{\dim V^j-\dim T^j\}$.

If $d_i=0$, that is, if $\dim G_i=0$ for all the $G_i$, then $-2(n_i+1)K_{X_i}$ is potentially birational, 
hence $|K_{X_i}-2(n_i+1)K_{X_i}|$ defines a 
birational map [\ref{HMX}, Lemma 2.3.4] which means $m_i\le 2n_i+1$ giving a contradiction as 
we can assume $m_i/n_i\gg 0$. Thus we can assume $d_i>0$, hence $\dim G_i>0$ for all the $G_i$  
appearing as general fibres of $V^j\to T^j$ for some $j$. 

Define $l_i\in\N$ to be the smallest number so that $\vol(-l_iK_{X_i}|_{G_i})>d^d$ for all 
the $G_i$ with positive dimension. 
Then we can assume there is $j$ so that if $G_i$ is a general fibre of 
$V^j\to T^j$, then $G_i$ is positive dimensional and $\vol(-(l_i-1)K_{X_i}|_{G_i})\le d^d$. \\
 
\emph{Step 3.} 
In this step we reduce the problem to the case when $\vol(-m_iK_{X_i}|_{G_i})$ is bounded from above. 
Assume $\frac{l_i}{n_i}$ is bounded from above by some natural number $a$.
Then after replacing $n_i$ with $dan_i$ and applying the second paragraph of \ref{ss-non-klt-centres} (2), 
for each $i$, we can replace the positive dimensional $G_i$ with new ones of smaller dimension, and 
replace the family accordingly, hence decrease the number $d_i$.  
Repeating the process we get to the situation in which we can assume $\frac{l_i}{n_i}$ is an 
increasing sequence approaching $\infty$ otherwise we get the case $d_i=0$ which yields a contradiction 
as in Step 2. On the other hand, if $\frac{m_i}{l_i}$ is not bounded from above, then we can assume 
$\frac{m_i}{l_i}$ is an increasing sequence approaching $\infty$, hence 
we can replace $n_i$ with $l_i$ in which case $\frac{l_i}{n_i}$ is bounded and we 
can argue as before. So we can assume $\frac{m_i}{l_i}$ is bounded from above.

In order to get a contradiction in the following steps it suffices to consider only those $G_i$  
which are positive dimensional and $\vol(-(l_i-1)K_{X_i}|_{G_i})\le d^d$. By Step 2, there is 
a sub-family of such $G_i$ appearing as general fibres of some $V^j\to T^j$.  
\emph{From now on when we mention $G_i$ we assume it is positive dimensional and it 
satisfies the  inequality just stated}.
  In particular,  
$$
\vol(-m_iK_{X_i}|_{G_i})=\left(\frac{m_i}{l_i-1}\right)^d\vol(-(l_i-1)K_{X_i}|_{G_i})\le \left(\frac{m_i}{l_i-1}\right)^dd^d
$$ 
is bounded from above, so $\vol(-m_iK_{X_i}|_{G_i})<v$ for some fixed number $v$.\\

\emph{Step 4.} 
Let $F_i$ be the normalisation of $G_i$. In this step we look at adjunction by restricting to $F_i$.
Since $G_i$ is a general member of a covering family, $X_i$ is $\Q$-factorial near the generic point 
of $G_i$. By Construction \ref{constr-adjunction-non-klt-centre} and Theorem \ref{t-subadjunction} (by taking 
$B=0$ and $\Delta=\Delta_i$), we can write 
$$
K_{F_i}+\Delta_{F_i}:=K_{F_i}+\Theta_{F_i}+P_{F_i}\sim_\R (K_{X_i}+\Delta_i)|_{F_i}
$$
where $P_{F_i}$ is pseudo-effective. Pick  $0\le Q_i\sim_\Q -n_iK_{X_i}$ not containing $x_i$. By definition of $\Theta_{F_i}$, 
adding $Q_i$ to $\Delta_i$ does not change $\Theta_{F_i}$ but changes $P_{F_i}$ to $P_{F_i}+Q_i|_{F_i}$. Thus replacing 
$n_i$ with $2n_i$ and changing $P_{F_i}$ up to $\R$-linear equivalence we can assume 
$P_{F_i}$ is effective and big.\\

\emph{Step 5.}
In this step we reduce to the situation in which $(F_i,\Delta_{F_i})$ is ${\epsilon}'$-lc 
for some $\epsilon'>0$ and that $\Theta_{F_i}=0$,  for every $i$.  
By construction,  $K_{X_i}+\Delta_i\sim_\Q -n_iK_{X_i}$ is ample. 
Assume $0\le M_i\sim -m_iK_{X_i}$ and $M_{F_i}:=M_i|_{F_i}$. Then we can assume 
$M_i-(K_{X_i}+\Delta_i)\sim_\Q -(m_i-n_i)K_{X_i}$ is also ample. 
Moreover, $\Delta_i\sim_\Q \frac{n_i+1}{m_i}M_i$.
Pick $0\le L_{F_i}\sim_\R M_{F_i}$. 

Let $\epsilon'<\epsilon$ be positive real numbers such that $\epsilon<{\epsilon_i}$ for every $i$.
Now let $t$ be the number given by Proposition \ref{p-bnd-sing-on-non-klt-centre} 
for the data $d,v,\epsilon, \epsilon'$. We can assume $\frac{n_i+1}{m_i}<t$ for every $i$.
Applying the proposition to $X_i,M_i,\Delta_i,G_i,F_i,\Theta_{F_i},P_{F_i}$ (here we take $C_i=0$), we deduce that  
$(F_i, \Delta_{F_i}+tL_{F_i})$ is ${\epsilon}'$-lc for every $i$, hence $(F_i, \Delta_{F_i})$ is 
${\epsilon}'$-lc for every $i$.

By Theorem \ref{t-subadjunction} and by the ACC for lc thresholds [\ref{HMX2}, Theorem 1.1], the 
coefficients of $\Theta_{F_i}$ belong to some fixed DCC set $\Psi$. 
We can assume $\epsilon_i$ is sufficiently close to $1$ and so we can also choose $\epsilon'$
to be close to $1$. This ensures that $\Theta_{F_i}=0$ by the $\epsilon'$-lc property of $(F_i, \Delta_{F_i})$  
and the fact that the coefficients of $\Theta_{F_i}$ are in the DCC set $\Psi$. \\

\emph{Step 6}.  
In this step we find a bounded birational model of $F_i$. 
By Lemma \ref{l-subadjunction-integral-div}, $\mu_DM_{F_i}=\mu_D(\Theta_{F_i}+M_{F_i})\ge 1$, for 
every component $D$ of $M_{F_i}$. Moreover, $M_{F_i}-K_{F_i}$ is big because 
$M_{F_i}-(K_{F_i}+\Delta_{F_i})$ is ample and $P_{F_i}$ is big. In addition, 
$|M_{F_i}|$ defines a birational map because $|M_i|$ defines a birational map and 
$G_i$ is a general member of a covering family of subvarieties.

Now applying 
Proposition \ref{p-log-bir-bnd-cert-pairs} (by taking $X=F_i, B=\Theta_{F_i}=0, M=M_{F_i}$), 
there is a bounded set of couples $\mathcal{P}$ independent of $i$ such that for each $i$, 
we can find a projective log smooth couple
 $({\overline{F}_i},{\Sigma}_{\overline{F}_i})\in\mathcal{P}$ and a birational map $\overline{F}_i\bir F_i$ such that 
\begin{itemize}
\item  $\Supp {\Sigma}_{\overline{F}_i}$ contains the exceptional 
divisors of  $\overline{F}_i\bir F_i$ and the birational transform of $\Supp M_{F_i}$;

\item if $F_i'\to F_i$ and $F_i'\to \overline{F}_i$ is a common resolution and 
$M_{\overline{F}_i}$ is the pushdown of $M_{F_i'}:=M|_{F_i'}$, then each coefficient of $M_{\overline{F}_i}$ is at most $c$;

\item $M_{F_i'}\sim A_{F_i'}+R_{F_i'}$ where $A_{F_i'}$ is big, $|A_{F_i'}|$ is base point free, $R_{F_i'}\ge 0$, 
and $A_{F_i'}\sim 0/\overline{F}_i$.
\end{itemize}

In addition we can assume $A_{F_i'}$ is reduced and that $A_{\overline{F}_i}\le {\Sigma}_{\overline{F}_i}$ 
where $A_{\overline{F}_i}$ is the pushdown of $A_{F_i'}$.\\ 

\emph{Step 7}.  
In this step we reduce to the situation in which $K_{\overline{F}_i}$ is pseudo-effective. 
By Lemma \ref{l-sub-bnd-on-gen-subvar}, we can write $K_{F_i}+\Lambda_{F_i}=K_{X_i}|_{F_i}$ 
where $\Lambda_{F_i}\le \Theta_{F_i}=0$ and $({F_i}+\Lambda_{F_i})$ is sub-$\epsilon_i$-lc.
 Let $K_{F_i'}+\Lambda_{F_i'}$ 
and $M_{F_i'}$ be the pullbacks of $K_{F_i}+\Lambda_{F_i}$ and $M_{F_i}$ respectively, and in turn 
$K_{\overline{F}_i}+\Lambda_{\overline{F}_i}$ and $M_{\overline{F}_i}$ be their pushdowns to $\overline{F}_i$. 
From $K_{X_i}+\frac{1}{m_i}M_i\sim_\Q 0$ we get 
$$
K_{\overline{F}_i}+\Lambda_{\overline{F}_i}+\frac{1}{m_i}M_{\overline{F}_i}\sim_\Q 0.
$$
Moreover, any component of $\Lambda_{\overline{F}_i}$ with positive coefficient is exceptional$/F_i$, 
hence a component of $\Sigma_{\overline{F}_i}$, and its coefficient in $\Lambda_{\overline{F}_i}$ is at most $1-\epsilon_i$. 
So the coefficients of $(\Lambda_{\overline{F}_i}+\frac{1}{m_i}M_{\overline{F}_i})^{\ge 0}$ get arbitrarily small 
as $i$ gets large. 
As $({\overline{F}_i},\Sigma_{\overline{F}_i})$ is log bounded and 
$$
\Supp \left(\Lambda_{\overline{F}_i}+\frac{1}{m_i}M_{\overline{F}_i}\right)^{\ge 0}\subseteq \Sigma_{\overline{F}_i},
$$ 
we can assume $K_{\overline{F}_i}$ is pseudo-effective for every $i$, by 
Lemma \ref{l-p-eff-thr-bnd-fam}.\\

\emph{Step 8.}
In this step we reduce to the situation in which $\kappa_\sigma(K_{\overline{F}_i})=0$ for every $i$. 
First assume $\kappa_\sigma(K_{\overline{F}_i})>0$ for every $i$. 
Perhaps after adding to $\Sigma_{\overline{F}_i}$ and replacing $\mathcal{P}$ accordingly, 
we can assume that there is a very ample divisor $0\le H_{\overline{F}_i}\le \Sigma_{\overline{F}_i}$, 
for each $i$. Now for each number $q\in\N$ there is a number $p\in\N$ 
such that $\vol(pK_{\overline{F}_i}+H_{\overline{F}_i})>q$ for every $i$, by Lemma \ref{l-vol-kappa(K)>0}.

Since $A_{\overline{F}_i}$ is big and $A_{\overline{F}_i}\le \Sigma_{\overline{F}_i}$, 
we can assume that there is $l\in\N$ such that 
$lA_{\overline{F}_i}-H_{\overline{F}_i}$ is big for each $i$.
Thus $\vol(pK_{\overline{F}_i}+lA_{\overline{F}_i})>q$ which implies $\vol(pK_{{F}_i'}+lA_{{F}_i'})>q$ 
and this in turn gives $\vol(pK_{{F}_i}+lA_{{F}_i})>q$ where $A_{F_i}$ is the pushdown of $A_{F_i'}$. 
Therefore, both sides of the inequality
$$
\vol\left(\frac{m_i}{n_i}(K_{{F}_i}+\Delta_{{F}_i})+lA_{{F}_i}\right)\ge 
\vol\left(\frac{m_i}{n_i}K_{{F}_i}+lA_{{F}_i}\right)
$$
go to $\infty$ as $i$ goes to $\infty$. 
But 
$$
\vol\left(\frac{m_i}{n_i}(K_{{F}_i}+\Delta_{{F}_i})+lA_{{F}_i}\right)=\vol\left(\frac{m_i}{n_i}(-n_iK_{{X}_i})|_{F_i}+lA_{{F}_i}\right)
$$
$$
\le \vol((-m_iK_{{X}_i}+lM_{i})|_{F_i})= \vol((-m_i(1+l) K_{X_i})|_{F_i})
$$
and the right hand side is bounded from above, a contradiction. 
Thus from now on we can assume $\kappa_\sigma(K_{\overline{F}_i})=0$ for every $i$.\\

\emph{Step 9}.
In this step we get a contradiction for the assumption that the sequence $\frac{m_i}{n_i}$ approaches $\infty$.
Since $\kappa_\sigma(K_{\overline{F}_i})=0$, there is $r\in \N$ such that $h^0(rK_{\overline{F}_i})\neq 0$ 
for every $i$, by Lemma \ref{l-non-van-kappa(K)=0}. Then  $h^0(rK_{{F}_i})\neq 0$ for every $i$, hence $rK_{F_i}\sim T_{F_i}$ 
for some integral divisor $T_{F_i}\ge 0$. First assume $T_{F_i}\neq 0$ for every $i$. 
Then 
$$
L_{F_i}:=\frac{m_i}{n_i} {\left( \Delta_{F_i}+\frac{1}{r}T_{F_i} \right)} \sim_\Q \frac{m_i}{n_i} ( K_{F_i}+\Delta_{F_i} )
\sim_\R  \frac{m_i}{n_i}( K_{X_i}+\Delta_i )|_{F_i}
$$
$$
\sim_\Q \frac{m_i}{n_i}(-n_iK_{X_i})|_{F_i}\sim_\Q M_i|_{F_i}=M_{F_i}
$$
and $L_{F_i}\ge \frac{m_i}{rn_i}T_{F_i}$. In particular, 
 $({F_i},\Delta_{F_i}+t L_{F_i})$ is not klt for any $i\gg 0$ where $t$ is as in Step 5, a contradiction. 
 
Now we can assume $T_{F_i}=0$ for every $i$. 
Then  
$$
h^0(-rK_{X_i}|_{{F}_i})=h^0(-r(K_{{F}_i}+\Lambda_{F_i}))=h^0(-r\Lambda_{F_i})\neq 0
$$ 
for every $i$ because $\Lambda_{F_i}\le 0$ by Step 7.
Thus by Step 5 and Proposition \ref{p-lift-section-lcc}, perhaps after replacing $r$ with a multiple,  
$h^0(-n_irK_{X_i})\neq 0$ for every $i$. Then $n_iK_{X_i}+N_i\sim_\Q 0$ for some $N_i\ge 0$ with coefficients 
$\ge \frac{1}{r}$. We can then apply Proposition \ref{p-eff-bir-delta-1} to deduce 
that $\frac{m_i}{n_i}$ is bounded from above, a contradiction.\\

\emph{Step 10.}
In this final step we get a contradiction for the assumption that $m_i$ is not bounded.
At this point we let $n_i\in\N$ be the smallest number so that $\vol(-n_iK_X)>(2d)^d$. 
We can assume $n_i>1$ for every $i$ since $\frac{m_i}{n_i}$ is bounded from above.
The volume $\vol(-m_iK_{X_i})$ is 
bounded from above because   
$$
\vol(-m_iK_{X_i})=\left(\frac{m_i}{n_i-1}\right)^d\vol(-(n_i-1)K_{X_i})\le \left(\frac{m_i}{n_i-1}\right)^d(2d)^d.
$$ 
Then by Proposition \ref{p-log-bir-bnd-cert-pairs}, 
there is a bounded set of couples $\mathcal{Q}$ and a number $c'\in\R^{>0}$ 
such that for each $i$ there is a projective log smooth couple $(\overline{X}_i,{\Sigma}_{\overline{X}_i})\in \mathcal{Q}$ 
and a birational map $\overline{X}_i\bir X_i$ such that 
\begin{itemize}
\item  $\Supp {\Sigma}_{\overline{X}_i}$ contains the exceptional 
divisor of  $\overline{X}_i\bir X_i$ and the birational transform of $\Supp ({M}_i)$, and 

\item if $W_i\to X_i$ and $W_i\to \overline{X}_i$ is a common resolution and 
$M_{\overline{X}_i}$ is the pushdown of $M_i|_{W_i}$, then each coefficient of $M_{\overline{X}_i}$ is at most $c'$. 
\end{itemize}

Let $K_{W_i}+\Lambda_{W_i}$ be the 
pullback of $K_{X_i}$ and let $K_{\overline{X}_i}+\Lambda_{\overline{X}_i}$ be its pushdown on $\overline{X}_i$.
The crepant pullback of $K_{X_i}+\frac{1}{m_i}M_{i}$ to $\overline{X}_i$ is  
$$
K_{\overline{X}_i}+\Lambda_{\overline{X}_i}+\frac{1}{m_i}M_{\overline{X}_i}\sim_\Q 0.
$$ 
Note that the coefficients of $\Lambda_{\overline{X}_i}$ are at most $1-\epsilon_i$ which are either 
negative or approach $0$ as $i$ goes to $\infty$. So if $i$ is sufficiently large, 
then the coefficients of 
$(\Lambda_{\overline{X}_i}+\frac{1}{m_i}M_{\overline{X}_i})^{\ge 0}$ are sufficiently small.
Thus  $K_{\overline{X}_i}$ is pseudo-effective for every $i\gg 0$, by Lemma 
\ref{l-p-eff-thr-bnd-fam}. This is a contradiction 
because $K_{W_i}$ is not pseudo-effective as $K_{X_i}$ is not pseudo-effective, hence $K_{\overline{X}_i}$ is not 
pseudo-effective for any $i$. Therefore, $m_i$ is bounded as required.

\end{proof}


\section{\bf Proof of Theorem \ref{t-BAB-good-boundary}}

In this section we prove Theorem \ref{t-BAB-good-boundary}. The main difficulty is to bound the anti-canonical volume 
which we tackle now before going into the proof of the theorem.

\begin{prop}\label{p-bnd-vol-good-boundary}
Let $d\in\N$ and $\epsilon,\delta \in \R^{>0}$. 
Then there is a number $v$ depending only on $d,\epsilon$, and $\delta$ such that for any  
$X$ as in Theorem \ref{t-BAB-good-boundary} we have $\vol(-K_X)<v$. 
\end{prop}
\begin{proof}
We first give a short summary of the proof. Using MMP we reduce to the case when $X$ is Fano. 
If $\vol(-K_X)$ is very large, then there is a very small $a>0$ such that $\vol(-aK_{X})>(2d)^d$. 
We will show that this leads to a contradiction by using arguments somewhat similar to the 
proof of \ref{p-eff-bir-delta-1}. Using divisors $0\le \Delta\sim_\Q -aK_X$ we create a 
covering family of non-klt centres $G$ of the pairs $(X,\Delta)$ such that the pair has other 
non-klt centres apart from $G$. The difficult case is  when $\dim G>0$ just as in the proof of \ref{p-eff-bir-delta-1}.  
Using the fact that $(X,\Delta)$ has other non-klt centres apart from $G$, 
we can create bad singularities on the normalisation $F$ of $G$. Finally making use of $B$ 
and applying \ref{p-bnd-sing-on-non-klt-centre} we derive a contradiction.\\

\emph{Step 1.}
In this step we reduce the problem to the situation in which $X$ is Fano, and 
introduce some notation.
If the statement is not true, then there is a sequence of pairs $(X_i,B_i)$ 
satisfying the properties listed in Theorem \ref{t-BAB-good-boundary} such that 
 $\vol(-K_{X_i})$ is an increasing sequence approaching $\infty$.
Taking a $\Q$-factorialisation we can assume $X_i$ is $\Q$-factorial.  
Since $B_i$ is big, $X_i$ is of Fano type.
Run an MMP on $-K_{X_i}\sim_\R B_i$ and let $X_i'$ be the resulting model. Since $B_i$ is big,  
$-K_{X_i'}$ is nef and big.  
Thus if $X_i'\to X_i''$ is the contraction defined by $-K_{X_i'}$, then $X_i''$ is  Fano. 
Moreover, $\vol(-K_{X_i''})=\vol(-K_{X_i})$, and  since $K_{X_i}+B_i\sim_\R 0$, $(X_i'',B_i'')$ is $\epsilon$-lc. 
Thus replacing $(X_i,B_i)$ with $(X_i'',B_i'')$, we can assume $X_i$ is Fano. Moreover, modifying $B_i$, 
we can assume it is a $\Q$-boundary.

By Proposition \ref{p-eff-bir-delta-2}, there is $m\in\N$ such that 
$|-mK_{X_i}|$ defines a birational map for every $i$. On the other hand, 
since  $\vol(-K_{X_i})$ is an increasing sequence approaching $\infty$, 
there is a strictly decreasing sequence $a_i\in \Q^{>0}$ approaching $0$ so that $\vol(-a_iK_{X_i})>(2d)^d$ 
for each $i$.\\ 

\emph{Step 2.} 
In this step we fix $i$ and create a covering family of non-klt centres on $X_i$. 
Applying \ref{ss-non-klt-centres} (2), 
 there exists a covering family of subvarieties of $X_i$  
such that for each pair of general closed points $x_i,y_i\in X_i$ 
there exist a general member $G_i$ of the family and a $\Q$-divisor 
$0\le \Delta_i\sim_\Q -a_iK_{X_i}$ such that $(X_i,\Delta_i)$ is lc at $x_i$ with a unique non-klt place 
whose centre contains $x_i$, that centre is $G_i$, and $(X_i,\Delta_i)$ is not klt at $y_i$. 
Moreover, perhaps after replacing $a_i$ with $3a_i$ and adding to $\Delta_i$ we can assume that 
 $(X_i,\Delta_i)$ is not lc at some fixed point of $X_i$: indeed applying \ref{ss-non-klt-centres} (2), 
there is $0\le \Xi_i\sim_\Q -a_iK_{X_i}$ such that $(X_i,\Xi_i)$ is not klt at some point $\xi_i\in X_i$; 
here $\Xi_i,\xi_i$ are fixed; then by adding $2\Xi_i$ to $\Delta_i$ we can assume $(X_i,\Delta_i)$ 
is not lc at $\xi_i$; note that since $x_i$ is general we are of course assuming $x_i\neq \xi_i$.

Recall from \ref{ss-cov-fam-subvar} that the $G_i$ in the last paragraph 
are among the general fibres of finitely many morphisms 
$V^j\to T^j$, and we can assume for each $j$ the points on $T^j$ corresponding to the $G_i$ are dense. 
We can assume $K_{X_i}+\Delta_i$ is anti-ample, so by the connectedness principle, $\dim G_i>0$.\\

 \emph{Step 3.} 
In this step we reduce the problem to the situation in which  $\vol(-mK_{X_i}|_{G_i})$ is bounded from above. 
For each $i$, let $b_i\in\Q$ be the smallest number so that $\vol(-b_i K_{X_i}|_{G_i})\ge d^d+1$ 
for all the $G_i$ in Step 2; equality holds on a subfamily of the $G_i$ which are 
general fibres of one of the morphisms $V^j\to T^j$.

Assume $b_i$ is not bounded from below, so we can assume the $b_i$ form a 
strictly decreasing sequence approaching $0$. 
Applying the second paragraph of \ref{ss-non-klt-centres} (2), replacing $a_i$ with $a_i+(d-1)b_i$  
and replacing the $\Delta_i$, we 
can replace each $G_i$ with one having smaller dimension. Introducing new $b_i$ as above and repeating 
the process leads us to the 
case when $b_i$ is bounded from below otherwise we get the case $\dim G_i=0$ which is not possible 
as mentioned above. 

In order to get a contradiction, it is enough, for each $i$, to consider a sub-family of the $G_i$ 
satisfying $\vol(-b_i K_{X_i}|_{G_i})=d^d+1$ and which are general fibres of one of the morphisms $V^j\to T^j$. 
\emph{From now on when we mention $G_i$ we mean one of those.}
In particular, $\vol(-mK_{X_i}|_{G_i})$ is bounded from above as $b_i$ is bounded from below.\\

\emph{Step 4.}
In this step we consider adjunction on non-klt centres.
For each $i$ pick a general $G_i$ as in the last paragraph, and let $F_i$ be its normalisation. 
In particular, $X_i$ is $\Q$-factorial near the generic point of $G_i$, and we can find a 
resolution $F_i'\to F_i$  so that we get an induced morphism $F_i'\to W_i$. 
Pick $0\le M_i\sim -mK_{X_i}$. Then $G_i$ is not contained in $\Supp M_i$ 
by the generality of $G_i$.

By Construction \ref{constr-adjunction-non-klt-centre} and Theorem \ref{t-subadjunction} 
(by taking $B=0$ and $\Delta=\Delta_i$) and the ACC for 
lc thresholds [\ref{HMX2}, Theorem 1.1], 
there is a $\Q$-boundary $\Theta_{F_i}$ with  
coefficients in a fixed DCC set $\Psi$ depending only on $d$ such that we can write 
$$
(K_{X_i}+\Delta_i)|_{F_i}\sim_\R K_{F_i}+\Delta_{F_i}:=K_{F_i}+\Theta_{F_i}+P_{F_i}
$$
where $P_{F_i}$ is pseudo-effective. Since $x_i$ is general, $x_i\notin \Supp M_i$. 
By definition of $\Theta_{F_i}$, adding $\lambda_i M_i$ to $\Delta_i$ does not change $\Theta_{F_i}$ 
but changes $P_{F_i}$ to $P_{F_i}+\lambda_i M_i|_{F_i}$ where $\lambda_i$ is a sufficiently 
small positive rational number. Thus replacing 
$a_i$ and changing $P_{F_i}$ up to $\R$-linear equivalence we can assume 
$P_{F_i}$ is effective and big.\\

\emph{Step 5.}
In this step we use Proposition \ref{p-bnd-sing-on-non-klt-centre} to derive a contradiction.
Pick $\epsilon'\in(0,\epsilon)$. By the connectedness 
principle, the non-klt locus of $(X_i,\Delta_i)$ is connected. Since $(X_i,\Delta_i)$ is not lc at some point $\xi_i$
by Step 2, the pair has a non-klt centre  intersecting $G_i$ but not equal to $G_i$. 
By Lemma \ref{l-subadjunction-non-epsilon-lc}, 
we can choose $P_{F_i}\ge 0$ so that $(F_i,\Delta_{F_i})$ is not $\epsilon'$-lc. 

On the other hand, by construction, both $K_{X_i}+B_i+\Delta_i\sim_\Q \Delta_i$ and 
$$
M_i-(K_{X_i}+B_i+\Delta_i)\sim_\Q M_i-\Delta_i\sim_\Q -(m-a_i)K_{X_i}
$$ 
are ample, $\Delta_i\sim_\Q \frac{a_i}{m}M_i$, and $\vol(M_i|_{F_i})$ 
is bounded from above. Moreover, decreasing $\epsilon$ we can assume $\epsilon\le \delta$, hence 
the coefficients of $B_i$ are $\ge \epsilon$. Therefore,  by Proposition \ref{p-bnd-sing-on-non-klt-centre}, 
$(F_i,B_i|_{F_i}+\Delta_{F_i})$ is $\epsilon'$-lc which contradicts the previous paragraph.

\end{proof}

\begin{proof}(of Theorem \ref{t-BAB-good-boundary})
Let $X'\to X$ be a small $\Q$-factorialisation. 
Run an MMP on $-K_{X'}$ and let $X''$ be the resulting model. Since $B$ is big, 
$-K_{X''}$ is nef and big, that is, $X''$ is a weak Fano. 
Assume $X''\to X'''$ is the contraction defined by $-K_{X''}$, hence $X'''$ is Fano. 
By Lemma \ref{l-bnd-Fano-change-model}, it is enough to show that 
such $X'''$ form a bounded family.
Therefore, replacing $X$ with $X''$ we can assume $X$ is Fano.

Pick $\epsilon'\in (0,\epsilon)$. Let $\Delta=(1+t)B$ for some $t>0$ so that 
$(X,\Delta)$ is $\epsilon'$-lc. By [\ref{HMX2}, Theorem 1.6], it is enough to show 
$(X,\Delta)$ is log birationally bounded which is equivalent to showing $(X,B)$ is log birationally bounded.

By Proposition \ref{p-eff-bir-delta-2}, there is $m\in\N$ depending only on $d,\epsilon,\delta$ such that 
$|-mK_X|$ defines a birational map.  Moreover, by Proposition \ref{p-bnd-vol-good-boundary}, 
$\vol(-mK_X)$ is bounded from above. Therefore, applying Proposition \ref{p-log-bir-bnd-cert-pairs} 
by taking some $0\le M\sim -mK_X$, we deduce that  $(X,\Supp B)$ is log birationally bounded as required.

\end{proof}


\section{\bf Boundedness of complements}

In this section we develop the theory of complements for generalised pairs following 
[\ref{shokurov-surf-comp}][\ref{PSh-I}][\ref{PSh-II}]. We prove various inductive statements 
before we come to the main result of this section (\ref{p-BAB-exc-to-compl}). Let $(X',B'+M')$ be as 
in Theorem \ref{t-bnd-compl}. Such pairs are of two types: non-exceptional and exceptional. 
The main point of this section is that we can treat the non-exceptional ones inductively by creating 
generalised non-klt centres or by using fibrations. 
The exceptional case is treated in the next section where the focus will be on proving that $X'$ is bounded. 

Assume $(X',B'+M')$ is non-exceptional. The main ideas in constructing a bounded complement 
for $K_{X'}+B'+M'$ are essentially as follows. 
By creating deep singularities and then modifying the pair we can assume 
 $(X',B'+M')$ is not generalised klt and $B'\in\mathfrak{R}$. If $K_{X'}+B'+M'\sim_\Q 0$ and $M'\sim_\Q 0$, 
we show the Cartier index of $K_{X'}+B'$ is bounded which implies we have a bounded complement 
(as in the proof of \ref{l-compl-non-exc}).  
We then can assume either  $K_{X'}+B'+M'\not\sim_\Q 0$ or $M'\not\sim_\Q 0$. 

Let $X'\to V'$ be the contraction defined by $-(K_{X'}+B'+M')$. If $M'$ is not big over $V'$, 
 then we pull back a complement from the base of some fibration $X'\to T'$ derived from $X'\to V'$
(as in the proof of \ref{p-bnd-compl-non-big}). Thus we assume $M'$ is big over $V'$. We then modify the setting 
and find a boundary $\Gamma'$ and number $\alpha\in (0,1)$ such that $(X',\Gamma'+\alpha M')$ is 
generalised plt with $S':=\rddown{\Gamma'}$ irreducible and $-(K_{X'}+\Gamma'+\alpha M')$ is ample 
(as in the  proof of \ref{p-bnd-compl-non-klt}). 
Using this plt pair and the ampleness mentioned we can apply  Kawamata-Viehweg vanishing 
on some resolution of $X'$ to lift a complement from $S'$ to $X'$ (\ref{p-bnd-compl-plt}).

Carrying out all the steps above and making sure that the required inductive assumptions are satisfied 
(e.g. \ref{l-fib-adj-dcc}) involves a lot of technicalities.

\subsection{General remarks}\label{ss-compl-remarks}
Let $(X',B'+M')$ be a projective generalised pair with data $\phi\colon X\to X'$ and $M$.

(1)
Assume there is ${B'}^+\ge B'$ such that $(X',{B'}^++M')$ is generalised lc, $nM$ is b-Cartier, and 
$n(K_{X'}+{B'}^++M')\sim 0$ for some $n\in\N$. We show $K_{X'}+{B'}^++M'$ is an $n$-complement 
of $K_{X'}+{B'}+M'$. 
Writing $B'=T'+\Delta'$ where $T'=\rddown{B'}$, we need to show  
$$
n{B'}^+\ge nT'+\rddown{(n+1)\Delta'}.
$$ 
Note that $T'$ and $\Delta'$ have no common components.

Let $D'$ be a prime divisor and 
$b'$ and ${b'}^+$ be its coefficients in ${B'}$ and ${B'}^+$. If ${b'}^+=1$, then either $b'=1$ 
in which case $n{b'}^+=nb'$, or $b'<1$ in which case $n{b'}^+=n\ge \rddown{(n+1)b'}$.
So assume ${b'}^+<1$, say $b'^+=\frac{i}{n}$. Then 
$$
nb'^+=i=\rddown{(n+1)b'^+}\ge \rddown{(n+1)b'}.
$$\

(2)
Assume $X'\bir X''$ is a birational map to a normal projective variety. Replacing $X$ we can assume 
the induced map $\psi\colon X\bir X''$ is a morphism. Let $M''=\psi_*M$ and assume 
$$
\phi^*(K_{X'}+B'+M')+P=\psi^*(K_{X''}+B''+M'')
$$
for some $P\ge 0$ and $B''\ge 0$. 
Suppose $K_{X''}+B''+M''$ has an $n$-complement $K_{X''}+{B''}^++M''$ with ${B''}^+\ge B''$. 
We claim  $K_{X'}+B'+M'$ also has an $n$-complement $K_{X'}+{B'}^++M'$ with ${B'}^+\ge B'$. 
Let $C''={B''}^+-B''$ and let ${B'}^+=B'+\phi_*(P+\psi^*C'')$. Then 
$$
K_{X'}+{B'}^++M'=K_{X'}+B'+M'+\phi_*P+n\phi_*\psi^*C''
$$
$$
=\phi_*\psi^*(K_{X''}+B''+M'')+\phi_*\psi^*C''=
\phi_*\psi^*(K_{X''}+{B''}^++M'') 
$$
which implies that 
$$
n(K_{X'}+{B'}^++M')\sim 0 ~~~\mbox{and}~~~
\phi^*(K_{X'}+B'^++M')=\psi^*(K_{X''}+B''^++M'').
$$
In particular, $(X',{B'}^++M')$ is generalised lc. Now apply (1).

(3) Assume $X'\bir X''$ is a partial MMP on $-(K_{X'}+B'+M')$ and $B'',M''$ are the pushdowns of $B',M'$. 
Then there is $P\ge 0$ as in (2). Thus if $K_{X''}+B''+M''$ has an $n$-complement 
$K_{X''}+{B''}^++M''$ with ${B''}^+\ge B''$, then  
 $K_{X'}+B'+M'$ has an $n$-complement $K_{X'}+{B'}^++M'$ with ${B'}^+\ge B'$.

\subsection{Hyperstandard coefficients under adjunction for fibre spaces}

To construct complements we sometimes come across fibrations $X\to Z$ along which a given log 
divisor $K_X+B$ is trivial. Using adjunction for fibre spaces (\ref{ss-adj-fib-spaces}) we can 
write $K_X+B$ as the pullback of $K_Z+B_Z+M_Z$ where $B_Z$ and $M_Z$ are the discriminant and 
moduli divisors. In order to apply induction we need to be able to control the coefficients of 
$B_Z$ and $M_Z$ in terms of the coefficients of $B$. We do this in the next proposition.  
The existence of $\mathfrak{S}$  is similar to [\ref{PSh-II}, Lemma 9.3(i)]. 

\begin{prop}\label{l-fib-adj-dcc}
Let $d\in\N$ and $\mathfrak{R}\subset [0,1]$ be a finite set of rational numbers. 
Assume Theorem \ref{t-bnd-compl-usual-local} holds in dimension $d$.
Then there exist $q\in \N$ and  a finite set of rational numbers $\mathfrak{S}\subset [0,1]$
depending only on $d,\mathfrak{R}$ satisfying the following. 
Assume $(X,B)$ is a pair and $f\colon X\to Z$ a contraction such that 
\begin{itemize}
\item $(X,B)$ is projective lc of dimension $d$, and $\dim Z>0$, 

\item $K_{X}+B\sim_\Q 0/Z$ and $B\in \Phi(\mathfrak{R})$,

\item $X$ is of Fano type over some non-empty open subset $U\subseteq Z$, and 

\item the generic point of each non-klt centre of $(X,B)$ maps into $U$.\\
\end{itemize}
Then we can write 
$$
q(K_X+B)\sim qf^*(K_Z+B_Z+M_Z)
$$
where $B_Z$ and $M_Z$ are the discriminant and moduli parts of adjunction (as in \ref{ss-adj-fib-spaces}),  
$B_Z\in \Phi(\mathfrak{S})$, and  for any high resolution $Z'\to Z$ the moduli divisor $qM_{Z'}$ is nef Cartier. 
\end{prop}
\begin{proof}
Here is a short summary of the proof. Using \ref{t-bnd-compl-usual-local} we find $q$ and make a specific 
choice of $M_Z$ as a Weil divisor. Next we aim to show the existence of $\mathfrak{S}$ and 
that $qM_Z$ is integral. By taking hyperplane sections on $Z$ we 
reduce this aim to the case when $Z$ is a curve. We then pick a closed point $z\in Z$ and use \ref{t-bnd-compl-usual-local} 
once more to create a $q$-complement $K_X+B^+$ over $z$ so that $B^+\ge B$ and  $(X,B^+)$ has a non-klt centre mapping to $z$,  
and this implies that the coefficient of $z$ in $B_Z$ belongs to special set of the form $\Phi(\mathfrak{S})$ which in turn 
implies that $qM_Z$ is integral. Finally, we go back to the general case of $Z$ and apply the 
previous arguments to a high enough resolution $Z'\to Z$ to show that $qM_{Z'}$ is Cartier (nefness is guaranteed 
by \ref{t-adj-fib-spaces}).\\

\emph{Step 1.} 
In this step we find $q$ and make a choice of $M_Z$.
Let $q=n$ be the number given by Theorem \ref{t-bnd-compl-usual-local} which depends only on $d,\mathfrak{R}$. 
Then there is a $q$-complement 
$K_X+B^+$ of $K_X+B$ over some 
point $z\in U$ with $B^+\ge B$. Since over $z$ we have $K_X+B\sim_\Q 0$ and $q(K_X+B^+)\sim 0$, 
and since $B^+\ge B$, we have 
$B^+=B$ near the generic fibre of $f$. Therefore, $q(K_X+B)\sim 0$ over the generic point of $Z$, 
hence there is a rational function $\alpha$ on $X$ such that $qL:=q(K_X+B)+\Div(\alpha)$ is zero  
over the generic point of $Z$. In particular, $q(K_X+B)\sim qL$ and $L$ is 
vertical$/Z$. Since $L\sim_\Q 0/Z$, $L=f^*L_Z$ for some $L_Z$ on $Z$. Let 
$M_Z:=L_Z-(K_Z+B_Z)$ where $B_Z$ is the discriminant part of adjunction for $(X,B)$ over $Z$. Thus  
$$
q(K_X+B)\sim qL=qf^*L_Z=qf^*(K_Z+B_Z+M_Z)
$$
and $M_Z$ is the moduli part of adjunction for $(X,B)$ over $Z$. Note that $M_Z$ is not unique: 
it depends on the choice of $\alpha$ and $K_Z$.\\  

\emph{Step 2.} 
Our aim until the end of Step 4 is to show the existence of $\mathfrak{S}$ and 
to show $qM_Z$ is integral. In this step we reduce this aim to the case $\dim Z=1$.  
Assume $\dim Z>1$.
Let $H$ be a general hyperplane section of $Z$ and $G$ its pullback to $X$. 
Let $K_G+B_G=(K_X+B+G)|_G$. Since $G$ is a general member of a free linear system, 
each non-klt centre of $(G,B_G)$ is a component of the intersection
of a non-klt centre of $(X,B)$ with $G$, hence its generic point maps into $U\cap H$. 
Moreover, $G$ is of Fano type over $U\cap H$. Let $B_H$ be the discriminant part of adjunction 
for $(G,B_G)$ over $H$. 

Let $g$ be the induced map $G\to H$.  Let $D$ be a prime divisor on $Z$ and let 
$C$ be a component of $D\cap H$. Let $t$ be the lc threshold of $f^*D$ with respect to 
$(X,B)$ over the generic point of $D$. Then there is a non-klt centre of $(X,B+f^*D)$ 
mapping onto $D$. This centre is also a non-klt centre of  $(X,B+G+f^*D)$, hence 
intersecting it with $G$ gives a non-klt centre of $(G,B_G+g^*C)$ mapping onto $C$, 
by inversion of adjunction [\ref{kawakita}]. Thus $t$ is the lc threshold of $g^*C$ with respect to $(G,B_G)$. 
Therefore,  $\mu_DB_Z=\mu_CB_H$ (see [\ref{B-sing-fano-fib}, proof of Lemma 3.2] for more details).
 
By Lemma \ref{l-div-adj-dcc}, there is a finite set of rational numbers 
$\mathfrak{T}\subset [0,1]$ depending only on $\mathfrak{R}$ such that 
$B_G\in \Phi(\mathfrak{T})$. 
So applying induction on dimension, there is 
a finite set of rational numbers 
$\mathfrak{S}\subset [0,1]$ depending only on $d-1,\mathfrak{T}$ hence depending only on 
$d, \mathfrak{R}$ such that $B_H\in \Phi(\mathfrak{S})$. Therefore, 
$B_Z\in \Phi(\mathfrak{S})$.

Pick a general $H'\sim H$ and 
let $K_H=(K_Z+H')|_H$: note that although $K_Z$ may not be $\Q$-Cartier but the restriction 
is well-defined as $H$ is a general hyperplane section, and $K_H$ is determined as a Weil divisor. Letting 
$M_H:=(L_Z+H')|_H-(K_H+B_H)$, we have    
$$
q(K_G+B_G)\sim q(L+G)|_G\sim qg^*(L_Z+H')|_H \sim qg^*(K_H+B_H+M_H)
$$
hence $M_H$ is the moduli part of $(G,B_G)$ over $H$. Moreover, $B_H+M_H=(B_Z+M_Z)|_H$, hence 
$
\mu_C(B_H+M_H)=\mu_D(B_Z+M_Z)
$
which implies $\mu_CM_H=\mu_DM_Z$ as $\mu_CB_H=\mu_DB_Z$. 
Therefore, $q\mu_DM_Z$ is integral iff $q\mu_CM_H$ is integral.   
So repeating the process we reduce the problem to the case $\dim Z=1$.\\

\emph{Step 3.}
In this  step we  prove existence of $\mathfrak{S}$. By Step 2 we can assume $Z$ is a curve. 
By Lemma \ref{l-FT-over-curve}, $X$ is of Fano type 
over $Z$. Pick a closed point $z\in Z$. 
Let $t$ be the lc threshold of $f^*z$ with respect to $(X,B)$. Let $\Gamma=B+tf^*z$ 
and let $({X}',{\Gamma}')$ be a $\Q$-factorial dlt model of $(X,\Gamma)$ so that 
$\rddown{{\Gamma}'}$ has a component mapping to $z$. Note that $K_{X'}+\Gamma'\sim_\Q 0/Z$. 
Then there is a boundary ${B}'\le \Gamma'$ such that $B'\in \Phi(\mathfrak{R})$, 
$\rddown{B'}$ has a component mapping to $z$, and $B^\sim\le B'$ where $B^\sim$ is the birational transform of $B$.  
Now $X'$ is of Fano type over $Z$  
and $-(K_{X'}+B')\sim_\Q \Gamma'-B'/Z$.
Run an MMP on $-(K_{X'}+B')$ over $Z$ and let $X''$ be the resulting model. 
Then $X''$ is of Fano type over $Z$, $B''\in \Phi(\mathfrak{R})$,  
and $-(K_{X''}+B'')$ is nef over $Z$. Moreover, $(X'',B'')$ is lc, as $(X'',\Gamma'')$ is lc 
and $B''\le \Gamma''$.

 By our choice of $q$ which comes from Theorem \ref{t-bnd-compl-usual-local}, 
$K_{X''}+B''$ has a $q$-complement $K_{X''}+B''^+$ over $z$ with $B''^+\ge B''$. 
Thus by  \ref{ss-compl-remarks}(3) (in the relative setting), there is a 
$q$-complement $K_{X'}+B'^+$ of $K_{X'}+B'$ over $z$ with $B'^+\ge B'$.
Pushing $K_{X'}+B'^+$ down to $X$ gives a 
$q$-complement $K_X+B^+$ of $K_X+B$ over $z$ with $B^+\ge B$ such that $(X,B^+)$ 
has a non-klt centre mapping to $z$.  
Now  $B^+-B\sim_\Q 0$ over $z$, hence $B^+-B$ is vertical over $Z$. Thus
over $z$, the divisor $B^+-B$ is just a multiple of the fibre $f^*z$. 
Therefore, $B^+=B+tf^*z$ over $z$ as $(X,B^+)$ 
has a non-klt centre mapping to $z$.

Recall that the coefficient of $z$ in $B_Z$ is $1-t$.
Pick a component $S$ of $f^*z$ and let $b$ and $b^+$ be its coefficients in $B$ and $B^+$. 
If $m$ is its coefficient in $f^*z$, then $b^+=b+tm$, hence $t=\frac{b^+-b}{m}$. 
Now $b=1-\frac{r}{l}$ for some $r\in \mathfrak{R}$ and $l\in \N$, so 
$t=\frac{s}{m}$ where  $s=b^+-1+\frac{r}{l}$. If $b^+=1$, then $t=\frac{r}{lm}$ and $\mu_zB_Z\in\Phi(\mathfrak{R})$. 
If $b^+<1$, then as $r\le 1$ and as $qb^+$ is integral we get 
$$
1-\frac{1}{l}\le b=1-\frac{r}{l}\le b^+\le 1-\frac{1}{q},
$$ 
so $l\le q$, hence in this case there are 
finitely many possibilities for $s$. Therefore, in any case $\mu_zB_Z\in\Phi(\mathfrak{S})$ for some 
fixed finite set $\mathfrak{S}\subset [0,1]$ of rational numbers.\\ 

\emph{Step 4.}
In this step we show $qM_Z$ is integral. By Step 2, we can assume $Z$ is a curve. 
By Step 1, $q(K_X+B)\sim 0$ over some non-empty open set $V\subseteq Z$ 
such that $\Supp B_Z\subseteq Z\setminus V$.
Let 
$$
\Theta=B+\sum_{z\in Z\setminus V} t_zf^*z
$$ 
where  
$t_z$ is the lc threshold of $f^*z$ with respect to $(X,B)$. 
If $\Theta_Z$ is the discriminant part 
of adjunction for $(X,\Theta)$ over $Z$, then 
$$
\Theta_Z=B_Z+\sum_{z\in Z\setminus V} t_zz,
$$
 hence 
$\Theta_Z$ is a reduced divisor. Moreover, by Step 3, $K_X+\Theta$ is a $q$-complement 
of $K_X+B$ over each $z\in Z\setminus V$, 
 hence $q(K_X+\Theta)\sim 0/Z$, by Lemma \ref{l-relatively-trivial-divisors}. Therefore, since 
$$
q(K_X+\Theta)=q(K_X+B)+q(\Theta-B)
$$
$$
\sim qf^*(K_Z+B_Z+M_Z)+qf^*(\Theta_Z-B_Z)=qf^*(K_Z+\Theta_Z+M_Z)
$$
we deduce $q(K_Z+\Theta_Z+M_Z)$ is Cartier. This implies $qM_Z$ is integral 
as $K_Z+\Theta_Z$ is integral.\\

\emph{Step 5.}
In this step we construct a birational model $X''$ of $X'$ using MMP. 
From now on we consider the general case of $Z$, that is, when it is not necessarily a curve.
Let $X'\to X$ be a log resolution of $(X,B)$ so that $X'\bir Z'$ is a morphism where $Z'\to Z$ is a 
high resolution. 
Let $U_0\subseteq U$ be a non-empty open set over which $Z'\to Z$ is an isomorphism. Let 
$\Delta'$ be the sum of the birational transform of $B$ and the reduced exceptional divisor of 
$X'\to X$ but with all the components mapping outside $U_0$ removed. We can assume the 
generic point of any non-klt centre of $(X',\Delta')$ maps into $U_0$. 
Run an MMP on $K_{X'}+\Delta'$ over $Z'\times_ZX$ with scaling of some ample divisor. 
By [\ref{B-lc-flips}, Theorem 1.9], the MMP terminates over $U_0'\subset Z'$, the inverse image of $U_0$. 
In fact we reach a model $X''$ such that over $U_0'$ the pair $(X'',\Delta'')$ is a $\Q$-factorial 
dlt model of $(X,B)$, hence $K_{X''}+\Delta''\sim_\Q 0$ over $U_0'$ and $X''$ is of Fano type over $U_0'$. 
Now by [\ref{B-lc-flips}, Theorem 1.4][\ref{HMX2}, Theorem 1.1], we can run an MMP$/Z'$ on $K_{X''}+\Delta''$ which 
terminates with a good minimal model 
over $Z'$ because the generic point of every non-klt centre of $(X'',\Delta'')$ is mapped into $U_0'$. 
Abusing notation, we denote the minimal model again by $X''$ which is of Fano type over $U_0'$.\\

\emph{Step 6.}
We are now ready to show that $qM_{Z'}$ is nef Cartier where $Z'\to Z$ is a high resolution. 
The nefness follows from Theorem \ref{t-adj-fib-spaces}, so 
we just need to show $qM_{Z'}$ is integral. 
We will use the construction of the previous step.  
Let $f''\colon X''\to Z''/Z'$ be the contraction defined by $K_{X''}+\Delta''$. 
By construction, on a common resolution $W$ of $X$ and $X''$, the pullbacks of $K_X+B$ 
and $K_{X''}+\Delta''$ are equal over $U_0''\subset Z''$, the inverse image of $U_0$. 
Let $K_{X''}+B''$ and $L''$ be the pushdown 
to $X''$ of the pullback of $K_X+B$ and $L$ to $W$, respectively, where $L$ is as in Step 1. Let $P''=\Delta''-B''$ 
which is vertical and $\sim_\Q 0$ over $Z''$, hence it is the pullback of some $\Q$-divisor $P_{Z''}$ on $Z''$.   
Denote by $\Delta_{Z''}$ the discriminant part of adjunction on ${Z''}$ defined for $(X'',\Delta'')$ over $Z''$. 
Then $\Delta_{Z''}=B_{Z''}+P_{Z''}$ where $B_{Z''}$ is the discriminant part 
of adjunction on $Z''$ defined for $(X,B)$ over $Z$. Moreover,  
$$
q(K_{X''}+\Delta'')=q(K_{X''}+B''+P'')\sim q(L''+P'')=qf''^*(L_{Z''}+P_{Z''})
$$
$$
=qf''^*(K_{Z''}+\Delta_{Z''}+M_{Z''})
$$
where $L_{Z''}$ is the pullback of $L_Z$ in Step 1, and $M_{Z''}=L_{Z''}-(K_{Z''}+B_{Z''})$
is the moduli part of both $(X'',\Delta'')$ over $Z''$ and $(X,B)$ over $Z$.
Now by Steps 2-4, $qM_{Z''}$ is an integral divisor, hence  $qM_{Z'}$ is integral as well which means 
it is Cartier as $Z'$ is smooth. 

\end{proof}

\subsection{Pulling back complements from the base of a fibration}

In this subsection we consider complements when a suitable fibration is present. 

\begin{prop}\label{p-bnd-compl-non-big}
Assume Theorem \ref{t-bnd-compl} holds in dimension $\le d-1$ and Theorem \ref{t-bnd-compl-usual-local} 
holds in dimension $d$. 
Then Theorem \ref{t-bnd-compl} holds in dimension $d$ for those $(X',B'+M')$ such that 
 there is a contraction $X'\to V'$ so that $K_{X'}+B'+M'\sim_\Q 0/V'$, $\dim V'>0$, and  
$M'$ is not big$/V'$.
\end{prop}
\begin{proof}
First we give a short summary of the proof. Modifying the setting we can assume $X'\to V'$ 
factors through a fibration $f'\colon X'\to T'$ such that $M'\sim_\Q 0/T'$. Applying adjunction and Proposition \ref{l-fib-adj-dcc} 
we can write 
$$
q(K_{X'}+B')\sim qf'^*(K_{T'}+B_{T'}+P_{T'})
$$
with $B_{T'}\in \Phi(\mathfrak{S})$ and $qP_T$ is nef Cartier where $q,\mathfrak{S}$ are fixed and $T$ is a 
resolution of $T'$. In addition we make sure $qM$ is linearly equivalent to the pullback of some nef Cartier divisor $qM_T$, 
and that $qM'$ is linearly equivalent to the pullback of $qM_{T'}$ where $M_{T'}$ is the pushdown of $M_T$. 
Next we show $({T'},B_{T'}+P_{T'}+M_{T'})$ is generalised lc, and we construct a bounded complement for 
$K_{T'}+B_{T'}+P_{T'}+M_{T'}$ and pull it back 
to a complement of $K_{X'}+B'+M'$.\\

\emph{Step 1}.
In this step we reduce the problem to the situation in which we have a contraction $X'\to T'/V'$ 
with $M'\sim_\Q 0/T'$. 
Replacing $(X',B'+M')$ with a $\Q$-factorial generalised dlt model as in \ref{ss-gpp}(3), 
we can assume $X'$ is $\Q$-factorial. 
Since $M'$ is not big$/V'$, $X'\to V'$ is not birational. 
After running an MMP$/V'$ on $M'$ and applying \ref{ss-compl-remarks}(2), 
we can assume $M'$ is semi-ample$/V'$. Note that we can run such MMP as $X'$ is of Fano type. 
So $X'\to V'$ factors through a 
contraction $f'\colon X'\to T'$ such that $\dim X'>\dim T'$ and $M'\sim_\Q 0/T'$.\\

\emph{Step 2}.
In this step we consider adjunction over $T'$.
By construction, $K_{X'}+B'\sim_\Q 0/T'$.
Thus by Proposition \ref{l-fib-adj-dcc} (which needs Theorem \ref{t-bnd-compl-usual-local} in dimension $d$) 
there exist $q\in \N$ and a finite set of rational numbers $\mathfrak{S}\subset [0,1]$ 
depending only on $d,\mathfrak{R}$ such that 
$$
q(K_{X'}+B')\sim qf'^*(K_{T'}+B_{T'}+P_{T'})
$$
where $B_{T'}$ and $P_{T'}$ are the discriminant 
and moduli divisors of adjunction for fibre spaces applied 
to $({X'},B')$ over $T'$, and such that $B_{T'}\in\Phi(\mathfrak{S})$ and $qP_T$ is nef Cartier 
for any high resolution $T\to T'$. We can assume $q$ is divisible by $p$.\\ 

\emph{Step 3}.
In this step we show that perhaps after replacing $X$, $pM$ is linearly the pullback of some Cartier divisor $pM_T$ on a 
resolution $T$ of $T'$. 
Pick a sufficiently high log resolution $\psi\colon T\to T'$ of $(T',B_T')$ so that the moduli part $P_T$ is nef and it satisfies 
the pullback property of Theorem \ref{t-adj-fib-spaces} (ii).
We consider $(T',B_{T'}+P_{T'})$ as a generalised pair with data $\psi\colon T\to T'$
and $P_T$. Since  $(X',B')$ is lc, the coefficients of the discriminant divisor $B_T$ on $T$ 
are at most $1$, hence $(T',B_{T'}+P_{T'})$ is generalised lc. Replace $X$ so that the induced map 
$f\colon X\bir T$ is a morphism.
Since $M$ is nef, $\phi^*M'=M+E$ for some exceptional$/X'$ 
and effective $\Q$-divisor $E$. Since $M'\sim_\Q 0/T'$, $E$ is vertical$/T'$, so there is a non-empty open subset  
of $T'$ over which $E=0$ and $M\sim_\Q 0$. Since $X'$ is of Fano type, the general 
fibres of $f'$ are also of Fano type, hence they are rationally connected  
which in turn implies the general fibres of $f$ are rationally connected [\ref{HM-rat-connected}][\ref{Zhang-rat-connected}].  
Thus perhaps after replacing $X$ and $T$ and applying Lemma \ref{l-nef-div-descent}, 
 $pM\sim pf^*M_T$ for a $\Q$-divisor 
$M_T$ on $T$ so that $pM_T$ is nef Cartier.\\ 

\emph{Step 4.}
In this step we show that  $qM'\sim qf'^*M_{T'}$ where $M_{T'}$ is the pushdown of $M_T$. 
Since $E$ is vertical and $\sim_\Q 0$ over $T$, $E=f^*E_T$ for some 
effective $\Q$-divisor $E_T$. Moreover, since $E$ is exceptional$/X'$, we 
deduce $E_T$ is exceptional over $T'$: otherwise $E_T$ has a component $D$ whose pushdown 
$D'$ on $T'$ is not zero; but then there is some prime divisor $C'$ on $X'$ mapping onto $D'$, and since 
$E=f^*E_T$, we get a component $C$ of $E$ mapping onto $C'$ contradicting the fact that 
$E$ is exceptional over $X'$. Therefore, $M_{T'}:=\psi_*M_T$ is $\Q$-Cartier as $M_T+E_T\sim_\Q 0/T'$. 
 By construction, $q(P_T+M_T)$ is nef Cartier. Moreover, 
$$
q\phi^*M'=q(M+E)\sim qf^*(M_T+E_T)=qf^*\psi^*M_{T'}=q\phi^*f'^*M_{T'}
$$ 
 which implies  $qM'\sim qf'^*M_{T'}$. 
\\ 

\emph{Step 5}.
Now we consider 
$(T',B_{T'}+P_{T'}+M_{T'})$ as a generalised pair 
with data $\psi\colon T\to T'$ and $P_T+M_T$.
We show it is generalised lc. We can assume 
$(T,B_{T}+E_T)$ is log smooth where as before $B_T$ is the discriminant divisor on $T$ 
defined for $(X,B)$ over $T'$. 
By construction
$$
K_T+B_T+E_T+P_T+M_T=\psi^*(K_{T'}+B_{T'}+P_{T'}+M_{T'}).
$$
So it is enough to show  
$(T,B_{T}+E_T)$ is sub-lc which in turn is  
equivalent to saying that every coefficient of $B_T+E_T$ is $\le 1$. 
Let $K_{X}+B$ be the pullback of $K_{X'}+B'$. Let 
$D$ be a prime divisor on $T$. By definition of the discriminant divisor, 
$\mu_DB_T=1-t_D$ where $t_D$ is the largest number so that $(X,B+t_Df^*D)$ is sub-lc over the 
generic point of $D$. Since $(X',B'+M')$ is generalised lc, $(X,B+E)$ is sub-lc, and since $E=f^*E_T$, we have 
$\mu_DE_T\le t_D$ which implies $\mu_DB_T+\mu_DE_T\le 1$ as required.\\ 

\emph{Step 6}.
To summarise we have proved:  $(T',B_{T'}+P_{T'}+M_{T'})$ is generalised lc, 
$B_{T'}\in \Phi(\mathfrak{S})$, $q(P_T+M_T)$ is Cartier, and 
$$
q(K_{X'}+B'+M')\sim qf'^*(K_{T'}+B_{T'}+P_{T'}+M_{T'})
$$ 
is anti-nef. Moreover,  $T'$ is of Fano type as $X'$ 
is of Fano type, by Lemma \ref{l-FT-base-contraction}. 

We will construct a complement on $T'$ and pull it back to $X'$.
Since we are assuming Theorem \ref{t-bnd-compl} in dimension $\le d-1$, 
$K_{T'}+B_{T'}+P_{T'}+M_{T'}$ has an $n$-complement
$K_{T'}+B_{T'}^++P_{T'}+M_{T'}$ for some $n$ divisible by $q$ and 
depending only on $\dim T',q, \mathfrak{S}$ such that $G_{T'}:=B^+_{T'}-B_{T'}\ge 0$. 
So $n$ depends only on $d,p,\mathfrak{R}$. 
Denote the pullback of $G_{T'}$ to $T,X,X'$ by $G_T,G,G'$, respectively.
Let ${B'}^+=B'+G'$. Then 
$$
n(K_{X'}+{B'}^++M')=n(K_{X'}+{B'}+M'+G')
$$
$$
\sim nf'^*(K_{T'}+B_{T'}+P_{T'}+M_{T'}+G_{T'})=nf'^*(K_{T'}+B_{T'}^++P_{T'}+M_{T'})\sim 0.
$$ 
 Thus by \ref{ss-compl-remarks}(1), $K_{X'}+{B'}^++M'$ is an 
$n$-complement of $K_{X'}+B'+M'$ if we show  $({X'},{B'}^++M')$ is generalised lc.\\ 

\emph{Step 7.}
In this final step we show that indeed $({X'},{B'}^++M')$ is generalised lc. The pair is clearly generalised 
lc over the generic point of $T'$ because by construction $B'^+=B'$ over the generic point of $T'$.
Let $C$ be a prime divisor on some birational model of $X'$. We want to show that 
$a(C,{X'},{B'}^++M')\ge 0$. This is true if $C$ is not vertical over $T'$,  by 
the previous sentence. Assume then that $C$ is vertical over $T'$.
Replacing $X,T$ we can assume $C$ is a divisor on $X$ and that 
its image on $T$ is a divisor, say $D$. The pullback of $K_{X'}+{B'}^++M'$ to $X$ is 
$K_X+B+E+G+M$. It is enough to show $\mu_C(B+E+G)\le 1$. 
 Since  $({T'},B_{T'}^++P_{T'}+M_{T'})$ is generalised lc, and since 
$$
K_T+B_T+E_T+G_T+P_T+M_T=K_T+B_T+P_T+G_T+E_T+M_T
$$
$$
=\psi^*(K_{T'}+B_{T'}+P_{T'}+G_{T'}+M_{T'})=
\psi^*(K_{T'}+B_{T'}^++P_{T'}+M_{T'})
$$ 
we have  $\mu_D(B_T+E_T+G_T)\le 1$.  
Letting $t_D$ be as in Step 5 we get $\mu_D(E_T+G_T)\le t_D$ which implies 
$(X,B+E+G)$ is sub-lc over the generic point of 
$D$. Therefore, $\mu_C(B+E+G)\le 1$ as required.

\end{proof}

\subsection{Lifting complements from a non-klt centre}

Until the end of this subsection we essentially give an inductive treatment of \ref{t-bnd-compl} when there are 
non-klt centres around from which we can lift complements. First we consider a key inductive statement.

\begin{prop}\label{p-bnd-compl-plt}
Assume Theorem \ref{t-bnd-compl} holds in dimension $d-1$. Then Theorem \ref{t-bnd-compl} 
holds in dimension $d$ for those $(X',B'+M')$ such that 
\begin{itemize}
\item $B'\in \mathfrak{R}$,

\item $({X'},\Gamma'+\alpha M')$ is $\Q$-factorial generalised plt for some $\Gamma'$ and $\alpha \in (0,1)$,

\item $-(K_{X'}+\Gamma'+\alpha M')$ is ample, and

\item $S'=\rddown{\Gamma'}$ is irreducible and it is a component of $\rddown{B'}$.
\end{itemize}
\end{prop}
\begin{proof}
We say a few words before going into the proof. The idea is to construct a complement  on $S'$ and then 
lift it to $X'$ using Kawamata-Viehweg vanishing theorem. Applying induction gives the required 
complement for $(K_{X'}+B'+M')|_{S'}$. However, we face
technical issues if we try to lift the complement directly to $X'$. Instead we assume $X\to X'$ is a resolution 
and extend information from $S$, the birational transform of $S'$, to $X$ and then push down to $X'$. 
This makes the notation a bit more complicated than desired. In the end we get something 
like $K_{X'}+B'^++M'$ which is the complement we need  except that 
we need to further argue that the singularities of $(X',B'^++M')$ are good away from $S'$ by using the connectedness 
principle.\\

\emph{Step 1}.
 In this step we consider adjunction  and  complements on $S'$.
We can assume the given map $\phi\colon X\to X'$ is a log resolution of $(X',B'+\Gamma')$ and that the induced map 
$\psi\colon S\bir S'$ is a morphism where $S$ is the birational transform of $S'$. Moreover, we can assume 
$pM$ is Cartier. Then, by \ref{ss-q-adjunction}(2), we have the generalised adjunction 
$$
K_{S'}+B_{S'}+M_{S'}\sim_\Q (K_{X'}+B'+M')|_{S'}
$$
such that $pM_S$ is Cartier and 
$$
p(K_{S'}+B_{S'}+M_{S'})\sim p(K_{X'}+B'+M')|_{S'}.
$$

By Lemma \ref{l-div-adj-dcc}, $B_{S'}\in \Phi(\mathfrak{S})$ for some finite set of rational numbers 
$\mathfrak{S}\subset [0,1]$ which only depends on $p,\mathfrak{R}$. 
Restricting $K_{X'}+\Gamma'+\alpha M'$ to $S'$ shows that $S'$ is of Fano type, by \ref{ss-gpp}(6). 
 Thus  
by Theorem \ref{t-bnd-compl} in dimension $d-1$, there is $n\in \N$ divisible by $p$ 
which depends only on $d-1, p, \mathfrak{S}$ such that $K_{S'}+B_{S'}+M_{S'}$ has 
an $n$-complement $K_{S'}+B_{S'}^++M_{S'}$ with $B_{S'}^+\ge B_{S'}$. 
Then $n$ depends only on $d,p, \mathfrak{R}$ and replacing it with $nI(\mathfrak{R})$ we can assume it is divisible by 
$I(\mathfrak{R})$. In particular, $nB'$ is integral as $B'\in \mathfrak{R}$.
We will show there is an $n$-complement $K_{X'}+{B}'^++M'$ of $K_{X'}+B'+M'$ with ${B'}^+\ge B'$.\\

\emph{Step 2}.
In this step we introduce basic notation. Write 
$$
N:=-(K_X+B+M):=-\phi^*(K_{X'}+B'+M')
$$
and let $T=\rddown{B^{\ge 0}}$ and $\Delta=B-T$. Define  
$$
L:=-nK_{X}-nT-\rddown{(n+1){\Delta}}-nM
$$
which is an integral divisor. Note that 
$$
L=n\Delta-\rddown{(n+1){\Delta}}+nN.
$$
Now write 
$$
K_X+\Gamma+\alpha M=\phi^*(K_{X'}+\Gamma'+\alpha M').
$$ 
Replacing $\Gamma'$ with $(1-a)\Gamma'+aB'$ and replacing $\alpha M$ with 
$((1-a)\alpha+a)M$ for some $a\in(0,1)$ sufficiently 
close to $1$, we can assume $\alpha$ is sufficiently close to $1$ and $B-\Gamma$ has sufficiently small 
(positive or negative) coefficients.\\ 

\emph{Step 3}.
In this step we define a divisor $P$ and study its properties.
Let $P$ be the unique integral divisor so that 
$$
\Lambda:=\Gamma+{n{\Delta}}-\rddown{(n+1){\Delta}}+P
$$ 
is a boundary, $(X,\Lambda)$ is plt, and $\rddown{\Lambda}=S$ (in particular, we are assuming $\Lambda\ge 0$). 
More precisely, we let $\mu_SP=0$ and for each prime divisor $D\neq S$, we let 
$$
\mu_DP=-\mu_D\rddown{\Gamma+{n{\Delta}}-\rddown{(n+1){\Delta}}}
$$
which satisfies 
$$
\mu_DP=-\mu_D\rddown{\Gamma-\Delta+\langle (n+1)\Delta\rangle}
$$ 
where $\langle (n+1)\Delta\rangle$ is the fractional part of $(n+1)\Delta$.
This implies $0\le \mu_DP\le 1$ for any prime divisor $D$: indeed we can assume $D\neq S$; if $D$ is a component of $T$, then 
$D$ is not a component of $\Delta$ but $\mu_D\Gamma\in(0,1)$, hence $\mu_DP=0$; if $D$ is not a 
component of $T$, then $\mu_D(\Gamma-\Delta)=\mu_D(\Gamma-B)$ is sufficiently small, hence $0\le \mu_DP\le 1$. 

We show $P$ is exceptional$/X'$. 
Assume $D$ is a component of $P$ which is not exceptional$/X'$. Then $D\neq S$, and since $nB'$ is integral,  
$\mu_{D}n\Delta$ is integral, hence  
$\mu_{D}\rddown{(n+1)\Delta}=\mu_Dn\Delta$ which implies 
$\mu_{D}P=-\mu_D\rddown{\Gamma}=0$, a contradiction.\\

\emph{Step 4.}
In this step we use Kawamata-Viehweg vanishing to lift sections from $S$ to $X$.
Let 
$$
A'=-(K_{X'}+\Gamma'+\alpha M')
$$
and let $A=\phi^*A'$.
Then 
$$
L+P= {n{\Delta}}-\rddown{(n+1){\Delta}}+ nN+P
$$
$$
=K_X+\Gamma+\alpha M+A+{n{\Delta}}-\rddown{(n+1){\Delta}}+ nN+P
$$
$$
=K_X+\Lambda+A+\alpha M+nN.  
$$
Since $A+\alpha M+nN$ is nef and big and $(X,\Lambda-S)$ is klt, 
$h^1(L+P-S)=0$ by the Kawamata-Viehweg vanishing theorem, hence  
$$
H^0(L+P)\to H^0((L+P)|_S)
$$
is surjective.\\

\emph{Step 5.}
In this step we define several divisors.
Let $R_{S'}:=B_{S'}^+-B_{S'}$ which satisfies 
$$
-n(K_{S'}+B_{S'}+M_{S'})\sim nR_{S'}\ge 0.
$$
Letting $R_S$ be the pullback of $R_{S'}$, we get 
$$
-n(K_{S}+B_{S}+M_{S}):=-n\psi^*(K_{S'}+B_{S'}+M_{S'})\sim nR_{S}.
$$
Then 
$$
nN|_S=-n(K_X+B+M)|_S \sim -n(K_S+B_S+M_S) \sim  nR_S\ge 0
$$
where the first linear equivalence follows from Step 1 as $n$ is divisible by $p$.

By construction, 
$$
(L+P)|_S=({n{\Delta}}-\rddown{(n+1){\Delta}}+ nN+P)|_S
$$
$$
\sim G_S:=nR_S+{n{\Delta_S}}-\rddown{(n+1){\Delta_S}}+P_S
$$
where $\Delta_S=\Delta|_S$ and $P_S=P|_S$.\\

\emph{Step 6.}
In this step we show $G_S\ge 0$ and that it lifts to some effective divisor $G$ on $X$.
Assume $C$ is a component of $G_S$ with negative coefficient. Then 
 there is a component $D$ of ${n{\Delta}}-\rddown{(n+1){\Delta}}$ with negative coefficient such that $C$ is a 
component of $D|_S$. But 
$$
\mu_C ({n{\Delta_S}}-\rddown{(n+1){\Delta_S}})=\mu_C (-\Delta_S+\langle(n+1)\Delta_S\rangle)\ge 
-\mu_C \Delta_S=-\mu_D \Delta>-1
$$ 
which gives $\mu_CG_S>-1$ and this in turn implies $\mu_CG_S\ge 0$ because $G_S$ is integral, a contradiction. 
Therefore $G_S\ge 0$, and by Step 4, $L+P\sim G$ for some effective divisor $G$ whose support does not contain $S$ 
and $G|_S=G_S$.\\ 

\emph{Step 7.}
In this step we introduce $B'^+$.
By the previous step and the fact that $P$ is exceptional$/X'$, we have     
$$
-nK_{X'}-nT'-\rddown{(n+1)\Delta'}-nM'=L'=L'+P'\sim G'\ge 0
$$
where $L'$ is the pushdown of $L$, etc.
Since $nB'$ is integral, $\rddown{(n+1)\Delta'}= n\Delta'$, so  
$$
-n(K_{X'}+B'+M')=-nK_{X'}-nT'-{n\Delta'}-nM'
=L'\sim nR':=G'\ge 0.
$$
Let ${B'}^+=B'+R'$. Then $n(K_{X'}+{B'}^++M')\sim 0$.\\ 

\emph{Step 8.}
It is enough to show that $(X',{B'}^++M')$ is generalised lc because then $K_{X'}+{B'}^++M'$ is 
an $n$-complement of $K_{X'}+{B'}+M'$, by \ref{ss-compl-remarks}(1).
First we want to show  $R'|_{S'}=R_{S'}$.
Since  
$$
nR:=G-P+\rddown{(n+1)\Delta}- n\Delta\sim L+\rddown{(n+1)\Delta}- n\Delta=nN\sim_\Q 0/X'
$$
and since $\rddown{(n+1)\Delta'}- n\Delta'=0$ as  
$n\Delta'$ is integral, we get  $\phi_*nR=G'=nR'$ and that $R$ is the pullback of $R'$. 
Now 
$$
nR_S=G_S-P_S+\rddown{(n+1)\Delta_S}- n\Delta_S
$$
$$
=(G-P+\rddown{(n+1)\Delta}- n\Delta)|_S=nR|_S
$$
which means $R_S=R|_S$, hence $R_{S'}=R'|_{S'}$ as required. 

The previous paragraph implies 
$$
n(K_{S'}+B_{S'}^++M_{S'})\sim n(K_{X'}+{B'}^++M')|_{S'}
$$
which gives the generalised adjunction 
$$
K_{S'}+B_{S'}^++M_{S'}\sim_\Q (K_{X'}+{B'}^++M')|_{S'}.
$$
 By generalised inversion of adjunction (\ref{l-inv-adjunction}), 
$(X',{B'}^++M')$ is generalised lc near $S'$. Let 
$$
\Omega':=a{B'}^++(1-a)\Gamma' ~~\mbox{and}~~ F=(a+(1-a)\alpha) M
$$ 
for some $a\in(0,1)$ close to $1$.
If $(X',{B'}^++M')$ is not generalised lc away from $S'$, 
then $(X',\Omega'+F')$ 
is also not generalised lc away from $S'$. But then 
$$
-(K_{X'}+\Omega'+F')=-a(K_{X'}+{B'}^++M')-(1-a)(K_{X'}+\Gamma'+\alpha M')
$$
is ample and the generalised 
non-klt locus of $(X',\Omega'+F')$  has at least two disjoint components one of which is $S'$. This contradicts the 
connectedness principle (\ref{l-connectedness}). Thus  $(X',{B'}^++M')$ is generalised lc.

\end{proof}

\begin{prop}\label{p-bnd-compl-non-klt}
Assume Theorem \ref{t-bnd-compl} holds in dimension $\le d-1$ and Theorem \ref{t-bnd-compl-usual-local} 
holds in dimension $d$. Then Theorem \ref{t-bnd-compl} holds 
in dimension $d$ for those $(X',B'+M')$ such that 
\begin{itemize}
\item $B'\in \mathfrak{R}$, 

\item  $(X',B'+M')$ is not generalised klt, and 

\item  either $K_{X'}+B'+M'\not\sim_\Q 0$ or $M'\not\sim_\Q 0$.
\end{itemize}
\end{prop}
\begin{proof}
We say a few words before going into the proof. We reduce to the case when $X'$ is $\Q$-factorial and 
$(X',B')$ is non-klt, 
 find a contraction $X'\to Z'$, and reduce to the case when $M'$ is nef and big over $Z'$ 
 via  \ref{p-bnd-compl-non-big}. Next we find $\alpha\in(0,1)$ such that 
 $-(K_{X'}+B'+\alpha M')\sim_\Q A'+G'$ is nef and big, $A'$ is ample, and $G'\ge 0$.  
If $\Supp G'$ does not contains non-klt centres of $(X',B'+\alpha M')$, then 
we can easily find a boundary $\Gamma'$ so that we can apply \ref{p-bnd-compl-plt}.
Otherwise we face some technicalities but we modify our pair so that in the end we can 
apply \ref{p-bnd-compl-plt} again.\\

\emph{Step 1.} 
In this step we reduce the problem to the case when $X'$ is $\Q$-factorial and 
$(X',B')$ is non-klt, find a contraction $X'\to Z'$, and reduce to the case when $M'$ is nef and big over $Z'$. 
Taking a $\Q$-factorial generalised dlt model of $(X',B'+M')$ we can assume $X'$ is $\Q$-factorial 
and that $(X',B')$ is not klt.
Let $X'\to Z'$ be the contraction defined by $-(K_{X'}+B'+M')$. 
Running an MMP on $M'$ over $Z'$ and replacing $X'$ with 
the resulting model we can assume $M'$ is nef$/Z'$. Note that since the MMP is an MMP on $-(K_{X'}+B')$, 
the non-klt property of $(X',B')$ is preserved. Moreover, if $M'\not\sim_\Q 0$, then this is also preserved 
by the MMP because $M'\sim_\Q 0$ implies $M\sim_\Q0$ as $M$ is nef.

Let $X'\to V'/Z'$ be the contraction defined by $M'$.   
If $\dim Z'>0$, then $\dim V'>0$. If $\dim Z'=0$, then $K_{X'}+B'+M'\sim_\Q 0$, hence $M'\not\sim_\Q 0$ 
so again $\dim V'>0$. 
 In particular, if $M'$ is not big over $Z'$, then we can apply 
Proposition \ref{p-bnd-compl-non-big}.
From now on we can assume $M'$ is nef and big over $Z'$.\\ 

\emph{Step 2.}
In this step we introduce numbers $\alpha,\beta$.
Since $M'$ is nef and big over $Z'$,
$$
-(K_{X'}+B'+\alpha M')=-(K_{X'}+B'+M')+(1-\alpha)M'
$$ 
is globally nef and big for some rational number $\alpha<1$ close to $1$ which will be fixed throughout the proof. 
The contraction defined by 
$-(K_{X'}+B'+\alpha M')$ is nothing but $X'\to V'$ which is birational. After running an MMP on ${B'}$ 
over $V'$ we can assume $B'$ is nef over $V'$, hence 
$$
-(K_{X'}+\beta {B'}+\alpha M')=-(K_{X'}+B'+\alpha M')+(1-\beta)B'
$$ 
is also globally nef and big for any rational number $\beta\in (\alpha,1)$ sufficiently close $1$. 
Note that since the latter MMP is $K_{X'}+B'$-trivial, the non-klt property of $(X',B')$ is again preserved.\\

\emph{Step 3.}
In this step we modify $X',B',M'$ and look at generalised non-klt centres of  $(X',B'+\alpha M')$.
 Since $X'$ is of Fano type and $\Q$-factorial, $(X',0)$ is klt. Thus since $(X',B'+M')$ is generalised lc, 
$({X'},\beta {B'}+\alpha M')$ is generalised klt. Let $(X'',B'')$ be a $\Q$-factorial dlt model of $(X',B')$, 
and let $M''$ be the pullback of $M'$: note that $M''$ is the pushdown of $M$, assuming $X\bir X''$ 
is a morphism, otherwise $(X',B',M')$ would not be generalised lc. Writing the pullback of 
$K_{X'}+\beta {B'}$ as $K_{X''}+\tilde{B}''$,  
perhaps after increasing $\beta$, we can assume the coefficients of $B''-\tilde{B''}$ are sufficiently small.
Replacing $X',B',M'$ with $X'',B'',M''$ and renaming $\tilde{B}''$ to $\tilde{B}'$, we have: 
$({X'},\tilde{B'}+\alpha M')$ is generalised klt, $-(K_{X'}+\tilde{B'}+\alpha M')$ is nef and big, 
and the coefficients of $B'-\tilde{B'}$ are sufficiently small. 
Moreover,  every generalised non-klt centre of $(X',B'+\alpha M')$ 
is a non-klt centre of $(X',B')$: if $D$ is a prime divisor on birational models of $X'$  
such that $a(D,X',{B'}+\alpha M')=0$,  then 
$$
0=a(D,X',{B'}+\alpha M')=\alpha a(D,X',{B'}+M')+(1-\alpha)a(D,X',{B'}),
$$
hence $a(D,X',B')=0$.\\

\emph{Step 4.}
In this step under some assumptions we introduce a boundary $\Gamma'$ and apply Proposition \ref{p-bnd-compl-plt}.
Write 
$$
-(K_{X'}+B'+\alpha M')\sim_\Q A'+G'
$$
where $A',G'\ge 0$ are $\Q$-divisors and $A'$ is ample. 
First assume that $\Supp G'$ does not contain any generalised non-klt centre of $({X'},B'+\alpha M')$. 
Then, for some small $\delta>0$,
$$
-(K_{X'}+B'+\alpha M'+\delta G')\sim_\Q (1-\delta)\left(\frac{\delta}{1-\delta} A'+A'+G'\right)
$$ 
is ample and $({X'},B'+\delta G'+\alpha M')$ is a generalised lc pair whose generalised non-klt 
locus is equal to the generalised non-klt locus of $({X'},B'+\alpha M')$  
which is in turn equal to 
the non-klt locus of $(X',B')$. In particular, $({X'},B'+\delta G')$ is dlt as $(X',B')$ is dlt.

Pick a component $S'$ of $\rddown{B'}$ and let 
$\Gamma'=S'+a(B'-S'+\delta G')$ for some $a<1$ close to $1$. 
Then $(X',\Gamma'+\alpha M')$ is generalised plt, $\rddown{\Gamma'}=S'$, 
and $-(K_{X'}+\Gamma'+\alpha M')$ is ample. 
Now apply Proposition \ref{p-bnd-compl-plt}. Thus from now on we assume that $\Supp G'$ 
contains some generalised non-klt centre of $({X'},B'+\alpha M')$. \\ 

\emph{Step 5.}
In this step we define a boundary $\Omega'$ and study some of its properties.
Let $t$ be the generalised lc threshold of $G'+B'-\tilde{B'}$ with respect to $({X'},\tilde{B'}+\alpha M')$ 
and let 
$$
\Omega':=\tilde{B'}+t(G'+B'-\tilde{B'}).
$$ 
As $({X'},\tilde{B'}+\alpha M')$ is generalised klt, $t>0$.  We can assume the given 
morphism $\phi\colon X\to X'$ is a log resolution of $({X'},{B'}+G')$. Write 
$$
K_{X}+B_\alpha+\alpha M=\phi^*(K_{X'}+B'+\alpha M')
$$
and 
$$
K_{X}+\tilde{B}_\alpha+\alpha M=\phi^*(K_{X'}+\tilde{B}'+\alpha M')
$$ 
from which we get $B_\alpha-\tilde{B}_\alpha=\phi^*(B'-\tilde{B}')$. 
Perhaps after replacing $\tilde{B}'$ with $b\tilde{B}'+(1-b)B'$ for some small $b>0$, 
we can assume the coefficients of $B_\alpha-\tilde{B}_\alpha$ are sufficiently small. 

Let $G=\phi^*G'$. Since $\Supp G'$ contains some generalised non-klt centre of $({X'},B'+\alpha M')$,
we can assume $G$ and $\rddown{B_\alpha}^{\ge 0}$ have a common component, say $T$. Now 
$$
K_{X}+\tilde{B}_\alpha+t(G+B_\alpha-\tilde{B}_\alpha)+\alpha M
=\phi^*(K_{X'}+\tilde{B}'+\alpha M')+t\phi^*(G'+B'-\tilde{B'})
$$
$$
=\phi^*(K_{X'}+\Omega'+\alpha M').
$$
Since $\mu_T\tilde{B}_\alpha$ is sufficiently close to $\mu_TB_\alpha=1$, we deduce  
$t$ is sufficiently small. Moreover, letting 
$$
\Omega=\tilde{B}_\alpha+t(G+B_\alpha-\tilde{B}_\alpha)
$$ 
we have 
$$
\Omega=(1-t)\tilde{B}_\alpha+tB_\alpha+tG \le B_\alpha+tG 
$$
and
$$
\rddown{\Omega}^{\ge 0}\le \rddown{{B}_\alpha+tG}^{\ge 0}= \rddown{{B}_\alpha}^{\ge 0}.
$$\

\emph{Step 6.} 
In this step we show that $-(K_{X'}+\Omega'+\alpha M')$ is ample. By construction
$$
-(K_{X'}+\tilde{B'}+\alpha M')= -(K_{X'}+{B'}+\alpha M')+B'-\tilde{B}'\sim_\Q A'+G'+B'-\tilde{B}'
$$
is nef and big.
Thus
$$
-(K_{X'}+\Omega'+\alpha M')=-(K_{X'}+\tilde{B'}+t(G'+B'-\tilde{B}')+\alpha M')
$$
$$
=-(K_{X'}+\tilde{B'}+\alpha M')-t(G'+B'-\tilde{B}')
$$
$$
\sim_\Q A'+G'+B'-\tilde{B}'-t(G'+B'-\tilde{B}')
$$
$$
=A'+(1-t)G'+(1-t)(B'-\tilde{B}')
$$
$$
= (1-t)\left(\frac{t}{1-t}A'+A'+G'+B'-\tilde{B}'\right)
$$
which is ample.\\
 
\emph{Step 7.}
In this step we settle the proposition in the case  $\rddown{\Omega'}\neq 0$.
Assume $\rddown{\Omega'}\neq 0$ and pick a component $S'$ of 
$\rddown{\Omega'}$. By Step 5, $\rddown{\Omega'}\le \rddown{B'}$, hence $S'$ is a component of $\rddown{B'}$. 
We then define $\Gamma'$ similar to 
Step 4 by perturbing the coefficients of $\Omega'$, say by letting $\Gamma'=S'+a(\Omega'-S')$ 
for some $a<1$ close to $1$, so that   
$\rddown{\Gamma'}=S'$, $({X'},\Gamma'+\alpha M')$ is generalised plt, and $-(K_{X'}+\Gamma'+\alpha M')$ is ample. 
Then we apply Proposition \ref{p-bnd-compl-plt}. We can then assume $\rddown{\Omega'}=0$.\\

\emph{Step 8.}
In this step we construct a birational model $X''$.
 Let $\Omega^\circ$ be the sum of the birational 
transform of $\Omega'$ and the reduced exceptional divisor of $X\to X'$. So $\Omega^\circ-\Omega$  
is effective and exceptional$/X'$.
Running an MMP$/X'$ on $K_X+\Omega^\circ+\alpha M$ contracts all the components of $\Omega^\circ-\Omega$ 
as 
$$
K_X+\Omega^\circ+\alpha M=K_X+\Omega+\alpha M+\Omega^\circ-\Omega\equiv \Omega^\circ-\Omega/X',
$$
hence we reach a model $X''/X'$ 
such that if $\Omega''$ and $M''$ are the pushdowns of $\Omega^\circ$ and $M$, then 
$(X'',\Omega''+\alpha M'')$ is a $\Q$-factorial generalised 
dlt model of $({X'},\Omega'+\alpha M')$. 
The exceptional prime divisors of $X''\to X'$ all have coefficient $1$ in $\Omega''$. 
Moreover,  any prime exceptional divisor $D$ of $X\to X'$ not contracted over $X''$ 
is a component of $\rddown{\Omega}^{\ge 0}$, hence a component of $\rddown{B_\alpha}^{\ge 0}$, by Step 5. Thus 
if $B''$ is the pushdown of $B_\alpha$, then 
$K_{X''}+B''+\alpha M''$ is the pullback of $K_{X'}+B'+\alpha M'$ to $X''$, and $B''$ is the sum of the 
birational transform of $B'$ and the reduced exceptional divisor of $X''\to X'$. 
Since $(X',B'+M')$ is generalised lc, $M''$ is the pullback of $M'$ and $K_{X''}+B''+M''$ 
is the pullback of $K_{X'}+B'+M'$.\\ 

\emph{Step 9.}
In this final step we finish the proof of the proposition again by applying \ref{p-bnd-compl-plt}.
Let $\tilde{\Delta}''$ be the sum of the birational transform of $\tilde{B}'$ and the 
reduced exceptional divisor of $X''\to X'$. Note that $\tilde{\Delta}''\le \Omega''$ as $\tilde{B}'\le \Omega'$, hence
$({X''},\tilde{\Delta}''+\alpha M'')$ is generalised dlt.
Run an MMP$/X'$ on $K_{X''}+\tilde{\Delta}''+\alpha M''$. The MMP ends with $X'$ because $X'$ 
is $\Q$-factorial and because the generalised klt property of $({X'},\tilde{B}'+\alpha M')$ ensures that 
$$
K_{X''}+\tilde{\Delta}''+\alpha M''\equiv Q''/X'
$$
where $Q''$ is effective whose support is the reduced exceptional divisor of $X''\to X'$.
The last step of the MMP 
is a divisorial contraction which contracts a component $S''$ of $\rddown{\Omega''}$. 
Abuse notation and replace $X''\to X'$ with that last contraction. 

By construction, $({X''},\tilde{\Delta}''+\alpha M'')$ is generalised plt 
and $-(K_{X''}+\tilde{\Delta}''+\alpha M'')$ is ample over $X'$. Defining 
$$
\Gamma''=a\tilde{\Delta}''+(1-a)\Omega''
$$ 
for a sufficiently small $a>0$ we can check that 
$({X''},\Gamma''+\alpha M'')$ is generalised plt, $S''=\rddown{\Gamma''}$, 
and $-(K_{X''}+\Gamma''+\alpha M'')$ is globally ample because 
$$
-(K_{X''}+\Gamma''+\alpha M'')=-a(K_{X''}+\tilde{\Delta}''+\alpha M'')-(1-a)(K_{X''}+\Omega''+\alpha M'')
$$
and because $-(K_{X''}+\Omega''+\alpha M'')$ is the pullback of the ample divisor $-(K_{X'}+\Omega'+\alpha M')$.
Now apply Proposition \ref{p-bnd-compl-plt} to $K_{X''}+B''+M''$.

\end{proof}

\begin{lem}\label{l-compl-s-non-exc}
Assume Theorem \ref{t-bnd-compl} holds in dimension $\le d-1$ and Theorem \ref{t-bnd-compl-usual-local} holds
in dimension $d$. Then Theorem \ref{t-bnd-compl} holds in dimension $d$ 
for those $(X',B'+M')$ such that 
\begin{itemize}
\item $B'\in \mathfrak{R}$, and 

\item $(X',B'+M')$ is strongly non-exceptional. 
\end{itemize}
\end{lem}
\begin{proof}
By definition of strongly non-exceptional pairs, there is $P'\ge 0$ such that 
$K_{X'}+B'+M'+P'\sim_\R 0$ and $(X',B'+P'+M')$ is not 
generalised lc. In particular, $P'\neq 0$. By Lemma \ref{l-s-non-exc-change-to-Q-div}, 
we can replace $P'$ so that is a $\Q$-divisor and that $\sim_\R$ becomes $\sim_\Q$. 

Let $t$ be the generalised lc threshold of $P'$ with respect to $(X',B'+M')$. 
Then $t<1$. Let $\Omega'=B'+tP'$ and let $(X'',\Omega''+M'')$ be a $\Q$-factorial generalised 
dlt model of $(X',\Omega'+M')$. There is a boundary $\Theta''$ such that $B'^\sim\le \Theta''\le \Omega''$, 
$\rddown{\Theta''}\neq 0$, and $\Theta''\in \mathfrak{R}$ where $B'^\sim$ is the birational 
transform of $B'$ (adding $1$ to $\mathfrak{R}$ we are assuming $1\in\mathfrak{R}$). 
Let $\pi$ denote $X''\to X'$ and let $P''$ be the pullback of $P'$. 
Then  $X''$ is of Fano type, and  
$$
-(K_{X''}+\Theta''+M'')=-(K_{X''}+\Omega''+M'')+\Omega''-\Theta''
$$
$$
=-\pi^*(K_{X'}+\Omega'+M')+\Omega''-\Theta'' =-\pi^*(K_{X'}+B'+tP'+M')+\Omega''-\Theta'' 
$$
$$
\sim_\Q \pi^*(1-t)P'+\Omega''-\Theta'' = (1-t)P''+\Omega''-\Theta''
$$
where $(1-t)P''+\Omega''-\Theta''$ is effective.

Run an MMP on $-(K_{X''}+\Theta''+M'')$ 
and let $X'''$ be the resulting model. By the previous paragraph, the MMP ends with a minimal model, that is, 
$-(K_{X'''}+\Theta'''+M''')$ is nef. 
Moreover, since $P''\ge 0$ is nef and non-zero, its pushdown $P'''\neq 0$, hence 
$$
-(K_{X'''}+\Theta'''+M''')\sim_\Q (1-t)P'''+\Omega'''-\Theta''' \not\sim_\Q 0.
$$ 
In addition, since $P''$ is 
semi-ample, there is $Q''\ge 0$ such that  
$$
K_{X''}+\Omega''+Q''+M''\sim_\Q 0
$$ 
and $({X''},\Omega''+Q''+M'')$ is generalised lc. Therefore, $({X'''},\Theta'''+M''')$ 
is generalised lc, however, it is not generalised klt as $({X''},\Theta''+M'')$ is not generalised klt.

By \ref{ss-compl-remarks}(3), if $K_{X'''}+\Theta'''+M'''$ has an $n$-complement $K_{X'''}+{\Theta'''}^++M'''$ with 
${\Theta'''}^+\ge {\Theta'''}$, then $K_{X''}+\Theta''+M''$ has an $n$-complement  $K_{X''}+{\Theta''}^++M''$ with ${\Theta''}^+\ge {\Theta''}$ which in turn gives an $n$-complement $K_{X'}+{B'}^++M'$ of $K_{X'}+B'+M'$ with ${B'}^+\ge {B'}$. 
 Now apply Proposition \ref{p-bnd-compl-non-klt} to $K_{X'''}+\Theta'''+M'''$.

\end{proof}

\begin{lem}\label{l-bnd-index-of-K+B'}
Let $d\in \N$ and assume Theorem \ref{t-bnd-compl} holds in dimension $\le d-1$ and 
Theorem \ref{t-bnd-compl-usual-local} holds in dimension $d$. 
Let $\mathfrak{R}\subset [0,1]$ be a finite set of rational numbers. 
Then there is a number $n\in \N$ depending only on $d,\mathfrak{R}$ such that if $(X',B')$ is a projective lc 
pair of dimension $d$ with $K_{X'}+B'\sim_\Q 0$ and $B'\in\mathfrak{R}$, and $X'$ is of Fano type, then $n(K_{X'}+B')\sim 0$.
\end{lem}
\begin{proof}
We say a few words before going into the proof. First we reduce the problem to the case 
when $X'$ is an $\epsilon$-lc Fano variety for some fixed $\epsilon>0$.
Next we find a bounded $n\in\N$ such that $|-nK_{X'}|$ defines a birational map. 
The rest of the proof is essentially a careful analysis of the linear system $|-nK_{X'}|$.\\

\emph{Step 1.}
In this step we reduce the problem to the case when $X'$ is an $\epsilon$-lc Fano variety 
for some fixed $\epsilon>0$. Taking a small $\Q$-factorialisation we can assume $X'$ is $\Q$-factorial.
By Lemma \ref{l-ep-to-0}, there is $\epsilon\in (0,1)$ depending only on $d,\mathfrak{R}$ 
such that if $D$ is any prime divisor on birational models of $X'$ with 
$a(D,X',0) <\epsilon$, then $a(D,X',B')=0$. 
Let $X''\to X'$ be the birational contraction which extracts exactly those $D$ with $a(D,X',0) <\epsilon$ 
if there is any otherwise let $X''\to X'$ be the identity. 
Then $X''$ is of Fano type and $\epsilon$-lc. Moreover, if $K_{X''}+B''$ is the pullback of $K_{X'}+B'$, then 
all the exceptional divisors of 
$X''\to X'$ appear in $B''$ with coefficient $1$. Replacing $(X',B')$ with $(X'',B'')$ 
we can assume $X'$ is $\epsilon$-lc. After running an MMP on $K_{X'}$ we can assume we have a 
$K_{X'}$-negative Mori fibre structure $X'\to T'$. 

If $\dim T'>0$, then applying Proposition \ref{p-bnd-compl-non-big}, there is an $n$-complement $K_{X'}+B'^+$ 
of $K_{X'}+B'$ 
for some bounded $n\in\N$ with $B'^+\ge B'$. Since  $K_{X'}+B'\sim_\Q 0$, we get $B'^+=B'$, hence 
$n(K_{X'}+B')\sim 0$. Thus  we can assume $\dim T'=0$, so $X'$ is an $\epsilon$-lc Fano variety.\\

\emph{Step 2.}
In this step we introduce divisors $A',R'$ with 
$$
n\left(K_{X'}+\frac{1}{n}R'+\frac{1}{n}A'\right)\sim 0
$$ 
and $({X'},\frac{1}{n}R'+\frac{1}{n}A')$ being lc.
By Lemma \ref{l-compl-s-non-exc}, there is a number $n$ depending only on $d$ such that if 
$Y'$ is any strongly non-exceptional Fano variety of dimension $d$ with klt singularities, then 
$K_{Y'}$ has an $n$-complement. We can assume $pI(\mathfrak{R})|n$.   On the other hand, 
by Proposition \ref{p-eff-bir-delta-2}, there is $m\in \N$ depending only on $d,\epsilon$ 
such that $|-mK_{X'}|$ defines a birational map. Replacing $n$ once more we can assume $m|n$. 
So $|-nK_{X'}|$ also defines a birational map. 

By Lemma \ref{l-mov-part-lin-system}, there is a log resolution $\phi\colon X\to X'$ of $(X',B')$ 
such that $\phi^*(-nK_{X'})\sim A+R$ 
where $A$ is the movable part, $|A|$ is base point free, and $R$ is the fixed part. 
We can assume $A$ is general in $|A|$. 
Then 
$$
n\left(K_{X'}+\frac{1}{n}R'+\frac{1}{n}A'\right)\sim 0
$$ 
where $R',A'$ are the pushdowns of $R,A$. 
We claim $({X'},\frac{1}{n}R'+\frac{1}{n}A')$ is lc. If not, then $(X',0)$ is strongly non-exceptional, 
hence by our choice of $n$ we have an $n$-complement $K_{X'}+{C'}^+$ of $K_{X'}$. Since  
$n{C'}^+\in |-nK_{X'}|$ and $({X'},{C'}^+)$ is lc, we deduce $({X'},\frac{1}{n}R'+\frac{1}{n}A')$ is also lc 
because $A'+R'\in |-nK_{X'}|$ is a general member. This is a contradiction.\\  

\emph{Step 3.}
In this step we introduce $(X',\Delta'+N')$.
Let 
$$
\Delta'=\frac{1}{2}B'+\frac{1}{2n}R'~~~\mbox{ and}~~~ N'=\frac{1}{2n}A'.
$$
 Since
$$
2n(K_{X'}+\Delta'+N')=2n\left( K_{X'}+\frac{1}{2}B'+\frac{1}{2n}R'+\frac{1}{2n}A' \right)
$$
$$
=n(K_{X'}+B')+nK_{X'}+R'+A' \sim n(K_{X'}+B'),
$$ 
it is enough to find a bounded $n$ so that  
$$
2n(K_{X'}+\Delta'+N')\sim 0.
$$

\emph{Step 4.} 
In this step we prove the lemma assuming $(X',\Delta'+N')$ is klt. 
Let $\epsilon'=\min\{ \frac{\epsilon}{2},\frac{1}{2n}\}$. We claim 
$(X',\Delta'+N')$ is $\epsilon'$-lc. If not, then there is some prime divisor $D$ 
with 
$$
0<a(D,X',\Delta'+N')<\epsilon'.
$$
Note that 
$$
 a(D,X',\Delta'+N')=\frac{1}{2}a(D,X',B')+\frac{1}{2}a(D,X',\frac{1}{n}R'+\frac{1}{n}A').
$$\\ 
 Then either $0<a(D,X',B')$ which implies $\epsilon\le a(D,X',B')$ 
by Lemma \ref{l-ep-to-0},
or 
$$
0<a(D,X',\frac{1}{n}R'+\frac{1}{n}A')
$$ 
which implies 
$$
\frac{1}{n}\le a(D,X',\frac{1}{n}R'+\frac{1}{n}A').
$$
In either case we get 
$$
 a(D,X',\Delta'+N')\ge \epsilon',
$$ 
a contradiction. So $(X',\Delta'+N')$ is $\epsilon'$-lc. 
Therefore, $X'$ belongs to a bounded family by 
[\ref{HMX2}, Corollary 1.7] as the coefficients of $\Delta'+N'$ belong to a fixed finite set, 
hence the Cartier index of $K_{X'}+\Delta'+N'$ is bounded by Lemma \ref{l-bnd-couples-bnd-Cartier-index-2}. 
Therefore, there is a bounded $n$ so that  $2n(K_{X'}+\Delta'+N')\sim 0$ because 
$\Pic(X')$ is torsion-free (cf. [\ref{Isk-Prokh}, Proposition 2.1.2]).\\ 
 
\emph{Step 5.}
Finally in this step we treat the case when  
 $(X',\Delta'+N')$ is not klt. Consider this pair as a generalised 
pair with data $\phi\colon X\to X'$ and  $N=\frac{1}{2n}A$. We show it is not generalised klt.
We can write 
$$
K_X+E:=\phi^*K_{X'}~~~\mbox{and}~~~K_X+B:=\phi^*(K_{X'}+B')
$$
 and 
$$
K_X+E+\frac{1}{n}R+\frac{1}{n}A=\phi^*\left(K_{X'}+\frac{1}{n}R'+\frac{1}{n}A'\right).
$$
Thus 
$$
K_X+\frac{1}{2}B+\frac{1}{2}E+\frac{1}{2n}R+\frac{1}{2n}A=\phi^*(K_{X'}+\Delta'+N')
$$ 
where 
$$
\left(X,\frac{1}{2}B+\frac{1}{2}E+\frac{1}{2n}R\right)
$$ 
is not sub-klt because 
$(X',\Delta'+N')$ is not klt (as a usual pair) and $A$ is general. Therefore, 
$(X',\Delta'+N')$ is not generalised klt.

Now obviously $N'\not\sim_\Q 0$, hence by Proposition \ref{p-bnd-compl-non-klt},  
$K_{X'}+\Delta'+N'$ has an $n$-complement $K_{X'}+\Delta'^++N'$ for some bounded $n\in\N$ 
where $\Delta'^+\ge \Delta'$. Since $K_{X'}+\Delta'+N'\sim_\Q 0$, we have $\Delta'^+=\Delta'$.
Therefore,  $2n(K_{X'}+\Delta'+N')\sim 0$ as required. 

\end{proof}

\begin{lem}\label{l-compl-non-exc}
Assume Theorem \ref{t-bnd-compl} holds in dimension $\le d-1$ and Theorem \ref{t-bnd-compl-usual-local} holds
in dimension $d$. Then Theorem \ref{t-bnd-compl} holds in dimension $d$ 
for those $(X',B'+M')$ such that 
\begin{itemize}
\item $B'\in \mathfrak{R}$, and 

\item $(X',B'+M')$ is non-exceptional. 
\end{itemize}
\end{lem}
\begin{proof}

By definition of non-exceptional pairs, there is $P'\ge 0$ such that $K_{X'}+B'+M'+P'\sim_\R 0$ and $(X',B'+P'+M')$ is not 
generalised klt. By Lemma \ref{l-s-non-exc-change-to-Q-div}, we can assume $P'$ is a $\Q$-divisor 
and can replace $\sim_\R$ with $\sim_\Q$.
We can assume $(X',B'+P'+M')$ is generalised lc otherwise $(X',B'+M')$ is strongly non-exceptional, so we 
can apply Lemma \ref{l-compl-s-non-exc}.

 Let $\Omega'=B'+P'$ and let $(X'',\Omega''+M'')$ be a $\Q$-factorial generalised 
dlt model of $(X',\Omega'+M')$. There is a boundary $\Theta''$ such that $B'^\sim\le \Theta''\le \Omega''$, 
$\rddown{\Theta''}\neq 0$, and $\Theta''\in \mathfrak{R}$ where $B'^\sim$ is the birational 
transform of $B'$ (adding $1$ to $\mathfrak{R}$ we are assuming $1\in\mathfrak{R}$). 
Let $\pi$ denote $X''\to X'$ and let $P''$ be the pullback of $P'$. 
Then  $X''$ is of Fano type, and  
$$
-(K_{X''}+\Theta''+M'')=-(K_{X''}+\Omega''+M'')+\Omega''-\Theta'' \sim_\Q \Omega''-\Theta''.
$$

Run an MMP on $-(K_{X''}+\Theta''+M'')$ 
and let $X'''$ be the resulting model. By the previous paragraph, the MMP ends with a minimal model, that is, 
$-(K_{X'''}+\Theta'''+M''')$ is nef. 
Moreover,  since $({X'''},\Omega'''+M''')$ is generalised lc, $({X'''},\Theta'''+M''')$ 
is generalised lc, however, it is not generalised klt as $({X''},\Theta''+M'')$ is not generalised klt.
By \ref{ss-compl-remarks}(3), if $K_{X'''}+\Theta'''+M'''$ has an $n$-complement $K_{X'''}+{\Theta'''}^++M'''$ with 
${\Theta'''}^+\ge {\Theta'''}$, then $K_{X''}+\Theta''+M''$ has an $n$-complement  $K_{X''}+{\Theta''}^++M''$ with ${\Theta''}^+\ge {\Theta''}$ which in turn gives an $n$-complement $K_{X'}+{B'}^++M'$ of $K_{X'}+B'+M'$ with ${B'}^+\ge {B'}$. 
Replacing $(X',B'+M')$ with  $({X'''},\Theta'''+M''')$, we can assume $(X',B'+M')$ is not generalised klt.

Applying Proposition \ref{p-bnd-compl-non-klt}, 
we can further assume $M'\sim_\Q 0$ and that $K_{X'}+B'+M'\sim_\Q 0$ which yield $K_{X'}+B'\sim_\Q 0$. In particular, 
since $M$ is nef, we get $M=\phi^*M'\sim_\Q 0$.
Since $X'$ is of Fano type, $\Pic(X')$ and $\Pic(X)$ are torsion-free (cf. [\ref{Isk-Prokh}, Proposition 2.1.2]).
In particular, $pM\sim 0$, so $pM'\sim 0$. 
Therefore, it is enough to find a bounded $n\in \N$ divisible by $p$ such that 
$n(K_{X'}+B')\sim 0$. Now apply Lemma \ref{l-bnd-index-of-K+B'}.

\end{proof}

\subsection{Boundedness of complements}

In this final subsection we prove the main  inductive result of this section.  

\begin{prop}\label{p-BAB-exc-to-compl}
Assume Theorems \ref{t-bnd-compl-usual-local} and \ref{t-BAB-exc} hold in dimension $\le d$. 
Then Theorem \ref{t-bnd-compl} holds in dimension $d$. 
\end{prop}
\begin{proof}
By induction on $d$ we can assume  Theorem \ref{t-bnd-compl} holds in dimension $\le d-1$. 
Let $(X',B'+M')$ be as in Theorem \ref{t-bnd-compl} in dimension $d$.
Replacing $(X',B'+M')$ with a $\Q$-factorial generalised dlt model we can assume $X'$ is $\Q$-factorial. 

Pick $\epsilon\in (0,1)$. Let $\Theta'$ be the boundary whose coefficients are 
the same as $B'$ except that we replace each coefficient in $(1-\epsilon,1)$ with $1$ (similar to \ref{p-B'-to-Theta'}). That is, 
we have $\Theta'=(B')^{\le 1-\epsilon}+\lceil (B')^{>1-\epsilon}\rceil$.
Run an MMP on $-(K_{X'}+\Theta'+M')$ and let $X''$ be the resulting model.
By Proposition \ref{p-B'-to-Theta'}, if $\epsilon$ is sufficiently small depending only on $d, p,\mathfrak{R}$, 
then:
\begin{itemize}
\item $(X',\Theta'+M')$ is generalised lc,

\item the MMP does not contract any component of $\rddown{\Theta'}$,

\item $-(K_{X''}+\Theta''+M'')$ is nef, and 

\item $({X''},\Theta''+M'')$ is generalised lc.
\end{itemize}
Since $1$ is the only accumulation point of $\Phi(\mathfrak{R})$, 
there is a finite set $\mathfrak{T}\subset [0,1]$  of rational numbers which includes $\mathfrak{R}$ and 
which depends only on 
$\epsilon, \mathfrak{R}$ such that $\Theta'\in \mathfrak{T}$ (note that by our choice of $\epsilon$, $\mathfrak{I}$ 
depends only on $d, p,\mathfrak{R}$).

By \ref{ss-compl-remarks}(3), if $K_{X''}+\Theta''+M''$ has an $n$-complement $K_{X''}+{\Theta''}^++M''$ 
with ${\Theta''}^+\ge {\Theta''}$, 
then we get an $n$-complement $K_{X'}+{\Theta'}^++M'$ of $K_{X'}+{\Theta'}+M'$ with ${\Theta'}^+\ge {\Theta'}$. 
Since $\Theta'\ge B'$, 
$K_{X'}+{\Theta'}^++M'$ would be an $n$-complement of $K_{X'}+{B'}+M'$.
Therefore, replacing $X',B',M'$ with $X'',\Theta'',M''$, and  
$\mathfrak{R}$ with $\mathfrak{T}$, we can assume $B'\in \mathfrak{R}$.

By Lemma \ref{l-compl-non-exc}, we can assume $(X',B'+M')$ is exceptional. Since we are assuming 
Theorem \ref{t-BAB-exc} in dimension $d$, $X'$ belongs to a bounded family. Thus we can choose a very ample divisor $A'$ 
so that $A'^d$ and $-A'^{d-1}K_{X'}$ are bounded from above. 
By construction, 
$$
L':=-q(K_{X'}+B'+M')
$$
 is nef and integral where 
$q=pI(\mathfrak{R})$. Moreover, $A'^{d-1}L'$ is bounded from above because $B'+M'$ is pseudo-effective 
which implies 
$$
A'^{d-1}L'\le A'^{d-1}(-qK_{X'})
$$
and the right hand side is bounded.
So by Lemma \ref{l-bnd-couples-bnd-Cartier-index-2}, there is a bounded number $n$ 
divisible by $q$ such that $-n(K_{X'}+B'+M')$ is nef and Cartier. 
Since $X'$ is of Fano type, we can use the effective base point free theorem [\ref{kollar-ebpf}], 
so we can assume $|-n(K_{X'}+B'+M')|$ is base point free. Now let 
$$
G'\in |-n(K_{X'}+B'+M')|
$$ 
be a general member and let ${B'}^+=B'+\frac{1}{n}G'$. Then $({X'},{B'}^++M')$ is generalised lc and 
$n(K_{X'}+{B'}^++M')\sim 0$, hence $K_{X'}+{B'}^++M'$ is an $n$-complement of 
$K_{X'}+B'+M'$.

\end{proof}


\section{\bf Boundedness of exceptional pairs}

The aim of this section is to treat exceptional pairs and exceptional generalised pairs as in 
Theorems \ref{t-BAB-exc-usual} and \ref{t-BAB-exc}, inductively. In the non-exceptional case, 
discussed in the previous section, the main inductive tools were lifting complements from the 
base of a fibration and from a non-klt centre. 

Assume $X$ is an exceptional weak Fano variety of dimension $d$, as in Theorem \ref{t-BAB-exc-usual}. 
Let $\tau$ be as in Proposition \ref{p-eff-bir-tau} in dimension $d$. If $X$ is $\tau$-lc, then by the 
proposition, $|-mK_X|$ defines a birational map for some bounded $m\in\N$. In particular,  
taking $M\in  |-mK_X|$ and letting $\Omega=\frac{1}{m}M$ we get a klt $m$-complement $K_X+\Omega$ 
as $X$ is exceptional, 
hence $X$ is bounded by [\ref{HX}, Theorem 1.3]. If $X$ is not 
$\tau$-lc, there is a prime divisor $D$ on birational models of $X$ with log discrepancy $a(D,X,0)<\tau$. Then 
there is a birational contraction $\phi\colon Y\to X$ from a $\Q$-factorial variety contracting only one divisor 
which is $D$. Thus  $K_Y+eD=\phi^*K_X$ where $e>1-\tau$. The idea here is to run some kind of 
MMP (as in \ref{l-the-simpler-directed-MMP}) in which in each step we try to increase 
$e$ but keeping the nefness of $-(K_Y+eD)$. Since $X$ is exceptional, the pair remains klt. 
In the end we get a model $Y'$ and a fibration $Y'\to Z'$ along which $K_{Y'}+\tilde{e}D'$ is numerically trivial 
where $\tilde{e}>e$. It turns out that $(Y',\tilde{e}D')$ is $\tilde{\epsilon}$-lc for some fixed $\tilde{\epsilon}>0$.
Applying \ref{t-BAB-good-boundary} we deduce $\tilde{e}$ belongs to a 
fixed finite set. If $\dim Z'>0$, then we can pull back a complement from the base and get a bounded complement
of $K_X$ and then apply [\ref{HX}, Theorem 1.3]. If $\dim Z'=0$, we apply  [\ref{HX}, Theorem 1.3] once more.

If $(X',B'+M')$ is exceptional as in Theorem \ref{t-BAB-exc}, adapting the strategy of the previous 
paragraph is much more difficult because of the presence of $B'+M'$.  We need to discuss and bound 
exceptional thresholds (\ref{l-exc-threshold}), bound anti-canonical volumes (\ref{l-exc-bnd-vol}) 
and bound anti-canonical singularities (\ref{l-exc-bnd-lct}) in order to be able to apply a general boundedness  
criterion (\ref{p-from-bnd-lct-to-bnd-var}). The latter criterion also plays an important role  in the 
proof of BAB [\ref{B-BAB}, proof of Theorem 1.1].

\subsection{Bound on singularities}

\begin{lem}\label{l-exc-bnd-sing}
Let $d,p\in\N$ and $\Phi\subset [0,1]$ be a DCC set.
Then there is a number $\epsilon>0$ depending only on $d,p,\Phi$ satisfying the following. 
Let $(X',B'+M')$ be a projective generalised pair  with data $\phi\colon X\to X'$ 
and $M$ such that 
\begin{itemize}
\item $(X',B'+M')$ is exceptional of dimension $d$, 

\item $B'\in \Phi$ and $pM$ is b-Cartier, and 

\item $X'$ is of Fano type.\\
\end{itemize} 
Then for any $0\le P'\sim_\R -(K_{X'}+B'+M')$, the 
pair $(X',B'+P'+M')$ is generalised $\epsilon$-lc (where we consider this generalised pair with 
boundary part $B'+P'$). 
\end{lem}
\begin{proof}  
Let $(X',B'+M')$ and $P'$ be as in the statement. Since $(X',B'+M')$ is exceptional, $(X',B'+P'+M')$ is generalised klt.
Taking a $\Q$-factorialisation we can assume $X'$ is $\Q$-factorial. 
Let $D''$ be a prime divisor on birational models of $X$ such that 
$$
a:=a(D'',X',B'+P'+M')
$$ 
is minimal.  
We can assume $a<1$. Let  $X''\to X'$ be the birational contraction which extracts exactly $D''$; 
it is the identity morphism if $D''$ is already a divisor on $X'$. 
Let $K_{X''}+B''+M''$ be the pullback of $K_{X'}+B'+M'$, and let $P''$ be the pullback of $P'$. 
Let $e$ and $c$ be the coefficients of $D''$ in $B''$ and $P''$ respectively (note that it is possible to have $e<0$).
 By assumption, $e+c=1-a>0$. By \ref{ss-gpp}(7), $X''$ is of Fano type. 

Running an MMP on 
$$
-(K_{X''}+B''+cD''+M'')\sim_\Q P''-cD''\ge 0
$$ 
we get  a model $X'''$ on which  $-(K_{X'''}+B'''+cD'''+M''')$ is nef. 
We can assume the induced maps 
 $\psi\colon X\bir {X}''$ and $\pi\colon X\bir {X}'''$ are morphisms. 
 Then 
$$
\pi^*(K_{{X}'''}+{B}'''+{c}{D}'''+{M}''')\ge \psi^*(K_{X''}+B''+cD''+M'').
$$ 

Let $\epsilon$ be the number given by Proposition \ref{p-B'-to-Theta'} for the data $d,p,\Phi$. 
We will show that $a\ge \epsilon$. Assume not, that is, assume $a<\epsilon$. We will derive a contradiction.
Let $\Theta'''$ be the same as $B'''+cD'''$ except that we replace each coefficient in $(1-\epsilon,1)$ with $1$.
 Next run an MMP on $-(K_{{X}'''}+\Theta'''+{M}''')$. 
  Let $\overline{X}$ be the resulting model. Then by \ref{p-B'-to-Theta'}, $-(K_{\overline{X}}+\overline{\Theta}+\overline{M})$ 
is nef. 
We can assume the induced map  
$\rho\colon X\bir \overline{X}$ is a morphism. 
Then  
$$
\rho^*(K_{\overline{X}}+\overline{\Theta}+\overline{M})\ge 
\pi^*(K_{{X}'''}+{\Theta}'''+{M}''')\ge \pi^*(K_{X''}+B'''+cD'''+M'')
$$
$$
\ge \psi^*(K_{X''}+B''+cD''+M'')\ge \phi^*(K_{X'}+B'+M'). 
$$

By construction, $\rddown{\Theta'''}\neq 0$, hence $({X}''',{\Theta}'''+{M}''')$ is not generalised klt 
which in turn implies $({\overline{X}},\overline{\Theta}+\overline{M})$ is not generalised klt. In particular, 
$({\overline{X}},\overline{\Theta}+\overline{M})$ is not exceptional as 
$-(K_{\overline{X}}+\overline{\Theta}+\overline{M})$ is nef and ${\overline{X}}$ is of Fano type. Therefore, 
$(X',B'+M')$ is not exceptional, by Lemma \ref{l-exc-compare-bigger-model}, a contradiction.

\end{proof}

\subsection{From complements to Theorem \ref{t-BAB-exc-usual}}

Before treating \ref{t-BAB-exc-usual} we prove a lemma.

\begin{lem}\label{l-the-simpler-directed-MMP}
Assume that 
\begin{itemize}
\item $(X,B)$ is a projective $\Q$-factorial pair, 

\item $-(K_X+B)$ is nef, 

\item $X$ is of Fano type, and that 

\item $D\neq 0$ is an effective $\R$-divisor on $X$.
 
\end{itemize}
Then there is a $-D$-MMP ending with a non-birational contraction 
$X'\to T'$  such that 
\begin{itemize}
\item $-(K_{X'}+B'+tD')$ is globally nef  and numerically trivial over $T'$ for some $t\ge 0$, and

\item the intersection of $K_X+B+tD$  with each extremal ray in the MMP is non-negative.
\end{itemize}
\end{lem}
\begin{proof}
Let $s$ be the largest real number such that  
$-(K_{X}+B+sD)$ 
is nef. Note that it is possible to have $s=0$, e.g. when $K_X+B\equiv 0$. Since $X$ is of Fano type, 
the Mori cone of $X$ is generated by finitely many extremal rays, hence  
there is an extremal ray $R$ such that 
$$
(K_{X}+B+sD)\cdot R=0 ~~~\mbox{and $D\cdot R>0$}.
$$
If $R$ defines a non-birational contraction, 
then we stop and let $t=s$ and let $X=X'\to T'$ be that contraction. 
If not, we let $X\bir Y$ be the divisorial contraction or flip defined by $R$. 
Then $-(K_Y+B_Y+sD_Y)$ is nef  where $B_Y,D_Y$ 
are the pushdowns of $B,D$. Moreover, $Y$ is of Fano type, and since $D\cdot R>0$, we have $D_Y\neq 0$.
Now let $u$ be the largest real number such that 
$-(K_Y+B_Y+uD_Y)$ is nef, and continue as above. 

The process gives a $-D$-MMP which eventually ends with a $-D$-Mori fibre space, that is, 
a non-birational contraction 
$X'\to T'$  such that $D'$ is ample over $T'$ because $D$ is effective. By construction, 
$-(K_{X'}+B'+tD')$ is nef globally and numerically trivial over $T'$ for some $t\ge 0$. 
In the first step of the MMP we have  
$$
(K_X+B+tD)\cdot R=(t-s)D\cdot R\ge 0.
$$ 
The same is true in each step, hence the intersection of $K_X+B+tD$  with each extremal ray in the MMP is non-negative.

\end{proof}

\begin{lem}\label{l-from-compl-to-BAB-exc-usual}
Assume Theorem \ref{t-bnd-compl} holds in dimension $\le d-1$ and Theorem \ref{t-bnd-compl-usual-local}
 holds in dimension $d$. 
Then Theorem \ref{t-BAB-exc-usual} holds in dimension $d$. 
\end{lem}
\begin{proof}
Assume the statement is not true. Then there is a sequence $X_i$ of exceptional weak Fano varieties 
of dimension $d$ which do not form a bounded family. Let $q_i\in\N$ be the smallest number such that 
$|-q_iK_{X_i}|$ is base point free. Then  $K_{X_i}$ has a $q_i$-complement which is necessarily klt 
because $X_i$ is exceptional. The set of the $q_i$ is not finite by [\ref{HX}, Theorem 1.3].
Replacing the sequence $X_i$ with a subsequence we can assume the $q_i$ form a strictly increasing sequence. 
In particular, no infinite subsequence of the $X_i$ forms a bounded family. 
Let $ X_i\to \tilde{X}_i$ be the contraction defined by $-K_{X_i}$. Then $\tilde{X}_i$ are Fano varieties,  
and they do not form a bounded family otherwise the $q_i$ would be in a finite set. 
Thus replacing $X_i$ with $\tilde{X}_i$ we can assume 
$X_i$ are Fano varieties. 

Let $\epsilon_i$ 
be the minimal log discrepancy of $X_i$. Let $\epsilon=\limsup \epsilon_i$. 
First assume $\epsilon<1$. Replacing the sequence we can assume $\epsilon_i\le \epsilon$ for every $i$. 
There is a birational contraction $X_i'\to X_i$ 
from a $\Q$-factorial variety which contracts only one prime divisor $D_i'$ and $a(D_i',X_i,0)=\epsilon_i$. 
Let $K_{X_i'}+e_iD_i'$ be the pullback of $K_{X_i}$. Then $e_i=1-\epsilon_i\ge 1-\epsilon$. 
By Lemma \ref{l-the-simpler-directed-MMP}, there is a $-D_i'$-MMP ending with a non-birational contraction 
$X_i''\to T_i''$  such that 
\begin{itemize}
\item $-(K_{X_i''}+e_iD_i''+t_iD_i'')$ is nef globally and numerically trivial over $T_i''$ for some $t_i\ge 0$, and

\item the intersection of $K_{X_i'}+e_iD_i'+t_iD_i'$  with each extremal ray in the MMP is non-negative.
\end{itemize}

Let $\tilde{e}_i=e_i+t_i$ which is $\ge 1-\epsilon$. 
Take common resolutions $\phi_i\colon W_i\to X_i$, $\psi_i\colon W_i\to X_i'$ and $\pi_i\colon W_i\to X_i''$. 
 Then 
$$
\pi_i^*(K_{X_i''}+\tilde{e}_iD_i'') \ge \psi_i^*(K_{X_i'}+\tilde{e}_iD_i')\ge \psi_i^*(K_{X_i'}+{e}_iD_i')=\phi_i^*K_{X_i}
$$
where the first inequality follows from the second item in the list of the properties of the MMP above. 
Thus  if $0\le P_i''\sim_\Q  -(K_{X_i''}+\tilde{e}_iD_i'')$ and if  $K_{X_i}+P_i$ is 
the crepant pullback of $K_{X_i''}+\tilde{e}_iD_i''+P_i''$ to $X_i$, then  $P_i\ge 0$.

Applying Lemma \ref{l-exc-bnd-sing}, $(X_i,P_i)$ is $\tilde{\epsilon}$-lc  for some 
$\tilde{\epsilon}>0$ independent of $i$. Thus $(X_i'',\tilde{e}_iD_i'')$ is also $\tilde{\epsilon}$-lc. 
Applying Theorem \ref{t-BAB-good-boundary} to the restriction of $K_{X_i''}+\tilde{e}_iD_i''$ to the general 
fibres of $X_i''\to T_i''$ shows that these fibres belong to a bounded family. Moreover, by Lemma \ref{l-bnd-fam-intersection}, 
 $\tilde{e}_i$ belongs to a finite set independent of $i$. 
 
If $\dim T_i''>0$, then by Proposition \ref{p-bnd-compl-non-big},  
$K_{X_i''}+\tilde{e}_iD_i''$ has an $n$-complement $K_{X_i''}+B_i''$  for some $n$ independent of $i$ 
such that $\tilde{e}_iD_i''\le B_i''$. On the other hand,  if $\dim T_i''=0$, then 
by  [\ref{HX}, Theorem 1.3] the varieties $X_i''$ form a bounded family, hence by 
Lemma \ref{l-bnd-couples-bnd-Cartier-index-2}, the Cartier index of $K_{X_i''}+\tilde{e}_iD_i''$ is bounded, so
again $K_{X_i''}+\tilde{e}_iD_i''$ has an $n$-complement $K_{X_i''}+B_i''$  for some $n$ independent of $i$ 
where in this case $\tilde{e}_iD_i''=B_i''$. Therefore, by \ref{ss-compl-remarks}(2), $K_{X_i'}+\tilde{e}_iD_i'$ has an $n$-complement 
$K_{X_i'}+B_i'$  for some $n$ such that $\tilde{e}_iD_i'\le B_i'$,  
and this in turn implies $K_{X_i}$ has an 
$n$-complement $K_{X_i}+B_i$.  Since $X_i$ is exceptional, $(X_i,B_i)$ is klt.  
So the $X_i$  form a bounded family by  [\ref{HX}, Theorem 1.3], a contradiction. 

Now we can assume $\epsilon=1$ and that the $\epsilon_i$ are sufficiently close to $1$. 
 By Proposition \ref{p-eff-bir-tau}, there exists $m\in \N$  
such that $|-mK_{X_i}|$ defines a birational map for every $i$. Pick  
$0\le C_i\sim -mK_{X_i}$. Since $X_i$ is exceptional, 
$(X_i,{B_i}:=\frac{1}{m}C_i)$ is klt, hence it is $\frac{1}{m}$-lc.
Therefore, the $X_i$ are bounded by  [\ref{HX}, Theorem 1.3], a contradiction. 

\end{proof}

\subsection{Bound on exceptional thresholds}

\begin{lem}\label{l-exc-threshold}
Let $d,p\in \N$ and let $\Phi\subset [0,1]$ be a DCC set. Then there is $\beta\in (0,1)$ 
depending only on $d,p,\Phi$ satisfying the following. Assume  
$(X',B'+M')$ is a projective generalised pair 
with data $\phi\colon X\to X'$ and $M$ such that 
\begin{itemize}
\item $(X',B'+M')$ is exceptional of dimension $d$,

\item $B'\in\Phi$, and $pM$ is b-Cartier, 

\item $-(K_{X'}+B'+M')$ is nef, and 

\item $X'$ is of Fano type and $\Q$-factorial.\\
\end{itemize}
Then $(X',B'+\alpha M')$ is exceptional for every $\alpha\in [\beta,1]$.
\end{lem}
\begin{proof}
Suppose the lemma is not true. Then there exist   
a strictly increasing sequence of numbers $\alpha_i$ approaching $1$ and a sequence 
$(X_i',B_i'+M_i')$ of generalised pairs  as in the statement 
such that $(X_i',B_i'+\alpha_i M_i')$ is not exceptional. In particular, $M_i'$ is not numerically trivial. 
Since
$$
 -(K_{X_i'}+B_i'+\alpha_i M_i')=-(K_{X_i'}+B_i'+M_i')+(1-\alpha_i)M_i'
$$ 
is pseudo-effective and $X_i'$ is of Fano type, there is 
$$
0\le P_i'\sim_\R -(K_{X_i'}+B_i'+\alpha_i M_i'),
$$
so it makes sense to say that $(X_i',B_i'+\alpha_i M_i')$ is non-exceptional.
In particular, we can choose $P_i'$ such that $(X_i',B_i'+P_i'+\alpha_iM_i')$ is not generalised klt.
Moreover, running an MMP on $P_i'$ we can assume it is nef. Note that since $N_\sigma(P_i')=0$, 
the MMP does not contract any divisor, hence all our assumptions are 
preserved except that $-(K_{X'}+B'+M')$ may no longer be nef but we will not need nefness.

Let $t_i$ be the generalised lc threshold of $P_i'$ with respect to $(X_i',B_i'+\alpha_iM_i')$.  
Then $t_i\le 1$.
Let $\Omega_i'=B_i'+t_iP_i'$ and let $(X_i'',\Omega_i''+\alpha_iM_i'')$ be a $\Q$-factorial generalised 
dlt model of $(X_i',\Omega_i'+ \alpha_iM_i')$.  
Adding $1$ to $\Phi$, we can find a boundary $\Gamma_i''\in \Phi$ such that $B_i'^\sim\le \Gamma_i''\le \Omega_i''$ and 
$\rddown{\Gamma_i''}\neq 0$ where $B_i'^\sim$ is the birational transform of $B_i'$. Let $G_i''=\Omega_i''-\Gamma_i''$. 
Then from 
$$
-(K_{X_i'}+\Omega_i'+\alpha_iM_i')=-(K_{X_i'}+B_i'+t_iP_i'+\alpha_iM_i')
$$
$$
=-(K_{X_i'}+B_i'+\alpha_iM_i')-t_iP_i'\sim_\R P_i'-t_iP_i'=(1-t_i)P_i'
$$
we get 
$$
-(K_{X_i''}+\Gamma_i''+\alpha_iM_i'')=
-(K_{X_i''}+\Omega_i''+\alpha_iM_i'')+G_i''\sim_\R (1-t_i)P_i''+G_i''\ge 0
$$ 
where $P_i''$ is the pullback of $P_i'$.

By \ref{ss-gpp}(7), $X_i''$ is of Fano type. Run an MMP on $-(K_{X_i''}+\Gamma_i''+M_i'')$ and 
let $X_i'''$ be the resulting model. First we argue that $({X_i'''},\Gamma_i'''+M_i''')$ is generalised 
lc for $i\gg 0$. Since $({X_i''},\Omega_i''+\alpha_iM_i'')$ 
is generalised lc and since $-(K_{X_i''}+\Omega_i''+\alpha_iM_i'')$ is nef hence semi-ample, 
we deduce that $({X_i'''},\Omega_i'''+\alpha_iM_i''')$ is generalised lc which in turn implies 
$({X_i'''},\Gamma_i'''+\alpha_i M_i''')$ is generalised lc. Now by the ACC for generalised lc thresholds [\ref{BZh}, Theorem 1.5],
$({X_i'''},\Gamma_i'''+M_i''')$ is generalised lc for $i\gg 0$,  hence we can assume this holds for every $i$.

We show that $X_i'''$ is a minimal model, that is, 
$-(K_{X_i'''}+\Gamma_i'''+M_i''')$ is nef, for $i\gg 0$. Assume not. Then we can assume 
the MMP ends with a Mori fibre space, 
that is, there is an extremal non-birational contraction $X_i'''\to T_i'''$ such that 
$K_{X_i'''}+\Gamma_i'''+M_i'''$ is ample over $T_i'''$. 
Let $\lambda_i\in [\alpha_i,1]$ be the 
smallest number such that $K_{X_i'''}+\Gamma_i'''+\lambda_iM_i'''$ is nef over $T_i'''$. 
Since
$$
-(K_{X_i'''}+\Gamma_i'''+\alpha_iM_i''')\sim_\R (1-t_i)P_i'''+G_i'''\ge 0,
$$ 
$K_{X_i'''}+\Gamma_i'''+\alpha_iM_i'''$ cannot be ample over $T_i'''$. Thus  
$M_i'''$ is ample over $T_i'''$ and $K_{X_i'''}+\Gamma_i'''+\lambda_iM_i'''\equiv 0/T_i'''$. 
By restricting $K_{X_i'''}+\Gamma_i'''+\lambda_iM_i'''$ to the general fibres of 
$X_i'''\to T_i'''$ and applying the global ACC for generalised pairs [\ref{BZh}, Theorem 1.6] we get a contradiction.

We can assume the induced maps $\psi_i\colon X_i\bir X_i''$ and $\pi_i\colon X_i\bir X_i'''$ are 
morphisms. Then
$$
-\pi_i^*(K_{X_i'''}+\Gamma_i'''+M_i''') \le -\psi_i^*(K_{X_i''}+\Gamma_i''+M_i'').
$$
Thus since $-(K_{X_i'''}+\Gamma_i'''+M_i''')$ is semi-ample, there is 
$$
0\le Q_i''\sim_\R -(K_{X_i''}+\Gamma_i''+M_i'')
$$ 
hence $({X_i''},\Gamma_i''+M_i'')$ is non-exceptional. Therefore, 
$({X_i'},\Gamma_i'+M_i')$ is non-exceptional where $\Gamma_i'$ is the pushdown of $\Gamma_i''$. 
This in turn implies $({X_i'},B_i'+M_i')$ is non-exceptional as $B_i'\le \Gamma_i'$, a contradiction.

\end{proof}

\subsection{Bound on anti-canonical volumes}

\begin{lem}\label{l-exc-bnd-vol}
Let $d,p\in \N$ and let $\Phi\subset [0,1]$ be a DCC set. Then there is $v$ 
depending only on $d,p,\Phi$ satisfying the following. Let $(X',B'+M')$ be a projective generalised pair 
 with data $\phi\colon X\to X'$ and $M$ such that 
\begin{itemize}
\item $(X',B'+M')$ is generalised klt of dimension $d$,

\item  $B'\in\Phi$, and $pM$ is b-Cartier and big, and

\item $K_{X'}+B'+M'\sim_\R 0$.\\
\end{itemize}
Then $\vol(-K_{X'})\le v$.
\end{lem}
\begin{proof}
First note that the assumptions imply that $X'$ is of Fano type. Indeed, since $M$ is big, $M'$ is big, 
so $M'\sim_\Q H'+D'$ where $H'$ is ample and $D'$ is effective. Now if $\lambda>0$ is small, then 
$$
(X',B'+\lambda D'+\lambda H'+(1-\lambda)M')
$$ 
is generalised klt which implies that 
$$
(X',B'+\lambda D'+(1-\lambda)M')
$$ 
is generalised klt too. Moreover,
$$
-(K_{X'}+B'+\lambda D'+(1-\lambda)M'))\sim_\Q \lambda H'
$$
is ample, so $X'$ is of Fano type by \ref{ss-gpp}(6).

Now if the lemma does not hold, then there is a sequence of generalised pairs 
$(X_i',B_i'+M_i')$ as in the statement such that the volumes $\vol(-K_{X_i'})$ form a 
strictly increasing sequence approaching $\infty$. 
After taking a small $\Q$-factorialisation 
we can assume $X_i'$ is $\Q$-factorial. Moreover, replacing $X_i$ we can assume $X_i\to X_i'$ is a 
resolution, in particular, $pM_i$ is Cartier. 
   Since $pM_i$ is nef and big, by Lemma \ref{l-large-boundaries}, $K_{X_i}+3dpM_i$ is big, 
hence $K_{X_i'}+3dpM_i'$ is big too. Thus 
$$
\vol(-K_{X_i'})<\vol(-K_{X_i'}+K_{X_i}+3dpM_i)=\vol(3dpM_i'),
$$ 
hence it is 
enough to show $\vol(M_i')$ is bounded from above. We can then assume the volumes $\vol(M_i')$   
form a strictly increasing sequence of numbers approaching $\infty$. 
   
There is a strictly decreasing sequence of numbers $\delta_i$ 
approaching zero such that $\vol(\delta_iM_i')>(2d)^d$. 
Thus letting $\alpha_i=1-\delta_i$,
$$
\vol(-(K_{X_i'}+B_i'+\alpha_iM_i'))=\vol(\delta_i M_i')>(2d)^d.
$$
On the other hand, running an MMP on $M_i'$ and replacing $X_i'$ with the resulting model, 
we can assume $M_i'$ is nef and big. Note that the MMP does not contract any divisor, so all 
the assumptions are preserved.
Now by \ref{ss-non-klt-centres}(2), there is some 
$$
0\le P_i'\sim_\R -(K_{X_i'}+B_i'+\alpha_iM_i') 
$$ 
such that $(X_i',P_i')$ is not klt which in turn implies 
$(X_i',B_i'+P_i'+\alpha_iM_i')$ is not generalised klt (note that $M_i'$ may not be ample but \ref{ss-non-klt-centres}(2) 
still applies as $M_i'$ is nef and big so it can be approximated by ample divisors with volume $>(2d)^d$). 
In particular,  $(X_i',B_i'+\alpha_iM_i')$ is non-exceptional. This contradicts  
Lemma \ref{l-exc-threshold} as $\lim \alpha_i=1$ and $(X_i',B_i'+M_i')$ are exceptional as 
$(X_i',B_i'+M_i')$ is generalised klt and $K_{X_i'}+B_i'+M_i'\sim_\R 0$.

\end{proof}

\subsection{Bound on lc thresholds}

\begin{lem}\label{l-exc-bnd-lct}
Let $d,p,l\in \N$ and let $\Phi\subset [0,1]$ be a DCC set. Then there is a positive real 
number $t$ depending only on $d,p,l,\Phi$ satisfying the following.  
Let $(X',B'+M')$ be a projective generalised pair  with data $\phi\colon X\to X'$ and $M$ such that 
\begin{itemize}
\item $(X',B'+M')$ is exceptional of dimension $d$,

\item  $B'\in\Phi$, and $pM$ is b-Cartier and big, 

\item $-(K_{X'}+B'+M')$ is nef, and

\item $X'$ is of Fano type and $\Q$-factorial.\\
\end{itemize}
Then for any $L'\in|-lK_{X'}|$, the pair $(X',tL')$ is klt.
\end{lem}
\begin{proof}
If the statement does not hold, then there exist a decreasing sequence of numbers $t_i$ 
approaching zero and a sequence 
$(X_i',B_i'+M_i')$ of generalised pairs as in the statement such that $(X_i',t_iL_i')$ is 
not klt for some $L_i'\in |-lK_{X_i'}|$. Replacing $X_i$ we can assume $X_i\to X_i'$ is a resolution, 
in particular, $pM_i$ is Cartier. 
Since $pM_i$ is nef and big, by Lemma \ref{l-large-boundaries}, $K_{X_i}+3dpM_i$ is big, hence  
$$
K_{X_i'}+3dpM_i'\sim_\Q -\frac{1}{l}L_i'+3dpM_i'
$$ 
is also big. 

By Lemma \ref{l-exc-threshold}, there is a rational number $\beta\in(0,1)$ such that  
$(X_i',B_i'+\beta M_i')$ is exceptional for every $i$. 
Let $s_i$ be the generalised lc threshold of $L_i'$ with respect to 
 $({X_i'},B_i'+\beta M_i')$. Then $s_i\le t_i$. We can assume $s_i<\frac{1-\beta}{3dpl}$ for every $i$. 
Thus    
$$
-(K_{X_i'}+B_i'+s_iL_i'+ \beta M_i')= -(K_{X_i'}+B_i'+M_i')+(1-\beta)M_i'-s_iL_i'
$$
is big by the previous paragraph. Therefore, there is 
$$
0\le P_i'\sim_\R -(K_{X_i'}+B_i'+s_iL_i'+ \beta M_i').
$$
Now 
$$
({X_i'},B_i'+s_iL_i'+P_i'+\beta M_i')
$$ 
is not generalised klt, so 
$({X_i'},B_i'+\beta M_i')$ is non-exceptional. This is a contradiction. 

\end{proof}

\subsection{From bound on lc thresholds to boundedness of varieties}

The next result is one of the key statements of this section. As mentioned before it plays a crucial 
role in the proof of BAB as well [\ref{B-BAB}, proof of Theorem 1.1].

\begin{prop}\label{p-from-bnd-lct-to-bnd-var}
Let $d,m,v\in \N$ and let $t_l$ be a sequence of positive real numbers. 
Assume Theorem \ref{t-bnd-compl} holds in dimension $\le d-1$ and Theorem \ref{t-bnd-compl-usual-local} 
holds in dimension $d$.
Let $\mathcal{P}$ be the set of projective varieties $X$ such that 
\begin{itemize}
\item $X$ is a klt weak Fano variety of dimension $d$, 

\item $K_X$ has an $m$-complement, 

\item $|-mK_X|$ defines a birational map,

\item $\vol(-K_X)\le v$, and 

\item for any $l\in\N$ and any $L\in |-lK_X|$,  the pair $(X,t_lL)$ is klt.\\
\end{itemize}
Then $\mathcal{P}$ is a bounded family. 
\end{prop}
\begin{proof}
We first give a short summary of the proof. Using the assumption that $|-mK_X|$ defines a birational map 
and $\vol(-K_X)\le v$, we find a bounded  smooth birational model $\overline{W}$ of $X$. We take an 
$m$-complement $K_X+B^+$ of $K_X$. If the complement is klt, we apply [\ref{HX}, Theorem 1.3]. 
Otherwise we let $K_{\overline{W}}+B^+_{\overline W}$ be the crepant pullback of $K_X+B^+$ to 
${\overline W}$, and try to perturb $B^+_{\overline W}$ to get $\Delta_{\overline W}\sim_\Q B^+_{\overline W}$ 
with $K_{\overline{W}}+\Delta_{\overline W}$ being sub-klt and with bounded Cartier index. 
We pull back the latter to $X$ to get 
$K_X+\Delta$. The main issue here is that $\Delta$ may not be effective. Using the final 
assumption of the proposition and some complement theory we construct $\Theta$ 
with coefficients in a fixed finite set so that 
$K_X+\Theta\sim_\Q 0$ is klt, and again apply  [\ref{HX}, Theorem 1.3].\\

\emph{Step 1.}
In this step we consider a birationally bounded model of $X$. 
By Lemma \ref{l-bnd-Fano-change-model}, 
we can take a small $\Q$-factorialisation of $X$, hence  assume $X$ is $\Q$-factorial. 
Let $M$ be a general element of $|-mK_X|$. Applying Prospotion \ref{p-log-bir-bnd-cert-pairs} (with $B=0$), 
 there is a bounded set of couples $\mathcal{P}$ 
and a number $c\in\R^{>0}$ depending only on $d,v$ such that 
there is a projective log smooth couple $(\overline{W},{\Sigma}_{\overline{W}})\in \mathcal{P}$ 
and a birational map $\overline{W}\bir X$ such that 
\begin{itemize}
\item  $\Supp {\Sigma}_{\overline{W}}$ contains the exceptional 
divisor of  $\overline{W}\bir X$ and the birational transform of $\Supp {M}$;

\item if $X'\to X$ and $X'\to \overline{W}$ is a common resolution and 
$M_{\overline{W}}$ is the pushdown of $M_{X'}:=M|_{X'}$, then each coefficient of $M_{\overline{W}}$ is at most $c$;

\item there is a resolution $\phi\colon W\to X$ such that $M_W:=M|_W\sim A_W+R_W$ where $A_W$ is 
the movable part of $|M_W|$, $|A_W|$ is base point free, and if $X'\to X$ factors through $W$, then 
$A_{X'}:=A_W|_{X'}\sim 0/\overline{X}$.
\end{itemize}

Note that since we assumed $M$ is a general element of $|-mK_X|$, we can assume 
$M_W=A_W+R_W$ and that $A_W$ is general in  $|A_W|$. In particular, if $A_{\overline{W}}$ is the 
pushdown of $A_W|_{X'}$, then $A_{\overline W}\le \Sigma_{\overline W}$.\\

\emph{Step 2.}
In this step we discuss complements.  
Let $M,A,R$ be the pushdowns of $M_W,A_W,R_W$ to $X$.
By assumption $K_X$ has an $m$-complement, say $K_X+B^+$. In particular, 
$mB^+$ is an element of the linear system $|M|=|-mK_X|$. But since $(X,B^+)$ is lc and since $A+R$ is a 
general element of $|-mK_X|$, 
we deduce that $(X,\frac{1}{m}A+\frac{1}{m}R)$ is lc too. Thus replacing $B^+$ we can assume 
$B^+=\frac{1}{m}A+\frac{1}{m}R$. 
By [\ref{HX}, Theorem 1.3], we can assume $(X,B^+)$ is not klt. 
 The idea is to construct another complement which is klt.
Replacing $m$ from the beginning we can assume $m>1$ so that  $A$ is not a component of $\rddown{B^+}$.\\

\emph{Step 3.} 
In this step we use the model $\overline W$ to construct a new lc complement $K_X+\Omega$. 
Since  $|A_{\overline W}|$ defines a birational contraction and since $A_{\overline W}\le \Sigma_{\overline W}$, 
there is $l\in \N$ depending only on the family $\mathcal{P}$ such that $lA_{\overline W}\sim G_{\overline W}$ 
for some $G_{\overline W}\ge 0$ whose support contains $\Sigma_{\overline W}$.
Now let $K_{\overline W}+B^+_{\overline W}$ be the crepant pullback of $K_X+B^+$ to ${\overline W}$. 
Then $({\overline W},B^+_{\overline W})$ is sub-lc and   
$$
\Supp B^+_{\overline W}\subseteq \Sigma_{\overline{W}}\subseteq \Supp G_{\overline{W}}.
$$ 
Let $G$ be the pushdown of $G_{X'}:=G_{\overline{W}}|_{X'}$ to $X$ where as in Step 1, 
$X'\to X$ and $X'\to \overline{W}$ is a common resolution. Since $A_{X'}\sim 0/\overline{W}$, 
we deduce that $A_{X'}$ is the pullback of $A_{\overline W}$. Thus from 
$lA_{\overline W}\sim G_{\overline W}$ we get $lA_{X'}\sim G_{X'}$ which in turn gives $lA\sim G$.
Therefore,  $G+lR\in |-lmK_X|$. 

Now, by assumption, $(X,t(G+lR))$ is klt where $t:=t_{lm}$. 
In particular, this means the coefficients of 
$t(G+lR)$ belong to a fixed finite set depending only on $t$. 
Decreasing $t$ we can assume it is rational and that $t<\frac{1}{lm}$.

If $(X,\frac{1}{lm}(G+lR))$
is lc, then we let $\Omega=\frac{1}{lm}(G+lR)$ and $n=lm$. But if it is not lc, then the pair $(X,t(G+lR))$
is strongly non-exceptional, hence by Lemma \ref{l-compl-s-non-exc}, there is $n$ depending only on $d,t$ 
such that there is $\Omega\ge t(G+lR)$ with 
 $(X,\Omega)$ is lc and $n(K_X+\Omega)\sim 0$.\\ 

\emph{Step 4.}
In this step we introduce $\Delta_{\overline W}$ and $\Delta$.
Let 
$$
\Delta_{\overline W}:=
B^+_{\overline W}+\frac{t}{m}A_{\overline W}-\frac{t}{lm}G_{\overline W}
$$
which satisfies  $K_{\overline W}+\Delta_{\overline W}\sim_\Q 0$. Since $A_{\overline W}$ is not 
a component of $\rddown{B^+_{\overline W}}$ and since $\Supp B^+_{\overline W}\subseteq \Supp G_{\overline{W}}$, 
there is $\epsilon>0$ depending only on $t,l,m$ 
such that  
$({\overline W},\Delta_{\overline W})$ is sub-$\epsilon$-lc.

Let $K_X+\Delta$ be the crepant pullback of $K_{\overline W}+\Delta_{\overline W}$ to $X$. 
Then $K_X+\Delta\sim_\Q 0$ 
and $(X,\Delta)$ is sub-$\epsilon$-lc. 
However, 
$$
\Delta=B^++\frac{t}{m}A-\frac{t}{lm}G
$$ 
has negative coefficients, so $(X,\Delta)$ is only a sub-pair.\\ 

\emph{Step 5.}
Finally we produce a suitable klt complement $K_X+\Theta$ and prove boundedness of $X$.
Let $\Theta= \frac{1}{2}\Delta+\frac{1}{2}\Omega$. Then 
$$
\Theta=\frac{1}{2}B^++\frac{t}{2m}A-\frac{t}{2lm}G+\frac{1}{2}\Omega\ge 
$$
$$
\frac{1}{2}B^++\frac{t}{2m}A-\frac{t}{2lm}G+\frac{t}{2}(G+lR) \ge -\frac{t}{2lm}G+\frac{t}{2}G \ge 0.
$$\\  
Since $(X,\Delta)$ is sub-$\epsilon$-lc and $(X,\Omega)$ is lc, $(X,\Theta)$ is $\frac{\epsilon}{2}$-lc.
Moreover, $K_X+\Theta\sim_\Q 0$, and 
the coefficients of $\Theta$ belong to a fixed finite set depending only on $t,l,m,n$. 
Now apply [\ref{HX}, Theorem 1.3].

\end{proof}

\subsection{From complements to Theorem \ref{t-BAB-exc}}

We arrive at the main result of this section. 

\begin{prop}\label{p-from-compl-to-BAB-exc}
Assume Theorem \ref{t-bnd-compl} holds in dimension $\le d-1$ and Theorem \ref{t-bnd-compl-usual-local} 
holds in dimension $d$. Then Theorem \ref{t-BAB-exc} holds in dimension $d$.
\end{prop}
\begin{proof}
We first give a short summary of the proof. We reduce boundedness of $(X', \Supp B')$ to boundedness of $X'$, 
and to the case $B'\in\mathfrak{R}$. We argue that it is enough to construct a bounded klt complement 
of $X'$. We consider an lc complement of $K_{X'}$, and then reduce the problem to the situation 
when $K_{X'}+B'+M'\sim_\Q 0$ and $M$ is big. Finally we apply \ref{p-from-bnd-lct-to-bnd-var} to 
the minimal model of $-K_{X'}$ to produce the required bounded klt complement of  $K_{X'}$.\\

\emph{Step 1.}
In this step we reduce the problem to boundedness of $X'$, and the case  $B'\in\mathfrak{R}$.
Let $(X',B'+M')$ be as in Theorem \ref{t-BAB-exc} in 
dimension $d$. It is enough to show $X'$ is bounded because then we can find a very ample 
Cartier divisor $H'$ so that $-K_{X'}\cdot H'^{d-1}$ is bounded from above 
 (note that although $X'$ may not be $\Q$-factorial but $D'\cdot H'^{d-1}$ 
is well-defined for any Weil divisor $D'$); then 
$$
B'\cdot H'^{d-1}\le B'\cdot H'^{d-1}-(K_{X'}+B'+M')\cdot H'^{d-1} \le -K_{X'}\cdot H'^{d-1}
$$ 
where the first inequality follows from nefness of $-(K_{X'}+B'+M')$ and the second inequality 
follows from the fact that $M'$ is pseudo-effective: thus $B'\cdot H'^{d-1}$ is bounded 
and this implies $(X',B')$ is log bounded as the coefficients of $B'$ belong to the 
DCC set $\Phi(\mathfrak{R})$. 

On the other hand, by Lemma \ref{l-exc-bnd-sing}, $(X',B'+M')$ is generalised 
$\epsilon$-lc for some $\epsilon>0$ depending only on $d,p,\mathfrak{R}$. In particular, 
the coefficients of $B'$ belong to a finite set depending only on $d,p,\mathfrak{R}$ because $1$ is the only accumulation 
point of $\Phi(\mathfrak{R})$. Extending $\mathfrak{R}$ we can assume $B'\in\mathfrak{R}$.\\

\emph{Step 2.}
In this step we reduce the problem to existence of a klt complement and the case when 
$X'$ is $\Q$-factorial. 
By [\ref{HX}, Theorem 1.3], 
it is enough to show that $K_{X'}$ has a klt $a$-complement for some bounded number $a\in \N$. 
This in turn follows from existence of a klt 
$a$-complement for $K_{X''}$ for a bounded $a\in \N$ where $X''\to X'$ is a small $\Q$-factorialisation. 
Let $K_{X''}+B''+M''$ be the 
pullback of $K_{X'}+B'+M'$. Then replacing $X',B',M'$ with $X'',B'',M''$,  we can assume 
$X'$ is $\Q$-factorial.\\

\emph{Step 3.}
In this step we consider a suitable lc complement of $K_{X'}$.
Run an MMP on $-K_{X'}$ and let $\tilde{X}'$ be the resulting model which is a klt weak Fano. 
Then $K_{\tilde{X}'}$ has an $n$-complement for some $n$ depending only on $d$ 
by applying Lemma \ref{l-from-compl-to-BAB-exc-usual} if $\tilde{X}'$ is exceptional or by applying
Lemma \ref{l-compl-non-exc} otherwise. This implies $K_{X'}$ 
also has an $n$-complement $K_{X'}+C'$, by \ref{ss-compl-remarks}(3). 

On the other hand, 
since $(X',B'+M')$ is generalised $\epsilon$-lc and $-(K_{X'}+B'+M')$ is semi-ample, 
$(\tilde{X}',\tilde{B}'+\tilde{M}')$ is generalised $\epsilon$-lc where $\tilde{B}',\tilde{M}'$ are the pushdowns of 
$B',M'$. Thus $\tilde{X}'$ is $\epsilon$-lc. 
Thus by Proposition \ref{p-eff-bir-delta-2}, $|-mK_{\tilde{X}'}|$ defines a 
birational map for some $m\in\N$ depending only on $d,\epsilon,n$ which in turn implies 
$|-mK_{X'}|$ also defines a birational map. Replacing 
both $m$ and $n$ by $pmn$, we can assume $m=n$ and that $p$ divides $m,n$. Moreover, 
by Lemma \ref{l-mov-part-lin-system}, replacing 
$\phi\colon X\to X'$ and $C'$ we can assume $C'=\frac{1}{m}A'+\frac{1}{m}R'$ 
where $\phi^*(-mK_{X'})\sim A+R$, $A$ is the movable part of $|\phi^*(-mK_{X'})|$, $|A|$ 
is base point free, and $R$ is the fixed part.\\

\emph{Step 4.} 
In this and the next step we reduce the problem to the situation in which 
 $K_{X'}+B'+M'\sim_\Q 0$ and that $M$ is big. We do this by introducing 
a new generalised pair $({X'},\Delta'+N')$. Let 
$$
\Delta':=\frac{1}{2}B'+\frac{1}{2m}R' ~~~\mbox{and} ~~~N:=\frac{1}{2}M+\frac{1}{2m}A.
$$ 
Then $({X'},\Delta'+N')$ is generalised lc and $-(K_{X'}+\Delta'+N')$ is nef. 
Note that the coefficients of $\Delta'$ belong to a fixed finite set depending only on $\mathfrak{R},m$, 
and that $2pmN$ is b-Cartier.
 
Assume $({X'},\Delta'+N')$ is non-exceptional. Then 
by Lemma \ref{l-compl-non-exc}, $K_{X'}+\Delta'+N'$ has an $l$-complement 
$K_{X'}+{\Delta'}^++N'$ for some $l$ depending only on $d,p,m,\mathfrak{R}$ such that 
$G':={\Delta'}^+-{\Delta'}\ge 0$. Then
$$
lm(K_{X'}+{B'}+2G'+M')\sim lm(K_{X'}+{B'}+2G'+M')+lm(K_{X'}+C')
$$
$$
=lm(2K_{X'}+B'+\frac{1}{m}R'+2G'+M'+\frac{1}{m}A')
$$
$$
=2lm(K_{X'}+{\Delta'}+G'+N')=2lm(K_{X'}+{\Delta'}^++N')\sim 0.
$$
Let $B'^+=B'+2G'$. Since $(X',B'+M')$ is exceptional, $({X'},{B'}^++M')$ is generalised klt. 
Thus 
$$
(X',\frac{1}{2}{B'}^++\frac{1}{2}{\Delta'}^++\frac{1}{2}M'+\frac{1}{2}N')
$$
is generalised klt, hence exceptional because 
$$
K_{X'}+\frac{1}{2}{B'}^++\frac{1}{2}{\Delta'}^++\frac{1}{2}M'+\frac{1}{2}N'\sim_\Q 0.
$$
 Now replace $B'$  with 
$\frac{1}{2}{B'}^++\frac{1}{2}{\Delta'}^+$ and replace $M$ with $\frac{1}{2}M+\frac{1}{2}N$. 
Replacing $p,\mathfrak{R}$ accordingly, we can then assume $K_{X'}+B'+M'\sim_\Q 0$ 
and that $M$ is big.\\ 

\emph{Step 5.} 
In this step we assume $({X'},\Delta'+N')$ is exceptional. 
By Lemma \ref{l-exc-threshold}, there is a rational number $\beta\in(0,1)$ depending only on 
$d,p,m,\mathfrak{R}$ such that $(X',\Delta'+\beta N')$ is exceptional. 
Since $N=\frac{1}{2}M+\frac{1}{2m}A$ and $A$ is base point free and big, there is $r\in\N$ such that 
$$
-r(K_{X'}+\Delta'+\beta N')=-r(K_{X'}+\Delta'+N')+r(1-\beta)N'
$$
is integral and potentially birational where $r$ depends only on $d,p,m,\beta,\mathfrak{R}$. Then 
$$
|K_{X'}-r(K_{X'}+\Delta'+\beta N')|
$$
defines a birational map by [\ref{HMX}, Lemma 2.3.4], hence 
$$
|mK_{X'}-rm(K_{X'}+\Delta'+\beta N')|
$$
also defines a birational map which in turn implies 
$$
|-rm(K_{X'}+\Delta'+\beta N')|=|mK_{X'}+mC'-rm(K_{X'}+\Delta'+\beta N')|
$$
defines a birational map as well. In particular, there is ${\Delta'}^+\ge \Delta'$ such that 
$$
rm(K_{X'}+{\Delta'}^++\beta N')\sim 0.
$$ 
Since $({X'},\Delta'+\beta N')$ is exceptional, 
$({X'},{\Delta'}^++\beta N')$ is generalised klt, hence exceptional.
Now replace $B'$ and $M$ with ${\Delta'}^+$ and $\beta N$, respectively.
Replacing $p,\mathfrak{R}$ accordingly, from now on we can then assume $K_{X'}+B'+M'\sim_\Q 0$ 
and that $M$ is big.\\

\emph{Step 6.}
Finally, we will use Proposition \ref{p-from-bnd-lct-to-bnd-var} to show $K_{X'}$ has a bounded 
klt complement as discussed in Step 2.
Let $\tilde{X}'$ be as in Step 3 which is the result of an MMP on $-K_{X'}$. It is enough to show that 
$K_{\tilde X'}$ has a klt $a$-complement for some bounded number $a\in \N$. 
Since $K_{X'}+B'+M'\sim_\Q 0$, 
we can replace $X'$ with $\tilde{X}'$, hence assume $X'$ is a weak Fano. 
By Step 3, $K_{X'}$ has an $m$-complement and $|-mK_{X'}|$ defines a birational map for some bounded $m\in\N$. 
Moreover, by Lemma \ref{l-exc-bnd-vol}, $\vol(-K_{X'})\le v$ for some number $v$ depending only on $d,p,\mathfrak{R}$.
On the other hand, by Lemma \ref{l-exc-bnd-lct}, for each $l\in\N$ there is a positive real 
number $t_l$ depending only on $d,p,l,\mathfrak{R}$ such that for any $L'\in|-lK_{X'}|$, the pair $(X',t_lL')$ is klt. 
Therefore, $X'$ belongs to a bounded family by Proposition \ref{p-from-bnd-lct-to-bnd-var}, and 
this implies the existence of the required bounded klt complement for $K_{X'}$, by Lemma \ref{l-bnd-couples-bnd-Cartier-index}.

\end{proof}


\section{\bf Boundedness of relative complements}

In this section we treat Theorem \ref{t-bnd-compl-usual-local} inductively. The results of Sections 
6 and 7 rely on this theorem. Proofs are similar to those in Section 6.

\begin{prop}\label{p-bnd-compl-usual-local-plt}
Assume Theorems \ref{t-bnd-compl-usual} and \ref{t-bnd-compl-usual-local} hold in dimension $d-1$. 
Then Theorem \ref{t-bnd-compl-usual-local} holds in dimension $d$ for those $(X,B)$ and $X\to Z$ such that 
\begin{itemize}
\item $B\in \mathfrak{R}$,

\item $({X},\Gamma)$ is $\Q$-factorial plt for some $\Gamma$,

\item $-(K_{X}+\Gamma)$ is ample over $Z$, 

\item $S:=\rddown{\Gamma}$ is irreducible and it is a component of $\rddown{B}$, and

\item $S$ intersects the fibre of $X\to Z$ over $z$.
\end{itemize}
\end{prop}

\begin{proof}
The proof is nearly identical to that of Proposition \ref{p-bnd-compl-plt} but 
 for convenience we will write a complete proof because we need some 
slight adjustments, e.g. Step 1,  and also the notation is different.\\ 

\emph{Step 1}.
In this step we show that the induced morphism $S\to f(S)$ is a 
contraction where $f$ denotes $X\to Z$ and $f(S)$ is the image of $S$ with reduced structure. 
 From the exact sequence 
$$
0\to \mathcal{O}_X(-S) \to \mathcal{O}_X \to \mathcal{O}_S \to 0
$$
we get the exact sequence 
$$
f_*\mathcal{O}_X \to f_*\mathcal{O}_S \to R^1f_*\mathcal{O}_X(-S)=0
$$
where the vanishing follows from the relative Kawamata-Viehweg vanishing theorem 
[\ref{KMM-mmp}, Theorem 1-2-5] as 
$$
-S= K_X+\Gamma-S-(K_X+\Gamma)
$$ 
with $(X,\Gamma-S)$ being klt and $-(K_X+\Gamma)$ ample over $Z$. Thus 
$f_*\mathcal{O}_X\to f_*\mathcal{O}_S$ is surjective. Therefore, if $\pi\colon V\to Z$ denotes the finite part of the 
Stein factorisation of $S\to Z$, then  
$$
\mathcal{O}_Z=f_*\mathcal{O}_X\to  f_*\mathcal{O}_S=\pi_*\mathcal{O}_V
$$ 
is surjective. But  $\mathcal{O}_Z\to\pi_*\mathcal{O}_V$ factors as 
$\mathcal{O}_Z\to \mathcal{O}_{f(S)}\to \pi_*\mathcal{O}_{V}$, hence $\mathcal{O}_{f(S)}\to \pi_*\mathcal{O}_{V}$ 
is surjective which is then an isomorphism as the induced morphism $V\to f(S)$ is finite.
 Therefore,  $V\to f(S)$ is an isomorphism and $S\to f(S)$ is a contraction.\\

\emph{Step 2}.
 In this step we consider adjunction  and  complements on $S$.
Consider a log resolution $\phi\colon X'\to X$  of $(X,B)$, let $S'$ be the birational transform of $S$, and 
$\psi\colon S'\to S$ be the induced morphism. By  adjunction, we can write 
$K_{S}+B_{S}:= (K_{X}+B)|_{S}$.
By Lemma \ref{l-div-adj-dcc}, $B_{S}\in \Phi(\mathfrak{S})$ for some finite set of rational numbers 
$\mathfrak{S}\subset [0,1]$ which only depends on $\mathfrak{R}$. 
Moreover, restricting $K_{X}+\Gamma$ to $S$ shows that $S$ is of Fano type over $f(S)$. In addition, 
$z\in f(S)$ by assumption.

Now applying Theorem \ref{t-bnd-compl-usual-local} in dimension $d-1$ if $\dim f(S)>0$, or applying 
Theorem \ref{t-bnd-compl-usual} in dimension $d-1$ if $\dim f(S)=0$, there is $n\in \N$ 
which depends only on $d-1, \mathfrak{S}$ such that $K_{S}+B_{S}$ has 
an $n$-complement $K_{S}+B_{S}^+$ over $z$, with $B_{S}^+\ge B_{S}$. 
Replacing $n$  with $nI(\mathfrak{R})$ we can assume $n$ is divisible by 
$I(\mathfrak{R})$. In particular, $nB$ is integral as $B\in \mathfrak{R}$.
In the subsequent steps we will lift the complement $K_{S}+B_{S}^+$ to an  
$n$-complement $K_{X}+{B}^+$ of $K_{X}+B$ over $z$, with ${B}^+\ge B$.\\

\emph{Step 3}.
In this step we introduce basic notation. Write 
$$
N':=-(K_{X'}+B'):=-\phi^*(K_{X}+B)
$$
and let $T'=\rddown{B'^{\ge 0}}$ and $\Delta'=B'-T'$. Define  
$$
L':=-nK_{X'}-nT'-\rddown{(n+1){\Delta'}}
$$
which is an integral divisor. Note that 
$$
L'=n\Delta'-\rddown{(n+1){\Delta'}}+nN'.
$$
Now write 
$$
K_{X'}+\Gamma':=\phi^*(K_{X}+\Gamma).
$$ 
Replacing $\Gamma$ with $(1-a)\Gamma+aB$ for some $a\in(0,1)$ sufficiently 
close to $1$, we can assume $B'-\Gamma'$ has sufficiently small 
(positive or negative) coefficients.\\ 

\emph{Step 4}.
In this step we define a divisor $P'$ and study its properties.
Let $P'$ be the unique integral divisor so that 
$$
\Lambda':=\Gamma'+{n{\Delta'}}-\rddown{(n+1){\Delta'}}+P'
$$ 
is a boundary, $(X',\Lambda')$ is plt, and $\rddown{\Lambda'}=S'$ (in particular, we are assuming $\Lambda'\ge 0$). 
More precisely, we let $\mu_{S'}P'=0$ and for each prime divisor $D'\neq S'$, we let 
$$
\mu_{D'}P'=-\mu_{D'}\rddown{\Gamma'+{n{\Delta'}}-\rddown{(n+1){\Delta'}}}
$$
which satisfies 
$$
\mu_{D'}P'=-\mu_{D'}\rddown{\Gamma'-\Delta'+\langle (n+1)\Delta'\rangle}
$$ 
where $\langle (n+1)\Delta'\rangle$ is the fractional part of $(n+1)\Delta'$.
This implies $0\le \mu_{D'}P'\le 1$ for any prime divisor $D'$: indeed we can assume $D'\neq S'$; 
if $D'$ is a component of $T'$, then 
$D'$ is not a component of $\Delta'$ but $\mu_{D'}\Gamma'\in(0,1)$, hence $\mu_{D'}P'=0$; if $D'$ is not a 
component of $T'$, then $\mu_{D'}(\Gamma'-\Delta')=\mu_{D'}(\Gamma'-B')$ is sufficiently small, hence 
$0\le \mu_{D'}P'\le 1$. 

We show $P'$ is exceptional$/X$. 
Assume $D'$ is a component of $P'$ which is not exceptional$/X$. Then $D'\neq S'$, and since $nB$ is integral,  
$\mu_{D'}n\Delta'$ is integral, hence  
$\mu_{D'}\rddown{(n+1)\Delta'}=\mu_{D'}n\Delta'$ which implies 
$\mu_{D'}P'=-\mu_{D'}\rddown{\Gamma'}=0$, a contradiction.\\

\emph{Step 5.}
In this step we use Kawamata-Viehweg vanishing to lift sections from $S'$ to $X'$.
Let $A:=-(K_{X}+\Gamma)$ and let $A'=\phi^*A$.
Then 
$$
L'+P'= {n{\Delta'}}-\rddown{(n+1){\Delta'}}+ nN'+P'
$$
$$
=K_{X'}+\Gamma'+A'+{n{\Delta'}}-\rddown{(n+1){\Delta'}}+ nN'+P'
$$
$$
=K_{X'}+\Lambda'+A'+nN'.  
$$
Shrinking $Z$ around $z$ we can assume $Z$ is affine. Now since $A'+nN'$ is nef and big over $Z$ 
and $(X',\Lambda'-S')$ is klt, we get 
$h^1(L'+P'-S')=0$ by the relative Kawamata-Viehweg vanishing theorem [\ref{KMM-mmp}, Theorem 1-2-5], hence  
$$
H^0(L'+P')\to H^0((L'+P')|_{S'})
$$
is surjective.\\

\emph{Step 6.}
In this step we define several divisors.
Let $R_{S}:=B_{S}^+-B_{S}$. Then, perhaps after shrinking $Z$ around $z$, we have 
$$
-n(K_{S}+B_{S})=-n(K_{S}+B^+_{S}+B_S-B^+_S)\sim -n(B_S-B^+_S)=nR_{S}\ge 0.
$$
Letting $R_{S'}$ be the pullback of $R_{S}$ we get 
$$
-n(K_{S'}+B_{S'}):=-n\psi^*(K_{S}+B_{S})\sim nR_{S'}\ge 0.
$$
Then 
$$
nN'|_{S'}=-n(K_{X'}+B')|_{S'}=-n(K_{S'}+B_{S'})\sim nR_{S'}.
$$
By construction, 
$$
(L'+P')|_{S'}=({n{\Delta'}}-\rddown{(n+1){\Delta'}}+ nN'+P')|_{S'}
$$
$$
\sim G_{S'}:=nR_{S'}+{n{\Delta_{S'}}}-\rddown{(n+1){\Delta_{S'}}}+P_{S'}
$$
where $\Delta_{S'}=\Delta'|_{S'}$ and $P_{S'}=P'|_{S'}$.\\

\emph{Step 7.}
In this step we show $G_{S'}\ge 0$ and that it lifts to some effective divisor $G'$ on $X'$.
Assume $C'$ is a component of $G_{S'}$ with negative coefficient. Then 
 there is a component $D'$ of ${n{\Delta'}}-\rddown{(n+1){\Delta'}}$ with negative coefficient such that $C'$ is a 
component of $D'|_{S'}$. But 
$$
\mu_{C'} ({n{\Delta_{S'}}}-\rddown{(n+1){\Delta_{S'}}})=\mu_{C'} (-\Delta_{S'}+\langle(n+1)\Delta_{S'}\rangle)\ge 
-\mu_{C'} \Delta_{S'}=-\mu_{D'} \Delta'>-1
$$ 
which gives $\mu_{C'}G_{S'}>-1$ and this in turn implies $\mu_{C'}G_{S'}\ge 0$ because $G_{S'}$ is integral, a contradiction. 
Therefore $G_{S'}\ge 0$, and by Step 5, $L'+P'\sim G'$ for some effective divisor $G'$ whose support does not contain $S'$ 
and $G'|_{S'}=G_{S'}$.\\ 

\emph{Step 8.}
In this step we introduce $B^+$.
By the previous step and the definition of $L'$ and the fact that $P'$ is exceptional$/X$, we have    
$$
-nK_{X}-nT-\rddown{(n+1)\Delta}=L=L+P\sim G\ge 0
$$
where $L$, etc, are the pushdowns of $L'$, etc.
Since $nB$ is integral, $\rddown{(n+1)\Delta}= n\Delta$, so  
$$
-n(K_{X}+B)=-nK_{X}-nT-{n\Delta}
=L\sim nR:=G\ge 0.
$$
Let ${B}^+=B+R$. Then $n(K_{X}+{B}^+)\sim 0$.\\ 

\emph{Step 9.}
In this step we show that $(X,{B}^+)$ is lc over some neighbourhood of $z$ 
which implies that $K_X+B^+$ is an $n$-complement 
of $K_X+B$ with $B^+\ge B$.
First we show $R|_{S}=R_{S}$. Since  
$$
nR':=G'-P'+\rddown{(n+1)\Delta'}- n\Delta'\sim L'+\rddown{(n+1)\Delta'}- n\Delta'=nN'\sim_\Q 0/X
$$
and since $\rddown{(n+1)\Delta}- n\Delta=0$ as  
$n\Delta$ is integral, we get  $\phi_*nR'=G=nR$ and that $R'$ is the pullback of $R$. 
Now 
$$
nR_{S'}=G_{S'}-P_{S'}+\rddown{(n+1)\Delta_{S'}}- n\Delta_{S'}
$$
$$
=(G'-P'+\rddown{(n+1)\Delta'}- n\Delta')|_{S'}=nR'|_{S'}
$$
which means $R_{S'}=R'|_{S'}$, hence $R_{S}=R|_{S}$.

The previous line implies 
$$
K_{S}+B_{S}^+=K_S+B_S+R_S=(K_X+B+R)|_S=(K_{X}+{B}^+)|_{S}.
$$
 By inversion of adjunction (\ref{l-inv-adjunction} or the usual version in [\ref{kawakita}]), 
 $(X,{B}^+)$ is lc near $S$. Let 
$$
\Omega:=a{B}^++(1-a)\Gamma 
$$ 
for some $a\in(0,1)$ close to $1$.
If $(X,{B}^+)$ is not  lc near the fibre over $z$, 
then $(X,\Omega)$ is also not  lc  near the fibre over $z$. Note that $(X,\Omega)$ is lc 
near $S$. But then 
$$
-(K_{X}+\Omega)=-a(K_X+B^+)-(1-a)(K_X+\Gamma)
$$
is ample over $Z$ and  the non-klt locus of $(X,\Omega)$ near the fibre over $z$ has at least two disjoint components 
one of which is $S$. This contradicts the 
connectedness principle (\ref{l-connectedness} or the usual version [\ref{Kollar-flip-abundance}, Theorem 17.4]). 
Therefore,  $(X,{B}^+)$ is  lc over some neighbourhood of $z$.

\end{proof}

\begin{prop}\label{p-bnd-compl-usual-local}
Assume Theorems \ref{t-bnd-compl-usual} and \ref{t-bnd-compl-usual-local} hold in dimension $d-1$. 
Then Theorem \ref{t-bnd-compl-usual-local} holds in dimension $d$.
\end{prop}
\begin{proof}
When $(X,B)$ is klt and $-(K_X+B)$ is nef and big over $Z$, the theorem is essentially [\ref{PSh-I}, Theorem 3.1]. 
Apart from Steps 1 and 2,  the proof below is similar to that of Proposition \ref{p-bnd-compl-non-klt}.\\

\emph{Step 1.}
In this step we reduce to the situation in which $\rddown{B}$ has a vertical component intersecting the fibre over $z$.
Pick an effective Cartier divisor 
$N$ on $Z$ passing through $z$. Let $t$ be the lc threshold of $f^*N$ with 
respect to $(X,B)$ over $z$ where $f$ denotes $X\to Z$. Let $\Omega=B+tf^*N$. 
Shrinking $Z$ we can assume $(X,\Omega)$ is lc everywhere.
Let $(X',\Omega')$ be a $\Q$-factorial 
dlt model of $(X,\Omega)$.  Then $X'$ is of Fano type over $Z$. Moreover, there is a boundary $\Delta'\le \Omega'$ 
such that $\Delta'\in \Phi(\mathfrak{R})$, 
 some component of $\rddown{\Delta'}$ is vertical over $Z$ intersecting the fibre over $z$, 
 and $B\le \Delta$ where $\Delta$ is the pushdown of $\Delta'$. 

Run an MMP$/Z$ on $-(K_{X'}+\Delta')$ and let $X''$ be the resulting model. 
Since 
$$
-(K_{X'}+\Delta')=-(K_{X'}+\Omega')+(\Omega'-\Delta')
$$ 
where $-(K_{X'}+\Omega')$ is nef$/Z$ and $\Omega'-\Delta'\ge 0$, the 
MMP ends with a minimal model, that is, $-(K_{X''}+\Delta'')$ is nef over $Z$. 
Moreover, if $K_{X''}+\Delta''$ has an $n$-complement $K_{X''}+{\Delta''}^+$ over $z$ with ${\Delta''}^+\ge {\Delta''}$, then  
$K_{X'}+\Delta'$ has an $n$-complement $K_{X'}+{\Delta'}^+$ over $z$ with ${\Delta'}^+\ge {\Delta'}$ which in turn implies 
$K_X+B$ also has an $n$-complement $K_X+B^+$ over $z$ with $B^+\ge B$. Since $-(K_{X'}+\Omega')$ is semi-ample 
over $Z$, $(X'',\Omega'')$ is lc, hence $(X'',\Delta'')$ is lc. In particular, no component of  $\rddown{\Delta'}$ 
is contracted by the MMP otherwise $a(S',X'',\Delta'')<0$ for any contracted component $S'$ of  $\rddown{\Delta'}$ 
contradicting the previous sentence.
Replacing $(X,B)$ with 
$({X''},\Delta'')$ we can assume $\rddown{B}$ has a component intersecting the fibre over $z$ and that 
$X$ is $\Q$-factorial.\\

\emph{Step 2.}
In this step we reduce to the case $B\in\mathfrak{R}$.
Let $\epsilon>0$ be a sufficiently small number. Let $\Theta$ be the boundary whose coefficients are the same as  
$B$ except that we replace each coefficient in $(1-\epsilon,1)$ with $1$. 
Run an MMP$/Z$ on $-(K_X+\Theta)$ and let $X'$ be the resulting model. By Proposition \ref{p-B'-to-Theta'},  
we can choose $\epsilon$ depending only on $d,\mathfrak{R}$ so that no component of $\rddown{\Theta}$ is 
contracted by the MMP, $(X',\Theta')$ is lc,  and that 
$-(K_{X'}+\Theta')$ is nef over $Z$. Moreover, the coefficients of $\Theta'$ belong to some fixed finite set depending only 
on $\mathfrak{R},\epsilon$ because $1$ is the only accumulation point of $\Phi(\mathfrak{R})$. 

If $K_{X'}+\Theta'$ has an $n$-complement $K_{X'}+{\Theta'}^+$ over $z$ with 
${\Theta'}^+\ge {\Theta'}$, then  $K_{X}+\Theta$ has an $n$-complement $K_{X}+{\Theta}^+$ over $z$ with 
${\Theta}^+\ge {\Theta}$ which in turn implies 
$K_X+B$ also has an $n$-complement $K_X+B^+$ over $z$ with $B^+\ge B$.
Replacing $(X,B)$ with $(X',\Theta')$ and extending $\mathfrak{R}$, from now on  
we can assume $B\in\mathfrak{R}$. In the following steps we try to mimic the arguments of the proof 
of Proposition \ref{p-bnd-compl-non-klt}.\\

\emph{Step 3.}
In this step we define a boundary $\Delta$.
Since $X$ is of Fano type over $Z$, $-K_X$ is big over $Z$. So since $-(K_X+B)$ is nef over $Z$, 
$$
-(K_X+\alpha B)=-\alpha(K_X+B)-(1-\alpha)K_X 
$$
is big over $Z$ for any $\alpha\in (0,1)$. We will assume $\alpha$ is sufficiently close to $1$. 
Define a boundary $\Delta$ as follows. Let $D$ be a prime divisor. If $D$ is vertical over $Z$, 
let $\mu_D\Delta=\mu_DB$ but if $D$ is horizontal over $Z$, let $\mu_D\Delta=\mu_D\alpha {B}$. 
Then $(X,\Delta)$ is lc, 
$\alpha {B}\le \Delta\le B$, $\rddown{\Delta}$ has a vertical component intersecting the fibre over $z$, 
and $-(K_X+\Delta)$ is big over $Z$ as $\Delta=\alpha {B}$ near the generic fibre.\\
  
\emph{Step 4.}
In this step we introduce a boundary $\tilde{\Delta}\le \Delta$ and reduce to the case when 
  $-(K_{X}+\Delta)$ and $-(K_{X}+\tilde{\Delta})$  are nef and big over $Z$, 
 some component of $\rddown{\Delta}$ 
 intersects the fibre over $z$, and that $(X,\tilde\Delta)$ is klt.  
Let $X\to V/Z$ be the contraction defined by $-(K_X+B)$.  
Run an MMP on $-(K_{X}+\Delta)$ over $V$ and let $X'$ be the resulting model. 
Then $-(K_{X'}+\Delta')$ is nef and big over $V$ but may not be nef over $Z$. However, after 
replacing $\Delta$ with $aB+(1-a)\Delta$ for some $a\in(0,1)$ sufficiently close to $1$   
(i.e. increasing $\alpha$ to get closer to $1$) we can assume $-(K_{X'}+\Delta')$ 
is nef and big over $Z$. The MMP does not contract any component of $\rddown{\Delta}$. Now 
replace $(X,B)$ with $(X',B')$ and replace $\Delta$ with $\Delta'$ so that we can assume $-(K_{X}+\Delta)$ 
is nef and big over $Z$. Let $X\to T/Z$ be the contraction defined by $-(K_X+{\Delta})$. 

Let $\tilde{\Delta}=\beta \Delta$ for some $\beta<1$. After running an MMP on 
$-(K_X+\tilde{\Delta})$ over $T$ 
we can assume  $-(K_X+\tilde{\Delta})$ is nef and big over $T$, hence also nef and big over 
$Z$ if we replace $\beta$ with a number sufficiently close to $1$. 
The MMP may contract the components of $\rddown{\Delta}$ 
but after replacing $(X,B)$ with a suitable  $\Q$-factorial dlt 
model, increasing $\alpha,\beta$, and replacing $K_X+{\Delta}$ and $K_X+\tilde{\Delta}$ with their 
pullbacks we can assume $(X,B)$ is $\Q$-factorial dlt and that there are boundaries 
$\tilde{\Delta}\le \Delta \le B$ so that $-(K_{X}+\Delta)$ and $-(K_{X}+\tilde{\Delta})$  are nef and big over $Z$, 
 some component of $\rddown{\Delta}$ 
 intersects the fibre over $z$, and that $(X,\tilde\Delta)$ is klt. Shrinking $Z$ around $z$ we can 
 assume every component of $\rddown{\Delta}$ intersects the fibre over $z$.\\

\emph{Step 5.}
In this step we introduce divisors $A,G$ and yet another boundary $\Gamma$.
We can write $-(K_X+\Delta)\sim_\R A+G/Z$ where $A\ge 0$ is ample and $G\ge 0$. 
Assume $\Supp G$ does not contain any non-klt centre of $(X,\Delta)$. 
Then $(X,\Delta+\delta G)$ is dlt for any sufficiently small $\delta>0$. Moreover,
$$
-(K_X+\Delta+\delta G)\sim_\R (1-\delta)\left(\frac{\delta}{1-\delta} A+A+G\right)/Z
$$
is ample over $Z$, hence by perturbing the coefficients of $\Delta+\delta G$ we can find a boundary 
$\Gamma$ such that $(X,\Gamma)$ is plt, $S:=\rddown{\Gamma}\subseteq \rddown{B}$ 
is irreducible intersecting the fibre over $z$, and $-(K_X+\Gamma)$ is ample over $Z$.
So we can apply Proposition \ref{p-bnd-compl-usual-local-plt}. 
From now on we assume $\Supp G$ contains some non-klt centre of $(X,\Delta)$.\\

\emph{Step 6.}
In this step we define another boundary $\Omega$.
Let $t$ be the lc threshold of $G+\Delta-\tilde{\Delta}$ with respect to $(X,\tilde{\Delta})$ over $z$.  
Replacing $\tilde{\Delta}$ we can assume $\Delta-\tilde{\Delta}$ is sufficiently small, hence $t$ is sufficiently small too. 
Then letting $\Omega=\tilde{\Delta}+t(G+\Delta-\tilde{\Delta})$, 
 any non-klt place of $(X,\Omega)$ is a non-klt place of $(X,\Delta)$ (this can be seen on a 
 log resolution of $(X,B+\Omega)$). By construction, over $Z$ we have
$$
-(K_X+\Omega)=-(K_X+\tilde{\Delta}+t(G+\Delta-\tilde{\Delta}))
$$
$$
=-(K_X+\Delta)+\Delta-\tilde{\Delta}-t(G+\Delta-\tilde{\Delta})
$$
$$
\sim_\R A+G-tG+(1-t)(\Delta-\tilde\Delta)
$$
$$
=(1-t)\left(\frac{t}{1-t}A+A+G+\Delta-\tilde\Delta\right)
$$
which implies $-(K_X+\Omega)$ is ample over $Z$ because 
$$
-(K_X+\tilde\Delta)=-(K_X+\Delta)+\Delta-\tilde\Delta\sim_\R A+G+\Delta-\tilde\Delta/Z
$$
is nef and big over $Z$.\\

\emph{Step 7.}
In this step we finish the proof of the proposition using  \ref{p-bnd-compl-usual-local-plt}.
If $\rddown{\Omega}\neq 0$, then there is a component $S$ of $\rddown{\Omega}\le \rddown{\Delta}\le \rddown{B}$ 
and there is a 
boundary $\Gamma$ so that $(X,\Gamma)$ is plt, $S=\rddown{\Gamma}$ 
intersects the fibre over $z$, and $-(K_X+\Gamma)$ is ample over $Z$. 
 So we can apply Proposition \ref{p-bnd-compl-usual-local-plt}.

Now assume $\rddown{\Omega}=0$. Let $(X',\Omega')$ be a $\Q$-factorial dlt model of $(X,\Omega)$.
Shrinking $Z$ we can assume every component of $\rddown{\Omega'}$ intersects the fibre over $z$.
Running an MMP on $K_{X'}+\rddown{\Omega'}$ over $X$ ends with $X$ because $\rddown{\Omega'}$ 
is the reduced exceptional divisor of $X'\to X$ and because $X$ is $\Q$-factorial klt. The last step of the 
MMP is a divisorial contraction $X''\to X$ contracting one prime divisor $S''$, and 
$(X'',S'')$ is plt and $-(K_{X''}+S'')$ is ample over $X$. Moreover, if we denote the pullbacks of 
$K_X+\Omega$ and $K_X+\Delta$ to $X''$ by $K_{X''}+\Omega''$ and 
$K_{X''}+\Delta''$, respectively, then 
$S''$ is a component of both $\rddown{\Omega''}$ and $\rddown{\Delta''}$. Now since $-(K_X+\Omega)$
is ample over $Z$ and $-(K_{X''}+S'')$ is ample over $X$, we can find a 
boundary $\Gamma''$ so that $(X'',\Gamma'')$ is plt, $S''=\rddown{\Gamma''}$
intersects the fibre over $z$, and $-(K_{X''}+\Gamma'')$ is ample 
over $Z$. In addition, if $K_{X'}+B''$ is the pullback of $K_X+B$, then $S''$ is a component of $\rddown{B''}$ 
since $\Delta''\le B''$. 
Now apply Proposition \ref{p-bnd-compl-usual-local-plt} to get an $n$-complement $K_{X''}+B''^+$ 
of $K_{X''}+B''$ over $z$ with $B''^+\ge B''$ for some bounded $n\in\N$. This gives 
an $n$-complement $K_{X}+B^+$ 
of $K_{X}+B$ over $z$ with $B^+\ge B$.

\end{proof}


\section{\bf Anti-canonical volume}

In this section we prove Theorem \ref{t-BAB-to-bnd-vol} which claims that the anti-canonical volumes 
of $\epsilon$-lc Fano varieties of a given dimension are bounded. 
Recall that we treated this boundedness for exceptional Fano varieties in \ref{l-from-compl-to-BAB-exc-usual}. 
To deal with the non-exceptional case we need Conjecture \ref{conj-BAB} in lower dimension.
We will also use Theorem \ref{t-eff-bir-e-lc}. Although \ref{t-eff-bir-e-lc} will be proved later in the final 
section but this is not a problem 
because no result of this paper relies on Theorem \ref{t-BAB-to-bnd-vol}. However, \ref{t-BAB-to-bnd-vol} 
is an important ingredient of the proof of BAB in [\ref{B-BAB}].

\begin{proof}(of Theorem \ref{t-BAB-to-bnd-vol})
We give a short summary of the proof. It is easy to derive the second claim of the theorem (birational boundedness) 
to the first claim which is the existence of $v$. Now if 
$\vol(-K_X)$ is too large, we find $0\le B\sim_\Q -aK_X$ with $a>0$ too small but with $\vol(B)>(2d)^d$ 
and $(X,B)$ exactly $\epsilon'$-lc for some $\epsilon'\in(0,\epsilon)$. We extract a prime divisor 
$D'$ with  $a(D',X,B)=\epsilon'$, say via $X'\to X$, run MMP on $-D'$ giving a Mori fibre space $X''\to Z$ 
and $K_{X''}+sD''\equiv 0/Z$ where $s$ is too large. The case $\dim Z>0$ is settled by induction 
and restriction to the fibres of $X''\to Z$. To treat the case $\dim Z=0$ we 
create a covering family of non-klt centres (similar to the proof of \ref{p-bnd-vol-good-boundary}) 
and use adjunction on these centres and BAB (\ref{conj-BAB}) in lower dimension  to get a contradiction.\\

\emph{Step 1.}
In this step we reduce the theorem to existence of $v$ and introduce basic notation.
The birational boundedness claim follows from existence of $v$ and Theorem \ref{t-eff-bir-e-lc}: 
indeed then there is $m\in\N$ depending only on $d,\epsilon$ such that $|-mK_X|$ defines a birational map,
 and $\vol(-mK_X)$ is bounded from above, so we can apply Proposition \ref{p-log-bir-bnd-cert-pairs} 
 by taking some $0\le M\sim -mK_X$. Moreover, we can assume $X$ is Fano by taking the contraction 
 $X\to \tilde{X}$ defined by $-K_X$ and by replacing $X$ with $\tilde{X}$. 
 
If there is no $v$ as in the statement, then there is a sequence $X_i$ of $\epsilon$-lc Fano varieties of 
dimension $d$ such that $\vol(-K_{X_i})$ is an increasing sequence approaching $\infty$.
We will derive a contradiction. 
Fix $\epsilon'\in (0,\epsilon)$. Then there exist a decreasing sequence of rational numbers $a_i$ 
approaching $0$, and $\Q$-boundaries 
$B_i\sim_\Q -{a_i}K_{X_i}$ such that $(2d)^d<\vol(B_i)$ and 
$(X_i,B_i)$ is $\epsilon'$-lc but not $\epsilon''$-lc for any $\epsilon''>\epsilon'$. We can assume $a_i<1$, hence
$-(K_{X_i}+B_i)\sim_\Q -(1-a_i)K_{X_i}$ is ample. 

For each $i$, there is a prime divisor $D_i'$ on birational models of 
$X_i$ such that $a(D_i',X_i,B_i)=\epsilon'$. If $D_i'$ is a divisor on $X_i$, then we let 
$\phi_i\colon X_i'\to X_i$ to be a small $\Q$-factorialisation, otherwise we let it be a 
birational contraction which extracts only $D_i'$ with $X_i'$ being $\Q$-factorial [\ref{BCHM}, Corollary 1.4.3]. 
Let 
$$
K_{X_i'}+e_iD_i'=\phi_i^*K_{X_i} ~~~\mbox{and}~~~ K_{X_i'}+B_{i}'=\phi_i^*(K_{X_i}+B_i).
$$
Then $e_i\le 1-\epsilon$ but $\mu_{D_i}B_{i}'=1-\epsilon'$.   
Therefore, the coefficient of $D_i'$ in 
$$
P_{i}':=\phi_i^*B_i=B_{i}'-e_iD_i'
$$ 
is at least $\epsilon-\epsilon'$.\\
 
\emph{Step 2.}
In this step we obtain Mori fibre spaces $X_i''\to Z_i$ and numbers $s_i$.
Let $H_i$ be a general ample $\Q$-divisor so that $K_{X_i}+B_i+H_i\sim_\Q 0$ and 
$({X_i},B_i+H_i)$ is $\epsilon'$-lc, and let 
$H_{i}'$ be its pullback to $X_i'$.
Run an MMP on $-D_i'$ which ends with a $-D_i'$-Mori fibre space, that is, an extremal 
non-birational contraction $X_i''\to Z_i$ where $D_i''$, the pushdown of $D_i'$, is ample over $Z_i$. 
Letting $b_i=\frac{1}{a_i}-1$ we get 
$$
b_iB_i=\frac{1}{a_i}B_i-B_i \sim_\Q -K_{X_i}-B_i\sim_\Q H_i
$$ 
and the $b_i$ form an increasing sequence 
approaching $\infty$. In particular, $K_{X_i}+B_i+b_iB_i\sim_\Q 0$. Thus 
$$
K_{X_i'}+B_i'+b_iP_i'=K_{X_i'}+B_i'+b_i\phi_i^*B_i\sim_\Q K_{X_i'}+B_i'+H_i'\sim_\Q 0
$$
which gives 
$$
K_{X_i''}+B_i''+b_iP_i''\sim_\Q K_{X_i''}+B_i''+H_i''\sim_\Q 0
$$
 where $B_i''$ denotes the pushdown of $B_i'$, etc. Moreover, $\mu_{D_i''}b_iP_i''\ge b_i(\epsilon-\epsilon')$.
So there is a number $s_i\ge b_i(\epsilon-\epsilon')$ so that $K_{X_i''}+s_iD_i''\sim_\Q 0/Z_i$.
In particular, $\lim s_i=\infty$.\\

\emph{Step 3.}
In this step we treat the case $\dim Z_i>0$ and modify the setting when $\dim Z_i=0$.
Assume $\dim Z_i>0$ for every $i$ and let $V_i$ be a general fibre of $X_i''\to Z_i$. 
By the previous step, $(X_i'',B_i''+H_i'')$ is $\epsilon'$-lc, hence $X_i''$ is $\epsilon'$-lc 
which implies $V_i$ is an $\epsilon'$-lc Fano variety. 
Since we are assuming Conjecture \ref{conj-BAB} in dimension $\le d-1$, 
$V_i$ belongs to a bounded family. Restricting $K_{X_i''}+s_iD_i''$ to $V_i$ we get 
$K_{V_i}+s_iD_{V_i}\sim_\Q 0$ where $D_{V_i}=D_i''|_{V_i}$. 
This contradicts Lemma \ref{l-bnd-fam-intersection}.
From now on we can assume $\dim Z_i=0$ for every $i$.

By construction, 
$$
\vol(-K_{X_i''})\ge \vol(b_iP_i'')\ge \vol(b_iP_{i}')=\vol(b_iB_i)>(2b_id)^d.
$$
Replacing $\epsilon$ with $\epsilon'$ and replacing $X_i$ with $X_i''$ 
we can assume there is a prime divisor $D_i$ on $X_i$ such that $K_{X_i}+s_iD_i\sim_\Q 0$ 
and that the $s_i$ form an increasing sequence approaching $\infty$.\\ 

\emph{Step 4.}    
In this step we fix $i$, create a family of non-klt centres on $X_i$, and consider adjunction. 
By \ref{ss-non-klt-centres}(2), there is a  covering family of subvarieties of $X_i$  
such that for any pair of general closed points $x_i,y_i\in X_i$ there exist a member $G_i$ of the family and a $\Q$-divisor 
$0\le \Delta_i\sim_\Q -a_iK_{X_i}$  so that $(X_i,\Delta_i)$ is lc at $x_i$ with a unique non-klt place whose centre 
contains $x_i$, that centre is $G_i$, and $(X_i,\Delta_i)$ is not klt at $y_i$. 
As $-(K_{X_i}+\Delta_i)$ is ample, 
$\dim G_i\neq 0$ by the connectedness principle. 

Let $F_i$ be the normalisation of $G_i$.
 By Construction \ref{constr-adjunction-non-klt-centre} and Theorem \ref{t-subadjunction} 
(by taking $B=0$ and $\Delta=\Delta_i$) and the 
 ACC for lc thresholds [\ref{HMX2}, Theorem 1.1], 
 there is a $\Q$-boundary $\Theta_{F_i}$ with  
coefficients in a fixed DCC set $\Phi$ depending only on $d$ such that we can write 
$$
(K_{X_i}+\Delta_i)|_{F_i}\sim_\Q K_{F_i}+\Delta_{F_i}:=K_{F_i}+\Theta_{F_i}+P_{F_i}
$$
where $P_{F_i}$ is pseudo-effective. Increasing $a_i$ and adding to $\Delta_i$ we can assume 
$P_{F_i}$ is big and effective.

Let $D_{F_i}:=D_i|_{F_i}$. Since $G_i$ is general, it is not contained in $\Supp D_i$. 
By Lemma \ref{l-subadjunction-integral-div},  
 each component of $D_{F_i}$ has coefficient at least $1$ in $\Theta_{F_i}+D_{F_i}$. 
Replacing $\Delta_i$ with $\Delta_i+D_i$ and replacing $P_{F_i}$ with $P_{F_i}+D_{F_i}$, 
we can assume  each component of 
$D_{F_i}$ has coefficient at least $1$ in $\Delta_{F_i}$. 
Note that we also need to replace $a_i$ with $a_i+\frac{1}{s_i}$ which we still can 
assume to form a decreasing sequence approaching $0$.\\

\emph{Step 5.}
Let $F_i'\to F_i$ be a log resolution of $(F_i,\Delta_{F_i})$. In this step we define a boundary $\Sigma_{F_i'}$.
Pick a rational number $\epsilon'\in (0,{\epsilon})$. By construction, $(F_i,\Delta_{F_i})$ is not $\epsilon'$-lc. 
Define a boundary $\Sigma_{F_i'}$ on $F_i'$ as follows. 
Let $S_i$ be a prime divisor and let $w_i$ be its coefficient in 
$\Delta_{{F_i'}}$ where $K_{F_i'}+\Delta_{{F_i'}}$ is the pullback of $K_{F_i}+\Delta_{{F_i}}$. 
If $w_i\le 0$, then let the coefficient of $S_i$ in 
$\Sigma_{{F_i'}}$ be zero. But if $w_i>0$, then let the coefficient of $S_i$ in 
$\Sigma_{{F_i'}}$ be the minimum of $w_i$ and $1-\epsilon'$. Then we can write 
$$
\Sigma_{{F_i'}}=\Delta_{F_i'}+E_{F_i'}-N_{F_i'}
$$ 
where $E_{F_i'},N_{F_i'}$ are effective with no common components, $E_{F_i'}$ is exceptional$/F_i$, 
each component of $N_{F_i'}$ has coefficient $> 1-\epsilon'$ in $\Delta_{F_i'}$, and $N_{F_i'}\neq 0$.
Note that $(F_i',\Sigma_{F_i'})$ is $\epsilon'$-lc.\\

\emph{Step 6.}
In this step we consider a birational model $F_i''$ from which we obtain a Mori fibre space 
$\tilde{F}_i\to T_i$.
Let $(F_i'',\Sigma_{{F_i''}})$ be a log minimal model of $(F_i',\Sigma_{F_i'})$ over $F_i$. 
By construction, 
$$
K_{F_i''}+\Sigma_{{F_i''}}=K_{F_i''}+\Delta_{F_i''}+E_{F_i''}-N_{F_i''}\sim_\Q E_{F_i''}-N_{F_i''}/F_i.
$$
So by the negativity lemma, 
$E_{F_i''}=0$, hence $\Delta_{F_i''}=\Sigma_{F_i''}+N_{F_i''}\ge 0$. Moreover, $N_{F_i''}\neq 0$ because 
the birational transform of each component of $D_{F_i}$ is a component of $N_{F_i''}$. 

Since $-(K_{X_i}+\Delta_i)$ is ample, $-(K_{F_i}+\Delta_{F_i})$ is ample, hence $-(K_{F_i''}+\Delta_{F_i''})$ 
is semi-ample. Let $0\le L_{F_i''}\sim_\Q -(K_{F_i''}+\Delta_{F_i''})$ be general with 
coefficients $\le 1-\epsilon'$. Then $({F_i''},\Sigma_{F_i''}+L_{F_i''})$ is $\epsilon'$-lc as 
$({F_i''},\Sigma_{F_i''})$ is $\epsilon'$-lc. Moreover, since $E_{F_i''}= 0$, we have
$$
K_{{F_i''}}+\Sigma_{{F_i''}}+L_{F_i''}+N_{{F_i''}}
=K_{F_i''}+\Delta_{F_i''}-N_{F_i''}+L_{F_i''}+N_{{F_i''}} \sim_\Q  0.
$$   
Thus running an MMP on $K_{F_i''}+\Sigma_{F_i''}+L_{F_i''}$ ends with a Mori fibre space $\tilde{F}_i\to T_i$. 
As we are assuming Conjecture \ref{conj-BAB} in dimension $\le d-1$, the 
general fibres of $\tilde{F}_i\to T_i$ are bounded because $K_{\tilde{F_i}}$ is 
$\epsilon'$-lc and anti-ample over $T_i$.\\

\emph{Step 7.}
In this step we define $\Lambda_{F_i}$ and study $\Delta_{F_i}-\Lambda_{F_i}$.
 By Lemma \ref{l-sub-bnd-on-gen-subvar}, we can write $K_{F_i}+\Lambda_{F_i}=K_{X_i}|_{F_i}$ 
where $({F_i},\Lambda_{F_i})$ is sub-$\epsilon$-lc and 
$$
\Delta_{F_i}-\Lambda_{F_i}\ge \Theta_{F_i}-\Lambda_{F_i}\ge 0.
$$ 
Moreover, 
$$
\Delta_{F_i}-\Lambda_{F_i}=K_{F_i}+\Delta_{F_i}-K_{F_i}-\Lambda_{F_i}\sim_\Q 
(K_{X_i}+\Delta_i)|_{F_i}-K_{X_i}|_{F_i}=\Delta_i|_{F_i}.
$$
Let 
$K_{F_i''}+\Lambda_{F_i''}$ be the pullback of $K_{F_i}+\Lambda_{F_i}$. Then  
$\Delta_{F_i''}-\Lambda_{F_i''}$ is the pullback of $\Delta_{F_i}-\Lambda_{F_i}$, hence 
$$
\Delta_{F_i''}-\Lambda_{F_i''}\sim_\Q \Delta_i|_{F_i''}\sim_\Q a_i{s_i} D_{F_i''}
$$
where $D_{F_i''}:=D_i|_{F_i''}$. 

On the other hand, by Step 6, 
$N_{\tilde{F}_i}$, the pushdown of $N_{F_i''}$, is ample over $T_i$. Let $\tilde{C}_i$ be one of its components that is  
ample over $T_i$. Let $C_i''$ on $F_i''$ be the birational transform of $\tilde{C}_i$. Since $C_i''$ is a component of 
$N_{F_i''}$, it is a component of $\Delta_{F_i''}$ with coefficient $>1-\epsilon'$ which in 
turn implies it is a component of 
$\Delta_{F_i''}-\Lambda_{F_i''}$ with coefficient $>\epsilon-\epsilon'$.\\   

\emph{Step 8.}
In this final step we get a contradiction by restricting to the general fibres of 
$\tilde{F}_i\to T_i$. Then 
$$
K_{X_i}+\Delta_i+s_i(1-a_i)D_i\sim_\Q 0,
$$  
hence 
$$
K_{F_i''}+\Delta_{F_i''}+s_i(1-a_i)D_{F_i''}\sim_\Q 0
$$ 
which in turn gives 
$$
K_{F_i''}+\Delta_{F_i''}+\frac{1-a_i}{a_i}(\Delta_{F_i''}-\Lambda_{F_i''})\sim_\Q 0
$$
and then 
$$
K_{\tilde{F}_i}+\Delta_{\tilde{F_i}}+\frac{1-a_i}{a_i}(\Delta_{\tilde{F_i}}-\Lambda_{\tilde{F_i}})\sim_\Q 0.
$$
But now $\tilde{C}_i$ is a component of $\frac{1-a_i}{a_i}(\Delta_{\tilde{F_i}}-\Lambda_{\tilde{F_i}})$ 
whose coefficient is at least 
$$
\frac{(1-a_i)(\epsilon-\epsilon')}{a_i}
$$ 
which approaches $\infty$ as 
$i$ grows large. 
Restricting to the general fibres of $\tilde{F}_i\to T_i$ and applying Lemma \ref{l-bnd-fam-intersection} 
gives a contradiction.

\end{proof}


\section{\bf Proofs of main results}

Recall that we proved Theorem \ref{t-BAB-good-boundary} in Section 5 and proved Theorem \ref{t-BAB-to-bnd-vol} 
in Section 9. We prove the other main results by induction so lets assume all the theorems in the introduction hold in dimension $d-1$. 
They can be verified easily in dimension $1$.

\begin{proof}(of Theorem \ref{t-bnd-compl-usual-local})
This follows from Theorems \ref{t-bnd-compl-usual} and \ref{t-bnd-compl-usual-local} in dimension $d-1$, and  
Proposition \ref{p-bnd-compl-usual-local}.

\end{proof}

\begin{proof}(of Corollary \ref{cor-bnd-index})
Shrinking $X$ around the generic point of $V$ we can assume $(X,\Delta)$ is klt. Then $X$ is of Fano type over itself. Thus 
by Theorem \ref{t-bnd-compl-usual-local} in dimension $d$, $K_X+B$ has an $n$-complement 
$K_X+B^+$ near the generic point $v$ of $V$ for some $n$ depending only on $d$ and $\mathfrak{R}$ such that $B^+\ge B$. 
Since $V$ is an lc centre of $(X,B)$, we deduce that $B^+=B$ near $v$ which in particular means $n(K_X+B)$ is Cartier near $v$.  

\end{proof}

\begin{proof}(of Theorem \ref{t-BAB-exc-usual})
This follows from Theorem \ref{t-bnd-compl} in dimension $\le d-1$, Theorem \ref{t-bnd-compl-usual-local} 
in dimension $d$, and Lemma \ref{l-from-compl-to-BAB-exc-usual}. 

\end{proof}

\begin{proof}(of Theorem \ref{t-BAB-exc})
This follows from Theorem \ref{t-bnd-compl} in dimension $\le d-1$, Theorem \ref{t-bnd-compl-usual-local} 
in dimension $d$, and Proposition \ref{p-from-compl-to-BAB-exc}. 

\end{proof}

\begin{proof}(of Theorem \ref{t-bnd-compl})
This follows from Theorems \ref{t-bnd-compl-usual-local} and \ref{t-BAB-exc} in dimension $d$, 
and  Proposition \ref{p-BAB-exc-to-compl}.

\end{proof}

\begin{proof}(of Theorem \ref{t-bnd-compl-usual})
This is a special case of Theorem \ref{t-bnd-compl}.

\end{proof}

\begin{proof}(of Theorem \ref{t-eff-non-van-fano})
This is a consequence of Theorem \ref{t-bnd-compl-usual}.

\end{proof}

\begin{proof}(of Theorem \ref{t-eff-bir-e-lc})
This follows from Theorem \ref{t-eff-non-van-fano} and Proposition \ref{p-eff-bir-delta-2}. 

\end{proof}


\vspace{2cm}

\textsc{DPMMS, Centre for Mathematical Sciences} \endgraf
\textsc{University of Cambridge,} \endgraf
\textsc{Wilberforce Road, Cambridge CB3 0WB, UK} \endgraf
\email{c.birkar@dpmms.cam.ac.uk\\}

\end{document}